\documentclass[11pt, letterpaper]{amsart}
\usepackage[margin=1in]{geometry}
\usepackage{preamble_article_FS}

\title{Fargues--Scholze correspondence and endoscopic classification for special orthogonal and unitary groups}
\author{Hao Peng}

\begin{document}
\maketitle

\begin{abstract}
Let $p$ be odd and let $K/\bb Q_p$ be unramified. For a special orthogonal group or a unitary group \(G\) over $K$ that splits over an unramified extension, we prove that the Fargues--Scholze local Langlands correspondence agrees with the semisimplification of the classical correspondence for $G$ constructed in the work of Arthur and others. As applications, we construct an unambiguous local Langlands correspondence for even special orthogonal groups, deduce the strong Kottwitz conjecture and the eigensheaf conjecture of Fargues, and establish new torsion vanishing results for orthogonal and unitary Shimura varieties.
\end{abstract}
\tableofcontents
\section{Introduction}

For a connected reductive group $G$ over a finite extension $K$ of $\bb Q_p$ for some rational prime $p$, the conjectural local Langlands correspondence is a map from the set $\Pi(G)$ of isomorphism classes of irreducible admissible representations of $G(K)$ to the set $\Phi(G)$ of conjugacy classes of $L$-parameters\footnote{Here and throughout the paper we take the Langlands $L$-group $\LL G$ in the Weil form.}
\begin{equation*}
\phi: W_K\times\SL_2(\bb C)\to \LL G,
\end{equation*}
which should have finite fibers called $L$-packets and satisfy various properties; see~\cite{Bor79}. When $G$ is a special orthogonal group or a unitary group over $K$, such a map is constructed by Arthur \cite{Art13}, Chen--Zou \cite{C-Z21}, and Ishimoto \cite{Ish24} in the special orthogonal case, and by Mok \cite{Mok15} and Kaletha--Minguez--Shin--White \cite{KMSW} in the unitary case. When $G$ is a special orthogonal group associated with a quadratic space of even dimension $2n$, the correspondence is only well-defined up to $\bx O_{2n}(\bb C)$-conjugation. In this case, set
\begin{equation*}
\tilde\Phi(G)\defining\Phi(G)/\bx O_{2n}(\bb C).
\end{equation*}
Also, when $G=\GSpin(V)$ for a quadratic space $V$ over $K$, the correspondence is constructed in \cite[Theorem~2.6.1]{G-T19} for those representations $\pi$ whose central character is the square of another character. These constructions ultimately rely on endoscopy and trace formula techniques. We denote the resulting correspondence by
\begin{equation*}
\rec_G: \Pi(G)\longrightarrow
\begin{cases}
\Phi(G), & \text{outside the even special orthogonal case},\\
\tilde\Phi(G), & \text{in the even special orthogonal case}.
\end{cases}
\end{equation*}

On the other hand, fix a rational prime $\ell\ne p$ and an isomorphism $\iota_\ell:\bb C\xr\sim\ovl{\bb Q_\ell}$. Using excursion operators on the moduli stack of $G$-bundles on the Fargues--Fontaine curve, Fargues and Scholze \cite{F-S24} constructed a candidate for a semisimplified local Langlands correspondence for every connected reductive group $G$, namely,
\begin{equation*}
\rec_G^\FS: \Pi(G)\to \Phi^\sems(G),\qquad \pi\mapsto \iota_\ell^{-1}\phi_{\iota_\ell\pi}^\FS.
\end{equation*}
Here $\Phi^\sems(G)$ denotes the set of conjugacy classes of continuous semisimple homomorphisms
\begin{equation*}
\phi: W_K\to \LL G
\end{equation*}
whose composition with the canonical projection $\LL G\to W_K$ is $\id_{W_K}$. Moreover, $\rec_G^\FS$ satisfies some desired properties listed in \cite[Theorem~1.9.6]{F-S24}. It is known that $\rec_G^\FS$ is independent of the choice of $\ell$; see~\cite{Sch26}.

It is both natural and nontrivial to ask whether $\rec_G$ and $\rec_G^\FS$ are compatible when they both exist, in the sense that there exists a commutative diagram
\begin{equation}\label{ieiemfiems}
\begin{tikzcd}[sep=large]
\Pi(G)\ar[r, "\rec_G"]\ar[rd, "\rec_G^\FS"] & \Phi(G)\ar[d, "(-)^\sems"]\\
&\Phi^\sems(G)
\end{tikzcd},
\end{equation}
where $(-)^\sems$ precomposes a parameter $\phi\in \Phi(G)$ with the map
\begin{equation*}
W_K\to W_K\times \SL_2(\bb C): g\mapsto \paren{g, \begin{bmatrix}\largel{g}_K^{1/2} &0\\ 0 & \largel{g}_K^{-1/2}\end{bmatrix}}.
\end{equation*}
Here $\largel{-}_K$ is defined to be the composition $W_K\to W_K^\ab\xr{\Art_K^{-1}}K^\times\xr{\largel{-}_K}\bb R_+$.

Our main result is the following theorem on compatibility of the Fargues--Scholze local Langlands correspondence with the ``classical local Langlands correspondence'' defined in \cite{Art13}, \cite{Mok15}, \cite{KMSW}, \cite{C-Z21}, \cite{C-Z21a}, and \cite{Ish24}:

\begin{mainthm}\label{compaiteiehdnifds}
Suppose $p>2$ and $K/\bb Q_p$ is unramified.
\begin{enumerate}
\item
If $G=\bx U(V)$ where $V$ is a Hermitian space with respect to the unramified quadratic extension $K_1/K$, then the diagram~\eqref{ieiemfiems} is commutative.
\item
If $G=\SO(V)$ where $V$ is a quadratic space over $K$ with $\dim(V)=2n+1$ for some positive integer $n$, then the diagram~\eqref{ieiemfiems} is commutative.
\item
If $G=\SO(V)$ where $V$ is a quadratic space over $K$ of dimension $2n$ for some positive integer $n$ such that $G$ splits over an unramified quadratic extension of $K$ (equivalently, $\ord_K(\disc(V))\equiv 0\modu2$; see~\textup{\S\ref{theogirneidnis}}), then the diagram~\eqref{ieiemfiems} is commutative up to conjugation by $\bx O(2n, \bb C)$.
\end{enumerate}
\end{mainthm}
\begin{rem*}
We have the following remarks concerning Theorem~\ref{compaiteiehdnifds}.
\begin{enumerate}
\item
When $G$ is an inner form of a general linear group over any finite extension of $\bb Q_p$, Hansen, Kaletha and Weinstein established in \cite[Theorem 6.6.1]{HKW22} that the Fargues--Scholze LLC coincides with the usual (semisimplified) LLC for inner forms of general linear groups given in \cite{DKV84, Rog83}.
\item
When $G$ is an inner form of $\Sp_4$ or $\GSp_4$ over an unramified finite extension of $\bb Q_p$ for $p>2$, the compatibility is established in \cite{Ham22}.
\item
For $G=\bx U(V)$ or $\GU(V)$ where $V$ is an odd-dimensional Hermitian space with respect to the unramified quadratic extension $\bb Q_{p^2}/\bb Q_p$, the compatibility is established in \cite[Theorem~1.1]{MHN24}, and their proof is different from ours. In fact, they established the Kottwitz conjecture first by proving an averaging formula of Shin for $\bx{GU}(V)$, and they restricted to the case $K=\bb Q_p$ because the Hasse principle holds for unitary similitude groups over $\bb Q$.
\item
The assumptions that $p>2$ and that the relevant splitting extension over \(\bb Q_p\) is unramified are necessary in order to apply Shen's result \cite{She20} that the relevant local shtuka spaces uniformize Shimura varieties of Abelian type. These assumptions in Theorem~\ref{compaiteiehdnifds} and its corollaries can be removed for such $G$ once the main result of \cite{She20} is established for a Shimura datum $(\bb G, \{\mu\})$ such that $\bb G\otimes\bb Q_p\cong \Res_{K/\bb Q_p}G$. During the review of this paper, we learned that \cite{DHKZ2} had removed these assumptions, building on computations carried out in the present work.
\end{enumerate}
\end{rem*}

Theorem~\ref{compaiteiehdnifds} is proved in \S\ref{secitonpaoroifnienocosmw}. Moreover, in the third case of Theorem~\ref{compaiteiehdnifds} (i.e., even special orthogonal groups), we use the compatibility property to construct an unambiguous version of the local Langlands correspondence for $G$, eliminating the ambiguity up to outer automorphisms by requiring compatibility with the Fargues--Scholze local Langlands correspondence, which is defined canonically without outer automorphisms. 

\begin{mainthm}\label{ioopwpwjn}
In the third case of \textup{Theorem~\ref{compaiteiehdnifds}}, there exists a map
\begin{equation*}
\rec_G^\natural: \Pi(G)\longrightarrow\Phi(G)
\end{equation*}
lifting the correspondence defined in Arthur \textup{\cite{Art13}} and Chen--Zou \textup{\cite{C-Z21}}. It is compatible with the Fargues--Scholze local Langlands correspondence in the sense that diagram~\eqref{ieiemfiems} is commutative. Moreover, $\rec_G^\natural$ satisfies the relevance, temperedness, discreteness, internal parametrization, Langlands-quotient, standard $\gamma$-factor, Plancherel-measure, and local-intertwining properties stated precisely in \textup{Theorem~\ref{LLCienifheis}}. In particular, Vogan's version of the local Langlands conjecture \textup{\cite{Vog93}} holds for unramified even special orthogonal groups.
\end{mainthm}

Theorem~\ref{ioopwpwjn} is proved in Theorem~\ref{LLCienifheis}.

Using the unambiguous local Langlands correspondence, we verify in \S\ref{naturalllFairngiesHOS} the naturality property of the Fargues--Scholze local Langlands correspondence for those $G$ appearing in Theorem~\ref{compaiteiehdnifds}, thereby confirming \cite[Assumption~6.5]{Ham24}. We also establish a weaker result for a central extension of $\Res_{K/\bb Q_p}G$, which will be used to deduce a torsion vanishing result for suitable Shimura varieties.

We next show that the classical Langlands correspondence, together with geometric techniques, provides sufficient input to verify part of the categorical local Langlands conjecture of Fargues--Scholze \cite[Conjecture \Rmnum{10}.1.4]{F-S24}.

\begin{mainthm}\label{teoreneieifes}
Suppose $G$ is one of the groups appearing in \textup{Theorem~\ref{compaiteiehdnifds}}, $p>2$, $K/\bb Q_p$ is unramified, and $\phi\in\Phi(G)$ is supercuspidal.
\begin{enumerate}
\item
The sheaf
\begin{equation*}
\mcl G_\phi=\bplus_{b\in B(G)_\bas}\bplus_{\pi_b\in\Pi_\phi(G_b)}i_{b!}(\pi_b)\in \bx D_{\bx{lis}}(\Bun_G, \ovl{\bb Q_\ell})
\end{equation*}
admits an action of $\mfk S_\phi\defining Z_{\hat G}(\phi)$ satisfying conditions (i)--(iv) of Fargues' conjecture \textup{\cite[Conjecture 4.4]{Far16}} for $G$. In particular, $\mcl G_\phi$ is a Hecke eigensheaf for $\phi$.
\item
The strong Kottwitz conjecture~\textup{\cite[Conjecture~1.0.1]{HKW22}} holds for $G$ and any conjugacy class of geometric cocharacters $\{\mu\}$ for $G_{\ovl K}$, up to a reparametrization of elements of the $L$-packet of $\phi$ that is independent of the choice of $\{\mu\}$. Moreover, assuming \textup{Hypothesis~\ref{hsisieteijeiureis}}, this reparametrization is trivial.
\end{enumerate}
\end{mainthm}
\begin{rem*}
The strong Kottwitz conjecture was previously shown when $G$ is an unramified odd unitary group over $\bb Q_p$ \cite{MHN24}, and a weaker version is known when $G$ is an inner form of $\GSp_4$ or $G=\Sp_4$ over an unramified extension of $\bb Q_p$ \cite{Ham22}. Our proof uses a new globalization method that connects the cohomology of local shtuka spaces with the endoscopic part of the cohomology of Shimura varieties.

Hypothesis~\ref{hsisieteijeiureis} describes the Galois representation appearing in the generic part of the middle degree $\ell$-adic cohomology of orthogonal and unitary Shimura varieties. It is known for orthogonal Shimura varieties over $\bb Q$ by \cite[Corollary~9.8.10]{Zhu24a}. In particular, Theorem~\ref{teoreneieifes} proves the strong form of the Kottwitz conjecture for special orthogonal groups over $\bb Q_p$ for $p>2$, generalizing \cite[Theorem~1.3]{Ham22}. In general, Hypothesis~\ref{hsisieteijeiureis} will follow from a sequel to \cite{KSZ21} as long as the full endoscopic classification for orthogonal and unitary groups is obtained.
\end{rem*}

Theorem~\ref{teoreneieifes} is proved in Theorems~\ref{enifeinifehniiwms} and~\ref{sinieikIInifhfiens}.

\subsection{Torsion vanishing for special orthogonal and unitary Shimura varieties}

We use the compatibility result to establish new torsion vanishing results for Shimura varieties of orthogonal or unitary type. We now introduce the necessary background and notation. Let $\bb G$ be a connected reductive group over $\bb Q$ with a Shimura datum $(\bb G, \bb X)$, and let $E\subset\bb C$ be the associated reflex field. Fix an odd rational prime $p$ that is coprime to $\#\pi_1([\bb G, \bb G])$, together with an isomorphism $\iota_p: \bb C\to \ovl{\bb Q_p}$ inducing an embedding $E\to \ovl{\bb Q_p}$. We denote by $\msf G$ the base change of $\bb G$ to $\bb Q_p$. We assume that $\msf G$ is unramified and equipped with a Borel pair $(\msf B, \msf T)$ and a hyperspecial subgroup $\mdc K_p$ of $\msf G(\bb Q_p)$. Let $\mdc K^p\le \bb G(\Ade_f^p)$ be a compact open subgroup such that  $\mdc K\defining \mdc K_p\mdc K^p\le \bb G(\Ade_f)$ is neat. Let $\ell$ be a rational prime that is coprime to $p\cdot\#\pi_0(Z(\msf G))$, and let the coefficient field $\Lbd$ be either $\ovl{\bb Q_\ell}$ or $\ovl{\bb F_\ell}$.

\begin{defi*}[{\cite[Definition~6.2]{H-L24}}]
Suppose $\msf G$ is an arbitrary quasisplit reductive group over a finite extension $K/\bb Q_p$ with a Borel pair $(\msf B, \msf T)$, and $\phi_{\msf T}\in\Phi^\sems(\msf T, \Lbd)$ is a semisimple $L$-parameter. Let $\phi_{\msf T}^\vee$ denote the Chevalley dual of $\phi_{\msf T}$. Then $\phi_{\msf T}$ is said to be \tbf{generic} (or of Langlands--Shahidi type) if for every positive coroot $\mu\in \Phi^\vee(\msf G, \msf T)^+\subset X_\bullet(\msf T)$, the following Galois cohomology complexes vanish:
\begin{equation*}
\bx R\Gamma(W_K, \LL\mcl T_\mu\circ\phi_{\msf T}), \quad \bx R\Gamma(W_K, \LL\mcl T_\mu\circ\phi_{\msf T}^\vee).
\end{equation*}
Here $\LL\mcl T_\mu$ denotes the extended highest weight tilting module of $\LL\msf T$ with $\Lbd$-coefficients associated with $\mu$; see \eqref{ieneineifeilsws}.
\end{defi*}

We now state the following torsion-vanishing conjecture for Shimura varieties.

\begin{conj*}[\cite{Car23, H-L24}]\label{conjeuejijrimoens}
Suppose $\phi\in\Phi^\sems(\msf G, \ovl{\bb F_\ell})$ is an unramified, semisimple, toral generic $L$-parameter, corresponding via the Satake isomorphism to a maximal ideal $\mfk m\subset \ovl{\bb F_\ell}[\mdc K_p\bsh\msf G(\bb Q_p)/\mdc K_p]$. Then the complex $\bx R\Gamma_c(\bSh_{\mdc K}(\bb G, \bb X)_{\ovl E}, \ovl{\bb F_\ell})_{\mfk m}$ (resp. $\bx R\Gamma(\bSh_{\mdc K}(\bb G, \bb X)_{\ovl E}, \ovl{\bb F_\ell})_{\mfk m}$) is concentrated in degrees $0\le i\le\dim_{\bb C}(\bb X)$ (resp. $\dim_{\bb C}(\bb X)\le i\le 2\dim_{\bb C}(\bb X)$).
\end{conj*}

This torsion vanishing conjecture has been established in \cite{C-S17, C-S24, Kos21} and \cite{H-L24} in the case where $(\bb G, \bb X)$ is a PEL-type Shimura datum of type $\bx A$ or $\bx C_2$ and $\msf G$ is a product of certain groups related either to $\GL_n$ over an unramified extension of $\bb Q_p$, or $\bx U_{2k+1}$ with respect to $\bb Q_{p^2}/\bb Q_p$, or $\bx U_2$ with respect to  a quadratic extension of unramified extensions of $\bb Q_p$, with $p$ and $\ell$ satisfying certain properties.

In this work, we extend the list of known cases, particularly when $(\bb G, \bb X)$ is an orthogonal or unitary Shimura datum associated with a quadratic or Hermitian space over a totally real number field $F$ with standard indefinite signature, $\msf G$ is among the unramified groups listed in Theorem~\ref{compaiteiehdnifds}, and $\ell$ is sufficiently large.

\begin{rem*}
The above torsion vanishing conjecture was established in~\cite{DHKZ} in the case where $\msf G=\bb G\otimes\bb Q_p$ is split and the Shimura variety is compact of Hodge type, under the hypothesis that the Fargues--Scholze correspondence for $\msf G$ is compatible with the so-called classical local Langlands correspondence; see~\cite[Assumption~6.5]{Ham24}. However, their result does not apply directly to the orthogonal Shimura variety $\bSh(\Res_{F/\bb Q}\SO(\mbf V), \mbf X)$ when the quadratic space $\mbf V$ over a totally real field $F$ has large rank, because it is not of Hodge type. A natural approach is to consider a Hodge type Shimura datum $(\mbf G^\sharp, \mbf X^\sharp)$ with a map of Shimura data $(\mbf G^\sharp, \mbf X^\sharp)\to (\Res_{F/\bb Q}\SO(\mbf V), \mbf X)$ such that the morphism $\mbf G^\sharp\to \Res_{F/\bb Q}\SO(\mbf V)$ is a central extension. However, $\mbf G^\sharp$ has derived subgroup $\Res_{F/\bb Q}\Spin(\mbf V)$, for which the so-called classical local Langlands correspondence has not yet been constructed. In this work, we modify the argument of \cite{DHKZ} and weaken the hypothesis of~\cite[Assumption~6.5]{Ham24} to our Hypothesis~\ref{isnishiefjemis} below. We then establish the torsion vanishing for $\bSh(\mbf G^\sharp, \mbf X^\sharp)$ and use the Hochschild--Serre spectral sequence to deduce the torsion vanishing for $\bSh(\Res_{F/\bb Q}\SO(\mbf V), \mbf X)$. We also treat certain cases where $p$ is not split in $F$.
\end{rem*}

If $\msf G_\ad$ is a product of Weil restrictions of split simple groups, then there exists a simpler proof without using geometric Eisenstein series. In this introduction, we state a more general result, which applies to the above-mentioned orthogonal or unitary case by constructing a central extension in \S\ref{seubsienfiehtoenfeis}. We first state a hypothesis on the Fargues--Scholze local Langlands correspondence: 

\begin{mainhyp}\label{isnishiefjemis}
Suppose $\msf G$ is a quasisplit connected reductive group over a $p$-adic number field $K$ with a Borel pair $(\msf B, \msf T)$ and $\phi\in\Phi^\sems(\msf G, \ovl{\bb Q_\ell})$ is a semisimple generic toral $L$-parameter. For any $\msf b\in B(\msf G)$ and any $\rho\in\Pi(\msf G_{\msf b}, \ovl{\bb Q_\ell})$, if the composition
\begin{equation*}
W_K\xr{\phi_\rho^\FS}\LL\msf G_{\msf b}(\ovl{\bb Q_\ell})\longrightarrow\LL\msf G(\ovl{\bb Q_\ell}),
\end{equation*}
where the second arrow is the twisted embedding defined in \textup{\cite[\S\Rmnum{9}.7.1]{F-S24}}, equals $\phi$, then $\msf b$ is unramified.
\end{mainhyp}

By our main theorem and \cite[Lemma~3.14]{Ham24}, this hypothesis holds for those groups $G/K$ appearing in Theorem~\ref{compaiteiehdnifds} that are quasisplit.

We now state our main theorem on torsion vanishing for certain Shimura varieties of Abelian type. This theorem is proved in Theorem~\ref{ieifmeimfeos}.

\begin{mainthm}
Suppose that the following assumptions hold:
\begin{enumerate}
\item
$\bSh_{\mdc K}(\bb G, \bb X)$ is proper, and there exists a Shimura datum of Hodge type $(\bb G^\sharp, \bb X^\sharp)$, and a morphism of Shimura data $(\bb G^\sharp, \bb X^\sharp)\to (\bb G, \bb X)$ such that $\bb G^\sharp_\ad\to \bb G_\ad$ is an isomorphism.
\item
There is an isomorphism
\begin{equation*}
\msf G_\ad\cong
\prod_{i=1}^k\Res_{L_i/\bb Q_p}\msf H_i,
\end{equation*}
where each $\msf H_i$ is a split simple group over $L_i$. Via $\iota_p$, the conjugacy class of Hodge cocharacters associated with $\bb X^\sharp$ induces a dominant cocharacter $\mu_\ad=(\mu_1,\ldots,\mu_k)$ of $\msf G_{\ad,\ovl{\bb Q_p}}$. Under the canonical decomposition
\begin{equation*}
\paren{\Res_{L_i/\bb Q_p}\msf H_i}_{\ovl{\bb Q_p}}\cong\prod_{\tau\in\Hom_{\bb Q_p}(L_i,\ovl{\bb Q_p})}\paren{\msf H_i\otimes_{L_i,\tau}\ovl{\bb Q_p}},
\end{equation*}
each $\mu_i$ is trivial on all but at most one simple factor.
\item
$\ell$ is a rational prime that is coprime to $p\cdot\#\pi_0(Z(\msf G^\sharp))\cdot\#\pi_0(Z(\msf G))$, and $\mfk m$ is a maximal ideal of the $\ell$-torsion Hecke algebra $\mcl H_{\mdc K_p}\defining \ovl{\bb F_\ell}[\mdc K_p\bsh \msf G(\bb Q_p)/\mdc K_p]$.
\end{enumerate}
If the semisimple toral $L$-parameter $\phi_{\mfk m}$ corresponding to $\mfk m$ is generic and \textup{Hypothesis~\ref{isnishiefjemis}} holds for every semisimple generic toral lift of $\phi_{\mfk m}$, then $\etH^i(\bSh_{\mdc K}(\bb G, \bb X)_{\ovl E}, \ovl{\bb F_\ell})_{\mfk m}$ vanishes unless $i=\dim_{\bb C}(\bb X)$.
\end{mainthm}
\begin{rem*}
The hypothesis that $\bSh_{\mdc K}(\bb G, \bb X)$ is proper is expected to be unnecessary, once one has constructed the minimally compactified Igusa stack for $\bSh(\bb G^\sharp, \bb X^\sharp)$ in the sense of \cite{Zha23} and compared the fibers of the Hodge--Tate map on it with the minimally compactified Igusa varieties. This has been done when $\bSh(\bb G^\sharp, \bb X^\sharp)$ is of PEL type A or C in \cite{H-L24}.
\end{rem*}

In particular, this theorem generalizes previous results of \cite{C-S17, Kos21, C-S24, H-L24} to compact orthogonal and unitary Shimura varieties. This is because we may construct a central extension $\mbf G^\sharp$ of $\mbf G\defining \Res_{F/\bb Q}\bx U(\mbf V)^\circ$ with a morphism of Shimura data $(\mbf G^\sharp, \mbf X^\sharp)\to (\mbf G, \mbf X)$, so that $(\mbf G^\sharp, \mbf X^\sharp)$ defines a Shimura datum of Hodge type. In the unitary case, we take
\begin{equation*}
\mbf G^\sharp=\mbf G\times\mbf Z^{\bb Q}.
\end{equation*}
Here
\begin{equation*}
\mbf Z^{\bb Q}=\{z\in \Res_{F_1/\bb Q}\GL_1: \Nm_{F_1/F}(z)\in\bb Q^\times\}
\end{equation*}
and $F_1/F$ is the CM-extension associated with the Hermitian space $\mbf V$. The desired map of Shimura data $(\mbf G^\sharp, \mbf X^\sharp)\to (\mbf G, \mbf X)$ is constructed by Rapoport, Smithling, and Zhang in \cite{RSZ20}. In the orthogonal case, following Carayol~\cite[p.~163]{Car86}, we construct, for each $\daleth\in\bb R_+\ii$ such that $\bb Q(\daleth)$ is an imaginary quadratic field, a group $\mbf G^\sharp$ fitting into an exact sequence
\begin{equation*}
1\to\mbf Z^{\bb Q}\to \mbf G^\sharp\to \mbf G\to 1,
\end{equation*}
where
\begin{equation*}
\mbf Z^{\bb Q}=\{z\in \Res_{F(\daleth)/\bb Q}\GL_1: \Nm_{F(\daleth)/F}(z)\in\bb Q^\times\}.
\end{equation*}
If moreover we assume that $p$ is unramified in $F$ and $\bb Q(\daleth)/\bb Q$ is split at $p$, then there exists an isomorphism
\begin{equation*}
\mbf G^\sharp\otimes\bb Q_p\cong \GL_1\times \Res_{F\otimes\bb Q_p/\bb Q_p}\GSpin(\mbf V\otimes\bb Q_p).
\end{equation*}

In fact, proving this torsion vanishing result for orthogonal Shimura varieties is one of the main motivations for this paper. In Euler system arguments via ``level-raising congruences'' in higher-dimensional Shimura varieties, as pioneered by Bertolini and Darmon \cite{B-D05} for Shimura curves, we need to construct elements in the cohomology of Shimura varieties via the Jacquet--Langlands correspondence. The natural way to do so is to take the Abel--Jacobi map of a globally defined cycle that is cohomologically trivial. The point is that a certain Hecke translate of a global cycle class becomes cohomologically trivial if the target cohomology group is itself trivial. For example, this strategy was applied in \cite{Liu16, Liu19, L-T20, LTXZZ} and will be used in the author's upcoming work~\cite{Pen26} on higher-dimensional analogues of Kolyvagin-type theorems for product of Shimura varieties of orthogonal type in the arithmetic Gross--Prasad setting. For further applications of torsion vanishing results, we refer the reader to Caraiani's ICM report~\cite{Car23}.

\subsection{An overview of the proof}

We summarize the proof of Theorem~\ref{compaiteiehdnifds}, adapting the method of Hamann~\cite{Ham22}. We exclude the case of even orthogonal groups in this introduction, as the additional outer automorphism complicates the notation. The proof proceeds by induction on the geometric rank of $G$, with low-rank cases verified by direct inspection. For higher ranks, if $\pi\in\Pi(G)$ is non-supercuspidal, we invoke the induction hypothesis along with the compatibility of $\rec_G$ (resp. $\rec_G^\FS$) with parabolic induction, since any proper Levi subgroup of $G$ is a product of (Weil restrictions of) general linear groups and a group of the same type as $G$, but with smaller geometric rank. We may therefore assume that $\pi$ is supercuspidal. We establish the compatibility for pure inner forms of $G$ simultaneously. If $\phi$ is the classical $L$-parameter of $\pi$ and $\Pi_\phi(G^*)$ contains a non-supercuspidal representation $\rho_{\bx{nsc}}$, then the compatibility is already known for $\rho_{\bx{nsc}}$. We then propagate this property to other representations in the $L$-packets of pure inner forms of $G$ with classical parameter $\phi$. The crucial input is a description of the cohomology of the local shtuka spaces $\Sht_{G,b,\{\mu\}}$ defined in \cite{S-W20}, where $\{\mu\}$ is a geometric conjugacy class of $G$ related to the Hodge cocharacter of a suitable global Shimura variety of orthogonal or unitary type, and $b\in B(G, \{\mu\})$ is the unique nontrivial basic element. This local shtuka space carries an action of $G_b(K)\times G(K)\times W_{E_{\{\mu\}}}$, where $E_{\{\mu\}}/K$ is the reflex field of $\{\mu\}$. For any $\rho\in \Pi_\phi(G_b)$, the complex $\bx R\Gamma^\flat_c(G, b, \{\mu\})[\rho]$ is isomorphic to the image of \(\rho\) under the relevant Hecke operator. Since Hecke operators and excursion operators commute, it follows that any representation of $G(K)$ occurring in $\bx R\Gamma^\flat_c(G, b, \{\mu\})[\rho]$ has Fargues--Scholze parameter equal to that of $\rho$.

To analyze which representations of $G$ appear, we use the weak Kottwitz conjecture established by Hansen, Kaletha and Weinstein \cite{HKW22}.\footnote{We remark that when $G$ is an even special orthogonal group, the hypothesis in \cite{HKW22} remains unproven, since only a version of the local Langlands correspondence up to conjugacy by the full orthogonal group is available. Instead, we use the weak endoscopic character identity established in \cite{Pen25b} and modify the arguments therein to establish a weaker version of the weak Kottwitz conjecture, Theorem~\ref{coalidhbinifw222}, valid up to conjugacy by the full orthogonal group. This weaker version is enough for the argument to work.} In fact, $\pi$ does not necessarily appear in the complex $\bx R\Gamma^\flat_c(G, b, \{\mu\})[\rho_{\bx{nsc}}]$, but we may iterate this process, replacing $\rho_{\bx{nsc}}$ by those representations that appear, until $\pi$ eventually appears. This ultimately depends on a detailed analysis of the combinatorics of the centralizer of the $L$-parameter $\phi$ in $\hat G$, see~\S\ref{combianosiLoifn}.

We are now left with the case in which $\Pi_\phi(G^*)$ consists entirely of supercuspidal representations. A result of M{\oe}glin and Tadi\'c, recalled in Proposition~\ref{chainsienicusonso}, then implies that $\phi$ is supercuspidal; in particular, it is discrete and trivial on the $\SL_2(\bb C)$-factor. Since Hecke operators commute with the excursion algebra, we may reduce to the quasisplit case, and it suffices to prove compatibility for each $\rho\in\Pi_\phi(G_b)$; see \cite[Lemma~3.15]{Ham22}.

Let $\hat\Std$ denote the standard representation of $\hat{G}\rtimes\Gal(K'/K_1)$, where $K'$ is the splitting field of $G$. Because $\phi$ is discrete, we may write
\begin{equation*}
\hat\Std\circ\phi|_{W_{K_1}}=\phi_1+\cdots+\phi_r,
\end{equation*}
where the $\phi_i$ are pairwise nonisomorphic irreducible representations, each occurring with multiplicity one. Hamann~\cite{Ham22} and Koshikawa~\cite{Kos21} showed that if every $\phi_i$ occurs as a subquotient of
\begin{equation*}
\bplus_{\rho'\in\Pi_\phi(G_b)}\bx R\Gamma^\flat_c(G, b, \{\mu\})[\rho']
\end{equation*}
as a $W_{K_1}$-representation, then every $\phi_i$ occurs in $\hat\Std\circ\phi^\FS_\rho|_{W_{K_1}}$. The latter representation is semisimple and has dimension $\sum_{i=1}^r\dim(\phi_i)$. Thus the containment of all the pairwise nonisomorphic constituents forces
\begin{equation*}
\hat\Std\circ\phi_\rho^\FS|_{W_{K_1}}\cong\hat\Std\circ\phi|_{W_{K_1}}.
\end{equation*}
The results of \cite{GGP12} then imply
$\phi_\rho^\FS=\phi$.

Thus it suffices to prove that
\begin{equation*}
\bplus_{\rho'\in\Pi_\phi(G_b)}\bx R\Gamma^\flat_c(G, b, \{\mu\})[\rho']
\end{equation*}
admits a subquotient isomorphic to $\hat\Std\circ\phi$ as a $W_{K_1}$-representation. This is where global inputs become necessary, i.e., by relating this complex to the cohomology of a relevant global Shimura variety. The cohomology of the relevant global Shimura variety is studied via the Langlands--Kottwitz method and related to automorphic forms, whose local components are governed by classical $L$-parameters via Arthur's multiplicity formula. To be more precise, we elaborate on the case when $G$ is an odd special orthogonal group as an example. According to a result of Shen~\cite{She20}, the local shtuka space uniformizes the basic Newton stratum of the generic fiber of a relevant Shimura variety, as defined in \cite{C-S17}. A relevant Shimura variety is given by $(\mbf G, \mbf X)$, where $\mbf G$ is a standard indefinite special orthogonal group over a totally real field $F$ with $p$ inert and $F_p\cong K$, such that  $\mbf G\otimes_FK\cong G$. Let $\Ade_F$ and $\Ade_{F, f}$ denote the ring of adeles and finite adeles of $F$, respectively, and let $\mdc K_p\le \mbf G(K)$ and $\mdc K^p\le \mbf G(\Ade_{F,f}^p)$ be sufficiently small level subgroups, thereby yielding the adic Shimura variety $\mcl S_{\mdc K_p\mdc K^p}(\Res_{F/\bb Q}\mbf G, \mbf X)$ defined over $\bb C_p$. For an algebraic representation $\xi$ of $\paren{\Res_{F/\bb Q}\mbf G}_{\bb C}$ with sufficiently regular highest weight, let $\mcl L_{\iota_\ell\xi}$ be the associated $\ovl{\bb Q_\ell}$-local system on $\mcl S_{\mdc K_p\mdc K^p}(\Res_{F/\bb Q}\mbf G, \mbf X)$, and consider the cohomology
\begin{equation*}
\bx R\Gamma_c(\mcl S_{\mdc K_p\mdc K^p}(\Res_{F/\bb Q}\mbf G, \mbf X), \mcl L_{\iota_\ell\xi}).
\end{equation*}

The basic uniformization result of Shen implies a $G(K)\times W_K$-equivariant map
\begin{align*}
\Theta: \bx R\Gamma_c(G, b, \uno, \{\mu\})\otimes\iota_\ell\largel{-}_{K_1}^{\frac{-\dim_{\bb C}(\mbf X)}{2}}[\dim_{\bb C}(\mbf X)]&\Ltimes_{G^*_b(K)}\mcl A\paren{\mbf G'(F)\bsh \mbf G'(\Ade_{F,f})/\mdc K^p, \mrs L_{\iota_\ell\xi}}\\
&\longrightarrow \bx R\Gamma_c\paren{\mcl S(\Res_{F/\bb Q}\mbf G, \mbf X)_{\mdc K^p}, \mcl L_{\iota_\ell\xi}},
\end{align*}
where $\mbf G'$ is an inner form of $\mbf G$ satisfying $\mbf G'\otimes_FK\cong G^*_b$ and such that $\mbf G'(F\otimes\bb R)$ is compact. Here
\begin{equation*}
\mcl A(\mbf G'(F)\bsh \mbf G'(\Ade_{F, f})/\mdc K^p, \mrs L_{\iota_\ell\xi})
\end{equation*}
denotes the space of $\mdc K^p$-invariant algebraic automorphic forms valued in $\iota_\ell\xi$. Next, we note that the pair $(\Res_{F/\bb Q}\mbf G, \{\mu_\Hdg\})$ is fully Hodge-Newton reducible as defined in \cite{GHN19}, where $\mu_\Hdg$ is the Hodge cocharacter associated with $\mbf X$. This implies that the flag variety $\Gr_{\Res_{K/\bb Q_p}G, \{\mu_\Hdg\}}$ is parabolically induced as a $G(K)$-space. This is called ``Boyer's trick''. Using the Hodge--Tate period map
\begin{equation*}
\pi_{\bx{HT}}: \mcl S_{\mdc K^p}(\mbf G, \mbf X)\to \Gr_{\Res_{K/\bb Q_p}G, \{\mu_\Hdg\}},
\end{equation*}
defined in~\cite{C-S17}, when we restrict to the summands on both sides of $\Theta$ where $G(K)$ acts by a supercuspidal representation, we obtain a $G(K)\times W_K$-equivariant isomorphism, which is also functorial with respect to $\mdc K^p$.

We globalize the given $\rho\in \Pi(G_b)$ to a cuspidal automorphic representation $\Pi'$ of $\mbf G'$, such that the $\mdc K^p$-fixed subspace of $\Pi'$ occurs as a $G_b(K)$-stable direct summand of
\begin{equation*}
\mcl A\paren{\mbf G'(F)\bsh \mbf G'(\Ade_{F,f})/\mdc K^p, \mrs L_{\iota_\ell\xi}}
\end{equation*}
for some $\xi$ with sufficiently regular highest weight, where
\begin{itemize}
\item
$\Pi'$ is an unramified twist of the Steinberg representation at some nonempty subset $\Pla^\St$ of places of $F$,
\item
$\Pi'$ is supercuspidal at some nonempty subset $\Pla^{\bx{sc}}$ of finite places of $F$ disjoint with $\Pla^\St$, and there exists $v\in \Pla^{\bx{sc}}$ such that $\Pi'_v$ has a simple supercuspidal $L$-parameter $\phi_v$, meaning that $\hat\Std\circ \phi_v$ is irreducible as a representation of $W_{F_v}$.
\item
$\Pi'$ is unramified outside some nonempty subset $\Pla$ of places of $F$ containing $\Pla^{\bx{sc}}\cup\Pla^\St\cup\infPla_F$, and $\mdc K^p$ decomposes as $\mdc K^p=\mdc K_{\Pla\setm\{p\}}\mdc K^\Pla$.
\end{itemize}

These conditions ensure that Arthur's multiplicity formula can be applied to analyze the cuspidal automorphic representations $\dot\Pi'$ in the near equivalence class of $\Pi'$. In particular, for each such $\dot\Pi'$, $\dot\Pi'_p$ has classical $L$-parameter $\phi$. If we consider the maximal ideal $\mfk m\subset \bb T^\Pla$ corresponding to $(\Pi^\prime)^\Pla$, where $\bb T^\Pla$ is the Hecke algebra of $\mbf G'$ away from $\Pla$, then $\mfk m$ is non-Eisenstein in the usual sense. Moreover, after localizing at $\mfk m$ and restricting to the summand on which $G(K)$ acts via a supercuspidal representation, we obtain an isomorphism
\begin{align*}
\Theta_{\mfk m, \bx{sc}}: \bx R\Gamma_c(G, b, \uno, \{\mu\})_{\bx{sc}}\otimes\iota_\ell\largel{-}_{K_1}^{\frac{-\dim_{\bb C}(\mbf X)}{2}}[\dim_{\bb C}(\mbf X)]&\Ltimes_{J(K)}\mcl A\paren{\mbf G'(F)\bsh \mbf G'(\Ade_{F,f})/\mdc K^p, \mrs L_{\iota_\ell\xi}}_{\mfk m}\\
&\xr\sim \bx R\Gamma_c\paren{\mcl S(\Res_{F/\bb Q}\mbf G, \mbf X)_{\mdc K^p}, \mcl L_{\iota_\ell\xi}}_{\mfk m}.
\end{align*}
The assertion then follows if we can prove that the right-hand side is concentrated in the middle degree $\dim_{\bb C}(\mbf X)$ and carries a $W_{K_1}$-action given by
\begin{equation*}
\hat\Std\circ\phi\otimes\iota_\ell\largel{-}_{K_1}^{\frac{-\dim_{\bb C}(\mbf X)}{2}}.
\end{equation*}
To prove this, we apply the Langlands--Kottwitz method in \S\ref{Lanlgan-Kotiemethiems} to compute traces of sufficiently large powers of Frobenius at all but finitely many finite places of $F$. Information at the place $p$ is then obtained from local--global compatibility for the Galois representation attached to the functorial transfer of $\dot\Pi'$, which is a self-dual cuspidal automorphic representation of $\GL_{2\rk(\mbf G_{\ovl{\bb Q}})}(\Ade_F)$.

\begin{rem*}
A natural question is whether the same method can be applied to prove compatibility for other reductive groups, for example $\GSpin_n, \GSp_{2n}, \Sp_{2n}$ and $\bx G_2$. For inner forms of $\GSp_4$ and $\Sp_4$, this is known by \cite{Ham22}. For $\GSpin_n$, it is possible to extend the method to prove compatibility of Fargues--Scholze's construction with the correspondence constructed by M{\oe}glin \cite{Moe14} in the quasisplit case, once the endoscopic character identities, as formulated in~\cite{Kal16a}, are proved for all of their inner twists. On the other hand, new ideas are needed to treat the cases of $\GSp_{2n}$ ($n\ge 3$) and $\bx G_2$, because a crucial step of the proof is to use the compatibility between local and global Shimura varieties to connect the construction of Fargues and Scholze with the so-called classical local Langlands correspondence through the cohomology of global Shimura varieties. The latter is studied via the Langlands--Kottwitz method, which can only give information about the (conjectural) global Galois representation $\rho: \Gal_F\to \LL G$ associated with cohomological automorphic forms after composition with the extended highest weight module $\LL\mcl T_{\{\mu\}}$ of $\LL G$, where $\{\mu\}$ is the conjugacy class of Hodge cocharacters of the Shimura datum. However, $\bx G_2$ admits no Shimura variety, and in the case of $G=\GSp_{2n}$, the extended highest weight module $\mcl T_{\{\mu\}}$ is the spin representation of $\GSpin_{2n+1}(\bb C)$, so it is hard to recover the Galois representation and its local components. See~\cite{K-S23} for related study of the cohomology of Siegel modular varieties.
\end{rem*}

In \S\S\ref{theogirneidnis}--\ref{theoenifniehifensis}, we review the classical local Langlands correspondence for special orthogonal and unitary groups and the statement of the endoscopic character identities. In \S\S\ref{compatibliseniniefs}--\ref{combianosiLoifn}, we analyze more properties of the local Langlands correspondence. In \S\S\ref{icoroenfiehsnw}--\ref{snihsnfifies}, we review the Fargues--Scholze local Langlands correspondence and the spectral action, and recall the related objects. In \S\ref{Kottwisniconeidinfs}, we prove a weaker version of the Kottwitz conjecture. In \S\S\ref{seubsienfiehtoenfeis}--\ref{coniierueiefiuhifhens}, we review the endoscopic classification of automorphic representations of relevant groups, and define a class of cohomological cuspidal automorphic representations with local constraints. In \S\ref{Lanlgan-Kotiemethiems}, we apply the Langlands--Kottwitz method to compute the Galois cohomology of relevant global Shimura varieties. In \S\ref{local-LGoabsiSHimief}, we apply basic uniformization and Boyer's trick to prove a key property of the cohomology of relevant local Shimura varieties, see Corollary~\ref{ckeyeeriefiensiw}. In \S\ref{secitonpaoroifnienocosmw}, we combine the previous results to prove the compatibility theorem, Theorem~\ref{compaiteiehdnifds}. In \S\ref{unambinloclianisske}, we use the compatibility property to construct an unambiguous local Langlands correspondence for even orthogonal groups. In \S\ref{naturalllFairngiesHOS}, we prove the naturality property of the Fargues--Scholze local Langlands correspondence. In \S\ref{Kotinitbeinss}, we prove Theorem~\ref{teoreneieifes} by combining the compatibility result with the spectral actions. In \S\ref{generisoemisiejfmes}, we study certain properties of generic toral $L$-parameters. In \S\ref{peomeofeniemfiesiws}, we use the naturality of Fargues--Scholze local Langlands correspondence to prove the torsion vanishing result for Shimura varieties of orthogonal or unitary type. In \S\ref{oniinssmows}, we review the endoscopy theory used in the main body.

\subsection{Notation and conventions}

We fix the following general notation.

\begin{note}\enskip
\begin{itemize}
\item
Let $\bb Z_+$ denote the set of positive integers and $\bb N$ denote the set of non-negative integers.
\item
For each $n\in\bb Z_+$, we define $[n]_+\defining\{1, 2, \ldots, n\}$. For each $n\in \bb N$, we define $[n]\defining\{0,1, \ldots, n\}$.
\item
For each $n\in \bb Z_+$, let $\Sym_n$ denote the symmetric group on $n$ letters.
\item
Suppose $X$ is a set.
\begin{itemize}
\item
Let $\#X$ denote the cardinality of $X$ and let $\mrs P(X)$ denote the power set of $X$.
\item
If $X$ has a distinguished trivial element, we denote it by $\uno$.
\item
For two elements $a, b$ in a set $X$, we define the Kronecker symbol
\begin{equation*}
\delta_{a, b}\defining\begin{cases} 1 &\If a=b\\ 0 &\If a\ne b\end{cases}.
\end{equation*}
\end{itemize}
\item
Let $\bb Q\subset \bb R\subset \bb C$ denote the fields of rational, real, and complex numbers, respectively. We fix a choice of square root $\ii$ of $-1$ in $\bb C$.
\item
When $A$ is a (topological/algebraic) group, we write $B\le A$ to mean that $B$ is a (closed) subgroup of $A$.
\item
For a finite group $A$, let $\Irr(A)$ denote the set of isomorphism classes of irreducible complex representations of $A$.
\item
All rings are assumed to be commutative and unital, and ring homomorphisms preserve units. Algebras, however, may be non-commutative and non-unital.
\item
The transpose of a matrix $M$ is denoted by $M^\top$. When $M$ is invertible, we write $M^{-\top}$ for $(M^{-1})^\top$.
\item
Let $J_n=(a_{ij})$ denote the anti-diagonal $n\times n$ matrix such that $a_{i, j}=\delta_{i, n+1-j}$ and $J_n'=(b_{ij})$ denote the anti-diagonal $n\times n$ matrix such that  $b_{i, j}=(-1)^{i+1}\delta_{i, n+1-j}$.
\item
If $S$ is a scheme over a commutative ring $R$ and $R'$ is a ring over $R$, we define $S_{R'}\defining S\times_{\Spec R}\Spec R'$.
\item
For a locally algebraic group scheme $G$ over a field $K$, let $Z(G)$ denote the center of $G$ and $G^\circ$ denote the identity component of $G$.
\item
Reductive groups are assumed to be connected.
\item
For a reductive group $G$ over a field $K$, let $W_G$ denote the relative Weyl group and $G^*$ denote the unique quasisplit inner form of $G$. A Borel pair for $G^*$ is defined to be a pair $(B^*, T^*)$ consisting of a Borel subgroup $B^*$ and a maximal torus $T^*$ contained in $B^*$.
\end{itemize}
\end{note}

We fix the following notation for a connected reductive group over a non-Archimedean local field of characteristic zero.

\begin{note}
Suppose $K/\bb Q_p$ is a finite extension and $\msf G$ is a connected reductive group over $K$.
\begin{itemize}
\item
Let $\kappa$ denote the residue field of $K$ with a fixed algebraic closure $\ovl\kappa$, and we fix a uniformizer $\varpi_K\in K^\times$.
\item
Let $\ord_K: K^\times\to \bb Z$ denote the additive valuation map that sends a uniformizer $\varpi_K$ to 1, and let $\largel{-}_K: K^\times\to\bb R_+$ denote the multiplicative valuation map such that  $\largel{x}_K=(\#\kappa)^{-\ord_K(x)}$.
\item
We fix an algebraic closure $\ovl K$ of $K$, and for each subfield $K'\subset\ovl K$, we define $\Gal_{K'}\defining\Gal(\ovl K/K')$.
\item
Denote by $W_K$ the Weil group of $K$ and by $I_K$ the inertia group of $K$. Let $\Art_K: K^\times\to W_K^\ab$ denote the Artin map. Fix an arithmetic Frobenius element $\vp_K\in W_K$. Set $\sigma_K\defining\vp_K^{-1}$, and we use the same symbol $\largel{-}_K$ to denote the composition $W_K\to W_K^\ab\xr{\Art_K^{-1}}K^\times\xr{\largel{-}_K}\bb R_+$.
\item
Let $\breve K$ denote the completion of the maximal unramified extension of $K$.
\item
We use the geometric normalization of the local class field theory, i.e., Artin maps are normalized so that they map uniformizers to geometric Frobenius classes.
\item
Let $\hat{\msf G}$ denote the Langlands dual group of $\msf G$, which is a Chevalley group whose based root data is dual to that of $\msf G$. It is equipped with an action of $\Gal_K$. Denote by $\LL{\msf G}\defining \hat{\msf G}\rtimes W_K$ the Langlands $L$-group of $\msf G$ in the Weil form. We usually conflate $\LL\msf G$ (respectively, $\hat{\msf G}$) with their $\bb C$-valued points, unless we write $\LL\msf G(\Lbd)$ (respectively, $\hat{\msf G}(\Lbd)$), which denotes its $\Lbd$-valued points for some ring $\Lbd$.
\item
Let $\mcl H(\msf G)$ denote the set of compactly supported locally constant $\bb C$-valued functions on $\msf G(K)$ that are bi-$\mdc K$-finite for some compact open subgroup $\mdc K\le \msf G(K)$.
\item
If $\msf G, \msf G'$ are reductive groups over $K$ and $\pi, \pi'$ are irreducible admissible representations of $\msf G(K)$ and $\msf G'(K)$, respectively, let $\pi\boxtimes\pi'$ denote the irreducible admissible representation of $\msf G(K)\times\msf G'(K)$ such that  $(\pi\boxtimes\pi')((g, g'))=\pi(g)\otimes \pi'(g')$.
\item
Let $\Pi(\msf G)$ denote the set of isomorphism classes of irreducible admissible representations of $\msf G(K)$, and let $\Pi_\temp(\msf G)$ (resp. $\Pi_2(\msf G)$, resp. $\Pi_{\bx{sc}}(\msf G)$) denote the subset of $\Pi(\msf G)$ consisting of tempered (resp. essentially square-integrable, resp. supercuspidal) representations. Set $\Pi_{2, \temp}(\msf G)\defining \Pi_2(\msf G)\cap\Pi_\temp(\msf G)$.
\item
If $\msf P\le\msf G$ is a parabolic subgroup with a Levi factor $\msf M$ and $\sigma\in \Pi(\msf M)$, we let $\delta_{\msf P}: \msf P(K)\to\bb R_+$ denote the modulus quasi-character. We define the normalized parabolic induction by
\begin{equation*}
\bx I_{\msf P}^{\msf G}(\sigma)\defining\Ind_{\msf P(K)}^{\msf G(K)}\paren{\delta_{\msf P}^{1/2}\otimes\sigma}.
\end{equation*}
\item
If $\msf G=\GL_n$ is a general linear group, we define a character $\nu\defining\largel{-}_K\circ \det: \msf G(K)\to\bb R_+$.
\item
Suppose $\ell$ is a rational prime different from $p$ and $\Lbd\in\{\ovl{\bb Q_\ell}, \ovl{\bb F_\ell}\}$. Let $\bx D(\msf G, \Lbd)$ denote the derived category of smooth representations of $\msf G(K)$ with coefficients in $\Lbd$, equipped with the natural $t$-structure. Let $\bx D^\adm(\msf G, \Lbd)$ denote the full subcategory of admissible complexes, i.e., those complexes whose invariants under any compact open subgroup $\mdc K\le \msf G(K)$ form a perfect complex.
\item
Suppose $\ell$ is a rational prime different from $p$ and $\Lbd\in\{\ovl{\bb Q_\ell}, \ovl{\bb F_\ell}\}$. For each conjugacy class of cocharacters $\{\mu\}$ for $\msf G_{\ovl K}$, there exists an indecomposable highest weight tilting module $\mcl T_{\{\mu\}}\in \bx{Rep}_\Lbd(\hat{\msf G})$ as defined in \cite{Rin91, Don93}; cf.~\cite[\S 9.1]{Ham24}.
\item
We define
\begin{equation*}
X_*(\msf G)\defining\Hom_K(\GL_{1, K},\msf G), \quad X_\bullet(\msf G)\defining\Hom_{\ovl K}(\GL_{1, \ovl K},\msf G_{\ovl K})
\end{equation*}
for the set of cocharacters and geometric cocharacters of $\msf G$, respectively, and define
\begin{equation*}
X^*(\msf G)\defining\Hom_K(\msf G, \GL_{1, K}), \quad X^\bullet(\msf G)\defining\Hom_{\ovl K}(\msf G_{\ovl K}, \GL_{1, \ovl K})
\end{equation*}
for the set of characters and geometric characters of $\msf G$, respectively.
\item
For any condensed $\infty$-category $\msc C$ and any finite index set $I$, let $\msc C^{\bx BW_K^I}$ denote the category of objects with continuous $W_K^I$-actions, as defined in \cite[\S \Rmnum 9.1]{F-S24}.
\item
For any subfield $\kappa'\subset \ovl\kappa$, let $\Perfd_{\kappa'}$ denote the category of affinoid perfectoid spaces over $\kappa'$.
\item
The six functor formalism of \cite{Sch22} and \cite{F-S24} on $\ell$-adic cohomology of diamonds and small Artin $v$-stacks is freely used. In particular, if $\ell$ is a rational prime different from $p$ and $\Lbd\in\{\ovl{\bb Q_\ell}, \ovl{\bb F_\ell}\}$, then for any small Artin $v$-stack $X$, let $\bx D_\blacksquare(X, \Lbd)$ denote the condensed $\infty$-category of solid $\Lbd$-sheaves on $X$ \cite[\S \Rmnum 7.1]{F-S24}, and let $\bx D_{\bx{lis}}(X, \Lbd)\subset \bx D_\blacksquare(X, \Lbd)$ denote the full subcategory of $\Lbd$-lisse-\eTale sheaves as defined in \cite[\S\Rmnum 7.6]{F-S24}.
\end{itemize}
\end{note}

We fix the following notation for a connected reductive group over a number field $F$.

\begin{note}\label{notaitnosidihe}
Suppose $F$ is a number field with a fixed embedding $\tau_0: F\to \bb C$ and $\bb G$ is an arbitrary connected reductive group over $F$.
\begin{itemize}
\item
We denote by $\fPla_F$ the set of finite places of $F$, by $\infPla_F$ the set of infinite places of $F$. Set $\Pla_F\defining \fPla_F\cup\infPla_F$.
\item
For each finite set $S$ of rational primes, let $\Pla_F(S)\subset \fPla_F$ denote the subset of all finite places of $F$ with residue characteristic in $S$.
\item
Let $\ovl F$ denote the algebraic closure of $F$ in $\bb C$.
\item
For each finite place $v$ of $F$, let $\kappa_v$ denote the residue field of $F_v$. Fix a decomposition group $D_v\subset\Gal_F$ and choose an element $\sigma_v\in D_v$ whose image in $\Gal(\ovl\kappa_v/\kappa_v)$ is geometric Frobenius. We also set $\norml{v}\defining\#\kappa_v$.
\item
Let $\Ade_F$ denote the ring of adeles of $F$, and let $\Ade_{F, f}$ denote the ring of finite adeles of $F$. We also write $\Ade\defining\Ade_{\bb Q}$ and $\Ade_f\defining\Ade_{\bb Q, f}$.
\item
If $\Pla\subset \fPla_F$ is a finite subset, set $\Ade_{F, f}^\Pla\defining \prod'_{v\in\fPla_F\setm\Pla}F_v$.
\item
For each discrete automorphic representation $\Pi$ of $\bb G(\Ade_F)$, let $m(\Pi)$ denote its multiplicity in the discrete automorphic spectrum of $\bb G$.
\end{itemize}
\end{note}

\section{Local Langlands correspondence via endoscopy}\label{ndinlocalLanlgnnis}

We begin by recalling the local Langlands correspondence defined via the theory of endoscopy. Let $K$ be a non-Archimedean local field of characteristic zero, and fix a nontrivial additive character $\uppsi_K$ of $K$, which extends to an additive character of any finite extension $K'/K$ by defining $\uppsi_{K'}\defining\uppsi_K\circ\tr_{K'/K}$. 

\subsection{The groups}\label{theogirneidnis}

Let $K_1/K$ be an unramified extension of degree at most two, and let $\cc\in \Gal(K_1/K)$ be the element with fixed field $K$. Let $\chi_{K_1/K}: K^\times\to\{\pm1\}$ denote the character associated with $K_1/K$ via local class field theory. Let $V$ be a vector space of dimension over $K_1$ equipped with a nondegenerate Hermitian $\cc$-sesquilinear form $\bra{-, -}$ on $V$, that is,
\begin{equation*}
\bra{au+bv, w}=a\bra{u, w}+b\bra{v, w}, \quad\text{and}\quad\bra{v, w}=\bra{w, v}^\cc
\end{equation*}
for all $a, b\in K_1$ and $u, v, w\in V$.

Fix an arbitrary orthogonal basis $\{v_1, \ldots, v_n\}$ of $V$ such that  $\bra{v_i, v_i}=a_i\in K^\times$. Define the discriminant of $V$ as
\begin{equation*}
\disc(V)=(-1)^{\binom{n}{2}}\prod_{i=1}^{\dim(V)}a_i.
\end{equation*}
The class of $\disc(V)$ in $K^\times/(K^\times)^2$ (resp. in $K^\times/\Nm_{K_1/K}(K_1^\times)$) when $K_1=K$ (resp. when $K_1\ne K$) is independent of the choice of orthogonal basis.

The (normalized) Hasse--Witt invariant of $V$ is defined as
\begin{equation*}
\eps(V)=
\begin{cases}\displaystyle
\paren{-1, (-1)^{\binom{\dim(V)}{4}}\cdot\disc(V)^{\binom{\dim(V)-1}{2}}}_K\cdot\prod_{i<j\in [\dim(V)]_+}(a_i, a_j)_K &\If K_1=K,\\
\chi_{K_1/K}(\disc(V)) &\If K_1\ne K.
\end{cases}
\end{equation*}
where
\begin{equation*}
(-, -)_K:\paren{K^\times/(K^\times)^2}\times\paren{K^\times/(K^\times)^2}\to \Br(K)[2]\cong\{\pm1\}
\end{equation*}
denotes the Hilbert symbol.

Recall from \cite[Theorem 2.3.7]{Ser73} that if $K_1=K$, the isometry class of $(V, \bra{-, -})$ is fully determined by the triple
\begin{equation*}
(\dim V, \disc(V), \eps(V))\in \bb Z_+\times\paren{K^\times/(K^\times)^2}\times\{\pm1\},
\end{equation*}
and, moreover, it follows from \cite[Proposition 2.3.6]{Ser73} that all such triples except $(1, d, -1)$ and $(2, 1, -1)$ occur. If $K_1\ne K$, the isometry class of $(V, \bra{-, -})$ is completely characterized by the pair
\begin{equation*}
(\dim(V), \eps(V))\in\bb Z_+\times\{\pm1\},
\end{equation*}
and all such pairs occur; see~\cite{M-H73}.

Let $G(V)$ denote the algebraic subgroup of $\GL(V)$ such that
\begin{equation*}
G(V)=\{g\in \GL(V): \bra{gv, gw}=\bra{v, w}\forall v, w\in V\},
\end{equation*}
and let $G=G(V)^\circ$ denote its identity component. Let $G^*$ denote the unique quasisplit inner form of $G$ over $K$. Exactly one of the following three cases holds:
\begin{itemize}
\item[\textbf{O1}]
$K_1=K$ and $\dim(V)=2n+1$ is odd. Then $G^*=\SO_{2n+1}$, the split orthogonal group in $2n+1$ variables.
\item[\textbf{O2}]
$K_1=K$ and $\dim(V)=2n$ is even. Then $G^*=\SO_{2n}^{\disc(V)}$, the quasisplit special orthogonal group associated with the quadratic space $V^*$ over $K$ of dimension $2n$, discriminant $\disc(V)$ and Hasse--Witt invariant 1.
\item[\textbf{U}]
$K_1\ne K$ and $\dim(V)=n$. Then $G^*=\bx U(n)$, the quasisplit unitary group associated with the Hermitian space of dimension $n$ with respect to the unramified quadratic extension $K_1/K$, with Hasse--Witt invariant 1.
\end{itemize}
We collectively refer to Cases O1 and O2 together as Case O. Note that
\begin{itemize}
\item
In Case O1, $G$ is split if $\eps(V)=1$ and non-quasisplit if $\eps(V)=-1$. In either case, $G$ splits over the unramified quadratic extension of $K$.
\item
In Case O2, $G$ is split if $\disc(V)=1, \eps(V)=1$; non-quasisplit if $\disc(V)=1, \eps(V)=-1$; and quasisplit but non-split if $\disc(V)\ne 1$. Moreover, $G$ splits over the unramified quadratic extension of $K$ if and only if $\ord_K(\disc(V))\equiv 0\modu 2$.
\item
In Case U, $G$ is non-quasisplit if $n$ is even and $\eps(V)=-1$; otherwise it is quasisplit but non-split. In all cases, $G$ splits over $K_1$.
\end{itemize}

To unify notation, let $n(G)=n(G^*)$ denote the geometric rank of $G$. Thus $n(\SO_{2n+1})=n(\SO_{2n}^{\disc(V)})=n(\bx U_n)=n$. We define the following invariants associated with $G$:
\begin{equation}\label{iefieheinfneis}
\begin{aligned}
N(G) &:= 
\begin{cases}
2n(G) & \text{in Case O}, \\
n(G) & \text{in Case U},
\end{cases} \\
d(G) &:= 
\begin{cases}
2n(G) + 1 & \text{in Case O1}, \\
2n(G) & \text{in Case O2}, \\
n(G) & \text{in Case U},
\end{cases} \\
b(G) &:=
\begin{cases}
-1 & \text{in Case O1}, \\
1 & \text{in Case O2}, \\
(-1)^{n(G)-1} & \text{in Case U}.
\end{cases}
\end{aligned}
\end{equation}
Here $N(G)$ is the integer $N$ for the group $\GL_N$ associated with the Langlands dual group of $G$; $d(G)$ is the dimension of the $\cc$-Hermitian space $V$ defining $G$, and $b(G)$ is the sign associated with $G$. In Case O2, we also define $\disc(G)\defining\disc(V)$.

Let $\mcl F=\{0=X'_0\subset X'_1\subset X'_2\subset\cdots\subset X'_r\}$ be a flag of isotropic $K_1$-subspaces of $V$. Then there exists an orthogonal direct sum decomposition
\begin{equation*}
V=(X'_r\oplus Y'_r)\perp V',
\end{equation*}
where $Y_r'$ is an isotropic subspace. The stabilizer $P\le G$ of this flag $\mcl F$ is a parabolic subgroup, and every parabolic subgroup of $G$ arises in this way. Moreover, if $X_i$ is a complement of $X'_{i-1}$ in $X'_i$ for each $i\in [r]_+$, then
\begin{equation*}
M=\Res_{K_1/K}\GL(X_1)\times\cdots\times\Res_{K_1/K}\GL(X_r)\times G(V')^\circ
\end{equation*}
is a Levi subgroup of $G$ (here $G(V')^\circ$ is trivial when $\dim V'=0$). Every Levi subgroup of $G$ arises in this way, and any two such Levi subgroups that are isomorphic are conjugate under $G(V)$.

We fix a pinning $(B^*, T^*, \{X^*_\alpha\}_{\alpha\in\Delta})$ of $G^*$ by identifying it with $G(V^*)^\circ$ for a suitable $\cc$-Hermitian space $V^*$ over $K_1$, and choosing a complete flag of totally isotropic subspaces in $V^*$. Recall that a Whittaker datum for $G^*$ is a $T^*(K)$-conjugacy class of generic characters of $N^*(K)$, where $N^*$ is the unipotent radical of $B^*$. Whittaker data for $G^*$ form a principal homogeneous space over the finite Abelian group
\begin{equation*}
E=\coker\paren{G^*(K)\to G^*_\ad(K)}=\ker(H^1(K, Z(G^*))\to H^1(K, G^*));
\end{equation*}
see~\cite[\S 9]{GGP12}. The fixed pinning of $G^*$, together with the additive character $\uppsi_K$ of $K$, determines a Whittaker datum $\mfk w$ for $G^*$; see~\cite[\S 5.3]{K-S99}. When $G$ is unramified, there exists a unique $G(K)$-conjugacy class of hyperspecial maximal compact open subgroups compatible with $\mfk w$, in the sense of \cite{C-S80}. In this case, ``unramified representations of $G(K)$'' refers to those unramified with respect to such a hyperspecial subgroup.

We define the Witt tower associated with $G$: For each $n_0\in [n(G)]$, let $G(n_0)$ denote the reductive group (possibly the trivial group $\uno$) of geometric rank $n_0$, such that there exists a Levi subgroup of $G$ isomorphic to
\begin{equation*}
\Res_{K_1/K}\GL_{\frac{n(G)-n_0}{[K_1: K]}}\times G(n_0).
\end{equation*}
By \cite[\S 4.4]{Tit79},
\begin{itemize}
\item
In Case O1, $G(n_0)$ exists if and only if $n_0\ge\frac{1-\eps(V)}{2}$,
\item
In Case O2, $G(n_0)$ exists if and only if $n_0\ge 1-\delta_{\disc(V), 1}\cdot \eps(V)$,
\item
In Case U, $G(n_0)$ exists if and only if $n(G)-n_0$ is even and moreover $n_0\ne 0$ when $G$ is non-quasisplit (i.e., when $n(G)$ is even and $\eps(V)=-1$).
\end{itemize}

We fix an isomorphism
\begin{equation*}
\hat G\cong \begin{cases}
\Sp_{N(G)}(\bb C) &\text{in Case O1}\\
\SO_{N(G)}(\bb C) &\text{in Case O2}\\
\GL_{N(G)}(\bb C) &\text{in Case U}\\
\end{cases},
\end{equation*}
and fix a pinning $(\hat T, \hat B, \{X_\alpha\})$ where $\hat T$ is the diagonal torus, $\hat B$ is the group of upper triangular matrices, and $\{X_\alpha\}$ is the set of standard root vectors. Let $\LL G=\hat G\rtimes W_K$ denote the Langlands $L$-group in the Weil form, where $W_K$ acts on $\hat G$ preserving the pinning, with the action as follows:
\begin{itemize}
\item
In Case O1, $W_K$ acts trivially on $\hat G$.
\item
In Case O2, $W_K$ acts via the quotient $\Gal(K(\sqrt{\disc(G)})/K)$. If $\disc(G)\ne 1$ and $\hat G$ is identified with the subgroup of $\SL(N(G), \bb C)$ preserving the nondegenerate bilinear form on $\bb C\{v_1, \ldots, v_{N(G)}\}$ defined by
\begin{equation*}
\bra{v_i, v_j}=\delta_{i, N(G)+1-j},
\end{equation*}
then the nontrivial element acts by conjugation via the element in $\bx O(N(G), \bb C)$ that exchanges $v_{n(G)}$ and $v_{n(G)+1}$ and fixes the others.
\item
In Case U, $W_K$ acts via the quotient $\Gal(K_1/K)$, where $\cc\in \Gal(K_1/K)$ acts by
\begin{equation*}
g\mapsto J_n'g^{-\top}(J_n')^{-1}.
\end{equation*}
\end{itemize}
Finally, note that $\hat G$ has a standard representation
\begin{equation*}
\hat\Std=\hat\Std_G: \hat G\to \GL_{N(G)}(\bb C).
\end{equation*}

Let
\begin{equation*}
\kappa_G: B(G)_\bas\xr\sim X^\bullet(Z(\hat G)^{\Gal_K})
\end{equation*}
denote the Kottwitz map \cite[Proposition 5.6]{Kot85}, which induces an identification of $H^1(K, G)$ with $X^\bullet\paren{\pi_0\paren{Z(\hat G)^{\Gal_K}}}$. There is a canonical isomorphism
\begin{equation*}
\bx H^1(K, G)=B(G)_\bas\cong\bb Z/2,
\end{equation*}
except in the case $G=\SO_2^1\cong \GL_1$, where $H^1(K, G)=1$ while $B(G)_\bas\cong\bb Z$; see~\cite[Lemma 2.1]{GGP12}. We exclude the case $G^*=\SO_2^1$ from now on. Except when $G^*\cong\SO_2^1$, every extended pure inner twist of $G^*$ is canonically a pure inner twist. We henceforth exclude the case $G^*\cong\SO_2^1$.

For every $b\in B(G)_\bas$, we may associate an \tbf{extended pure inner twist} $(G_b, \varrho_b, z_b)$ as defined in~\cite[\S 3.3, 3.4]{Kot97}, where
\begin{itemize}
\item
$G_b(K)=\Brace{g\in G(\breve K)|\dot b\vp_K(g)\dot b^{-1}=g}$, where $\dot b\in G(\breve K)$ maps to $b\in B(G)_\bas$,
\item
$\varrho_b: G \otimes_K \ovl K \xrightarrow{\sim} G_b \otimes_K \ovl K$ is an isomorphism over $\ovl K$, and
\item
$z_b$ is a 1-cocycle in $Z^1(W_K, G_\mathrm{ad})$ representing the class corresponding to $b$.
\end{itemize}
The pointed set $\bx H^1(K,G)$ and the extended pure inner twists of $G$ correspond bijectively to forms $V'$ of the space $V$ with its sesquilinear form $(-, -)$; cf.~\cite[\S29D and \S29E]{KMRT}. Choose an extended pure inner twist $(G^*_{b_0},\varrho_{b_0},z_{b_0})$ associated with $b_0\in B(G)_\bas$ such that
\begin{equation*}
G\cong G^*_{b_0}.
\end{equation*}
For each $b\in B(G)_\bas$, the inner twists $\varrho_{b_0}$ and $\varrho_b$ canonically identify $\hat{G_b}$, $\hat G$, and $\hat{G^*}$.

For every parabolic pair $(M, P)$ of $G$, there exists a unique standard parabolic pair $(M^*, P^*)$ of $G^*$ corresponding to $(M, P)$ under $\varrho_{b_0}$. This determines an equivalence class of extended pure inner twists of $M^*$, which we also denote by $(\varrho_{b_0},z_{b_0})$.

For each $b_0\in B(G^*)_\bas$, we write $\kappa_{b_0}$ for the character of $\pi_0(Z(\hat G)^{\Gal_K})$ corresponding to $b_0$ under the Kottwitz map. Since $Z(\hat G)^{\Gal_K}$ is finite, the Kottwitz maps for $G$ and $G^*$ give isomorphisms
\begin{equation}\label{ienufheiehfniess}
B(G)_\bas\cong X^\bullet(Z(\hat G)^{\Gal_K})\cong B(G^*)_\bas,
\end{equation}
which sends $b$ to $b+b_0$; cf.~\cite[p.~202]{Kot85}. In particular, $G_b\cong G^*_{b_0+b}$.

There is an automorphism $\theta$ of $G^\GL\defining \Res_{K_1/K}\GL_{N(G)}$ given by
\begin{equation*}
\theta(g)=J_{N(G)}'\cc(g)^{-\top}(J_{N(G)}')^{-1}
\end{equation*}
for
\begin{equation*}
g\in G^\GL(K)=\GL_{N(G)}(K_1).
\end{equation*}
We fix a standard $\Gal_{K_1}$-invariant pinning $(B^\GL, T^\GL, \{X_\alpha^\GL\})$ of $G^\GL$ stabilized by $\theta$. 
 
The quasisplit group $G^*$ determines an elliptic twisted endoscopic datum for $G^\GL\rtimes\theta$, hence a class in $\mcl E_\ellip(G^\GL\rtimes\theta)$; see \cite[pp.~3--4, 7]{Mok15} in Case U and \cite[\S1.2]{Art13} in Case O. Denote the group $\OAut_{G^\GL}(G^*)$ from \eqref{equirenrinfids} by $\OAut_N(G^*)$. Then $\OAut_N(G^*)$ is trivial in Case O1 and Case U, and $\OAut_N(G^*)=\bx O_{2n(G^*)}(\bb C)/\SO_{2n(G^*)}(\bb C)$ in Case O2.

We recall the following list of isomorphism classes of elliptic endoscopic triples $\mfk e\in \mcl E_\ellip(G^*)$ from~\cite[\S 4.6]{Rog90} and \cite{Wal10}, where we only describe $G^{\mfk e}$ and $\OAut(\mfk e)$:
\begin{itemize}
\item
In Case U, $G^{\mfk e}\cong \bx U_{K_1/K}(a)\times \bx U_{K_1/K}(b)$. Here $a, b\in [n(G)]$ satisfy $a+b=n(G)$, and $\OAut_{G^*}(\mfk e)$ is trivial except when $a=b$, where there exists a unique nontrivial outer automorphism swapping the two factors of $\hat{G^{\mfk e}}\cong \GL_a\times \GL_b$.
\item
In Case O1, $G^{\mfk e}\cong\SO_{2a+1}\times \SO_{2b+1}$. Here $a, b\in [n(G)]$ satisfy $a+b=n(G)$, and $\OAut_{G^*}(\mfk e)$ is trivial except when $a=b$, in which case there exists a unique nontrivial outer automorphism swapping the two factors of $\hat{G^{\mfk e}}\cong \Sp_{2a}(\bb C)\times \Sp_{2b}(\bb C)$.
\item
In Case O2, $G^{\mfk e}\cong \SO_{2a}^\beta\times\SO_{2b}^\gamma$. Here $a,b\in[n(G)]$ satisfy $a+b=n(G), \beta\gamma=\disc(G)$, and moreover $\beta=1$ if $a=0$, $\gamma=1$ if $b=0$, and $(a, \beta)\ne (1, 1), (b, \gamma)\ne (1, 1)$. If $ab>0$, there exists an outer automorphism acting by conjugation action of an element of $\bx O_{2a}(\bb C)\times \bx O_{2b}(\bb C)$ on the two factors of $\hat{G^{\mfk e}}\cong \SO_{2a}(\bb C)\times \SO_{2b}(\bb C)$. There are no other nontrivial outer automorphisms except when $a=b$ and $\disc(G)=1$, in which case there exists one swapping the two factors of $\hat{G^{\mfk e}}\cong \SO_{2a}(\bb C)\times \SO_{2b}(\bb C)$, and so $\OAut(\mfk e)\cong \bb Z/2\times \bb Z/2$ in this case.
\end{itemize}
For each $\mfk e\in \mcl E_\ellip(G^*)$, we fix a choice of $\LL\xi^{\mfk e}$ as in \cite[\S 1.8]{Wal10}, such that  if $G^{\mfk e}=H_1\times H_2$, then $\hat\Std_G\circ\LL\xi^{\mfk e}$ is conjugate to $(\hat\Std_{H_1}\oplus\hat\Std_{H_2})\circ\iota$, where $\iota: \LL G^{\mfk e}\inj \LL H_1\times\LL H_2$ is the natural inclusion.

Non-elliptic endoscopic triples $\mfk e\in \mcl E(G^*)$ are described similarly, in which case $G^{\mfk e}$ is a product of groups of the same type as $G$ with geometric rank smaller than $n(G)$ and restrictions of scalars of general linear groups, see for example \cite[\S 3.1.3]{Ish24} in Case O1.

\subsection{The \texorpdfstring{$L$}{L}-parameters}\label{IFNieniehifeniws}

We recall the description of $L$-parameters for $G$ and $G^*$, and explain their relation to conjugate self-dual representations of Weil groups. Let $\Phi(G)$ denote the set of $L$-parameters for $G$, in the sense of \cite[\S8.2]{Bor79}. We describe this set explicitly below.

For each positive integer $m\in\bb Z_+$, an $L$-parameter $\phi$ for $\GL_{m, K'}$ over any finite extension $K'/K$ may be regarded as  an isomorphism class of $m$-dimensional representations of $W_{K'}\times \SL_2(\bb C)$. Every such representation is isomorphic to a finite direct sum of representations of the form $\rho\boxtimes\sp_a$ where $\rho$ is an irreducible smooth representation of $W_{K'}$ and $\sp_a$ is the unique irreducible algebraic representation of $\SL_2(\bb C)$ of dimension $a$. 

An $L$-parameter $\phi$ for $\GL_m$ over $K_1$ is called conjugate self-dual and irreducible if $\phi$ is irreducible and isomorphic to $\phi^\theta\defining (\phi^s)^\vee$ as representations, where $s\in W_K$ is a lift of $\cc\in W_K/W_{K_1}\cong \Gal(K_1/K)$ and $\phi^s$ is the conjugate action $\phi^s(g)=\phi(sgs^{-1})$. Following \cite[\S 3]{GGP12}, we introduce the sign of a conjugate self-dual irreducible $L$-parameter $\phi$: There exists an isomorphism $f: \phi\xr\sim \phi^\theta$ such that $(f^\vee)^{s}=b(\phi)f$ for some $b(\phi)\in\{\pm1\}$. The value $b(\phi)$ is independent of the choice of $f$, and is called the sign of $\phi$. If $\phi=\rho\boxtimes\sp_a$, where $\rho$ is an irreducible representation of $W_{K_1}$, then
\begin{equation*}
b(\phi)=b(\rho)(-1)^{a-1};
\end{equation*}
see~\cite[Lemma 3.2]{GGP12}, \cite[\S 1.2.4]{KMSW}.

Then it follows from \cite[Theorem~8.1]{GGP12} and \cite[p. 365]{A-G17} that there exists a natural identification
\begin{equation*}
\Phi(G)=\Phi(G^*)=
\begin{cases}
\{\text{admissible }\phi: W_{K_1}\times\SL_2(\bb C)\to \Sp_{N(G)}(\bb C)\}/\Sp_{N(G)}(\bb C) &\text{in Case O1},\\
\left\{
\begin{aligned}
&\text{admissible }\phi: W_{K_1} \times \SL_2(\bb C) \to \bx O_{2n}(\bb C) \\
&\text{such that } \det(\phi) = (\Art_K^{-1}(-), \disc(V))_K
\end{aligned}
\right\} / \SO_{2n}(\bb C)&\text{in Case O2},\\
\left\{\begin{aligned}
&\text{admissible }\phi: W_{K_1}\times \SL_2(\bb C)\to \GL_{N(G)}(\bb C)\\
&\text{that is conjugate self-dual of sign }(-1)^{n(G)-1}
\end{aligned}
\right\}/\GL_{N(G)}(\bb C) &\text{in Case U}.
\end{cases}
\end{equation*}
Here $\phi$ is called admissible if 
\begin{itemize}
\item
$\phi(\sigma_{K_1})$ is semisimple,
\item
$\phi|_{I_{K_1}}$ is smooth, and
\item
$\phi|_{\SL_2(\bb C)}$ is algebraic.
\end{itemize}
Strictly speaking, in the definition of \cite[\S8.2]{Bor79}, an $L$-parameter for $G$ is a homomorphism from $W_K\times \SL_2(\bb C)$ to $\LL G$ up to $\hat G$-conjugacy. In the description above, we instead encode such a parameter by the corresponding homomorphism with source $W_{K_1}\times \SL_2(\bb C)$. Thus, by a slight abuse of terminology, we will refer to this latter homomorphism as an $L$-parameter and denote it by $\phi$. The associated parameter in the sense of \cite[\S8.2]{Bor79} will be denoted by
\begin{equation*}
\phi^\natural: W_K\times \SL_2(\bb C)\longrightarrow \LL G^*.
\end{equation*}

Then an $L$-parameter $\phi\in \Phi(G^*)$ is
\begin{itemize}
\item
tempered (or bounded) if and only if $\phi(W_{K_1})$ is a relatively compact subset of the target.
\item
discrete if and only if $\Im(\phi)$ is not contained in the normalizer of a proper parabolic subgroup of $\hat G^*$.
\item
semisimple if and only if it is trivial on the $\SL_2(\bb C)$-factor.
\item
supercuspidal if and only if it is discrete and semisimple.
\end{itemize}
The subset of tempered (resp. discrete, resp. semisimple, resp. supercuspidal) $L$-parameters for $G^*$ is denoted by $\Phi_\temp(G^*)$ (resp. $\Phi_2(G^*)$, resp. $\Phi^\sems(G^*)$, resp. $\Phi^{\bx{sc}}(G^*)$).

An $L$-parameter $\phi\in\Phi(G^*)$ can be regarded as an $N(G)$-dimensional conjugate self-dual representation $\phi^\GL$ of $W_{K_1}\times\SL_2(\bb C)$ via the standard representation $\hat\Std_G$. $\phi$ is determined by $\phi^\GL$ in Case O1 and Case U, but only determined up to $\bx O_{N(G)}(\bb C)$-conjugation in Case O2; see~\cite[Theorem 8.1]{GGP12}. We can write $\phi^\GL$ as
\begin{equation*}
\phi^\GL=\bplus_{i\in I_\phi^+}m_i\phi_i\oplus \bplus_{i\in I_\phi^-}2m_i\phi_i\oplus\bplus_{i\in J_\phi}m_i(\phi_i\oplus\phi_i^\theta),
\end{equation*}
where $m_i$ are positive integers, and $I_\phi^+, I_\phi^-, J_\phi$ index mutually inequivalent irreducible representations of $W_{K_1}\times\SL_2(\bb C)$ such that 
\begin{itemize}
\item
for $i\in I_\phi^+$, $\phi_i$ is conjugate self-dual with sign $b(G)$.
\item
for $i\in I_\phi^-$, $\phi_i$ is conjugate self-dual with sign $-b(G)$.
\item
for $i\in J_\phi$, $\phi_i$ is not conjugate self-dual.
\end{itemize}
Then $\phi$ is discrete if and only if $m_i=1$ for $i\in I_\phi^+$, and $I_\phi^-=J_\phi=\vn$. Moreover, we call $\phi$ a \tbf{simple $L$-parameter} if it is discrete and $\#I_\phi^+=1$.

For any $\phi\in\Phi(G^*)$, define
\begin{equation*}
S_\phi^\sharp\defining\prod_{i\in I_\phi^+}\bx O_{m_i}(\bb C)\times\prod_{i\in I_\phi^-}\Sp_{2m_i}(\bb C)\times\prod_{i\in J_\phi}\GL_{m_i}(\bb C).
\end{equation*}
This group formally represents the centralizer of $\phi$ in $\hat G\rtimes\Gal_{K'/K}$, where $K'$ is a minimal splitting field of $G$. Its formal component group is
\begin{equation*}
\mfk S_\phi^\sharp\defining \pi_0(S_\phi^\sharp)\cong \bplus_{i\in I_\phi^+}(\bb Z/2)e_i.
\end{equation*}
For each $i\in I_\phi^+$, let $e_i^\vee\in\Irr(\mfk S_\phi^\sharp)$ be the character determined by
\begin{equation*}
e_i^\vee(e_j)=(-1)^{\delta_{i,j}}\qquad(i,j\in I_\phi^+).
\end{equation*}
Writing tensor products of characters additively, we obtain
\begin{equation*}
\Irr(\mfk S_\phi^\sharp)\cong\bplus_{i\in I_\phi^+}(\bb Z/2)e_i^\vee.
\end{equation*}

In Case O2, define
\begin{equation*}
S_\phi\defining Z_{\hat G}(\phi).
\end{equation*}
It is naturally a subgroup of $S_\phi^\sharp$ of index at most two. The homomorphism
\begin{equation*}
\det\nolimits_\phi: \mfk S_\phi^\sharp\to \bb Z/2, \qquad \sum_{i\in I_\phi^+}x_ie_i\mapsto \sum_{i\in I_\phi^+}x_i\dim(\phi_i)
\end{equation*}
has kernel
\begin{equation*}
\mfk S_\phi\defining\ker(\det_\phi)=\pi_0(S_\phi).
\end{equation*}
To unify notation, in Case O1 and Case U, set $S_\phi\defining S_\phi^\sharp$ and $\mfk S_\phi\defining\mfk S_\phi^\sharp$.

We define the central element
\begin{equation*}
z_\phi \defining\sum_{i \in I_\phi^+} m_i e_i \in \mfk S_\phi,
\end{equation*}
and define the reduced component group
\begin{equation*}
\ovl{\mfk S}_\phi \defining \mfk S_\phi/\bra{z_\phi}.
\end{equation*}

When $(G=G^*_{b_0}, \varrho_{b_0}, z_{b_0})$ is a pure inner form of $G^*$, an $L$-parameter for $G^*$ is called $(G, \varrho_{b_0})$-relevant if every Levi subgroup $\LL M$ of $\LL G$ such that  $\Im(\phi)\subset\LL M$ is relevant, i.e., $\LL M$ is a Levi component of a $(G, \varrho_{b_0})$-relevant parabolic subgroup of $\LL G$; see~\cite[\S 0.4.2]{KMSW}.

Let $M$ be a standard Levi subgroup of $G$ of the form
\begin{equation*}
\Res_{K_1/K}\GL_{n_1}\times\cdots\times\Res_{K_1/K}\GL_{n_k}\times G(n_0),
\end{equation*}
where
\begin{equation*}
[K_1: K](n_1+\cdots+n_k)+n_0=n(G).
\end{equation*}
Define the $K_1$-group
\begin{equation*}
M^\GL\defining\Res_{K_1/K}\paren{\GL_{n_1}\times\cdots\times\GL_{n_k}\times\GL_{\frac{N(G)}{n(G)}n_0}}.
\end{equation*}
Then $\Phi(M^\GL)$ is canonically identified with the set of tuples $(\phi_1, \ldots, \phi_k, \phi_0)$, where
\begin{equation*}
\phi_i\in \Phi(\GL_{n_i, K_1})\quad\text{and}\quad\phi_0\in \Phi(\GL_{\frac{N(G)}{n(G)}n_0, K_1})
\end{equation*}
Then the canonical map $\Phi(M)\to \Phi(G)$ fits into the commutative diagram
\begin{equation}\label{ienfieunienfis}
\begin{tikzcd}[sep=large]
\Phi(M)\ar[r, "(-)^\GL"]\ar[d] &\Phi(M^\GL)\ar[d]\\
\Phi(G)\ar[r, "(-)^\GL"] &\Phi(G^\GL),
\end{tikzcd}
\end{equation}
where the right vertical map is given by
\begin{equation}\label{isniheiherehs}
(\phi_1, \ldots, \phi_k, \phi_0)\mapsto \phi^\GL=\phi_1\oplus\cdots\oplus\phi_k\oplus\phi_0\oplus\phi_1^\theta\oplus\cdots\oplus\phi_k^\theta.
\end{equation}

\subsection{The correspondence}\label{theoenifniehifensis}

We state the local Langlands correspondence for pure inner twist $(G, \varrho, z)$ of $G^*$. In Case O1, it is established by Arthur \cite{Art13} when $G$ is quasisplit, and by Ishimoto \cite[Theorem 3.15]{Ish24} when $G$ is not quasisplit. In Case U it is established by Mok~\cite[Theorem 2.5.1, Theorem 3.2.1]{Mok15} when $G$ is quasisplit, and by Kaletha, Minguez, Shin and White \cite[Theorem 1.6.1]{KMSW} when $G$ is not quasisplit. In Case U, more properties of this correspondence are established in \cite{C-Z21a}.

In Case O2, only a weak version of the local Langlands correspondence is established. This is due to the intrinsic nature of the endoscopy method, because when $G^*$ is regarded as a twisted endoscopic group of $\GL_{N(G)}$, $\phi$ can only be recovered from $\phi^\GL$ up to $\bx O_{N(G)}(\bb C)$-conjugation \cite[Theorem 8.1]{GGP12}. When $G$ is quasisplit, the weak LLC is established by Arthur \cite{Art13} (see also \cite[Theorem 3.6]{A-G17}), and when $G$ is not quasisplit, it is established by Chen and Zou~\cite[Theorem~A.1]{C-Z21}.

To state the weak version in Case O2, for each pure inner twist $(G=G^*_{b_0}, \varrho_{b_0}, z_{b_0})$ of $G^*$, we introduce an equivalence relation $\sim_\varsigma$ on $\Pi(G)$. Note that there exists an outer automorphism $\varsigma$ of $G^*$ which preserves $\mfk w$: in fact $\varsigma$ can be realized as an element of the corresponding orthogonal group of determinant $-1$; see~\cite[p. 847]{Tai19}. Via $\varrho_{b_0}$, the element $\varsigma$ acts by a rational outer automorphism on $G$; see~\cite[Lemma 9.1.1]{Art13}. For each $\pi\in \Pi(G)$, its $\varsigma$-conjugate $\pi^\varsigma$ is defined by $\pi^\varsigma(h)=\pi(h^\varsigma)$, and the equivalence relation $\sim_\varsigma$ is defined by
\begin{equation*}
\pi\sim_\varsigma\pi^\varsigma.
\end{equation*}
For each $\pi\in \Pi(G)$, we write $\tilde\pi$ for the image of $\pi$ under the quotient map $\Pi(G)\to\tilde\Pi(G)$, where $\tilde\Pi(G)\defining\Pi(G)/\sim_\varsigma$. It is clear that the sets $\Pi_\temp(G), \Pi_2(G)$ and $\Pi_{\bx{sc}}(G)$ are preserved under this equivalence relation, so temperedness, discreteness, and supercuspidality are well-defined for equivalence classes $\tilde\pi\in\tilde\Pi(G)$.

Similarly, we define
\begin{equation*}
\tilde\Phi(G^*)=\{\text{admissible }\phi: W_{K_1}\times \SL_2(\bb C)\to \bx O_{N(G)}(\bb C)|\det(\phi)=(\Art_K^{-1}(-), \disc(V))_K\}/\bx O_{N(G)}(\bb C),
\end{equation*}
together with a natural map $\Phi(G^*)\to \tilde\Phi(G^*)$. Note that if $\phi_1$ and $\phi_2$ are conjugate up to $\bx O_{N(G)}(\bb C)$, then $\phi_1^\GL=\phi_2^\GL$. In particular, $\phi^\sems, \phi^\GL, I_\phi^+, I_\phi^-, J_\phi, S_\phi^\sharp, S_\phi, \mfk S_\phi, \ovl{\mfk S}_\phi, z_\phi$ are well-defined functions for $\tilde\phi\in \tilde\Phi(G^*)$. We also set $\tilde\Phi(G)\defining\tilde\Phi(G^*)$.

For uniformity of notation, in Case O1 and Case U set $\varsigma=\id_G$ and $\tilde\Pi(G)=\Pi(G)$, $\tilde\Phi(G)=\Phi(G)$. Define a subspace $\tilde{\mcl H}(G)\subset\mcl H(G)$ of test functions on $G(K)$ as follows. In Case O2, following Arthur \cite{Art13}, let $\tilde{\mcl H}(G)$ denote the subspace of $\mcl H(G)$ consisting of $\varsigma$-invariant functions on $G(K)$; so that irreducible smooth representations of $\tilde{\mcl H}(G)$ correspond to $\bx O(V)$-conjugacy classes of irreducible admissible representations of $G(K)$. Similarly, for each $\mfk e\in \mcl E_\ellip(G)$, set
\begin{equation*}
\tilde{\mcl H}(G^{\mfk e})=\tilde{\mcl H}(G_1)\otimes \tilde{\mcl H}(G_2),
\end{equation*}
since $G^{\mfk e}$ is a product of two (possibly trivial) even special orthogonal groups over $K$. In Case O1 and Case U, simply take $\tilde{\mcl H}(G)=\mcl H(G)$ and $\tilde{\mcl H}(G^{\mfk e})=\mcl H(G^{\mfk e})$.

We can now reformulate the local Langlands correspondence as follows.

\begin{thm}[\cite{Art13, Mok15, KMSW, C-Z21, C-Z21a, Ish24}]\label{coeleninifheihsn}
If $(G=G^*_{b_0}, \varrho_{b_0}, z_{b_0})$ is a pure inner twist of $G^*$, then there exists a map
\begin{equation*}
\rec_G:\tilde\Pi(G)\to\tilde\Phi(G)
\end{equation*}
with finite fibers. For any $\tilde\phi\in\tilde\Phi(G)$, we write $\tilde\Pi_{\tilde\phi}(G)$ for $\rec^{-1}_G(\tilde\phi)$, called the (ambiguous) $L$-packet for $\tilde\phi$. This map satisfies the following properties:
\begin{enumerate}
\item
If $\tilde\phi\in\tilde\Phi(G)$ is not relevant for $G$ in the sense of \textup{\cite[Definition 0.4.14]{KMSW}}, then $\tilde\Pi_{\tilde\phi}(G)=\vn$.
\item
For each $\tilde\phi\in\tilde\Phi(G)$ and $\tilde\pi\in\tilde\Pi_{\tilde\phi}(G)$, $\tilde\pi$ is tempered if and only if $\tilde\phi$ is tempered, and $\tilde\pi$ is a discrete series representation if and only if $\tilde\phi$ is discrete. 
\item
$\rec_G$ only depends on $G$ but not on $\varrho_{b_0}$ and $z_{b_0}$. For the fixed Whittaker datum $\mfk w$ of $G^*$ which induces a Whittaker datum for each standard Levi factor of $G$, there exists a canonical bijection
\begin{equation*}
\iota_{\mfk w, b_0}: \tilde\Pi_{\tilde\phi}(G)\cong\Irr(\mfk S_{\tilde\phi}; \kappa_{b_0})
\end{equation*}
for each $\tilde\phi\in\tilde\Phi(G)$, where $\Irr(\mfk S_{\tilde\phi}; \kappa_{b_0})$ is the set of characters $\eta$ of $\mfk S_{\tilde\phi}$ such that  $\eta(z_{\tilde\phi})=\kappa_{b_0}(-\uno)$. We write $\tilde\pi=\tilde\pi_{\mfk w, b_0}(\tilde\phi, \eta)$ if $\tilde\pi\in\tilde\Pi_{\tilde\phi}(G)$ corresponds to $\eta\in\Irr(\mfk S_{\tilde\phi})$ via $\iota_{\mfk w, b_0}$.
\item(Compatibility with Langlands quotient)
For $\tilde\phi\in\tilde\Phi(G)$, suppose
\begin{equation*}
\tilde\phi^\GL=\paren{\phi_1\otimes\largel{-}_{K_1}^{s_1}}\oplus\cdots\oplus\paren{\phi_r\otimes\largel{-}_{K_1}^{s_r}}\oplus\tilde\phi_0^\GL\oplus\paren{\phi_r^\theta\otimes\largel{-}_{K_1}^{-s_r}}\oplus\cdots\oplus\paren{\phi_1^\theta\otimes\largel{-}_{K_1}^{-s_1}},
\end{equation*}
where 
\begin{itemize}
\item
$\phi_i\in\Phi_{2, \temp}(\GL_{d_i, K_1})$ for each $i\in[r]_+$, where $d_i>0$,
\item
$\tilde\phi_0\in\tilde\Phi_\temp(G^*(n_0))$,
\item
$s_1\ge s_2\ge\cdots\ge s_r>0$,
\item
$[K_1: K](d_1+\cdots+d_r)+n_0=n(G)$.
\end{itemize}
Let $\tau_i\in\Pi_{2, \temp}(\GL_{d_i, K_1})$ correspond to $\phi_i$ for each $i\in [r]_+$, then there is a canonical identification
\begin{equation*}
\tilde\Pi_{\tilde\phi}(G)=\Brace{\ovl{\bx I}_P^G\paren{\paren{\tau_1\otimes\nu^{s_1}}\boxtimes\cdots\boxtimes\paren{\tau_r\otimes\nu^{s_r}}\boxtimes\tilde\pi_0}\bigg| \tilde\pi_0\in \tilde\Pi_{\tilde\phi_0}(G(n_0))}.
\end{equation*}
Here $P$ is a parabolic subgroup of $G$ with a Levi factor
\begin{equation*}
M\cong \Res_{K_1/K}\GL_{d_1}\times\cdots\times\Res_{K_1/K}\GL_{d_r}\times G(n_0),
\end{equation*}
such that $M=\varrho_{b_0}(M^*)$ where $M^*$ is a standard Levi subgroup of $G^*$, and
\begin{equation*}
\ovl{\bx I}_P^G\paren{\paren{\tau_1\otimes\nu^{s_1}}\boxtimes\cdots\boxtimes\paren{\tau_r\otimes\nu^{s_r}}\boxtimes\tilde\pi_0}
\end{equation*}
is the unique irreducible quotient of 
\begin{equation*}
\bx I_P^G\paren{\paren{\tau_1\otimes\nu^{s_1}}\boxtimes\cdots\boxtimes\paren{\tau_r\otimes\nu^{s_r}}\boxtimes\tilde\pi_0},
\end{equation*}
whose $\varsigma$-equivalence class is well-defined in $\tilde\Pi(G)$. Moreover, if $\mfk w_0$ is the induced Whittaker datum on $M^*$, there exists a natural identification $\mfk S_{\tilde\phi_0}\cong\mfk S_{\tilde\phi}$, and $\iota_{\mfk w, b_0}(\tilde\pi)=\iota_{\mfk w_0, b_0}(\tilde\pi_0)$ if
\begin{equation*}
\tilde\pi=\ovl{\bx I}_P^G\paren{\paren{\tau_1\otimes\nu^{s_1}}\boxtimes\cdots\boxtimes\paren{\tau_r\otimes\nu^{s_r}}\boxtimes\tilde\pi_0}.
\end{equation*}
\item(Compatibility with standard $\gamma$-factors)
Suppose $\tilde\pi\in\tilde\Pi(G)$ with $\tilde\phi\defining\rec_G(\tilde\pi)$, then for any quasi-character $\chi$ of $K^\times$, viewed as a character of $W_K$ and then restricted to $W_{K_1}$,
\begin{equation*}
\gamma(\tilde\pi, \chi, \uppsi_K; s)=\gamma(\tilde\phi^\GL\otimes\chi, \uppsi_{K_1}; s).
\end{equation*}
Here the left-hand side is the standard $\gamma$-factor defined by Lapid--Rallis using the doubling zeta integral \textup{\cite{L-R05}} but modified in \textup{\cite{G-I14}}, and the right-hand side is the $\gamma$-factor defined in \textup{\cite{Tat79}}.
\item(Compatibility with Plancherel measures)
Suppose $\tilde\pi\in\tilde\Pi(G)$ with $\tilde\phi\defining \rec_G(\tilde\pi)$, then for any $d\in \bb Z_+$ and any $\tau\in\Pi(\GL_{d, K_1})$ with $L$-parameter $\phi_\tau$,
\begin{align*}
\mu_{\uppsi_K}(\tau\otimes\nu^s\boxtimes\tilde\pi)=&\gamma(\phi_\tau\otimes(\tilde\phi^\GL)^\vee, \uppsi_{K_1}; s)\cdot\gamma(\phi_\tau^\vee\otimes\tilde\phi^\GL, \uppsi_{K_1}^{-1}, -s)\\
&\cdot \gamma(\bx R_{G^*}\circ\phi_\tau, \uppsi_K; 2s)\cdot \gamma(\bx R_{G^*}\circ\phi_\tau^\vee, \uppsi_K^{-1}; -2s).
\end{align*}
Here the left-hand side is the Plancherel measure defined in \textup{\cite[\S 12]{G-I14}} (cf.~\textup{\cite[\S A.7]{G-I16}}), and in the right-hand side $\bx R_{G^*}$ is the representation
\begin{equation*}
\bx R_{G^*}=
\begin{cases}
\Sym^2 &\text{in Case O1}\\
\wedge^2 &\text{in Case O2}\\
\Asai^{(-1)^{N(G)}} &\text{in Case U}\\
\end{cases}
\end{equation*}
where $\Asai^+$ and $\Asai^-$ are the two Asai representations of $\Res_{K_1/K}\GL_{N(G)}$.
\item(Local intertwining relations)
Suppose $\tilde\phi\in\tilde\Phi(G)$ admits a decomposition
\begin{equation*}
\tilde\phi^\GL=\phi_\tau+\tilde\phi_0^\GL+\phi_\tau^\theta,
\end{equation*}
where $\tau\in\Pi_{2, \temp}(\GL_{d, K_1})$ has $L$-parameter $\phi_\tau\in\Phi(\GL_{d, K_1})$, and $\tilde\phi_0\in\tilde\Phi_\temp(G(n(G)-[K_1: K]d))$ is tempered. Let $P\le G$ be a maximal parabolic subgroup with a Levi factor
\begin{equation*}
M\cong \Res_{K_1/K}\GL_d\times G(n(G)-[K_1: K]d).
\end{equation*}
Assume that $M=\varrho_{b_0}(M^*)$, where $M^*$ is a standard Levi subgroup of $G^*$, and let $\mfk w_0$ be the induced Whittaker datum on $M^*$. Then, for each $\eta_0\in\Irr(\mfk S_{\tilde\phi_0})$,
\begin{equation*}
\bx I_P^G(\tau\boxtimes\tilde\pi_{\mfk w_0, b_0}(\tilde\phi_0, \eta_0))=\bplus_\eta\tilde\pi_{\mfk w, b_0}(\tilde\phi, \eta)
\end{equation*}
as a $\tilde{\mcl H}(G)$-module, where $\eta$ runs through characters of $\mfk S_{\tilde\phi}$ that restrict to $\eta_0$ under the natural embedding $\mfk S_{\tilde\phi_0}\inj \mfk S_{\tilde\phi}$. Moreover, if $\phi_\tau$ is conjugate self-dual of sign $b(G)$, let
\begin{equation*}
\bx R_{\mfk w}(w, \tau\boxtimes\tilde\pi_{\mfk w_0, b_0}(\tilde\phi_0, \eta_0))\in\End_{G(K)}\paren{\bx I_P^G(\tau\boxtimes\tilde\pi_{\mfk w_0, b_0}(\tilde\phi_0, \eta_0))}
\end{equation*}
be the normalized intertwining operator defined in \textup{\cite[\S 7.1]{C-Z21}} in Case O2 and in \textup{\cite[\S 5.2]{C-Z21a}} in Case $\bx U$, and in \textup{\cite[\S 4]{Ish24}} in Case O1. Here $w$ is the unique nontrivial element in the relative Weyl group for $M$. Then the restriction of $\bx R_{\mfk w}(w, \tau\boxtimes\tilde\pi_{\mfk w_0, b_0}(\tilde\phi_0, \eta_0))$ to $\tilde\pi_{\mfk w, b_0}(\tilde\phi, \eta)$ is the scalar multiplication by
\begin{equation*}
\begin{cases}
\eta(e_\tau) &\text{in Case O}\\
\eta(e_\tau)\kappa_{b_0}^d(-\uno) &\text{in Case U}
\end{cases}
\end{equation*}
where $e_\tau$ is the element of $\mfk S_{\tilde\phi}$ corresponding to $\phi_\tau$.
\end{enumerate}
\end{thm}
\begin{rem}\enskip
\begin{enumerate}
\item
The independence of $\rec_G$ with respect to $\varrho$ and $z$ is established in Case O by \cite[Remark 4.6(2)]{C-Z21} and in Case U by the argument before \cite[Theorem 2.5.5]{C-Z21a}.
\item
The compatibility of $\rec_G$ with standard $\gamma$-factors and Plancherel measures can be used to characterize $\rec_G$ and show compatibility with LLC constructed via exceptional isomorphisms in low dimensions, via \cite[Lemma A.6]{G-I16}. For example, in Case O2, when $G$ is quasisplit, it is shown that the construction of Chen--Zou using theta correspondence is compatible with that defined by Arthur \cite[Theorem 9.1]{C-Z21}, and in Case U, when $G$ is quasisplit, it is shown that the construction of Chen--Zou using theta correspondence is compatible with that defined by Mok \cite[Theorem 7.1.1]{C-Z21a}.
\item
The local intertwining relation can be used to characterize the finer structure of $L$-packets, i.e., $\iota_{\mfk w, b_0}$. In Case O1, it follows from \cite[Proposition 2.3.1 and Theorems 2.2.1, 2.2.4, 2.4.1 and 2.4.4]{Art13} when $G$ is quasisplit and follows from the corresponding propositions in \cite[\S 4]{Ish24} when $G$ is not quasisplit. In Case U, it follows from \cite[Theorem 3.2.1, 3.4.3]{Mok15} when $G$ is quasisplit and \cite[Theorem 2.5.1]{C-Z21a} when $G$ is not quasisplit. In Case O2, it follows from \cite[Proposition 2.3.1 and Theorems 2.2.1, 2.2.4, 2.4.1 and 2.4.4]{Art13} when $G$ is quasisplit and follows from \cite[Theorem A.1]{C-Z21} when $G$ is not quasisplit. For the details, see, for example, \cite[Theorem 2.2]{Ato17}.
\item
Note that in Case U, the LLC stated in \cite{Mok15} and \cite{KMSW} is compatible with the arithmetic normalization of the local class field theory instead of the geometric normalization of the local class field theory, i.e., the Artin map sends a uniformizer to an arithmetic Frobenius class rather than to a geometric Frobenius class. One can pass between the two normalizations by using the compatibility of local Langlands correspondence with contragredients \cite{Kal13}; cf.~\cite[Theorem 2.5]{M-N23}.
\end{enumerate}
\end{rem}

We will always call this correspondence the \tbf{classical local Langlands correspondence (LLC)}, as opposed to the Fargues--Scholze local Langlands correspondence to be defined in \S\ref{local-abglancorrepsodeifn}.

For later use, we define what a local functorial transfer is from $G$ to $G^\GL$:

\begin{defi}\label{localinsishinfiems}
Let $\pi\in\Pi(G)$ have classical $L$-parameter $\tilde\phi$. The unique irreducible admissible representation $\pi^\GL\in\Pi(G^\GL)$ whose classical $L$-parameter satisfies $\phi_{\pi^\GL}=\tilde\phi^\GL$ is called the \tbf{local functorial transfer} of $\pi$.
\end{defi}

We also need certain endoscopic character identities for the local Langlands correspondence. In Case O1, they are established by Arthur \cite[Theorem 2.2.1]{Art13} when $G$ is quasisplit, and by Ishimoto \cite[Theorem 3.15]{Ish24} when $G$ is not quasisplit. In Case U, they are established by Mok \cite{Mok15} when $G$ is quasisplit, and by Kaletha, Minguez, Shin and White \cite[Theorem 1.6.1]{KMSW} when $G$ is not quasisplit. In Case O2, they are established by Arthur \cite[Theorem 2.2.1]{Art13} when $G$ is quasisplit, and are established by \cite[Theorem A]{Pen25b} when $G$ is not quasisplit. We adopt the notation for endoscopy from \S\ref{oniinssmows}.

\begin{thm}\label{endoslicindiikehrnieiis}
Fix a pure inner twist $(G=G^*_{b_0}, \varrho_{b_0}, z_{b_0})$ of $G^*$.
\begin{enumerate}
\item
Suppose $\mfk e\in \mcl E_\ellip(G)$ and $\phi^{\mfk e}\in \Phi_{2,\temp}(G^{\mfk e})$ with $\phi=\LL\xi^{\mfk e}\circ\phi^{\mfk e}$ and $s^{\mfk e}\in S_\phi=\mfk S_\phi$. Then each $f\in \tilde{\mcl H}(G)$ has a transfer $f^{G^{\mfk e}}$ contained in $\tilde{\mcl H}(G^{\mfk e})$, and
\begin{equation*}
\sum_{\tilde\pi^{\mfk e}\in\tilde\Pi_{\tilde\phi^{\mfk e}}(G^{\mfk e})}\tr\paren{f^{G^{\mfk e}}|\tilde\pi^{\mfk e}}=e(G)\sum_{\tilde\pi\in\tilde\Pi_{\tilde\phi}(G)}\iota_{\mfk w, b_0}(\tilde\pi)(s^{\mfk e})\tr\paren{f|\tilde\pi}.
\end{equation*}
Here $e(G)$ is the Kottwitz sign of $G$ as defined in~\cite{Kot83}, $\Delta[\mfk w, \mfk e, z_{b_0}]$ is the transfer factor normalized in~\textup{\S\ref{transofenrienfieslsss}}, and $\tilde\Pi_{\tilde\phi^{\mfk e}}(G^{\mfk e})$ is defined to be the (ambiguous) $L$-packet associated with $\tilde\phi^{\mfk e}$ as before.\footnote{This $L$-packet is well defined because $G^{\mfk e}$ is a product of special orthogonal or unitary groups.}

In particular, by the Weyl integration formula and its stable analogue (cf.~\cite[Appendix~A.1, especially~(A.1.2)]{HKW22}),\footnote{See also \cite[Lemma 2.11.2]{A-K24} for the same Weyl-integration formula argument.} the preceding distributional identity is equivalent to
\begin{equation*}
\sum_{h\in G^{\mfk e}(K)_\sreg/G^{\mfk e}(\ovl K)\dash\Conj}\Delta[\mfk w, \mfk e, z_{b_0}](h, g)\bx S\Theta_{\tilde\phi^{\mfk e}}(h)=e(G)\sum_{\tilde\pi\in\tilde\Pi_{\tilde\phi}(G)}\iota_{\mfk w, b_0}(\tilde\pi)(s^{\mfk e})\Theta_{\tilde\pi}(g)
\end{equation*}
for any strongly regular semisimple element $g\in G(K)_\sreg$, where
\begin{equation*}
\bx S\Theta_{\tilde\phi^{\mfk e}}\defining\sum_{\tilde\pi\in\tilde\Pi_{\tilde\phi^{\mfk e}}(G^{\mfk e})}\Theta_{\tilde\pi},
\end{equation*}
and $\Theta_{\tilde\pi}$ is the average of $\Theta_\pi$ where $\pi$ runs through the preimage of $\tilde\pi$ under the map
\begin{equation*}
\Pi(G^{\mfk e})\to \tilde\Pi(G^{\mfk e}).
\end{equation*}
\item
For any tempered $L$-parameter $\tilde\phi\in\tilde\Phi_\temp(G^*)$, if $f^\GL\in \mcl H(G^\GL)$, then $f^\GL$ has a transfer $f^*$ to $G^*$ contained in $\tilde{\mcl H}(G^*)$, and
\begin{equation*}
\sum_{\tilde\rho\in\tilde\Pi_{\tilde\phi}(G^*)}\tr(f^*|\tilde\rho)=\tr_\theta(f^\GL|\pi_{\tilde\phi^\GL}).
\end{equation*}
Here the right-hand side is the $\theta$-twisted trace, and $\pi_{\tilde\phi^\GL}\in\Pi(G^\GL)$ is associated with the $L$-parameter $\tilde\phi^\GL$ via LLC, and the left-hand side is a stable distribution of $f^*$, i.e., it depends only on the stable orbital integrals of $f^*$.
\end{enumerate}
\end{thm}

Finally, we remark that the classical local Langlands correspondence for inner twists (or rather $K$-groups) of special orthogonal or unitary groups over $\bb R$ and (ordinary) endoscopic character identities are known in complete generality \cite{ABV92, Vog93, She08, She10, Art13}, where we replace irreducible admissible representations by smooth \Frechet representations of moderate growth whose associated Harish--Chandra modules are admissible, as introduced by Casselman \cite{Cas89} and Wallach \cite{Wal92}. Moreover, the twisted endoscopic character identities are known by Arthur \cite{Art13}. The real case is analogous to the $p$-adic case, except that there are more inner twists. In fact, we will only use results concerning discrete series $L$-packets. For a modern exposition in the discrete series case, see \cite{A-K24}.

\subsection{Compatibility with parabolic induction}\label{compatibliseniniefs}

In this subsection, we recall the definition of extended cuspidal support of a discrete series representation of $G(K)$, and deduce a compatibility property of the classical semisimplified $L$-parameters with parabolic induction. In Case U, the assertion is established in \cite[Proposition 2.11]{MHN24}, but our proof is slightly simpler than the proof given there.

First, we recall that it is a theorem of Bernstein and Zelevinsky~\cite[Theorem 2.5, Theorem 2.9]{B-Z77} that all irreducible smooth representations of $G(K)$ can be constructed by parabolic induction from supercuspidal representations: 

\begin{thm}
For any irreducible smooth representation $\pi\in \Pi(G)$, there exists a unique pair $(M, \sigma)$ up to conjugacy by $G(K)$, where $M$ is a Levi component of some parabolic subgroup $P\le G$ and $\sigma\in \Pi_{\bx{sc}}(M)$ is a supercuspidal representation, such that  $\pi$ is a subrepresentation of $\bx I_P^G(\sigma)$. Such pairs $(M, \sigma)$ are called \tbf{cuspidal supports} of $\pi$.
\end{thm}

For a discrete series representation $\pi\in\Pi_{2, \temp}(G)$, we write $\pi^\GL\in\Pi(\GL_{N(G), K_1})$ for the representation corresponding to the $L$-parameter $\tilde\phi_\pi^\GL$, and write $\bx{SuppCusp}^+(\pi)$ for the cuspidal support of $\pi^\GL$, called the \tbf{extended cuspidal support} of $\pi$.

\begin{lm}\label{cussinifienss}
If $\pi\in\Pi_{2, \temp}(G)$ is a discrete series representation, and if $(M, \pi_M)$ is a cuspidal support of $\pi$, then $\bx{SuppCusp}^+(\pi_M)=\bx{SuppCusp}^+(\pi)$.
\end{lm}
\begin{proof}
If $G$ is quasisplit, this is established in \cite[p. 309]{Moe14}. We note that by the endoscopic character identities Theorem~\ref{endoslicindiikehrnieiis}, the classical $L$-packets are the same as the $L$-packets defined in \cite[\S 6.4]{Moe14}. For details, the reader is referred to \cite[Proposition 2.10]{MHN24}.

When $G$ is not quasisplit, in view of the endoscopic character identities Theorem~\ref{endoslicindiikehrnieiis}, we can reduce the theorem to the quasisplit case using the alternative definition of $\bx{SuppCusp}^+(\pi)$ via Jacquet modules \cite{Moe14} and the fact that transfers are compatible with parabolic induction and the formation of Jacquet modules; see~\cite[Proposition 3.4.2]{She82} and \cite[\S 2.6]{Moe14}.
\end{proof}

\begin{prop}\label{compaitebirelclaisirelsi}
If $P\le G$ is a proper parabolic subgroup with Levi factor $M$, $\pi_M\in\Pi(M)$ with semisimplified $L$-parameter $\tilde\phi_{\pi_M}^\sems: W_{K_1}\to \LL M$,\footnote{The $L$-parameter $\tilde\phi_{\pi_M}^\sems$ is well-defined because $M$ is the product of a special orthogonal or unitary group with restrictions of general linear groups.} and $\pi\in\Pi(G)$ is a subquotient of the normalized parabolic induction $\bx I_P^G(\pi_M)$, then the semisimplified $L$-parameter $\tilde\phi_\pi^\sems$ is given by
\begin{equation*}
\tilde\phi_\pi^\sems: W_{K_1}\xr{\tilde\phi_{\pi_M}^\sems}\LL M\to \LL G.
\end{equation*}
\end{prop}
\begin{proof}
First, the assertion for general linear groups follows from the compatibility of the local Langlands correspondence for general linear groups with cuspidal support, by a result of Bernstein and Zelevinsky; cf.~\cite[Th\'eor\`eme 4]{Rod82}. We argue by induction on $n(G)$.

For nested Levi subgroups $M_1\le M_2$, write
\begin{equation*}
\iota^L_{M_1,M_2}:\LL M_1\longrightarrow\LL M_2
\end{equation*}
for the canonical $L$-homomorphism.

We first reduce to the case in which $(M,\pi_M)$ is a cuspidal support of $\pi$. Let $(M',\pi_{M'})$ be a cuspidal support of $\pi_M$. Transitivity of normalized parabolic induction implies that $(M',\pi_{M'})$ is also a cuspidal support of $\pi$. By the induction hypothesis applied to the classical-group factor of $M$, together with the general linear case,
\begin{equation*}
\tilde\phi_{\pi_M}^\sems=\iota^L_{M',M}\circ\tilde\phi_{\pi_{M'}}^\sems.
\end{equation*}
Consequently, once the assertion is proved for the cuspidal support $(M',\pi_{M'})$ of $\pi$, we obtain
\begin{equation*}
\tilde\phi_\pi^\sems=\iota^L_{M',G}\circ\tilde\phi_{\pi_{M'}}^\sems=\iota^L_{M,G}\circ\tilde\phi_{\pi_M}^\sems.
\end{equation*}
We may therefore assume that $(M,\pi_M)$ is a cuspidal support of $\pi$.

Suppose first that $\pi$ is non-tempered. By the Langlands classification of irreducible representations in terms of irreducible tempered representations (see~\cite{Sil78} and \cite[Theorem 3.5]{Kon03}), there are a parabolic subgroup $P_L\le G$ with Levi factor
\begin{equation*}
M_L\cong\Res_{K_1/K}\GL_{d_1}\times\cdots\times\Res_{K_1/K}\GL_{d_r}\times G(n_0),
\end{equation*}
tempered representations $\tau_i\in\Pi_\temp(\GL_{d_i,K_1})$ and $\pi_0\in\Pi_\temp(G(n_0))$, and real numbers $s_1\ge\cdots\ge s_r>0$, such that $\pi$ is the unique irreducible quotient of
\begin{equation*}
\bx I_{P_L}^G\paren{\paren{\tau_1\otimes\nu^{s_1}}\boxtimes\cdots\boxtimes\paren{\tau_r\otimes\nu^{s_r}}\boxtimes\pi_0}.
\end{equation*}
Compatibility of the classical correspondence with Langlands quotients, Theorem~\ref{coeleninifheihsn}, gives
\begin{align*}
\tilde\phi_\pi^\GL={}&\paren{\phi_{\tau_1}\otimes\largel{-}_{K_1}^{s_1}}\oplus\cdots\oplus\paren{\phi_{\tau_r}\otimes\largel{-}_{K_1}^{s_r}}\oplus\tilde\phi_{\pi_0}^\GL\\
&\oplus\paren{\phi_{\tau_r}^\theta\otimes\largel{-}_{K_1}^{-s_r}}\oplus\cdots\oplus\paren{\phi_{\tau_1}^\theta\otimes\largel{-}_{K_1}^{-s_1}}.
\end{align*}
The general linear case and the induction hypothesis for $G(n_0)$, together with transitivity of normalized parabolic induction and uniqueness of cuspidal support, now give the desired formula for $\tilde\phi_\pi^\sems$.

We may henceforth assume that $\pi$ is tempered. By the classification of tempered representations of classical groups \cite{Jan14}, there are a Levi subgroup $M_L\le G$ and a discrete series representation $\sigma\in\Pi_{2,\temp}(M_L)$ such that $\pi$ is an irreducible constituent of $\bx I_{P_L}^G(\sigma)$,
and
\begin{equation*}
\tilde\phi_\pi=\iota^L_{M_L,G}\circ\tilde\phi_\sigma.
\end{equation*}
More explicitly, if $\pi$ is not a discrete series representation and we write
\begin{equation*}
\tilde\phi_\pi^\GL=\bplus_{i\in I_{\tilde\phi}^+}m_i\phi_i\oplus \bplus_{i\in I_{\tilde\phi}^-}2m_i\phi_i\oplus\bplus_{i\in J_{\tilde\phi}}m_i(\phi_i\oplus\phi_i^\theta),
\end{equation*}
as in \S\ref{IFNieniehifeniws}, then either $m_i>1$ for some $i\in I_{\tilde\phi}^+$ or $m_i>0$ for some $i\in I_{\tilde\phi}^-\cup J_{\tilde\phi}$. In either case we can write
\begin{equation*}
\tilde\phi_\pi^\GL=\phi_\tau+\tilde\phi_{\pi_0}^\GL+\phi_\tau^\theta,
\end{equation*}
where $\tau\in\Pi_{2, \temp}(\GL_{d,K_1})$ is a discrete series representation and $\pi_0\in\Pi_\temp(G(n(G)-[K_1: K]d))$ is tempered. Then it follows from the local intertwining relations, Theorem~\ref{coeleninifheihsn}, that $\pi$ is a subrepresentation of $\bx I_{P_L}^G(\tau\otimes \pi_0)$, where $P_L\le G$ is a maximal parabolic subgroup with Levi factor
\begin{equation*}
M_L\cong\Res_{K_1/K}\GL_d\times G(n(G)-[K_1: K]d).
\end{equation*}

If $M_L$ is proper, its classical-group factor has geometric rank strictly smaller than $n(G)$. The induction hypothesis applied to $\sigma$, followed by transitivity of the canonical $L$-homomorphisms, proves the assertion for $\pi$. Thus it remains only to treat the case in which $\pi$ is a discrete series representation.

Write
\begin{equation*}
M\cong \Res_{K_1/K}\GL_{d_1}\times\cdots\times\Res_{K_1/K}\GL_{d_r}\times G(n_0)
\end{equation*}
and
\begin{equation*}
\pi_M=\tau_1\otimes\cdots\otimes\tau_r\otimes\pi_0\in\Pi_{\bx{sc}}(M).
\end{equation*}
Let
\begin{align*}
M_\diamond^\GL\defining{}&\Res_{K_1/K}\GL_{d_1}\times\cdots\times\Res_{K_1/K}\GL_{d_r}\times\Res_{K_1/K}\GL_{\frac{N(G)}{n(G)}n_0}\\
&\times\Res_{K_1/K}\GL_{d_r}\times\cdots\times\Res_{K_1/K}\GL_{d_1},
\end{align*}
which is a Levi subgroup of $G^\GL$. The diagram \eqref{ienfieunienfis}, together with \eqref{isniheiherehs}, identifies the image of $\tilde\phi_{\pi_M}^{\sems}$ in $\tilde\Phi^\sems(G)$ with a semisimple parameter for $M_\diamond^\GL$.

Let
\begin{equation*}
\tau_{M_\diamond}\in\Pi(M_\diamond^\GL)\quad\text{and}\quad\tau_G\in\Pi(G^\GL)
\end{equation*}
be the representations corresponding, respectively, to
\begin{equation*}
(\tilde\phi_{\pi_M})_\diamond^{\sems,\GL}\quad\text{and}\quad\tilde\phi_\pi^{\sems,\GL}.
\end{equation*}
These representations need not be tempered. By Lemma~\ref{cussinifienss}, the full general linear transfers associated with $\pi_M$ and $\pi$ have the same cuspidal support. In the general linear case, replacing a Weil--Deligne parameter by its semisimplification does not change the cuspidal support. Hence $\tau_{M_\diamond}$ and $\tau_G$ have the same cuspidal support. Compatibility of semisimplified general linear parameters with parabolic induction therefore gives
\begin{equation*}
(\tilde\phi_{\pi_M})_\diamond^{\sems,\GL}=\tilde\phi_\pi^{\sems,\GL}.
\end{equation*}
Finally, the standard representation determines the corresponding class in $\tilde\Phi^\sems(G)$, including in Case O2 after passage to $\bx O_{N(G)}(\bb C)$-conjugacy. Thus
\begin{equation*}
\tilde\phi_\pi^\sems=\iota^L_{M,G}\circ\tilde\phi_{\pi_M}^\sems,
\end{equation*}
as required.
\end{proof}

\subsection{Supercuspidal \texorpdfstring{$L$}{L}-packets}

We recall a result of M{\oe}glin and Tadi\'c that characterizes supercuspidal $L$-parameters in terms of their corresponding $L$-packets.

Let $\tilde\phi\in\tilde\Phi_2(G^*)$ be a discrete $L$-parameter. We write $\bx{Jord}(\tilde\phi)$ for the set of irreducible subrepresentations of $W_{K_1}\times\SL_2(\bb C)$ contained in $\tilde\phi^\GL$.

An $L$-parameter $\tilde\phi$ is said to be \tbf{without gaps} if for every $\rho\boxtimes\sp_a\in\bx{Jord}(\tilde\phi)$ with $a>2$, one also has $\rho\boxtimes\sp_{a-2}\in\bx{Jord}(\tilde\phi)$.

We recall the following characterization of supercuspidal representations of $G(K)$:

\begin{prop}\label{chainsienicusonso}
Suppose $\tilde\phi\in\tilde\Phi_2(G^*)$ is a discrete $L$-parameter and $\eta\in \Irr(\ovl{\mfk S}_{\tilde\phi})$ view as a character of $\mfk S_{\tilde\phi}$ via inflation., then $\tilde\pi_{\mfk w, \uno}(\tilde\phi, \eta)\in\tilde\Pi_{\tilde\phi}(G^*)$ is supercuspidal if and only if the following two conditions hold:
\begin{itemize}
\item
$\tilde\phi$ is without gaps,
\item
$\eta(e_{\rho\boxtimes\sp_a})=-\eta(e_{\rho\boxtimes\sp_{a-2}})$ for each $\rho\boxtimes\sp_a\in\bx{Jord}(\tilde\phi)$ with $a\ge 2$, where we assume that $\eta(e_{\rho\boxtimes\sp_0})=1$.
\end{itemize}
\end{prop}
\begin{proof}
This is established in Case O in \cite{Moe02, M-T02}, see also \cite[Theorem 1.5.1]{Moe11} and \cite[Theorem 3.3]{Xu17}, and in Case U in \cite[Theorem 8.4.4]{Moe07}. Note that by endoscopic character identities Theorem~\ref{endoslicindiikehrnieiis}, the LLC defined in \cite{Mok15} is the same as the LLC defined in \cite{Moe07}; cf.~\cite[Proposition 2.10]{MHN24}.
\end{proof}

We then obtain the following corollary characterizing supercuspidal $L$-packets:

\begin{cor}\label{superisnidnLpaifniesues}
For a discrete $L$-parameter $\tilde\phi\in\tilde\Phi_2(G^*)$, $\tilde\Pi_{\tilde\phi}(G^*)$ consists of supercuspidal representations if and only if $\tilde\phi$ is a supercuspidal $L$-parameter.
\end{cor}
\begin{proof}
If $\tilde\phi$ is supercuspidal, then clearly it is without gaps, and it follows from Proposition~\ref{chainsienicusonso} that $\tilde\Pi_{\tilde\phi}(G^*)$ consists of supercuspidal representations.

Conversely, suppose $\tilde\Pi_{\tilde\phi}(G^*)$ consists of supercuspidal representations, then $\tilde\phi$ is without gaps by Proposition~\ref{chainsienicusonso}. If $\rho\boxtimes\sp_a\in \bx{Jord}(\tilde\phi)$ and $a\ge 2$, then there exists a character $\eta$ of $\mfk S_{\tilde\phi}$ such that  $\eta(z_{\tilde\phi})=1$ and $\eta(e_{\rho\boxtimes\sp_a})=\eta(e_{\rho\boxtimes\sp_{a-2}})$. So it follows from Proposition~\ref{chainsienicusonso} that $\tilde\pi_{\mfk w, \uno}(\tilde\phi, \eta)\in \tilde\Pi_{\tilde\phi}(G^*)$ is not supercuspidal, contradiction. Thus we conclude that $a=1$ for each $\rho\boxtimes\sp_a\in \bx{Jord}(\tilde\phi)$, which means that $\tilde\phi$ is supercuspidal.
\end{proof}

\subsection{Combinatorics on \texorpdfstring{$L$}{L}-parameters}\label{combianosiLoifn}

Following \cite[\S 2.2.4]{MHN24}, we give a combinatorial description of discrete $L$-packets that will be used in the explicit computations in the proof of the Kottwitz conjecture.

For $b\in B(G)_\bas$, if $\tilde\phi\in \tilde\Phi_2(G^*)$ is a discrete $L$-parameter and $\tilde\pi\in\tilde\Pi_{\tilde\phi}(G), \tilde\rho\in\tilde\Pi_{\tilde\phi}(G_b)$, we can define a character
\begin{equation*}
\delta[\tilde\pi, \tilde\rho]=\iota_{\mfk w, b_0}(\tilde\pi)^\vee\otimes\iota_{\mfk w, b_0+b}(\tilde\rho)\in \Irr(\mfk S_{\tilde\phi}),
\end{equation*}
which can be thought of as measuring the relative position of $\tilde\pi$ and $\tilde\rho$. This character is independent of the Whittaker datum $\mfk w$ and also $b_0$, because changing those shifts $\iota_{\mfk w, b_0}$ by a character of $\mfk S_{\tilde\phi}$ \cite[Lemma 2.3.3]{HKW22}.

Let $b_1\in B(G)_\bas$ be the unique nontrivial basic element. Because $\tilde\phi$ is discrete, the packet $\tilde\Pi_{\tilde\phi}(G)$ (resp. $\tilde\Pi_{\tilde\phi}(G_{b_1})$) has size $\#\mfk S_{\tilde\phi}/2$ and corresponds via $\iota_{\mfk w, b_0}$ (resp. $\iota_{\mfk w, b_0+b_1}$) to those characters $\eta$ of $\mfk S_{\tilde\phi}$ such that  $\eta(z_{\tilde\phi})=\kappa_{b_0}(-\uno)$ (resp. $\eta(z_{\tilde\phi})=-\kappa_{b_0}(-\uno)$).

Write
\begin{equation*}
\tilde\phi^\GL=\phi_1\oplus\cdots\oplus\phi_k\oplus\phi_{k+1}\oplus\cdots\oplus\phi_r,
\end{equation*}
where the $\phi_i$ are irreducible representations of $W_{K_1}\times\SL_2(\bb C)$ and $\dim\phi_i$ is odd if and only if $i\le k$. For subsets $I, J\subset [r]_+$, define their symmetric difference by
\begin{equation*}
I\oplus J=(I\setm J)\cup(J\setm I),
\end{equation*}
In Case O2, define an equivalence relation $\sim_k$ on $\mrs P([r]_+)$ by
\begin{equation*}
I_1\sim_k I_2\iff I_1=I_2\quad\text{or}\quad I_1=I_2\oplus[k]_+.
\end{equation*}
Then we can define the symmetric differences of equivalence classes in $\mrs P([r]_+)/\sim_k$. Note that we can talk about parity of cardinality of equivalence classes of $\mrs P([r]_+)$, because $k$ is always an even number in Case O2.

To unify notation, in Case O1 or Case U, let $\sim_k$ be the trivial equivalence relation. There is a bijection
\begin{equation*}
\mrs P([r]_+)/\sim_k\xr\sim\Irr(\mfk S_{\tilde\phi}),\qquad[I]\longmapsto\eta_{[I]}\defining\sum_{i\in I}e_i^\vee.
\end{equation*}
via the map
\begin{equation*}
[I]\in\mrs P([r]_+)/\sim_k\mapsto \eta_{[I]}\defining\sum_{i\in I}e_i^\vee\in \Irr(\mfk S_{\tilde\phi}).
\end{equation*}
The right-hand side is independent of the representative $I$ of $[I]$. If
\begin{equation*}
\#[I]\equiv\frac{\kappa_{b_0}(-\uno)-1}{2}\modu2,
\end{equation*}
set
\begin{equation*}
\tilde\pi_{[I]}\defining\tilde\pi_{\mfk w,b_0}(\tilde\phi,\eta_{[I]}).
\end{equation*}
If
\begin{equation*}
\#[I]\equiv\frac{\kappa_{b_0}(-\uno)+1}{2}\modu2,
\end{equation*}
set
\begin{equation*}
\tilde\pi_{[I]}\defining\tilde\pi_{\mfk w,b_0+b_1}(\tilde\phi,\eta_{[I]}).
\end{equation*}
For any $[I],[J]\in\mrs P([r]_+)/\sim_k$, we then have
\begin{equation*}
\delta[\tilde\pi_{[I]},\tilde\pi_{[J]}]=\eta_{[I\oplus J]}=\sum_{i\in I\oplus J}e_i^\vee.
\end{equation*}

Let $\mu_1\in X_\bullet(G)$ be the dominant cocharacter given by
\begin{equation}\label{idnifdujiherheins}
z\mapsto
\begin{cases}\diag(z, \underbrace{1, \ldots, 1}_{d(G)-1\text{-many}}) &\text{in Case U}\\
\diag(z, \underbrace{1, \ldots, 1}_{d(G)-2\text{-many}}, z^{-1})&\text{in Case O}
\end{cases},
\end{equation}
(i.e., $\mu_1=\omega_1^\vee$ in the standard notation), where in Case U we use the standard realization of $G_{\ovl K}\cong \GL_{n(G), \ovl K}$, and in Case O we use the standard realization of $G_{\ovl K}$ as the subgroup of $\SL_{d(G),\ovl K}$ preserving the nondegenerate bilinear form on $\ovl K\{v_1, \ldots, v_{d(G)}\}$ defined by $\bra{v_i, v_j}=(1+\delta_{i, j})\delta_{i, d(G)+1-j}$.

Then the highest weight tilting module $\mcl T_{\mu_1}$ of $\hat G$ is the representation $\hat\Std_G$ of $\hat G$.

\begin{thm}\label{pdidiiidjifenis}
Let $\tilde\phi\in\tilde\Phi_2(G)$ be a discrete $L$-parameter with a decomposition
\begin{equation*}
\tilde\phi^\GL=\phi_1\oplus\cdots\oplus\phi_k\oplus\phi_{k+1}\oplus\cdots\oplus\phi_r,
\end{equation*}
where the $\phi_i$ are distinct irreducible representations of $W_{K_1}\times\SL_2(\bb C)$ of dimension $d_i$, and $d_i$ is odd if and only if $i\le k$. Let $b_1\in B(G)_\bas$ be the unique nontrivial basic element. Suppose that
\begin{equation*}
\tilde\rho=\tilde\pi_{[I]}\in\tilde\Pi_{\tilde\phi}(G_{b_1}),
\end{equation*}
and write $\tilde\pi=\tilde\pi_{[J]}\in\tilde\Pi_{\tilde\phi}(G)$. Then
\begin{equation*}
\Hom_{\mfk S_{\tilde\phi}}\paren{\delta[\tilde\pi,\tilde\rho],\tilde\phi^\GL}\cong\bplus_{\substack{t\in[r]_+\\
[J]=[I\oplus\{t\}]}}\phi_t
\end{equation*}
as representations of $W_{K_1}\times\SL_2(\bb C)$.

In particular, the direct sum on the right has at most one summand, except in Case O2 with $k=2$. In that exceptional case,
\begin{equation*}
[I\oplus\{1\}]=[I\oplus\{2\}],
\end{equation*}
and the corresponding summand is $\phi_1\oplus\phi_2$.
\end{thm}
\begin{proof}
By the preceding description,
\begin{equation*}
\delta[\tilde\pi,\tilde\rho]=\eta_{[I\oplus J]}.
\end{equation*}
The group $\mfk S_{\tilde\phi}$ acts on the summand $\phi_t$ of $\tilde\phi^\GL$ through the character $e_t^\vee=\eta_{[\{t\}]}$. Hence
\begin{align*}
\Hom_{\mfk S_{\tilde\phi}}\paren{\delta[\tilde\pi,\tilde\rho],\phi_t}\ne0
&\iff\eta_{[I\oplus J]}=\eta_{[\{t\}]}\\
&\iff[I\oplus J]=[\{t\}]\\
&\iff[J]=[I\oplus\{t\}].
\end{align*}
This proves the displayed direct-sum formula.

It remains to determine when two different indices can contribute. If
\begin{equation*}
[I\oplus\{t\}]=[I\oplus\{u\}]\qquad (t\ne u),
\end{equation*}
then the equivalence relation must be nontrivial, so we are in Case O2, and
\begin{equation*}
\{t,u\}=[k]_+.
\end{equation*}
Since the left-hand side has cardinality two, this is possible exactly when $k=2$ and $\{t,u\}=\{1,2\}$. Thus the only repeated class contributes $\phi_1\oplus\phi_2$.
\end{proof}

\section{Local Langlands correspondence via moduli of local shtukas}\label{local-abglancorrepsodeifn}

In this section, we review the construction of the Fargues--Scholze local Langlands correspondence.

Let $p$ be a rational prime, let $K/\bb Q_p$ be a finite extension, and let $\msf G$ be a connected reductive group over $K$. Let $\ell$ be a rational prime different from $p$, and let $\Lbd\in\{\ovl{\bb Q_\ell},\ovl{\bb F_\ell}\}$ be such that $\#\pi_0(Z(\msf G))$ is invertible in $\Lbd$.\footnote{This is the coefficient hypothesis used in \cite{F-S24} to avoid additional complications in the $\ell$-modular setting.} Fix an isomorphism $\iota_\ell: \bb C\to \ovl{\bb Q_\ell}$ sending the chosen complex square root of $p$ to a chosen square root in $\ovl{\bb Q_\ell}$, whose reduction determines a square root in $\ovl{\bb F_\ell}$.

We use $L$-groups and $L$-parameters with $\Lbd$-coefficients. Let $\Phi(\msf G, \Lbd)$ denote the set of $L$-parameters
\begin{equation*}
W_K\times \SL_2(\Lbd)\longrightarrow \LL\msf G(\Lbd),
\end{equation*}
and let $\Pi(\msf G, \Lbd)$ denote the set of irreducible smooth representations of $\msf G(K)$ with coefficients in $\Lbd$.

\subsection{The correspondence}\label{icoroenfiehsnw}

We briefly recall the construction of the Fargues--Scholze local Langlands correspondence. The Kottwitz set $B(\msf G)$, as defined in \cite{Kot85}, consists of $\vp_K$-equivalence classes of $\msf G(\breve K)$, i.e.,
\begin{equation*}
\msf b\sim\msf b'\iff\msf b'=\msf g^{-1}\msf b\vp_K(\msf g) \text{ for some }\msf g\in\msf G(\breve K).
\end{equation*}
Each element $\msf b\in B(\msf G)$ is determined by two invariants: The Kottwitz invariant $\kappa_{\msf G}(\msf b)$ and the $\msf G(\breve K)$-conjugacy class of its slope homomorphism (or Newton map)
\begin{equation*}
\nu_{\msf b}:\mbf D_{\breve K}\longrightarrow\msf G_{\breve K},
\end{equation*}
where $\mbf D$ is the pro-algebraic diagonalizable group whose character group is $\bb Q$.

An element $\msf b\in B(\msf G)$ is called \tbf{basic} if $\nu_{\msf b}$ is central in $\msf G$; the set of basic elements in $B(\msf G)$ is denoted $B(\msf G)_\bas$. An element $\msf b\in B(\msf G)$ is called \tbf{unramified} if it lies in the image of the map $B(\msf T)\to B(\msf G)$ for some maximal torus $\msf T\le\msf G$. Denote by $B(\msf G)_{\bx{un}}\subset B(\msf G)$ the subset of unramified elements. These are precisely those $\msf b\in B(\msf G)$ for which the twisted centralizer $\msf G_{\msf b}$ is quasisplit; see~\cite[Lemma~2.12]{Ham24}.

We recall some material from~\cite{S-W20} and~\cite{Far16} regarding the relative Fargues--Fontaine curve. For any $S\in\Perfd_\kappa$, the associated Fargues--Fontaine curve $X_S$ is defined as in~\cite{F-F18}. When $S=\Spa(R, R^+)$ is affinoid with pseudo-uniformizer $\varpi$, the adic space $X_S$ is defined as follows:
\begin{equation*}
Y_S=\Spa(W_{\mcl O_K}(R^+))\setm\{p[\varpi]=0\},
\end{equation*}
\begin{equation*}
X_S=Y_S/\vp_K^{\bb Z},
\end{equation*}

For any affinoid perfectoid space $S$ over $\kappa$, the following sets are canonically in bijection:
\begin{enumerate}
\item
$\Spd(K)(S)$,
\item
Untilts $S^\sharp$ of $S$ over $K$,
\item
Cartier divisors of $Y_S$ of degree 1.
\end{enumerate}
For any untilt $S^\sharp$ of $S$, we denote by $D_{S^\sharp}\subset Y_S$ the corresponding divisor.

By~\cite[Theorem \Rmnum{3}.0.2]{F-S24}, the presheaf $\Bun_{\msf G}$ on $\Perfd_{\ovl\kappa}$, which assigns to each $S$ the groupoid of $\msf G$-torsors on $X_S$, is a small Artin $v$-stack. For any $S\in\Perfd_{\ovl\kappa}$, there exists a functor $\msf b\mapsto \mcl E^{\msf b}$ from the category of isocrystals with $\msf G$-structure to $\Bun_{\msf G}(S)$. When $S=\Spd(\ovl\kappa)$, this map induces a bijection from $B(\msf G)$ to the set of isomorphism classes of $\msf G$-bundles on $X_S$. More precisely, there exists a homeomorphism $\largel{\Bun_{\msf G}}\to B(\msf G)$ sending $\mcl E^{\msf b}$ to $\msf b$; see~\cite{Far20, Ans23, Vie24}.

For any $\msf b\in B(\msf G)$, the sub-functor
\begin{equation*}
\Bun^{\msf b}_{\msf G}\defining \Bun_{\msf G}\times_{\largel{\Bun_{\msf G}}}\{\msf b\}\subset \Bun_{\msf G}
\end{equation*}
is locally closed. It can be identified with the classifying stack $[\Spd(\ovl\kappa)/\tilde{\msf G}_{\msf b}]$, where $\tilde{\msf G}_{\msf b}$ denotes the $v$-sheaf of groups given by $S\mapsto \Aut_{X_S}(\mcl E^{\msf b})$; see~\cite[Proposition \Rmnum3.5.3]{F-S24}. The natural morphism $\Bun_{\msf G}^{\msf b}\to [\Spd(\ovl\kappa)/\udl{\msf G_{\msf b}(K)}]$ induces equivalences of categories
\begin{equation*}
\bx D(\msf G_{\msf b}, \Lbd)\cong \bx D_{\bx{lis}}([\Spd(\ovl\kappa)/\udl{\msf G_{\msf b}(K)}], \Lbd)\cong \bx D_{\bx{lis}}(\Bun_{\msf G}^{\msf b}, \Lbd),
\end{equation*}
see~\cite[Theorem \Rmnum7.7.1]{F-S24}.

Writing $i_{\msf b}$ for the inclusion $\Bun^{\msf b}_{\msf G}\subset \Bun_{\msf G}$, any $\pi\in\Pi(\msf G_{\msf b}, \Lbd)$ may be regarded as an object in $\bx D_{\bx{lis}}(\Bun_{\msf G}, \Lbd)$ via the extension by zero $(i_{\msf b})_!$, which is well-defined by~\cite[Proposition \Rmnum7.7.3, Proposition \Rmnum7.6.7]{F-S24}. Moreover, when $\msf b$ is basic, the map  $\Bun_{\msf G}^{\msf b}\to [\Spd(\ovl\kappa)/\udl{\msf G_{\msf b}(K)}]$ is an isomorphism; see~\cite[Proposition \Rmnum3.4.5]{F-S24}.

We now introduce the Hecke operators. For each finite index set $I$, let $\bx{Rep}_\Lbd(\LL\msf G^I)$ denote the category of algebraic representations of $I$ copies of $\LL\msf G(\Lbd)$ over $\Lbd$, and let $\Div^I$ be the $I$-fold product of the mirror curve $\Div^1\defining\Spd(\breve K)/\vp_K^{\bb Z}$. The diamond $\Div^1$ represents the functor that sends $S\in \Perfd_{\ovl\kappa}$ to the set of degree 1 Cartier divisors on $X_S$.

We then have the Hecke correspondence diagram
\begin{equation}\label{isnihinieifheis}
\begin{tikzcd}
 &\Hec_{\msf G, I}\ar[ld, "h^\from"']\ar[rd, "h^\to\times\supp"]&\\
 \Bun_{\msf G} & & \Bun_{\msf G}\times \Div^I
\end{tikzcd},
\end{equation}
where $\Hec_{\msf G, I}$ represents the functor sending $S\in \Perfd_{\ovl\kappa}$ to the groupoid of tuples
\begin{equation*}
(\mcl E_1, \mcl E_2, \beta, (D_i)_{i\in I}),
\end{equation*}
where $D_i\subset X_S$ are Cartier divisors, $\mcl E_1, \mcl E_2$ are $\msf G$-torsors on $X_S$, together with an isomorphism
\begin{equation*}
\beta: \mcl E_1|_{X_S\bsh\cup_{i\in I}D_i}\xr\sim \mcl E_2|_{X_S\bsh\cup_{i\in I}D_i}.
\end{equation*}
The morphism $h^\from$ sends the tuple to $\mcl E_1$, while $h^\to\times\supp$ sends it to $(\mcl E_2, (D_i)_{i\in I})$.

For each $W\in \bx{Rep}_\Lbd(\LL\msf G^I)$, Fargues and Scholze define a solid $\Lbd$-sheaf $\mcl S'_W\in \bx D_\blacksquare(\Hec_{\msf G, I}, \Lbd)$ via the geometric Satake correspondence; see~\cite[Theorem~\Rmnum1.6.3]{F-S24}. The corresponding Hecke operator
\begin{equation}\label{Heivneinekdnfienis}
\bx T_W: \bx D_{\bx{lis}}(\Bun_{\msf G}, \Lbd)\to \bx D_\blacksquare(\Bun_{\msf G}\times\Div^I, \Lbd): A\mapsto \bx R(h^\to\times\supp)_\natural(h^{\from*}(A)\Ltimes \mcl S'_W).
\end{equation}
Here $\bx R(h^\to\times\supp)_\natural$ is the natural push-forward, left adjoint to restriction; \cite[Proposition \Rmnum{7}.3.1]{F-S24}. The results of \cite[Theorem~\Rmnum{1}.7.2, Proposition~\Rmnum{9}.2.1, Corollary~\Rmnum{9}.2.3]{F-S24} show that, after taking the geometric fiber over $\Div^I$ used in the construction of excursion operators, this object is lisse and carries a canonical continuous $W_K^I$-action. By abuse of notation, we therefore also write
\begin{equation*}
\bx T_W: \bx D_{\bx{lis}}(\Bun_{\msf G},\Lbd)\longrightarrow\bx D_{\bx{lis}}(\Bun_{\msf G},\Lbd)^{\bx BW_K^I}
\end{equation*}
for the resulting functor.

For any $\pi\in\Pi(\msf G, \Lbd)$, the Fargues--Scholze $L$-parameter $\phi_\pi^\FS$ is extracted from the action of excursion operators on $\pi$. Consider a tuple
\begin{equation*}
(I, W, (\gamma_i)_{i\in I}, \alpha, \beta),
\end{equation*}
where
\begin{itemize}
\item
$I$ is a finite index set, and $\Delta_I: \LL\msf G\to\LL\msf G^I$ for the diagonal embedding;
\item
$(r_W, W)\in \bx{Rep}_\Lbd(\LL\msf G^I)$ is an algebraic representation;
\item
$\gamma_i\in W_K$ for every $i\in I$;
\item
\begin{equation*}
\alpha:\uno\longrightarrow\Delta_I^*W,\qquad\beta:\Delta_I^*W\longrightarrow\uno
\end{equation*}
are morphisms in $\bx{Rep}_\Lbd(\LL\msf G)$.
\end{itemize}
The associated excursion operator is the natural transformation of the identity functor given by
\begin{equation*}
\id=\bx T_\uno\xr{\bx T_\alpha}\bx T_{\Delta_I^*W}\cong\bx T_W\xr{(\gamma_i)_{i\in I}}\bx T_W\cong\bx T_{\Delta_I^*W}\xr{\bx T_\beta}\bx T_\uno=\id.
\end{equation*}
Evaluating this natural transformation at $\pi$ gives a scalar by Schur's lemma. The relations among these scalars, together with Lafforgue's reconstruction theorem \cite[Proposition~11.7]{Laf18}, determine the Fargues--Scholze parameter $\phi_\pi^\FS: W_K\to \LL\msf G(\Lbd)$ characterized by the requirement that the above excursion operator acts on $\pi$ through the scalar defined by
\begin{equation*}
\Lbd\xrightarrow{\alpha}W\xrightarrow{r_W\paren{\paren{\phi_\pi^\FS(\gamma_i)}_{i\in I}}}W\xrightarrow{\beta}\Lbd.
\end{equation*}
See~\cite[Definition/Proposition~\Rmnum{9}.4.1]{F-S24}.

Then Fargues and Scholze~\cite[Theorem \Rmnum{1}.9.6]{F-S24} showed that their construction has various desirable properties: 

\begin{thm}\label{compaitbsilFaiirfies}
The Fargues--Scholze LLC $\rec_{\msf G}^\FS: \Pi(\msf G, \Lbd)\to \Phi^\sems(\msf G, \Lbd)$ satisfies the following compatibility properties: 
\begin{enumerate}
\item
(Compatibility with local class field theory) If $\msf G$ is a torus, then $\rec_{\msf G}^\FS$ is the usual local Langlands correspondence for tori constructed from local class field theory,
\item
(Compatibility with natural operations) $\rec_{\msf G}^\FS$ is compatible with character twists, central characters, and taking contragredients,
\item
(Compatibility with products) If $\msf G=\msf G_1\times\msf G_2$ is a product of two groups and $\pi_i\in \Pi(\msf G_i, \Lbd)$ for each $i\in\{1, 2\}$, then $\rec_{\msf G}^\FS(\pi_1\boxtimes\pi_2)=\rec_{\msf G_1}^\FS(\pi_1)\times\rec_{\msf G_2}^\FS(\pi_2)$,
\item
(Compatibility with central extensions) If $\msf G'\to \msf G$ is a map of reductive groups inducing an isomorphism on adjoint groups, $\pi\in \Pi(\msf G, \Lbd)$, and $\pi'\in\Pi(\msf G', \Lbd)$ is an irreducible admissible subquotient of $\pi|_{\msf G'(K)}$, then $\rec_{\msf G'}^\FS(\pi')$ is given by $\rec^\FS_{\msf G}(\pi)$ composed with the natural map $\LL G(\Lbd)\to \LL G'(\Lbd)$. In particular, $\rec_{\msf G'}^\FS$ commutes with contragredients and Chevalley involutions.
\item
(Compatibility with parabolic induction) If $\msf P\le\msf G$ is a parabolic subgroup with Levi factor $\msf M$ and $\pi_{\msf M}\in\Pi(\msf M, \Lbd)$, then for any subquotient $\pi$ of the normalized parabolic induction $\bx I_{\msf P}^{\msf G}(\pi_{\msf M})$, $\rec_{\msf G}^\FS(\pi)$ is given by the composition
\begin{equation*}
W_K\xr{\rec_{\msf M}^\FS(\pi_{\msf M})}\LL\msf M(\Lbd)\to \LL\msf G(\Lbd),
\end{equation*}
where $\LL\msf M(\Lbd)\to \LL\msf G(\Lbd)$ is the canonical embedding.
\item
(Compatibility with Harris--Taylor/Henniart LLC) If $\Lbd=\ovl{\bb Q_\ell}$ and $\msf G=\GL_n$, then $\rec_{\GL_n}^\FS$ coincides with the (semisimplified) local Langlands correspondence given by Harris--Taylor and Henniart in the sense of \textup{Theorem~\ref{compaiteiehdnifds}}.
\item
(Compatibility with Weil restriction) If $\msf G=\Res_{K'/K}\msf G'$ for some reductive group $\msf G'$ over a finite extension $K'/K$, then $\rec_{\msf G'}^\FS(\pi)=\rec_{\msf G}^\FS(\pi)|_{W_{K'}}$ for any $\pi\in \Pi(\msf G', \Lbd)=\Pi(\msf G, \Lbd)$.
\item
(Compatibility with contragredients) For $\pi\in \Pi(\msf G, \Lbd)$, we have $\rec_{\msf G}^\FS(\pi^\vee)=\vartheta_{\hat{\msf G}}\paren{\rec_{\msf G}^\FS(\pi)}$; see~\textup{\cite[Proposition~\Rmnum{9}.5.3]{F-S24}}. Here we recall that for each Borel pair $(\hat{\msf B}, \hat{\msf T})$ of $\hat{\msf G}$, there exists a Cartan involution (also called the Chevalley involution) $\vartheta_{\hat{\msf G}}$ of $\hat{\msf G}$ which preserves the fixed pinning and acts as $t\mapsto w_0(t^{-1})$ on $\hat{\msf T}$, where $w_0$ is the longest Weyl element taking $\hat{\msf B}$ to the opposite Borel group. $\vartheta_{\hat{\msf G}}$ extends to an involution of $\LL\msf G$ because the action of $W_K$ on $\hat{\msf G}$ preserves the pinning.
\item
(Compatibility with reduction modulo $\ell$) If $\pi\in \Pi(\msf G, \ovl{\bb Q_\ell})$ admits a $\msf G(K)$-stable $\ovl{\bb Z_\ell}$-lattice, and the reduction modulo $\ell$ representation of $\pi$ has an irreducible subquotient $\ovl\pi'$, then $\rec_{\msf G}^\FS(\pi): W_K\to \LL\msf G(\ovl{\bb Q_\ell})$ factors through $\LL\msf G(\ovl{\bb Z_\ell})$, and its reduction modulo $\ell$ equals $\rec_{\msf G}^\FS(\ovl\pi')$; see~\textup{\cite[\S \Rmnum{9}.5.2]{F-S24}}.
\end{enumerate}
\end{thm}

Moreover, Hansen, Kaletha and Weinstein established in \cite[Theorem 6.6.1]{HKW22} that the Fargues--Scholze LLC coincides with the usual (semisimplified) LLC for inner forms of general linear groups given in \cite{DKV84, Rog83} via the Jacquet--Langlands correspondence.

When $\msf G^\sharp$ is an inner form of $\GL_n$ and $\msf G=[\msf G^\sharp, \msf G^\sharp]$, the local Langlands correspondence constructed in \cite{G-K82, Tad92, H-S12, ABPS} assigns to each $\pi^\sharp\in\Pi(\msf G^\sharp)$ a parameter $\phi_{\pi^\sharp}$ with the property that, for any irreducible admissible subrepresentation $\pi\subset\pi^\sharp|_{\msf G(K)}$, $\phi_\pi$ equals $\phi_{\pi^\sharp}$ composed with the natural map $\LL\msf G^\sharp\to \LL\msf G$. This prescription uniquely determines the local Langlands correspondence for $\msf G$, since every $\pi\in\Pi(\msf G)$ occurs as a subrepresentation of $\pi^\sharp|_{\msf G(K)}$ for some $\pi^\sharp\in\Pi(\msf G^\sharp)$, by \cite[Lemma 2.3]{G-K82}.

Combining the compatibility of the Fargues--Scholze LLC with the LLC for
inner forms of general linear groups with its compatibility under
homomorphisms inducing an isomorphism on adjoint groups
(see~Theorem~\ref{compaitbsilFaiirfies}), we obtain the following.

\begin{prop}\label{inenfomeieSLIwninsfds}
Let $\msf G$ be an inner form of $\SL_N$ over $K$ for some positive integer $N\in \bb Z_+$, and suppose that
$\Lbd=\ovl{\bb Q_\ell}$. Then the Fargues--Scholze LLC is compatible with the
LLC constructed in \textup{\cite{G-K82, Tad92, H-S12, ABPS}} in the sense of
\textup{Theorem~\ref{compaiteiehdnifds}}.
\end{prop}

\begin{proof}
Choose an inner form $\msf G^\sharp$ of $\GL_N$ such that $[\msf G^\sharp,\msf G^\sharp]=\msf G$,
and let
\begin{equation*}
q:\LL\msf G^\sharp\longrightarrow\LL\msf G
\end{equation*}
be the natural map. Let $\pi\in\Pi(\msf G)$. By \textup{\cite[Lemma~2.3]{G-K82}}, there exists
$\pi^\sharp\in\Pi(\msf G^\sharp)$ such that $\pi$ is an irreducible
constituent of $\pi^\sharp|_{\msf G(K)}$. By the construction of the classical LLC for inner forms of $\SL_N$, $\rec_{\msf G}(\pi)=q\circ\rec_{\msf G^\sharp}(\pi^\sharp)$ up to $\hat{\msf G}$-conjugacy.

On the other hand, the compatibility of the Fargues--Scholze LLC under maps
inducing an isomorphism on adjoint groups, as stated in
Theorem~\ref{compaitbsilFaiirfies}, gives
\[
\rec_{\msf G}^{\FS}(\pi)
=
q\circ\rec_{\msf G^\sharp}^{\FS}(\pi^\sharp).
\]
By \textup{\cite[Theorem~6.6.1]{HKW22}}, the Fargues--Scholze LLC for
$\msf G^\sharp$ agrees with the semisimplification of the classical LLC.
Consequently,
\[
\begin{aligned}
\rec_{\msf G}^{\FS}(\pi)
&=
q\circ\rec_{\msf G^\sharp}^{\FS}(\pi^\sharp)\\
&=
q\circ\rec_{\msf G^\sharp}(\pi^\sharp)^\sems\\
&=
\paren{q\circ\rec_{\msf G^\sharp}(\pi^\sharp)}^\sems\\
&=
\rec_{\msf G}(\pi)^\sems.
\end{aligned}
\]
This is precisely the compatibility asserted in
Theorem~\ref{compaiteiehdnifds}.
\end{proof}

\subsection{Local shtuka spaces}\label{secitonautnciniauotenries}

In this subsection we introduce the local shtuka spaces whose cohomology will be one of the main objects of study of this paper.

Let $\msf G^*$ be the quasisplit inner form of $\msf G$ with a fixed Borel subgroup $\msf B^*$ containing a maximal torus $\msf T^*$, and let $\mu$ be a dominant cocharacter of $\msf G^*_{\ovl K}$ with reflex field $E_\mu$. We first recall the notion of neutral $\mu$-acceptable elements from \cite[Definition 2.3]{R-V14}: Let $\mu^\natural$ be the image of $\mu$ in $X^\bullet\paren{Z(\hat{\msf G^*})^{\Gal_K}}$, and let
\begin{equation*}
\ovl\mu\defining\frac{1}{[E_\mu: K]}\sum_{\gamma\in \Gal_K/\Gal_{E_\mu}}\gamma(\mu)\in X_\bullet(\msf T^*)_{\bb Q}^{+, \Gal_K}.
\end{equation*}
For $\msf b_0\in B(\msf G^*)$, define the set of neutral $\mu$-acceptable elements by
\begin{equation}\label{seinieunifnws}
B(\msf G^*,\msf b_0, \mu)\defining \{\msf b\in B(\msf G^*)|\kappa_{\msf G^*}(\msf b)-\kappa_{\msf G^*}(\msf b_0)=\mu^\natural,\quad \ovl\mu-(\nu_{\msf b}-\nu_{\msf b_0})\in \bb R_+\Phi^\vee(\msf G^*,\msf T^*)^+\}
\end{equation}
where we regard the Newton maps $\nu_{\msf b}, \nu_{\msf b_0}: \mbf D(\bb Q)\to\msf G^*_{\ovl K}$ as elements of $X_\bullet(\msf T^*)_{\bb Q}^{+, \Gal_K}$.

For example, for the quasisplit group $G^*$ defined in \S\ref{theogirneidnis} with a pure inner twist $(G=G^*_{b_0}, \varrho_{b_0}, z_{b_0})$ and the geometric cocharacter $\mu_1\in X_\bullet(G^*)$ defined in \S\ref{combianosiLoifn}, the unique basic element $b_1\in B(G^*, b_0, \mu_1)_\bas$ is just the unique nontrivial basic element $b_1$ of $B(G)_\bas$ under the isomorphism~\eqref{ienufheiehfniess}.

We then recall Scholze's definition \cite[\S 23]{S-W20} of the local shtuka space in the basic case: For each $\msf b_0\in B(\msf G^*)_\bas$ and $\msf b\in B(\msf G^*, \msf b_0, \mu)_\bas$, the local shtuka space $\Sht(\msf G^*,\msf b,\msf b_0,\mu)$ is a local spatial diamond over $\Spd(\breve KE_\mu)$ that represents the functor that maps $S\in \Perfd_{\breve KE_\mu}$ to the set of isomorphisms
\begin{equation*}
\gamma: \mcl E^{\msf b}|_{X_{S^\flat}\setm D_S}\xr\sim \mcl E^{\msf b_0}|_{X_{S^\flat}\setm D_S},
\end{equation*}
of $\msf G^*$-bundles that are meromorphic along the divisor $D_S\subset X_{S^\flat}$ and bounded by $\mu$ pointwise on $\Spa(S)$, as defined in \cite[p.~11]{HKW22}.\footnote{We note that our definition of $\Sht(\msf G^*,\msf b,\msf b_0,\mu)$ agrees with the definition in \cite[\S 3.1]{Ham22}, \cite[\S2.3.2]{MHN24}, \cite[\S 10]{Ham24}, but differs from the definitions in \cite{S-W20} and \cite[Definition~2.4.1]{HKW22} by changing $\mu$ to $\mu^{-1}$. Our definition simplifies certain presentations.}

The automorphism groups of $\mcl E^{\msf b}$ and $\mcl E^{\msf b_0}$ are the constant group diamonds $\udl{\msf G^*_{\msf b}(K)}$ and $\udl{\msf G^*_{\msf b_0}(K)}$, respectively, so $\Sht(\msf G^*,\msf b,\msf b_0, \mu)$ is equipped with a commuting action of $\msf G^*_{\msf b}(K)$ and $\msf G^*_{\msf b_0}(K)$ by pre- and post-composing on $\gamma$.

We record the compatibility of the local shtuka space under Weil restrictions. Let $K/K_0$ be an unramified extension and set $\msf G^*_0=\Res_{K/K_0}\msf G^*$. Suppose that, under the identification
\begin{equation*}
(\msf G_0^*)_{\ovl K_0}\cong\prod_{\Hom(K,\ovl K_0)}\msf G^*_{\ovl K},
\end{equation*}
the cocharacter $\mu_0$ is trivial on every factor except one, where it equals a cocharacter $\mu$ of $\msf G^*_{\ovl K}$. Then $\mu_0$ and $\mu$ have the same reflex field. For $\msf b_0\in B(\msf G_0^*)_\bas$ and $\msf b\in B(\msf G_0^*, \msf b_0, \mu_0)_\bas$, use the natural isomorphism $B(\msf G_0^*)\cong B(\msf G^*)$ to regard $\msf b_0$ and $\msf b$ as basic elements of $B(\msf G^*)$. The definitions then give a natural isomorphism of diamonds
\begin{equation*}
\Sht(\msf G^*, \msf b, \msf b_0, \mu)\cong \Sht(\msf G_0^*, \msf b, \msf b_0, \mu_0).
\end{equation*}

For each compact open subgroup $\mdc K\le\msf G^*_{\msf b_0}(K)$, we define $\Sht_{\mdc K}(\msf G^*, \msf b, \msf b_0, \mu)=\Sht(\msf G^*,\msf b,\msf b_0, \mu)/\mdc K$, which is also a locally spatial diamond defined over $\Spd(\breve KE_\mu)$ \cite[Theorem 23.1.4]{S-W20}. Let $\mcl S'_\mu$ be the $\Lbd$-sheaf corresponding to the highest weight Tilting module $\mcl T_\mu\in \bx{Rep}_\Lbd(\hat{\msf G^*})$ via the geometric Satake equivalence \cite[Theorem \Rmnum1.6.3]{F-S24}. There is a natural map
\begin{equation*}
\pr: \Sht(\msf G^*,\msf b, \msf b_0, \mu)\to \Hec_{\msf G^*, 1},
\end{equation*}
and we denote the pullback of $\mcl S_\mu'$ along $\pr$ by $\mcl S_\mu$. Note that $\pr$ factors through the quotient of $\Sht(\msf G^*,\msf b, \msf b_0, \mu)$ by the actions of $\msf G_{\msf b}^*(K)$ and $\msf G_{\msf b_0}^*(K)$, so $\mcl S_\mu$ is equivariant with respect to these actions. We define
\begin{equation*}
\bx R\Gamma_c(\Sht_{\mdc K}(\msf G^*,\msf b, \msf b_0, \mu), \mcl S_\mu)\defining\ilim_U\bx R\Gamma_c(U, \mcl S_\mu)
\end{equation*}
where $U$ runs through quasi-compact open subsets of $\Sht_{\mdc K}(\msf G^*,\msf b, \msf b_0, \mu)$, and also
\begin{equation*}
\bx R\Gamma_c(\msf G^*, \msf b, \msf b_0, \mu)\defining \ilim_{\mdc K}\bx R\Gamma_c(\Sht_{\mdc K}(\msf G^*,\msf b, \msf b_0, \mu), \mcl S_\mu)
\end{equation*}
where $\mdc K$ runs through open compact subgroups of $\msf G_{\msf b_0}^*(K)$.\footnote{Note that our $\bx R\Gamma_c(\msf G^*, \msf b, \msf b_0, \mu)$ agrees with $\bx R\Gamma_c(\mcl M_{\msf G^*_{\msf b_0},\msf b_0, \{\mu^{-1}\}, \infty}, \Lbd)$ defined in \cite[\S 8.5.8]{DHKZ}, by \cite[Proposition~8.5.9]{DHKZ}; but differs from $\bx R\Gamma_c(\msf G^*_{\msf b_0}, \msf b, \mu)$ defined in \cite[Definition~3.2]{Ham22} by a Tate twist and a $t$-shift.} 

Following \cite{Shi11}, we now define a map $\Mant_{\msf G^*, \msf b, \msf b_0, \mu}:\bx K_0(\msf G_{\msf b}^*, \Lbd)\to\bx K_0(\msf G_{\msf b_0}^*(K)\times W_{E_\mu}, \Lbd)$ describing the cohomology of $\Sht(\msf G^*,\msf b,\msf b_0,\mu)$, where for each $\msf b\in B(\msf G^*)_\bas$, we denote by $\bx K_0(\msf G^*_{\msf b}, \Lbd)$ the Grothendieck group of the category of finite-length admissible representations of $\msf G^*_{\msf b}(K)$ with $\Lbd$-coefficients, and denote by $\bx K_0(\msf G^*_{\msf b}(K)\times W_{E_\mu}, \Lbd)$ the Grothendieck group of the category of finite-length admissible representations of $\msf G^*_{\msf b}(K)$ with $\Lbd$-coefficients equipped with a continuous action of $W_{E_\mu}$ commuting with the $\msf G^*_{\msf b}(K)$-action.

\begin{defi}
For $\rho\in\Pi(\msf G^*_{\msf b}, \Lbd)$, define
\begin{equation}\label{emfneimfeos}
\bx R\Gamma^\flat_c(\msf G^*,\msf b, \msf b_0, \mu)[\rho]\defining\bx R\Hom_{\msf G^*_{\msf b}(K)}\paren{\bx R\Gamma_c(\msf G^*, \msf b, \msf b_0, \mu), \rho},
\end{equation}
and
\begin{equation*}
\bx R\Gamma_c(\msf G^*,\msf b, \msf b_0, \mu)[\rho]\defining\bx R\Gamma_c(\msf G^*, \msf b, \msf b_0, \mu)\Ltimes_{\msf G^*_{\msf b}(K)}\rho.
\end{equation*}
It follows from \cite[Corollary \Rmnum{1}.7.3 and p.~317]{F-S24} that $\bx R\Gamma^\flat_c(\msf G^*,\msf b,\msf b_0,\mu)[\rho]$ is a finite-length $W_{E_\mu}$-equivariant object of $\bx D(\msf G_{\msf b_0}^*, \Lbd)$. We denote its image in $\bx K_0(\msf G_{\msf b_0}^*(K)\times W_{E_\mu}; \Lbd)$ by
\begin{equation*}
\Mant_{\msf G^*, \msf b, \msf b_0, \mu}(\rho).
\end{equation*}
\end{defi}

Note that $\bx R\Gamma_c(\msf G^*,\msf b, \msf b_0, \mu)[\rho]$ is much more natural from the point of view of geometric arguments on $\Bun_{\msf G^*}$ as it involves the much simpler extension by zero functor, while the complex $\bx R\Gamma^\flat_c(\msf G^*,\msf b, \msf b_0, \mu)[\rho]$ is studied in \cite{HKW22}. It follows from Hom-Tensor duality that
\begin{equation}\label{IIEHienitehiremids}
\bx R\Gamma^\flat_c(\msf G^*,\msf b, \msf b_0, \mu)[\rho^*]\cong\bx R\Hom\paren{\bx R\Gamma_c(\msf G^*,\msf b, \msf b_0, \mu)[\rho], \Lbd}.
\end{equation}
Moreover, we have the following result of Meli, Hamann and Nguyen~\cite[Proposition~2.25]{MHN24} for representations with supercuspidal Fargues--Scholze $L$-parameters: 

\begin{prop}\label{speeorjoiehgiejmsiws}
If $\rho\in\Pi(\msf G_{\msf b}^*, \Lbd)$ has supercuspidal Fargues--Scholze $L$-parameter $\phi_\rho^\FS$, then there exists an isomorphism
\begin{equation*}
\bx R\Gamma_c(\msf G^*,\msf b, \msf b_0, \mu)[\rho]\cong \bx R\Gamma^\flat_c(\msf G^*,\msf b, \msf b_0, \mu)[\rho]
\end{equation*}
of representations of $\msf G_{\msf b_0}^*(K)\times W_{E_\mu}$.
\end{prop}

Finally, we recall the following crucial result relating the cohomology of local shtuka spaces and Fargues--Scholze $L$-parameters. Before that, recall that for any dominant cocharacters $\mu$ for $\msf G$ with reflex field $E_\mu$, the highest weight tilting module $\mcl T_\mu\in\bx{Rep}_\Lbd(\hat{\msf G})$ extends naturally to a representation of $\hat{\msf G}\rtimes W_{E_\mu}$, and we define the extended highest weight tilting module
\begin{equation}\label{ieneineifeilsws}
\LL\mcl T_\mu\defining \Ind^{\LL\msf G}_{\hat{\msf G}\rtimes W_{E_\mu}}\mcl T_\mu\in \bx{Rep}_\Lbd(\LL\msf G).
\end{equation}
The isomorphism class of $\LL\mcl T_\mu$ only depends on the $\Gal_K$-orbits of $\mu$; see~\cite[\S9.1]{Ham24}. Moreover, we recall that the tilting module $\mcl T_\mu$ equals the usual highest weight module (defined by unnormalized parabolic induction) when $\mu$ is minuscule.

\begin{prop}\label{ienfimeoifiems}
Suppose $\rho\in\Pi(\msf G^*_{\msf b}, \Lbd), \pi\in\Pi(\msf G^*_{\msf b_0}, \Lbd)$ and suppose that $\pi$ appears in $\Mant_{\msf G^*, \msf b, \msf b_0, \mu}(\rho)$. Here we omit the action of $W_{E_\mu}$ on $\Mant_{\msf G^*, \msf b, \msf b_0, \mu}(\rho)$. We then have $\phi_\pi^\FS=\phi_\rho^\FS\in\Phi^\sems(\msf G^*, \Lbd)$.
\end{prop}
\begin{proof}
This is established in \cite[Corollary 3.15]{Ham22}. Strictly speaking, the cited result only proves the assertion when $\mu$ is minuscule. However, the general proof is the same and we reproduce here.

Regard $\rho$ as an element of $\bx D(\msf G^*_{\msf b}, \Lbd)\cong \bx D_{\bx{lis}}(\Bun_{\msf G^*}^{\msf b}, \Lbd)$, then there is an isomorphism
\begin{equation*}
\bx R\Gamma_c(\msf G^*, \msf b, \msf b_0, \mu)[\rho]\cong i_{\msf b_0}^*\bx T_\mu i_{\msf b, !}(\rho)\in \bx D(\msf G^*_{\msf b_0}, \Lbd)^{\bx BW_{E_\mu}},
\end{equation*}
by \cite[\S \Rmnum{9}.3]{F-S24}. Here $\bx T_\mu$ is the Hecke operator associated with the extended highest weight tilting module $\LL\mcl T_\mu$ associated with $\mu$ \eqref{ieneineifeilsws}. Then each Schur irreducible subquotient $A\in \bx D_{\bx{lis}}(\Bun_{\msf G^*}, \Lbd)$ of $\bx T_\mu i_{\msf b, !}(\rho)$ has Fargues--Scholze parameter $\phi_\rho^\FS$ because the excursion algebra
commutes with Hecke operators. Here we omit the $W_{E_\mu}$-action. Furthermore, each $\pi\in\Pi(\msf G^*_{\msf b_0}, \Lbd)\subset \bx D_{\bx{lis}}(\Bun_{\msf G^*}^{\msf b_0}, \Lbd)$ appearing in $i_{\msf b_0}^*(A)$ has Fargues--Scholze parameter equal to that of $A$ under the identification $\Phi^\sems(\msf G_{\msf b_0}^*, \Lbd)=\Phi^\sems(\msf G^*, \Lbd)$, by \cite[\S \Rmnum{9}.7.1]{F-S24} (see also \cite[Proposition 3.14]{Ham22}). Thus the assertion follows.
\end{proof}

\subsection{Spectral actions}\label{snihsnfifies}

In this subsection, we assume $\msf G$ is quasisplit. We recall the spectral actions on sheaves on $\Bun_{\msf G}$ by sheaves on the stack of Langlands parameters; cf.~\cite[\S \Rmnum{10}]{F-S24}, \cite{Ham22}.

For $\phi\in\Phi^\sems(\msf G, \Lbd)$, we define $\bx D_{\bx{lis}}(\Bun_{\msf G}, \Lbd)_\phi\subset \bx D_{\bx{lis}}(\Bun_{\msf G}, \Lbd)$ to be the full subcategory of objects $A$ such that  the endomorphism induced by any $f\in \mcl O_{\mfk X_{\hat{\msf G}}}\setm\mfk m_\phi$ is an isomorphism on $A$. 


We recall the Act-functors defined in \cite[\S 3.2]{Ham22} and \cite[\S 2.3.3]{MHN24} via the spectral action of the moduli stack of $L$-parameters: Let $\mfk X_{\hat{\msf G}}\defining [Z^1(W_K, \hat{\msf G})_{\ovl{\bb Q_\ell}}/\hat{\msf G}]$ be the moduli stack of semisimple Langlands parameters defined in \cite{DHKM, Zhu21} and \cite[Theorem \Rmnum{8}.1.3]{F-S24}, and let $\Perf(\mfk X_{\hat{\msf G}})$ be the derived category of perfect complexes on $\mfk X_{\hat{\msf G}}$. Let $\bx D_{\bx{lis}}(\Bun_{\msf G}, \ovl{\bb Q_\ell})^\omega\subset \bx D_{\bx{lis}}(\Bun_{\msf G}, \ovl{\bb Q_\ell})$ be the triangulated sub-category consisting of compact objects. Then it follows from \cite[Corollary \Rmnum{10}.1.3]{F-S24} that for each finite index set $I$, there exists a $\ovl{\bb Q_\ell}$-linear action
\begin{equation*}
\Perf(\mfk X_{\hat{\msf G}})^{\bx BW_K^I}\to \End\paren{\bx D_{\bx{lis}}(\Bun_{\msf G}, \ovl{\bb Q_\ell})^\omega}^{\bx BW_K^I}: C\mapsto\{A\mapsto C\star A\},
\end{equation*}
which is monoidal in the sense that there exists a natural equivalence of functors:
\begin{equation*}
(C_1\Ltimes C_2)\star(-)\cong C_1\star(C_2\star(-)).
\end{equation*}

Fix for the rest of this subsection a supercuspidal $L$-parameter $\phi\in\Phi^{\bx{sc}}(\msf G, \ovl{\bb Q_\ell})$, then there exists a connected component $C_\phi$ of $\mfk X_{\hat{\msf G}}$ consisting of unramified twists of the parameter $\phi$, equipped with a natural map $C_\phi\to [\Spd(\ovl{\bb Q_\ell})/\mfk S_\phi]$; see~\cite[\S \Rmnum{10}.2]{F-S24}. We then have a direct summand
\begin{equation*}
\Perf(C_\phi)\subset \Perf(\mfk X_{\hat{\msf G}}),
\end{equation*}
and the spectral action gives rise to a direct summand
\begin{equation*}
\bx D_{\bx{lis}}^{C_\phi}(\Bun_{\msf G}, \ovl{\bb Q_\ell})^\omega\subset \bx D_{\bx{lis}}(\Bun_{\msf G}, \ovl{\bb Q_\ell})^\omega.
\end{equation*}
Then there exists a decomposition \cite[p.~34]{Ham22}
\begin{equation*}
\bx D_{\bx{lis}}^{C_\phi}(\Bun_{\msf G}, \ovl{\bb Q_\ell})^\omega\cong\bplus_{\msf b\in B(\msf G)_\bas}\bx D^{C_\phi}(\msf G_{\msf b}, \ovl{\bb Q_\ell})^\omega,
\end{equation*}
where $\bx D^{C_\phi}(\msf G_{\msf b}, \ovl{\bb Q_\ell})^\omega\subset \bx D(\msf G_{\msf b}, \ovl{\bb Q_\ell})^\omega$ is a full subcategory of the subcategory of compact objects in $\bx D(\msf G_{\msf b}, \ovl{\bb Q_\ell})$. 
Let $\chi$ be the character of $Z(\msf G)(K)$ determined by $\phi$ as in \cite[\S 10.1]{Bor79}, then for each $\msf b\in B(\msf G)$, the subcategory
\begin{equation*}
\bx D^{C_\phi, \chi}(\msf G_{\msf b}, \ovl{\bb Q_\ell})^\omega\subset \bx D^{C_\phi}(\msf G_{\msf b}, \ovl{\bb Q_\ell})^\omega
\end{equation*}
spanned by the compact objects with fixed central character $\chi$ of $Z(\msf G_{\msf b})\cong Z(\msf G)$ is semisimple because supercuspidal representations are both injective and projective in the category of smooth representations with fixed central character. Thus we can identify $\bx D^{C_\phi, \chi}(\msf G_{\msf b}, \ovl{\bb Q_\ell})^\omega$ with
\begin{equation*}
\bplus_{\pi_{\msf b}\in \Pi_\phi(\msf G_{\msf b})}\iota_\ell\pi_{\msf b}\otimes\Perf(\ovl{\bb Q_\ell}).
\end{equation*}
This category is preserved under the spectral action of $\Perf(C_\phi)$; see~\cite[p.~34]{Ham22}.

We now recall the Act-functors: For each $\eta\in \Irr(\mfk S_\phi)$, we get a line bundle $\mcl L_\eta$ on $C_\phi$ by pulling back along the natural map $C_\phi\to [\Spd(\ovl{\bb Q_\ell})/\mfk S_\phi]$, and we define $\bx{Act}_\eta$ to be the spectral action of this line bundle on $\bx D^{C_\phi, \chi}(\msf G_{\msf b}, \ovl{\bb Q_\ell})^\omega$. In particular, Act-functors are symmetric monoidal, i.e., $\bx{Act}_\uno$ is the identity functor, and for any $\eta, \eta'\in \Irr(\mfk S_\phi)$, there exists a natural equivalence of functors
\begin{equation*}
\bx{Act}_\eta\circ\bx{Act}_{\eta'}\cong \bx{Act}_{\eta\eta'}.
\end{equation*}
From this it is easy to show that each Act-functor sends an irreducible admissible representation
\begin{equation*}
\iota_\ell\pi\in \bigcup_{\msf b\in B(\msf G)_\bas}\iota_\ell\Pi_\phi(\msf G_{\msf b})
\end{equation*}
to another irreducible admissible representation
\begin{equation*}
\iota_\ell\pi'\in \bigcup_{\msf b\in B(\msf G)_\bas}\iota_\ell\Pi_\phi(\msf G_{\msf b})
\end{equation*}
with a $t$-shift; see~\cite[Lemma 2.28]{MHN24}.

\subsection{Weak version of the Kottwitz conjecture}\label{Kottwisniconeidinfs}

In this subsection, let $G^*$ be a special orthogonal or unitary group as defined in \S\ref{theogirneidnis}, and let $(G=G^*_{b_0}, \varrho_{b_0}, z_{b_0})$ be the pure inner twist of $G^*$ associated with $b_0\in B(G^*)_\bas$. Let $\mu$ be a dominant cocharacter of $G^*_{\ovl K}$, and let $b\in B(G^*, b_0, \mu)_\bas$ be the unique nontrivial basic element, regarded as a basic element in $B(G)_{\bas}$ via the isomorphism~\eqref{ienufheiehfniess}, so we can adopt the notation from~\S\ref{secitonautnciniauotenries}. Thus $G_b\cong G^*_{b_0+b}$, and we define
\begin{equation*}
\Mant_{G, b, \mu}\defining \Mant_{G^*, b_0+b, b_0, \mu}: \bx K_0(G_b, \ovl{\bb Q_\ell})\to \bx K_0(G, \ovl{\bb Q_\ell}),
\end{equation*}
after forgetting the $W_{E_\mu}$-action.

We will use the weak version of the Kottwitz conjecture from \cite{HKW22} describing the cohomology $\Mant_{G, b, \mu}(\iota_\ell\rho)$ for $\rho\in\Pi_\phi(G_b)$, where $\phi$ is a discrete $L$-parameter. The (generalized) Kottwitz conjecture describes $\Mant_{G, b, \mu}(\iota_\ell\rho)$ in terms of the local Langlands correspondence, where $\rho$ lies in a supercuspidal $L$-packet.

In \cite[Theorem 1.0.2]{HKW22}, a weak version of the Kottwitz conjecture is established for all discrete $L$-parameters, but disregarding the action of the Weil group, and modulo a virtual representation whose character vanishes on the locus of elliptic elements. Their proof is conditional on the refined local Langlands conjecture of \cite[Conjecture G]{Kal16a} (in fact, as $G$ is always a pure inner form of $G^*$, the isocrystal version \cite[Conjecture G]{Kal16a} suffices), but in Case O2 (with sufficiently high rank), only the weak version of this conjecture stated in \S\ref{ndinlocalLanlgnnis} is known. To remedy this, we will use weak versions of the endoscopic character identities Theorem~\ref{endoslicindiikehrnieiis} to prove a weak version of \cite[Theorem 1.0.2]{HKW22}.

To state results uniformly, for every $b\in B(G)_\bas$, define
\begin{equation*}
\tilde{\bx K}_0(G_b,\ovl{\bb Q_\ell})\defining\bx K_0(G_b,\ovl{\bb Q_\ell})/\bra{[\pi]-[\pi^\varsigma]: \pi\in\Pi(G_b,\ovl{\bb Q_\ell})}.
\end{equation*}
The natural $\varsigma$-action on the local shtuka space induces a well-defined map
\begin{equation*}
\Mant_{G, b, \mu}: \tilde{\bx K}_0(G_b, \ovl{\bb Q_\ell})\to \tilde{\bx K}_0(G, \ovl{\bb Q_\ell}).
\end{equation*}

The set of elliptic elements of $G_b(K)$ is invariant under the action of $\varsigma$, so it makes sense to talk about an object of $\tilde{\bx K}_0(G_b, \ovl{\bb Q_\ell})$ whose character vanishes on the locus of elliptic elements of $G_b(K)$, and these objects are exactly those coming from a proper Levi subgroup of $G_b$ \cite[Theorem C.1.1]{HKW22}.

Similarly, for $\tilde\pi\in\tilde\Pi(G_b)$, we can define the Harish-Chandra character
\begin{equation*}
\Theta_{\tilde\pi}=\frac{1}{2}(\Theta_{\pi}+\Theta_{\pi^\varsigma})\in C(G_b(K)_\sreg\sslash G_b(K), \bb C),
\end{equation*}
where $\pi\in\Pi(G_b)$ is an arbitrary representative of $\tilde\pi$.

For $\tilde\rho\in\tilde\Pi(G_b)$, the Fargues--Scholze parameter $\iota_\ell^{-1}\phi_{\tilde\rho}^\FS: W_K\to \LL G^*$ is well-defined up to $\bx O_{N(G)}(\bb C)$-conjugation in Case O2, by the compatibility of Fargues--Scholze LLC with central extensions, Theorem~\ref{compaitbsilFaiirfies}, so it is unambiguous to ask whether it is a supercuspidal $L$-parameter.

Then our theorem, which slightly generalizes the main theorem of \cite{HKW22}, is stated as follows (a stronger version will be established in~\S\ref{Kotinitbeinss}):

\begin{thm}\label{coalidhbinifw222}
If $\tilde\phi\in\tilde\Phi_2(G^*)$ is a discrete $L$-parameter, and $\tilde\rho\in \tilde\Pi_{\tilde\phi}(G_b)$, then
\begin{equation*}
\Mant_{G, b, \mu}(\iota_\ell\tilde\rho)=\sum_{\tilde\pi\in\tilde\Pi_{\tilde\phi}(G)}\dim\Hom_{\mfk S_{\tilde\phi}}(\delta[\tilde\pi,\tilde\rho],\mcl T_\mu)[\iota_\ell\tilde\pi]+\Err
\end{equation*}
in $\tilde{\bx K}_0(G, \ovl{\bb Q_\ell})$, where $\Err\in\tilde{\bx K}_0(G, \ovl{\bb Q_\ell})$ is a virtual representation whose character vanishes on $G(K)_{\sreg, \ellip}$.

Moreover, if the packet $\tilde\Pi_{\tilde\phi}(G)$ consists entirely of supercuspidal representations and the Fargues--Scholze $L$-parameter $\tilde\phi_{\tilde\rho}^\FS$ is supercuspidal, then $\Err=0$.
\end{thm}

We can apply Theorem~\ref{coalidhbinifw222} to $(b,\mu)=(b_1, \mu_1)$ defined in \S\ref{combianosiLoifn}. Then it follows from Theorem~\ref{pdidiiidjifenis} that we obtain the following corollary:

\begin{cor}\label{Haninidkeiehiens}
Suppose $\tilde\phi\in\tilde\Phi_2(G^*)$ is a discrete $L$-parameter and
\begin{equation*}
\tilde\phi^\GL=\phi_1\oplus\cdots\oplus\phi_k\oplus\phi_{k+1}\oplus\cdots\oplus\phi_r
\end{equation*}
where the $\phi_i$ are irreducible representations of $W_{K_1}\times\SL_2(\bb C)$ of dimension $d_i$ for each $i$, and $d_i$ is odd if and only if $i\le k$. Let $\mu_1$ be the dominant cocharacter of $G^*_{\ovl K}$ defined in \eqref{idnifdujiherheins}, and $b_1\in B(G^*, b_0, \mu_1)_\bas$ be the unique basic element. For
\begin{equation*}
\tilde\rho=\tilde\pi_{[I]}\in\tilde\Pi_{\tilde\phi}(G_{b_1}),\qquad\#I\equiv\frac{\kappa_{b_0}(-\uno)+1}{2}\modu2,
\end{equation*}
one has
\begin{equation*}
\Mant_{G, b_1, \mu_1}(\iota_\ell\tilde\pi_{[I]})=\Mant_{G, b_1, \mu_1^\bullet}(\iota_\ell\tilde\pi_{[I]})=\sum_{i\in [r]_+}d_i[\iota_\ell\tilde\pi_{[I\oplus\{i\}]}]+\Err
\end{equation*}
in $\tilde{\bx K}_0(G, \ovl{\bb Q_\ell})$. Here $\mu_1^\bullet=-w_0(\mu_1)$ is the dominant cocharacter conjugate to $\mu_1^{-1}$, and $\Err\in\tilde{\bx K}_0(G, \ovl{\bb Q_\ell})$ is a virtual representation whose character vanishes on $G(K)_{\sreg, \ellip}$.
\end{cor}
\begin{cor}\label{coroellaoslisenifes}
If $\tilde\phi\in\tilde\Phi_{\bx{sc}}(G^*)$ is a supercuspidal $L$-parameter, then all representations $\tilde\pi\in \tilde\Pi_{\tilde\phi}(G)$ have the same Fargues--Scholze $L$-parameter $\tilde\phi_{\iota_\ell\tilde\pi}^\FS$.
\end{cor}
\begin{proof}
Let $I\subset [r]_+$ with $\#[I]\equiv\frac{\kappa_{b_0}(-\uno)+1}{2}\modu2$. For every $i\in [r]_+$, it follows from Corollary~\ref{Haninidkeiehiens} that $\tilde\pi_{[I\oplus\{i\}]}$ appears in $\Mant_{G, b_1, \mu_1}(\iota_\ell\tilde\pi_{[I]})$. Here we use that the error term has no supercuspidal constituents. Proposition~\ref{ienfimeoifiems} therefore
gives
\begin{equation*}
\tilde\phi_{\iota_\ell\tilde\pi_{[I\oplus\{i\}]}}^\FS=\tilde\phi_{\iota_\ell\tilde\pi_{[I]}}^\FS=\tilde\phi_{\iota_\ell\tilde\pi_{[I\oplus\{j\}]}}^\FS
\end{equation*}
for any $i, j\in [r]_+$. This implies that $\tilde\phi_{\iota_\ell\tilde\pi_{[J]}}^\FS=\tilde\phi_{\iota_\ell\tilde\pi_{[J']}}^\FS$ for any $J\subset J'\subset [r]_+$ with $\#J'=\#J+1$. Clearly, this implies that $\tilde\phi_{\iota_\ell\tilde\pi_{[J]}}^\FS=\tilde\phi_{\iota_\ell\tilde\pi_{[J']}}^\FS$ for any $J, J'\subset [r]_+$ with $\#J\equiv \#J'\modu2$.
\end{proof}

We first recall from~\cite[Definition 3.2.4]{HKW22} the transfer map from conjugation-invariant functions on $G_b(K)_\sreg$ to conjugation-invariant functions on $G(K)_\sreg$ when $b\in B(G)_\bas$ is basic.

\begin{defi}
There is a diagram of topological spaces
\begin{equation*}
\begin{tikzcd}
&\bx{Rel}_b\ar[ld]\ar[rd] &\\
G(K)_\sreg\sslash G(K) &  & G_b(K)_\sreg\sslash G_b(K),
\end{tikzcd}
\end{equation*}
where $\bx{Rel}_b$ is the set of $G(K)\times G_b(K)$-conjugacy classes of triples $(g, g', \lbd)$ such that  
\begin{itemize}
\item
$g\in G(K)_\sreg$ and $g'\in G_b(K)_\sreg\subset G(\breve K)$ are stably conjugate, i.e., conjugate under the action of $G(\breve K)$,
\item
$\lbd\in X_\bullet(Z_G(g))$ satisfies $\kappa_{Z_G(g)}(\inv[b](g, g'))$ agrees with the image of $\lbd$ in $X_\bullet(Z_G(g))_{\Gal_K}$. Here $\inv[b](g, g')$ is the class of $y^{-1}b\vp_K(y)$ in $B(Z_G(g))$, where $y\in G(\breve K)$ satisfies $g'=ygy^{-1}$. This class is independent of the choice of $y$; see~\cite[Definition 3.2.2, Fact 3.2.3]{HKW22}.
\end{itemize}
Here $(z, z')\in G(K)\times G_b(K)$ acts by conjugation on such triples by
\begin{equation*}
\ad(z, z')(g, g', \lbd)=(\ad(z)g, \ad(z')g', \ad(z)\lbd),
\end{equation*}
and $\bx{Rel}_b$ is given the subspace topology by the inclusion $\bx{Rel}_b\subset (G(K)\times G_b(K)\times X_\bullet(G))/(G(K)\times G_b(K))$ with $X_\bullet(G)$ being discrete.
\end{defi}

We recall the following Hecke transfer map from \cite[Definition 3.2.7, Definition 6.3.4]{HKW22}.
\begin{defi}
We define a Hecke transfer map
\begin{equation*}
T_{b, \mu}^{G_b\to G}: C(G_b(K)_\sreg\sslash G_b(K))\to C(G(K)_\sreg\sslash G(K))
\end{equation*}
such that 
\begin{equation*}
\Bkt{T_{b, \mu}^{G_b\to G}(f')}(g)=(-1)^{\bra{\mu, 2\rho_{G^*}}}\sum_{(g, g', \lbd)\in \bx{Rel}_b}f'(g')\dim\mcl T_\mu[\lbd].
\end{equation*}
The representation $\mcl T_\mu$ has only finitely many weights, and the relevant fibers of $\bx{Rel}_b$ over a fixed strongly regular element are finite. Hence the above sum is finite. The same finiteness, together with the properness property of the relation used in the definition of the transfer map, shows that $T_{b,\mu}^{G_b\to G}(f')$ is compactly supported whenever $f'$ is compactly supported.

Moreover, on the strongly regular elliptic locus, $T_{b, \mu}^{G_b\to G}$ can be extended to a Hecke map on invariant distributions
\begin{equation*}
\mcl T_{b, \mu}^{G_b\to G}: \Dist(G_b(K)_{\sreg, \ellip})^{G_b(K)}\to \Dist(G(K)_{\sreg, \ellip})^{G(K)}.
\end{equation*}
\end{defi}

We then have the following result of Hansen, Kaletha and Weinstein~\cite[Theorem 6.5.2]{HKW22}:

\begin{prop}\label{ienicehifnHneisn}
For any $\rho\in\Pi(G_b)$ and $f\in C_c(G(K)_{\sreg, \ellip})$,
\begin{equation*}
\tr\paren{f|\iota_\ell^{-1}\Mant_{G, b, \mu}(\iota_\ell\rho)}=\Bkt{\mcl T_{b, \mu}^{G_b\to G}(\Theta_\rho)}(f).
\end{equation*}
In particular, the virtual character of $\iota_\ell^{-1}\Mant_{G, b, \mu}(\iota_\ell\rho)$ restricted to $G(K)_{\sreg, \ellip}$ is equal to $T_{b, \mu}^{G_b\to G}(\Theta_\rho)$.
\end{prop}

We now use the weak endoscopic character identities to prove an analogue of \cite[Theorem 3.2.9]{HKW22}, which relates the Hecke transfer map $T_{b, \mu}^{G_b\to G}$ to classical LLCs of $G$ and $G_b$:

\begin{prop}\label{hteoeomiiiehinidiss}
Assume that $b\in B(G^*, b_0, \mu)_\bas$ is basic, and let $\phi\in\Phi_2(G^*)$ be a discrete $L$-parameter. Then, for every $\tilde\rho\in\tilde\Pi_{\tilde\phi}(G_b)$ and every $g\in G(K)_\sreg$ that transfers to $G_b(K)$,
\begin{equation*}
\Bkt{T_{b, \mu}^{G_b\to G}\Theta_{\tilde\rho}}(g)=\sum_{\tilde\pi\in\tilde\Pi_{\tilde\phi}(G)}\dim\Hom_{\mfk S_{\tilde\phi}}(\delta[\tilde\pi, \tilde\rho],\mcl T_\mu)\Theta_{\tilde\pi}(g).
\end{equation*}
\end{prop}
\begin{proof}
We modify the argument of \cite[\S 3.3]{HKW22}. We first clarify
the three related endoscopic elements that occur in the comparison
between the rigid and isocrystal normalizations. Let
\begin{equation*}
\pr_\phi,\nu_\phi:S_\phi^+\longrightarrow S_\phi
\end{equation*}
denote, respectively, the natural projection and the comparison
homomorphism defined in \cite[(4.7)]{Kal18}. Since $\phi$ is
discrete, we identify $S_\phi^\natural=S_\phi=\mfk S_\phi$.

Choose $\dot s\in S_\phi^+$, and set
\begin{equation*}
s=\pr_\phi(\dot s), \qquad s^\natural=\nu_\phi(\dot s).
\end{equation*}
Here $\dot s$ is the refined endoscopic element used in the rigid
normalization, $s$ is its ordinary image and determines the underlying
endoscopic group, while $s^\natural$ is the corresponding element in
the isocrystal normalization and is the element at which the packet
characters and $\mcl T_\mu$ are evaluated. By
\cite[\S 4.2]{Kal18}, the elements $s$ and $s^\natural$ differ by an
element of $Z(\hat G)^{\Gal_K}$. Consequently, they have the same image
\begin{equation*}
\ovl s\in\ovl{\mfk S}_\phi
\end{equation*}
and determine the same underlying endoscopic group and $L$-embedding.

Let
\begin{equation*}
\mfk e=(G^{\mfk e},s^\natural,\LL\xi^{\mfk e})
\end{equation*}
be the extended endoscopic triple associated with $s^\natural$, and
let $\phi^{\mfk e}\in\Phi_2(G^{\mfk e})$ be the parameter associated
with $\phi$. Its rigid refinement
\begin{equation*}
\dot{\mfk e}= (G^{\mfk e},\LL G^{\mfk e},\dot s,\LL\xi^{\mfk e})
\end{equation*}
is the refined endoscopic datum attached to $\dot s$ in \cite[(A.1.1)]{HKW22}.

For every $(g, g', \lbd)\in \bx{Rel}_b$, we have
\begin{eqnarray*}
&&e(G_b)\sum_{\tilde\rho'\in\tilde\Pi_{\tilde\phi}(G_b)}\iota_{\mfk w, b_0+b}(\tilde\rho')(s^\natural)\Theta_{\tilde\rho'}(g')\\
&\overset{\text{Theorem~\ref{endoslicindiikehrnieiis}}}{=}&\sum_{h\in G^{\mfk e}(K)_{\sreg}/{\bx{st.conj}}}\Delta[\mfk w, \mfk e, z_{b_0+b}](h, g')\bx S\Theta_{\tilde\phi^{\mfk e}}(h)\\
&\overset{\text{\cite[Lemma A.1.1]{HKW22}}}{=}&\sum_{h\in G^{\mfk e}(K)_{\sreg}/{\bx{st.conj}}}\Delta[\mfk w, \mfk e, z_{b_0}](h, g)\bra{\inv[b](g, g'), s_{h, g}^\natural}\bx S\Theta_{\tilde\phi^{\mfk e}}(h)\\
&=&\sum_{h\in G^{\mfk e}(K)_{\sreg}/{\bx{st.conj}}}\Delta[\mfk w, \mfk e, z_{b_0}](h, g)\lbd(s_{h, g}^\natural)\bx S\Theta_{\tilde\phi^{\mfk e}}(h).
\end{eqnarray*}
Here $s_{h,g}^\natural$ is obtained by transporting $s^\natural$ to
\begin{equation*}
\widehat{Z_G(g)}^{\Gal_K}
\end{equation*}
through the admissible isomorphism between the centralizers of $h$ and $g$. It is this transported element that is paired with $\inv[b](g,g')$ in \cite[Lemma A.1.1]{HKW22}.

We now multiply this expression by the kernel function $\dim\mcl T_\mu[\lbd]$, and sum over all $g'\in G_b(K)_\sreg\sslash G_b(K)$ and $\lbd\in X_\bullet(Z_G(g))$ such that  $(g, g', \lbd)\in \bx{Rel}_b$.
\begin{eqnarray*}
&&e(G_b)\sum_{(g', \lbd)}\sum_{\tilde\rho'\in\tilde\Pi_{\tilde\phi}(G_b)}\iota_{\mfk w, b_0+b}(\tilde\rho')(s^\natural)\Theta_{\tilde\rho'}(g')\dim\mcl T_\mu[\lbd]\\
&=&\sum_{h\in G^{\mfk e}(K)_{\sreg}/{\bx{st.conj}}}\Delta[\mfk w, \mfk e, z_{b_0}](h, g)\bx S\Theta_{\tilde\phi^{\mfk e}}(h)\sum_{(g', \lbd)}\lbd(s_{h, g}^\natural)\dim\mcl T_\mu[\lbd]\\
&=&\sum_{h\in G^{\mfk e}(K)_{\sreg}/{\bx{st.conj}}}\Delta[\mfk w, \mfk e, z_{b_0}](h, g)\bx S\Theta_{\tilde\phi^{\mfk e}}(h)\tr(\mcl T_\mu(s_{h, g}^\natural))\\
&=&\tr(\mcl T_\mu(s^\natural))\sum_{h\in G^{\mfk e}(K)_{\sreg}/{\bx{st.conj}}}\Delta[\mfk w, \mfk e, z_{b_0}](h, g)\bx S\Theta_{\tilde\phi^{\mfk e}}(h)\\
&\overset{\text{Theorem~\ref{endoslicindiikehrnieiis}}}{=}&\tr(\mcl T_\mu(s^\natural))e(G)\sum_{\tilde\pi\in\tilde\Pi_{\tilde\phi}(G)}\iota_{\mfk w, b_0}(\tilde\pi)(s^\natural)\Theta_{\tilde\pi}(g).
\end{eqnarray*}
Here the second and third equations hold by the argument of~\cite[p.~17]{HKW22}.

We multiply the above equation by $\iota_{\mfk w, b_0+b}(\tilde\rho)(s^\natural)^{-1}$, then as a function of $s^\natural\in \mfk S_{\tilde\phi}$, the left-hand side is invariant under translation by $z_\phi$, so the same holds for the right-hand side, and both sides become functions of $\ovl s\in\ovl{\mfk S}_\phi$. We now average over $\ovl s\in\ovl{\mfk S}_\phi$ to get
\begin{align*}
&\frac{1}{\#\ovl{\mfk S}_\phi}e(G_b)\sum_{\ovl s\in\ovl{\mfk S}_\phi}\sum_{(g', \lbd)}\sum_{\tilde\rho'\in\tilde\Pi_{\tilde\phi}(G_b)}\frac{\iota_{\mfk w, b_0+b}(\tilde\rho')}{\iota_{\mfk w, b_0+b}(\tilde\rho)}(s^\natural)\Theta_{\tilde\rho'}(g')\dim\mcl T_\mu[\lbd]\\
&=\frac{1}{\#\ovl{\mfk S}_\phi}e(G)\sum_{\ovl s\in\ovl{\mfk S}_\phi}\tr(\mcl T_\mu(s^\natural))\sum_{\tilde\pi\in\tilde\Pi_{\tilde\phi}(G)}\frac{\iota_{\mfk w, b_0}(\tilde\pi)}{\iota_{\mfk w, b_0+b}(\tilde\rho)}(s^\natural)\Theta_{\tilde\pi}(g),
\end{align*}
where we recall that $s^\natural\in\mfk S_\phi$ is a lift of $\ovl s$. By Fourier inversion, the left-hand side equals
\begin{equation*}
e(G_b)\sum_{(g', \lbd)}\Theta_{\tilde\rho}(g')\dim\mcl T_\mu[\lbd]=(-1)^{\bra{\mu, 2\rho_{G^*}}}e(G_b)\Bkt{T_{b, \mu}^{G_b\to G}\Theta_{\tilde\rho}}(g),
\end{equation*}
and the right-hand side equals
\begin{equation*}
e(G)\sum_{\tilde\pi\in\tilde\Pi_{\tilde\phi}(G)}\Theta_{\tilde\pi}(g)\frac{1}{\#\ovl{\mfk S}_{\tilde\phi}}\sum_{\ovl s\in\ovl{\mfk S}_{\tilde\phi}}\tr(\mcl T_\mu(s^\natural))\delta[\tilde\pi, \tilde\rho]^{-1}(s^\natural)=e(G)\sum_{\tilde\pi\in\tilde\Pi_{\tilde\phi}(G)}\dim\Hom_{\mfk S_{\tilde\phi}}(\delta[\tilde\pi, \tilde\rho],\mcl T_\mu)\Theta_{\tilde\pi}(g).
\end{equation*}
So the assertion is reduced to the identity $e(G)=(-1)^{\bra{\mu, 2\rho_{G^*}}}e(G_b)$, which is exactly \cite[(3.3.3)]{HKW22}.
\end{proof}

Now Theorem~\ref{coalidhbinifw222} follows from the above propositions in the same way as in the proof of \cite[Theorem 6.5.1]{HKW22}.

\begin{proof}[Proof of \textup{Theorem~\ref{coalidhbinifw222}}]
The equality in $\tilde{\bx K}_0(G(K))$ follows from Propositions~\ref{ienicehifnHneisn} and~\ref{hteoeomiiiehinidiss}.

It remains to prove the assertion about $\Err$. Choose a representative $\rho\in\Pi(G_b)$ of $\tilde\rho$. Let $\pi\in\Pi(G)$ have nonzero coefficient in
\begin{equation*}
\iota_\ell^{-1}\Mant_{G,b,\mu}(\iota_\ell\rho).
\end{equation*}
By Proposition~\ref{ienfimeoifiems},
\begin{equation*}
\phi_{\iota_\ell\pi}^\FS=\phi_{\iota_\ell\rho}^\FS.
\end{equation*}
The parameter on the right is supercuspidal by assumption. Suppose that $\pi$ were not supercuspidal. By the theory of cuspidal support, $\pi$ would be an irreducible subquotient of a normalized parabolic induction from a proper Levi subgroup of $G$. Compatibility of the Fargues--Scholze correspondence with parabolic induction, as stated in Theorem~\ref{compaitbsilFaiirfies}, would then imply that $\phi_{\iota_\ell\pi}^\FS$ factors through the $L$-group of a proper Levi subgroup. This contradicts its supercuspidality. Therefore
\begin{equation*}
\Mant_{G,b,\mu}(\iota_\ell\rho)
\end{equation*}
is a virtual sum of supercuspidal representations. The same is true of the explicit packet term in the formula, so $\Err$ is a virtual sum of supercuspidal representations.

On the other hand, the character of $\Err$ vanishes on the strongly regular elliptic locus. By \cite[Theorem~C.1.1]{HKW22}, its class is generated by normalized parabolic inductions from proper Levi subgroups. Such induced representations have no supercuspidal irreducible constituents. Since irreducible classes form a basis of the Grothendieck group, the subgroups generated respectively by supercuspidal and non-supercuspidal irreducible representations have zero intersection. Hence $\Err=0$.
\end{proof}

\section{Cohomology of orthogonal and unitary Shimura varieties}\label{codehfijiniSHimreis}\label{aoienitrniniishehneisi}

In this section, we compute the $\Pi$-isotypic component of the cohomology of Shimura varieties $\bSh(\mbf G, \mbf X)$ of orthogonal or unitary type related to the local group $G$ defined in \S\ref{theogirneidnis}, where $\Pi$ is a special cuspidal automorphic representation of $\mbf G(\Ade_f)$. These results are related to the cohomology of local shtuka spaces via the basic uniformization theorem stated in the next section, \S\ref{local-LGoabsiSHimief}.

Let $F$ be a totally real number field, and fix a nontrivial additive character $\uppsi_F$ of $F\bsh \Ade_F$, which extends to an additive character of any finite extension $F'/F$ by defining $\uppsi_{F'}\defining\uppsi_F\circ\tr_{F'/F}$. Let $F_1$ be either $F$ or a CM field containing $F$, and let $\cc\in \Gal(F_1/F)$ be the element with fixed field $F$. Let $\tau_0: F_1\to \bb C$ be a fixed embedding. Denote by $\chi_{F_1/F}: \Ade_F^\times/F^\times\to\{\pm1\}$ the character associated with $F_1/F$ via global class field theory. If $F_1\ne F$, choose a totally imaginary element $\daleth\in F_1^\times$, so each embedding $\tau: F\to \bb C$ extends to an embedding $\tau: F_1\to \bb C$ sending $\daleth$ to $\bb R_+\ii$. 

\subsection{The groups}\label{seubsienfiehtoenfeis}

Let $\mbf V$ be a vector space over $F_1$ equipped with a nondegenerate Hermitian $\cc$-sesquilinear form $\bra{-, -}$ on $\mbf V$, i.e.
\begin{equation*}
\bra{au+bv, w}=a\bra{u, w}+b\bra{v, w},
\end{equation*}
\begin{equation*}
\bra{v, w}=\bra{w, v}^\cc.
\end{equation*}

In Case O2, choose an arbitrary diagonal basis $\{v_1, \ldots, v_{\dim(\mbf V)}\}$ of $\mbf V$ over $F_1$ such that $\bra{v_i, v_i}=a_i\in F^\times$. Define
\begin{equation*}
\disc(\mbf V)=(-1)^{\binom{\dim(\mbf V)}{2}}2^{-\dim(\mbf V)}\prod_{i=1}^{\dim(\mbf V)}a_i
\end{equation*}
whose image in $F^\times/(F^\times)^2$ is independent of the basis chosen, and we write $\disc(\mbf G)=\disc(\mbf V)$.

Let $\bx U(\mbf V)\le \GL(\mbf V)$ be the algebraic subgroup defined by
\begin{equation*}
\bx U(\mbf V)=\{g\in \GL(\mbf V): \bra{gv, gw}=\bra{v, w}\forall v, w\in\mbf V\},
\end{equation*}
and let $\mbf G=\bx U(\mbf V)^\circ$ be the neutral component of $\bx U(\mbf V)$. Let $\mbf G^*$ be the unique quasisplit inner form of $\mbf G$ over $F$. Exactly one of the following cases occurs:
\begin{itemize}
\item[\textbf{O1}]
If $F_1=F$ and $\dim(\mbf V)=2n+1$ is odd, then $\mbf G^*=\SO_{2n+1}$.
\item[\textbf{O2}]
If $F_1=F$ and $\dim(\mbf V)=2n$ is even, then $\mbf G^*=\SO_{2n}^{\disc(\mbf V)}$ is the special orthogonal group associated with the quadratic space $\mbf V^*$ over $F$ of dimension $2n$, discriminant $\disc(\mbf V)$ such that the Hasse--Witt invariants of $\mbf V^*\otimes F_v$ are 1 for each $v\in\Pla_F$.
\item[\textbf{U}]
If $F_1\ne F$ and $\dim(\mbf V)=n$, then $\mbf G^*=\bx U(n)$ is the unitary group associated with the Hermitian space $\mbf V^*$ of dimension $n$ with respect to the quadratic extension $F_1/F$ such that the Hasse--Witt invariants of $\mbf V^*\otimes F_v$ are $1$ for each $v\in\Pla_F$.
\end{itemize}
We refer to cases O1 and O2 together as Case O. We assume further that
\begin{equation*}
\dim\mbf V\ge\begin{cases}2 &\text{in Case U}\\
5&\text{in Case O1}\\
6&\text{in Case O2}
\end{cases},
\end{equation*}
so that $\mbf G_\ad$ is always geometrically simple.

To unify notation, we write $n(\mbf G)=n(\mbf G^*)$ for the rank of $\mbf G_{\ovl F}$, define $N(\mbf G), d(\mbf G), b(\mbf G)$ analogous to the local case \eqref{iefieheinfneis}, and define $\disc(\mbf G)\defining\disc(\mbf V)$ in Case O2.

We fix a pinning $(\mbf B^*,\mbf T^*, \{X^*_\alpha\}_{\alpha\in\Delta})$ of $\mbf G^*$ by identifying $\mbf G^*=\bx U(\mbf V^*)^\circ$ for a suitable $\cc$-Hermitian space $\mbf V^*$ over $F_1$, and choosing a complete flag of totally isotropic subspaces in $\mbf V^*$. Together with the additive character $\uppsi_F$, this pinning determines a Whittaker datum $\mfk w$ for $\mbf G^*$.

It follows from the theorem of Hasse--Minkowski and Landherr \cite[Theorem~2.1, Theorem~3.1]{Gro21} and \cite[Lemma 2.1]{GGP12} that pure inner twists of $\mbf G^*$ are in bijection with isometry classes of $\cc$-Hermitian spaces $\mbf V$ with respect to  $F_1/F$ of dimension $d(\mbf G^*)$ (and also with discriminant $\disc(\mbf V)$ in Case O2), and these isometry classes are determined by isometry classes of localizations $\mbf V_v$ for each $v\in\Pla_F$. In particular, $\mbf G$ can always be realized as a pure inner twist $(\mbf G, \bm\varrho, \bm z)$ of $\mbf G^*$.

We fix an isomorphism
\begin{equation*}
\hat{\mbf G}\cong \begin{cases}
\Sp_{N(\mbf G)}(\bb C) &\text{in Case O1}\\
\SO_{N(\mbf G)}(\bb C) &\text{in Case O2}\\
\GL_{N(\mbf G)}(\bb C) &\text{in Case U}\\
\end{cases},
\end{equation*}
and fix a pinning $(\hat{\mbf T}, \hat{\mbf B}, \{X_\alpha\}_{\alpha\in\Delta})$ where $\hat{\mbf T}$ is the diagonal torus, $\hat{\mbf B}$ is the group of upper triangular matrices, and $\{X_\alpha\}_{\alpha\in\Delta}$ is the set of standard root vectors. We write $\LL{\mbf G}=\hat{\mbf G}\rtimes W_F$ for the Langlands L-group in the Weil form. Note that $\hat{\mbf G}$ has a standard representation $\hat\Std=\hat\Std_{\mbf G}: \hat{\mbf G}\to \GL_{N(\mbf G), \bb C}$.

As in the local case \S\ref{theogirneidnis}, there exists an automorphism $\theta$ of $\mbf G^\GL\defining \Res_{F_1/F}\GL_{N(\mbf G)}$ such that  $\mbf G^*$ determines an element in $\mcl E_\ellip(\mbf G^\GL\rtimes\theta)$, and the description of the isomorphism classes of elliptic endoscopic triples $\mfk e\in \mcl E_\ellip(\mbf G)$ is similar to the local case.

Finally, we define a central extension $\mbf G^\sharp$ of $\Res_{F/\bb Q}\mbf G$ as follows: 
\begin{itemize}
\item
In Case O, we imitate \cite[p.~163]{Car86}. Let $\Cl(\mbf V)$ and $\Cl^\circ(\mbf V)$ be the Clifford algebra and even Clifford algebra, respectively. Note that there exists an embedding $\mbf V\subset \Cl(\mbf V)$ and an anti-involution $*$ on $\Cl(\mbf V)$ (the main involution) \cite[\S 1.1]{Mad16}. Let $\GSpin(\mbf V)$ be the stabilizer in $\Cl^\circ(\mbf V)^\times$ of $\mbf V\subset \Cl(\mbf V)$ with respect to  the conjugation action of $\Cl^\circ(\mbf V)^\times$ on $\Cl(\mbf V)$, which is a reductive group over $F$. The conjugation action of $\GSpin(\mbf V)$ on $\mbf V$ induces an exact sequence of reductive groups over $F$:
\begin{equation}\label{osneihfnemsos}
1\to \GL_{1, F}\to \GSpin(\mbf V)\to \mbf G\to 1.
\end{equation}
There is a similitude map $\nu: \GSpin(\mbf V)\to \GL_{1, F}: g\mapsto g^*g$ whose restriction on the central torus is $z\mapsto z^2$. The kernel $\Spin(\mbf V)\defining \ker(\nu)$ is called the spinor group of $\mbf V$.

Fix an imaginary quadratic element $\daleth\in \bb R_+\ii$ (in particular $\daleth^2\in \bb Q_-$). Define
\begin{equation*}
\mbf G_\daleth=\paren{(\Res_{F/\bb Q}\GSpin(\mbf V))\times\Res_{F(\daleth)/\bb Q}\GL_1}/\Res_{F/\bb Q}\GL_1,
\end{equation*}
where $\Res_{F/\bb Q}\GL_1$ is embedded anti-diagonally. Define
\begin{equation*}
\nu^\sharp:\mbf G_\daleth\to\Res_{F/\bb Q}\GL_1:
(g, t)\mapsto \nu(g)\Nm_{F(\daleth)/F}(t),
\end{equation*} 
and we define $\mbf G^\sharp\subset \mbf G_\daleth$ to be the inverse image of the subtorus $\GL_1\subset \Res_{F/\bb Q}\GL_1$ under $\nu^\sharp$. Then
\begin{equation*}
[\mbf G^\sharp, \mbf G^\sharp]=\Res_{F/\bb Q}\Spin(\mbf V),
\end{equation*}
and the short exact sequence~\eqref{osneihfnemsos} induces a short exact sequence
\begin{equation*}
1\to\mbf Z^{\bb Q}\to \mbf G^\sharp\to \Res_{F/\bb Q}\mbf G\to 1,
\end{equation*}
where
\begin{equation*}
\mbf Z^{\bb Q}=\{z\in \Res_{F(\daleth)/\bb Q}\GL_1: \Nm_{F(\daleth)/F}(z)\in\bb Q^\times\}.
\end{equation*}
\item
In Case U, following \cite{RSZ20}, let $\GU^{\bb Q}(\mbf V)$ be the reductive group over $\bb Q$ defined by
\begin{equation*}
\GU^{\bb Q}(\mbf V)=\{(g, \lbd)\in \GL(\mbf V)\times \GL_1: \bra{gv, gw}=\lbd\bra{v, w}\}.
\end{equation*}
It is equipped with a similitude character
\begin{equation*}
\nu: \GU^{\bb Q}(\mbf V)\to \GL_1, \qquad (g, \lbd)\mapsto \lbd.
\end{equation*}
Note that $\GU^{\bb Q}(\mbf V)$ is a subgroup of the Weil restriction of the unitary similitude group $\GU(\mbf V)$. Define
\begin{equation*}
\mbf Z^{\bb Q}=\{z\in \Res_{F_1/\bb Q}\GL_1: \Nm_{F_1/F}(z)\in\bb Q^\times\}.
\end{equation*}
It is naturally equipped with a map to $\GL_1$. We then define a reductive group $\mbf G^\sharp$ over $\bb Q$ by
\begin{equation*}
\mbf G^\sharp=\GU^{\bb Q}(\mbf V)\times_{\GL_1}\mbf Z^{\bb Q}.
\end{equation*}
It is isomorphic to $\Res_{F/\bb Q}\mbf G\times\mbf Z^{\bb Q}$ via the isomorphism
\begin{equation*}
\mbf G^\sharp\cong \Res_{F/\bb Q}\mbf G\times\mbf Z^{\bb Q}: (g, z)\mapsto (z^{-1}g, z).
\end{equation*}
\end{itemize}

\subsection{Endoscopic classification of automorphic representations}

Let $\mbf G^*$ be as in~\S\ref{seubsienfiehtoenfeis} and $(\mbf G, \bm\varrho, \bm z)$ be a pure inner form of $\mbf G^*$. To prepare for the Langlands--Kottwitz method in~\S\ref{Lanlgan-Kotiemethiems}, we recall some results on endoscopic classifications of automorphic representations of orthogonal and unitary groups over a totally real field, following \cite{Art13, KMSW, Ish24, C-Z24}.

For each Archimedean place $\tau$ of $F$, fix a maximal compact subgroup $\mdc K_\tau$ of $\mbf G(F_\tau)$. For all but finitely many finite places $v$ of $F$, the inner twist $(\varrho_v, z_v)$ is trivial, and we fix a hyperspecial subgroup $\mdc K_v^\hs$ of $\mbf G(F_v)$ compatible with $\mfk w_v$, the localization at $v$ of the Whittaker datum $\mfk w$. In Case O2, there exists a nontrivial outer automorphism $\varsigma$ of $\mbf G^*$ which preserves $\mfk w$, thus also $\mdc K_v^\hs$; in the other cases we take $\varsigma=1$. Let $\mcl A(\mbf G)$ denote the space of automorphic forms for $\mbf G$ in the sense of~\cite{B-J79}, with respect to the fixed compact subgroup $\mdc K_\infty\defining \prod_{\tau\in\infPla_F}\mdc K_\tau$. We write $\mcl A_2(\mbf G)$ for the space of square-integrable automorphic forms for $\mbf G$. It is a module over the restricted tensor product
\begin{equation*}
\tilde{\mcl H}(\mbf G(\Ade_F))\defining\bigotimes_v\nolimits'\tilde{\mcl H}(\mbf G(F_v))
\end{equation*}
of the $\varsigma$-invariants of local Hecke algebras with respect to the characteristic function of $\mdc K_v^\hs$ (defined for all but finitely many finite places $v$ of $F$).

The space $\mcl A_2(\mbf G)$ can be decomposed into near equivalence classes of representations. Two irreducible representations $\pi=\otimes_v'\pi_v$ and $\pi'=\otimes'_v\pi_v'$ are called nearly equivalent if $\pi_v$ and $\pi_v'$ are isomorphic for all but finitely many places $v\in\Pla_F$. The decomposition into near equivalence classes will be expressed in terms of elliptic global $A$-parameters. An elliptic global $A$-parameter $\bm\psi\in\Psi_\ellip(\mbf G)$ is a formal finite sum of pairs
\begin{equation*}
\bm\psi=\sum_i(\Pi_i, d_i),
\end{equation*}
where each $\Pi_i$ is an irreducible cuspidal automorphic representation of $\GL_{n_i}(\Ade_{F_1})$ that is conjugate self-dual of sign $(-1)^{d_i-1}b(\mbf G)$ (defined similarly as in \S\ref{IFNieniehifeniws}), such that 
\begin{itemize}
\item
$\sum_in_id_i=N(\mbf G)$.
\item
$(\Pi_i, d_i)\ne (\Pi_j, d_j)$ if $i\ne j$,
\item
In Case O2, we assume $\prod_i\omega_i^{d_i}=\chi_{F\paren{\sqrt{\disc(\mbf G)}}/F}$, where $\omega_i$ is the central character of $\Pi_i$, and $\chi_{F\paren{\sqrt{\disc(\mbf G)}}/F}$ is the quadratic character of $\Ade_F^\times/F^\times$ corresponding to the extension $F(\sqrt{\disc(\mbf G)})/F$ via global class field theory.
\end{itemize}
The parameter $\bm\psi$ is called \tbf{generic} (or \tbf{tempered}) if $d_i=1$ for all $i$, in which case we abbreviate $(\Pi_i, 1)$ to $\Pi_i$.

Given an elliptic global $A$-parameter
\begin{equation*}
\bm\psi=(\Pi_1, d_1)+\cdots+(\Pi_k, d_k)
\end{equation*}
for $\mbf G$, define the formal extended component group
\begin{equation*}
\mfk S_{\bm\psi}^\sharp\defining \bplus_i(\bb Z/2)e_i,
\end{equation*}
where $e_i$ is a formal coordinate associated with the summand $\Pi_i$. In Case O2, there is a map
\begin{equation*}
\det\nolimits_{\bm\psi}: \mfk S_{\bm\psi}^\sharp\to \bb Z/2, \quad \sum_{1\le i\le k}x_ie_i\mapsto \sum_{1\le i\le k}n_id_ix_i,
\end{equation*}
and we define the formal component group $\mfk S_{\bm\psi}\defining \ker(\det_{\bm\psi})$. To unify notation, in Case O1 and Case U, we set $\mfk S_{\bm\psi}=\mfk S_{\bm\psi}^\sharp$.  Moreover, we define the quotient group
\begin{equation*}
\ovl{\mfk S}_{\bm\psi}\defining \mfk S_{\bm\psi}/\bra{e_1+\cdots+e_k}.
\end{equation*}
Finally, Arthur defines a canonical character $\ve_{\bm\psi}$ of $\mfk S_{\bm\psi}$ as in~\cite[Equation~(1.5.6)]{Art13}, which is trivial if $\bm\psi$ is generic.

For each $v\in\Pla_F$, we can define the formal localization $\tilde{\bm\psi}_v\defining\sum_i(\phi_{i, v}, d_i)$: If $\Pla_{F_1}(v)$ is a singleton, which we also denote by $v$, then $\phi_{i, v}$ is a $n_i$-dimensional representation of $W_{(F_1)_v}\times \SL_2(\bb C)$ that corresponds to $\Pi_{i, v}$ via the local Langlands correspondence, which is conjugate self-dual of sign $(-1)^{d_i-1}b(\mbf G)$. We can associate to $\tilde{\bm\psi}_v$ a formal sum
\begin{equation*}
\tilde\phi_{\tilde{\bm\psi}_v}^\GL\defining\sum_i\paren{\paren{\phi_{i, v}\otimes\largel{-}_{(F_1)_v}^{\frac{d_i-1}{2}}}\oplus\paren{\phi_{i, v}\otimes\largel{-}_{(F_1)_v}^{\frac{d_i-3}{2}}}\oplus\cdots\oplus\paren{\phi_{i, v}\otimes\largel{-}_{(F_1)_v}^{\frac{1-d_i}{2}}}},
\end{equation*}
which may be regarded as an element of $\tilde\Phi(\mbf G_v)$.

On the other hand, if $\#\Pla_{F_1}(v)=2$, then we are in Case U. If we write $\Pla_{F_1}(v)=\{w, w^\cc\}$, then $\Pi_{i, w}\cong \Pi_{i, w^\cc}^\vee$ under the identifications $\GL_{n_i}(F_{1,w})\cong \GL_{n_i}(F_{1, w^\cc})$. We define $\phi_{i, v}$ to be the $n_i$-dimensional representation of $W_{F_v}\times \SL_2(\bb C)$ that corresponds to $\Pi_{i, w}$ under the identification $w: F_{1,w}\xr\sim F_v$. This is independent of the choice of $w$ after taking into account the corresponding change of the identification $\mbf G_v\simeq \GL_{N(\mbf G),F_v}$. We define $\tilde\phi_{\tilde{\bm\psi}_v}^\GL$ as before.

We then have the following theorem, usually called ``Arthur's multiplicity formula'':

\begin{thm}[\cite{Art13, Mok15, KMSW, Ish24, C-Z24}]\label{endoslcinidhnfineism}
If $\bm\psi$ is an elliptic global $A$-parameter for $\mbf G$, we denote by $\mcl A_{2, \bm\psi}(\mbf G)$ the direct sum of irreducible admissible representations $\pi$ in $\mcl A_2(\mbf G)$ such that $\tilde\phi_{\pi_v}^\GL\cong \tilde\phi_{\tilde{\bm\psi}_v}^\GL$ for all but finitely many places $v\in\Pla_F$ (here if $\#\Pla_{F_1}(v)=2$ with $\Pla_{F_1}(v)=\{w, w^\cc\}$, then we write $\tilde\phi_{\pi_v}^\GL$ for the classical $L$-parameter of $\mbf G_v\cong \GL_{N(\mbf G), F_v}$ corresponding to $\pi_v$ composed with the identification $w: F_{1,w}\xr\sim F_v$). Then there is a natural decomposition of $\tilde{\mcl H}(\mbf G(\Ade_F))$-modules
\begin{equation*}
\mcl A_2(\mbf G)=\bplus_{\bm\psi}\mcl A_{2, \bm\psi}(\mbf G),
\end{equation*}
where $\bm\psi$ runs through elliptic global $A$-parameters for $\mbf G$. An irreducible subrepresentation $\pi$ of $\mcl A_{2, \bm\psi}(\mbf G)$ is said to have formal parameter $\bm\psi$.

Furthermore, for each generic elliptic global $A$-parameter $\bm\psi=\Pi_1+\cdots+\Pi_k$ for $\mbf G$, there exists a natural diagonal map
\begin{equation*}
\Delta: \mfk S_{\bm\psi}\to \mfk S_{\bm\psi, \Ade_F}\defining \prod_{v\in\Pla_F}\mfk S_{\tilde\phi_{\tilde{\bm\psi}_v}},
\end{equation*}
and a decomposition of $\tilde{\mcl H}(\mbf G(\Ade_F))$-modules:
\begin{equation*}
\mcl A_{2, \bm\psi}(\mbf G)\cong m_{\bm\psi}\bplus_{\substack{\eta\in\Irr(\mfk S_{\bm\psi, \Ade_F})\\ \Delta^*(\eta)=\ve_{\bm\psi}}}\tilde\pi_{\mfk w, z}(\eta),
\end{equation*}
where each $\tilde\pi_{\mfk w, z}(\eta)=\otimes_v\tilde\pi_{\mfk w_v, z_v}(\tilde{\bm\psi}_v, \eta_v)$ is the global restricted tensor product of local representations, cf.~\textup{Theorem~\ref{coeleninifheihsn}}. Moreover, the multiplicity $m_{\bm\psi}$ satisfies $m_{\bm\psi}=1$ unless we are in case \textup{O2} and $n_i$ is even for each $i$, in which case $m_{\bm\psi}=2$.
\end{thm}
\begin{proof}
In Case O1, this is established in \cite{Art13} when $\mbf G$ is quasisplit, and established in \cite[Theorem 3.16, 3.17]{Ish24} when $\mbf G$ is non-quasisplit. In Case O2, this is established in \cite{Art13} when $\mbf G$ is quasisplit, and in \cite[Theorem 2.1, 2.6]{C-Z24} when $\mbf G$ is non-quasisplit. In Case U, this is established in \cite{Mok15} when $\mbf G$ is quasisplit and established in \cite[Theorem 1.7.1]{KMSW} when $\mbf G$ is non-quasisplit.
\end{proof}

This theorem implies the following result on strong functorial transfer and strong multiplicity one for cuspidal automorphic representations of $\mbf G(\Ade_F)$: 

\begin{cor}\label{strongiaufneifniesm}\label{sotnienfiehineiws}
Let $\pi$ be a cuspidal automorphic representation of $\mbf G(\Ade_F)$ whose formal parameter $\bm\psi$ is generic. Then, for every finite place $v\in\fPla_F$, one has
\begin{equation*}
\tilde\phi_{\pi_v}^{\GL}\cong\tilde\phi_{\tilde{\bm\psi}_v}^{\GL}.
\end{equation*}
We write $\pi^\GL\defining\bm\psi$ and call it the \tbf{strong functorial transfer} of $\pi$.

Moreover, if $\pi'$ is another cuspidal automorphic representation of $\mbf G(\Ade_F)$ or $\mbf G^*(\Ade_F)$ with formal parameter $\pi^\GL$, then
\begin{equation*}
\tilde\phi_{\pi'_v}\cong\tilde\phi_{\pi_v}
\end{equation*}
for every finite place $v$ of $F$, and $\pi^\GL$ is also the strong functorial transfer of $\pi'$.
\end{cor}

\subsection{Controlled cuspidal automorphic representations}\label{coniierueiefiuhifhens}

Let $\mbf G^*$ be as in \S\ref{seubsienfiehtoenfeis}. To specify the local conditions imposed on the automorphic representations under consideration, we will use the following notion of control tuples:

\begin{defi}\label{deieiutesuroeneis}
A \tbf{control tuple} for $\mbf G^*$ is a tuple $\bigstar=(\Pla^\circ, \Pla^\St, \Pla^{\bx{sc}}, \Pla, \xi)$ where
\begin{itemize}
\item
$\Pla^\St$ and $\Pla^{\bx{sc}}$ are disjoint nonempty finite sets of finite places of $F$.
\item
$\Pla^\circ\subset \Pla^\St\cup \Pla^{\bx{sc}}$ and $\Pla^\St\cup\Pla^{\bx{sc}}\cup\infPla_F\subset \Pla$ are finite sets of places of $F$.
\item
$\xi=\otimes_{\tau\in\Hom(F, \bb C)}\xi_\tau$ is an irreducible representation of $\paren{\Res_{F/\bb Q}\mbf G^*}\otimes_{\bb Q}\bb C$ with regular algebraic highest weight.
\item
$\mbf G^*_v$ is unramified for any finite place $v\in \fPla_F$ not in $\Pla$.
\end{itemize}
\end{defi}

\begin{defi}\label{oeinifenieis}
Let $\bigstar$ be a control tuple for $\mbf G^*$. A pure inner twist $(\mbf G, \bm\varrho,\bm z)$ of $\mbf G^*$ over $F$ is called a \tbf{$\bigstar$-good pure inner form} of $\mbf G^*$ if $(\varrho_v, z_v)$ is trivial for each $v\in\fPla_F\setm\Pla^\circ$.

If $(\mbf G, \bm\varrho,\bm z)$ is a $\bigstar$-good pure inner form of $\mbf G^*$, then for each $v\in\fPla_F\setm\Pla$, $\mbf G_v$ has a reductive integral model $\mcl G_v$ over $\mcl O_{F_v}$ coming from the fixed reductive integral model $\mcl G^*_v$ of $\mbf G^*_v$ via $\varrho_v$. We also write $\mdc K_v^\hs$ for the corresponding hyperspecial maximal compact subgroup and define the abstract Hecke algebra away from $\Pla$:
\begin{equation*}
\tilde{\bb T}^\Pla\defining \bigotimes_{v\in\fPla_F\setm\Pla}\nolimits'\tilde{\mcl H}(\mbf G(F_v), \mdc K^\hs_v).
\end{equation*}
\end{defi}

\begin{defi}\label{Soainsilsienicmosn}
Let $\bigstar$ be a control tuple for $\mbf G^*$ and let $(\mbf G, \bm\varrho,\bm z)$ be a $\bigstar$-good pure inner twist of $\mbf G^*$. Suppose $\Pla'\subset \Pla^\circ$ is a subset. A compact open subgroup $\mdc K^{\Pla'}\le \mbf G(\Ade_{F, f}^{\Pla'})$ is called a \tbf{$\bigstar$-split subgroup} if it is of the form
\begin{equation*}
\mdc K^{\Pla'}=\prod_{v\in\fPla_F\setm\Pla'}\mdc K_v,
\end{equation*}
where $\mdc K_v=\mdc K_v^\hs$ for $v$ not in $\Pla$.
\end{defi}

\begin{defi}\label{olakisniuateineeoifneis}
For a control tuple $\bigstar$ for $\mbf G^*$ and a $\bigstar$-good pure inner twist $(\mbf G, \bm\varrho,\bm z)$ of $\mbf G^*$, a \tbf{$\bigstar$-good automorphic representation} of $\mbf G(\Ade_F)$ is a cuspidal automorphic representation $\pi=\otimes'_v\pi_v$ of $\mbf G(\Ade_F)$ such that  
\begin{itemize}
\item
$\pi_v$ is unramified for all $v\in \fPla_F\setm\Pla$;
\item
$\pi_v$ is an unramified twist of the Steinberg representation for any $v\in\Pla^\St$ (see Definition~\ref{unraimfiehidtiensitineis});
\item
$\pi_v$ has supercuspidal classical $L$-parameter for each $v\in\Pla^{\bx{sc}}$. Moreover, $\pi_v$ has simple classical $L$-parameter for some $v\in\Pla^{\bx{sc}}$ (as defined in \S\ref{IFNieniehifeniws});
\item
$\pi_\infty$ is cohomological for $\xi$, i.e.,
\begin{equation*}
\bx H^i(\Lie(\mbf G(F\otimes\bb R)), \mdc K_\infty, \pi_\infty\otimes_{\bb C}\xi)\ne 0
\end{equation*}
for some $i\in\bb N$.
\end{itemize}
\end{defi}

We then have the following controlled strong transfer result:

\begin{thm}\label{strongleienineirnes}
For any control tuple $\bigstar$ for $\mbf G^*$ and $\bigstar$-good pure inner twists $(\mbf G, \bm\varrho,\bm z), (\mbf G',\bm \varrho',\bm z')$ of $\mbf G^*$, if $\pi$ is a $\bigstar$-good automorphic representation of $\mbf G$, then there exists a $\bigstar$-good automorphic representation $\tau$ of $\mbf G'$ such that 
\begin{itemize}
\item
$\tau^\Pla\cong \pi^\Pla$ as $\tilde{\bb T}^\Pla$-modules via the isomorphism
\begin{equation*}
\bm\varrho'\circ\bm\varrho^{-1}: \mbf G^\Pla\xr\sim(\mbf G')^\Pla.
\end{equation*}
\item
for any $v\in\fPla_F$, $\tau_v$ has the same classical $L$-parameter as $\pi_v$.
\end{itemize}
Such a $\tau$ is called a \tbf{$\bigstar$-good transfer} of $\pi$ to $\mbf G'$.
\end{thm}
\begin{proof}
Let $\bm\psi=\pi^\GL$. By Definition~\ref{olakisniuateineeoifneis}, $\bm\psi$ is a simple generic elliptic global $A$-parameter. For each finite place $v$, let $\tilde\phi_v$ be the classical $L$-parameter of $\pi_v$.

If $v\notin\Pla^\circ$, the two pure inner twists are locally identified, and we choose the member of
$\tilde\Pi_{\tilde\phi_v}(\mbf G'_v)$ corresponding to $\pi_v$. If $v\in\Pla^\circ$, then $v\in\Pla^\St\cup\Pla^{\bx{sc}}$, so $\tilde\phi_v$ is discrete. The local packet parametrization in Theorem~\ref{coeleninifheihsn} therefore gives a nonempty packet $\tilde\Pi_{\tilde\phi_v}(\mbf G'_v)$. At every Archimedean place $v$, choose a member of the cohomological discrete-series packet associated with $\xi_v$.

Let $\eta_v$ denote the character of the local component group corresponding to the chosen local representation. Since $\bm\psi$ is simple, $\mfk S_{\bm\psi}$ has order $2$ and is generated by its distinguished central element. The value of $\eta_v$ on its localization is the local Kottwitz sign determined by $(\varrho'_v,z'_v)$. The product formula for the global pure inner twist $(\mbf G',\bm\varrho',\bm z')$ therefore gives
\begin{equation*}
\Delta^*\big(\otimes_v\eta_v\big)=\uno=\ve_{\bm\psi}.
\end{equation*}
Arthur's multiplicity formula, Theorem~\ref{endoslcinidhnfineism}, now produces a discrete
automorphic representation
\begin{equation*}
\tau=\otimes_v'\tau_v
\end{equation*}
of $\mbf G'(\Ade_F)$ having the prescribed local members.

For every $v\in\Pla^{\bx{sc}}$, the parameter $\tilde\phi_v$ is
supercuspidal, and hence $\tau_v$ is supercuspidal. Since
$\Pla^{\bx{sc}}\ne\vn$, it follows that $\tau$ is cuspidal.
The unramified and supercuspidal conditions are preserved because
the local parameters are unchanged.

Finally, if $v\in\Pla^\St$, then
$\mfk S_{\tilde\phi_v}$ has order $2$ and
$z_{\tilde\phi_v}$ is nontrivial. Hence the relevant packet on
$\mbf G'_v$ is a singleton, whose unique member is an unramified
twist of the Steinberg representation. Thus $\tau$ is
$\bigstar$-good and has the asserted local properties.
\end{proof}

We now recall the following result, due to the work of Clozel, Kottwitz, Harris--Taylor \cite{H-T01}, Shin \cite{Shi11}, Chenevier--Harris \cite{C-H13} and others, which allows us to construct $\ell$-adic representations attached to $\bigstar$-good automorphic representations:

\begin{thm}[\cite{Clo90, H-T01, T-Y07, Shi11, Car12,C-H13, Car14}]\label{Galosiniautoenrieeis}
For any control tuple $\bigstar$ for $\mbf G^*$ and $\bigstar$-good pure inner twist $(\mbf G, \bm\varrho,\bm z)$ of $\mbf G^*$, if $\pi$ is a $\bigstar$-good automorphic representation of $\mbf G$, then $\pi_v$ is tempered for every $v\in\fPla_F$, and there exists a continuous irreducible representation, unique up to isomorphism,
\begin{equation*}
\rho_{\pi, \ell}: \Gal_{F_1}\to \GL_{N(\mbf G)}(\ovl{\bb Q_\ell}),
\end{equation*}
such that
\begin{equation}\label{lsisinenieuifes}
\WD\paren{\rho_{\pi,\ell}|_{W_{F_{1,w}}}}^{F\dash\sems}\cong\iota_\ell\paren{\tilde\phi_{\pi_v}^\GL\otimes\largel{-}_{F_{1,w}}^{\frac{1-N(\mbf G)}{2}}}
\end{equation}
for any finite place $w$ of $F_1$ with underlying finite place $v$ of $F$. Here $\tilde\phi_{\pi_v}$ is the classical $L$-parameter of $\pi_v$. Note that, if $\#\Pla_{F_1}(v)=2$ with $\Pla_{F_1}(v)=\{w, w^\cc\}$ then we write $\tilde\phi_{\pi_v}^\GL$ for the classical $L$-parameter of $\mbf G_v\cong \GL_{N(\mbf G), F_v}$ corresponding to $\pi_v$ composed with the identification $F_v\cong F_{1, w}$.

Moreover, if $\pi'$ is another cuspidal automorphic representation of $\mbf G(\Ade_F)$ or $\mbf G^*(\Ade_F)$ such that  $\pi_v$ and $\pi_v'$ have the same classical $L$-parameter $\tilde\phi_{\pi_v}=\tilde\phi_{\pi'_v}$ for all but finitely many finite places $v$ of $F$, then \textup{Equation~\eqref{lsisinenieuifes}} is satisfied for each $v\in\fPla_F$ with $\pi_v$ replaced by $\pi'_v$.
\end{thm}
\begin{proof}
Let $\pi^\GL$ be the strong functorial transfer of $\pi$ to $\mbf G^\GL$ (Corollary~\ref{strongiaufneifniesm}). Then $\pi^\GL$ is conjugate self-dual and cohomological with regular highest weight. We let $\rho_{\pi, \ell}$ be the Galois representation associated with $\pi^\GL$. Such a $\rho_{\pi, \ell}$ is constructed in \cite[Theorem 3.2.3]{C-H13} and the local-global compatibility is established in \cite[Theorem 1.1]{Car12} and \cite[Theorem 1.1]{Car14}. The temperedness is established for Archimedean places by Clozel \cite[Lemma 4.9]{Clo90} and established for finite places by \cite{H-T01, T-Y07, Shi11, Car12, Clo13, Car14}. The irreducibility follows from local-global compatibility and from the existence of a place $v\in\Pla^{\bx{sc}}$ at which $\pi_v$ has simple supercuspidal classical $L$-parameter.

The last assertion follows from Corollary~\ref{sotnienfiehineiws}.
\end{proof}

\subsection{The Shimura data}\label{Shimreiejvanieniocinis}

We first define the relevant Shimura varieties following \cite{Mad16} and \cite{RSZ20}.

Let $F_1/F$, $\daleth\in F_1^\times$ and $\mbf V, \mbf G, \mbf G^*, \mbf G^\sharp$ be as in \S\ref{seubsienfiehtoenfeis}.

\begin{defi}\label{staninihidnfies}
A pure inner twist $(\mbf V, \mbf G=\bx U(\mbf V)^\circ)$ of $\mbf G^*$ is called 
\begin{itemize}
\item
\tbf{standard definite} if $\mbf V\otimes_{F, \tau}\bb R$ is positive definite for each $\tau: F\to \bb R$.
\item
\tbf{standard indefinite} if $\mbf V\otimes_{F, \tau}\bb R$ has signature $(\dim\mbf V-2, 2)$ (resp. $(\dim\mbf V-1, 1)$) in Case O (resp. in Case U) for $\tau=\tau_0$ and positive definite (i.e., signature $(\dim\mbf V, 0)$) for each $\tau\in\infPla_F\setm\{\tau_0\}$.
\end{itemize}
\end{defi}

Suppose $(\mbf V, \mbf G=\bx U(\mbf V)^\circ)$ is a standard indefinite pure inner twist of $\mbf G^*$, and $\mbf G^\sharp$ is the central extension of $\Res_{F/\bb Q}\mbf G$ defined in \S\ref{seubsienfiehtoenfeis}. We have the Shimura datum $(\mbf G^\sharp, \mbf X^\sharp)$, where $\mbf X^\sharp$ is the conjugacy class of a Deligne homomorphism
\begin{equation}\label{delfieheihsojiems}
h^\sharp_0: \Res_{\bb C/\bb R}\GL_1\to \mbf G^\sharp\otimes\bb R
\end{equation}
defined as follows:
\begin{itemize}
\item
In Case O, let
\begin{equation*}
h_{0, \natural}: \Res_{\bb C/\bb R}\GL_1\to\paren{\Res_{F/\bb Q}\GSpin(\mbf V)}\otimes\bb R\cong \prod_{\tau\in\Hom(F, \bb R)}\GSpin(\mbf V\otimes_{F, \tau}\bb R)
\end{equation*}
be the homomorphism that is trivial on $\GSpin(\mbf V\otimes_{F, \tau}\bb R)$ for $\tau\ne \tau_0$, and on $\GSpin(\mbf V\otimes_{F, \tau_0}\bb R)$ it is induced by
\begin{equation*}
h^\sharp_{0, \tau_0}: \bb C^\times\to \GSpin(\mbf V\otimes_{F, \tau_0}\bb R): a+b\ii\mapsto a+be_{\tau_0, 1}e_{\tau_0, 2}
\end{equation*}
where $e_{\tau_0, 1}, e_{\tau_0, 2}$ are two orthogonal vectors in $\mbf V\otimes_{F, \tau_0}\bb R$ such that  $\norml{e_{\tau_0, 1}}=\norml{e_{\tau_0, 2}}=-1$. We also define
\begin{equation*}
h_{0, \daleth}: \Res_{\bb C/\bb R}\GL_1\to\paren{\Res_{F(\daleth)/F}\GL_1}\otimes\bb R\cong \prod_{\tau\in\Hom(F, \bb R)}\Res_{\bb C/\bb R}\GL_1
\end{equation*}
to be the homomorphism that is trivial on the $\tau_0$-factor and is the identity map on the other factors. We then define
\begin{equation*}
h_0^\sharp=(h_{0, \natural}, h_{0, \daleth}): \Res_{\bb C/\bb R}\GL_1\to \mbf G_\daleth\otimes\bb R,
\end{equation*}
which factors through $\mbf G^\sharp\otimes\bb R\subset \mbf G_\daleth\otimes\bb R$.
\item
In Case U, for each $\tau\in\Hom(F, \bb C)$, we fix a $\bb C$-basis $\mbf v=(v_1, \ldots, v_{n(G)})^\top$ of $\mbf V\otimes_{F_1, \tau}\bb C$ such that 
\begin{equation*}
J_\tau\defining (\bra{v_i, v_j}_{i, j})=\begin{cases}\diag(\underbrace{1, \ldots, 1}_{(n(G)-1)\text{-many}}, -1) &\If \tau=\tau_0\\ \uno &\If \tau\ne \tau_0\end{cases}.
\end{equation*}
Let
\begin{equation*}
h_{0, \natural}: \Res_{\bb C/\bb R}\GL_1\to \GU^{\bb Q}(\mbf V)\otimes\bb R\cong \prod_{\tau\in\Hom(F, \bb R)}\GU(\mbf V\otimes_{F_1, \tau}\bb C)
\end{equation*}
be the homomorphism such that  on each $\tau$-factor it is given by
\begin{equation*}
a+b\ii\mapsto  a+b\ii J_\tau.
\end{equation*}
Let
\begin{equation*}
h_{0, \daleth}: \Res_{\bb C/\bb R}\GL_1\to \mbf Z^{\bb Q}\otimes\bb R\subset \prod_{\tau\in\Hom(F, \bb R)}\Res_{\bb C/\bb R}\GL_1
\end{equation*}
be the diagonal embedding. We then define
\begin{equation*}
h_0^\sharp=(h_{0, \natural}, h_{0, \daleth}): \Res_{\bb C/\bb R}\GL_1\to\mbf G^\sharp\otimes\bb R.
\end{equation*}
\end{itemize}
It is routine to check that $(\mbf G^\sharp, \mbf X^\sharp)$ is a Shimura datum.

We define
\begin{equation*}
h_0: \Res_{\bb C/\bb R}\GL_1\to \Res_{F/\bb Q}\mbf G
\end{equation*}
to be the composition of $h^\sharp_0$ with the central extension $\mbf G^\sharp\to \Res_{F/\bb Q}\mbf G$, and let $\mbf X\defining\{h_0\}$ be the $G(F\otimes\bb R)$-conjugacy class of $h_0$. Then $(\Res_{F/\bb Q}\mbf G, \mbf X)$ is a Shimura datum. Its Hodge cocharacter is 
\begin{equation*}
\mu: \GL_{1,\bb C}\xr{z\mapsto (z, 1)}\paren{\Res_{\bb C/\bb R}\GL_1}_{\bb C}\xr{(h_0)_{\bb C}}\paren{\Res_{F/\bb Q}\mbf G}_{\bb C}.
\end{equation*}
Its reflex field $E$ is $F_1$, except in Case U with $n(\mbf G)=2$, where $E=F$. We use the embedding $E\inj\bb C$ induced by $\tau_0$.

For any rational prime $p$ and any fixed isomorphism $\iota_p:\bb C\xr\sim \ovl{\bb Q_p}$, the induced cocharacter of
\begin{equation*}
(\Res_{F/\bb Q}\mbf G)_{\ovl{\bb Q_p}}\cong \prod_{v:F\hookrightarrow\ovl{\bb Q_p}}\mbf G\otimes_{F,v}\ovl{\bb Q_p}.
\end{equation*}
is conjugate to the inverse of the cocharacter $\mu_1$ defined in~\eqref{idnifdujiherheins} on the factor corresponding to $\iota_p\circ \tau_0|_F$, and is trivial on all other factors.

We define $\mdc K_\infty$ and $\mdc K_\infty^\sharp$ to be the centralizer of $h_0$ and $h_0^\sharp$, respectively. Then
\begin{equation*}
\mbf X\cong\mbf G(F\otimes\bb R)/\mdc K_\infty,\qquad\mbf X^\sharp\cong\mbf G^\sharp(\bb R)/\mdc K_\infty^\sharp.
\end{equation*}

By Deligne's theory, there is a projective system
\begin{equation*}
\{\bSh_{\mdc K}(\Res_{F/\bb Q}\mbf G, \mbf X)\}
\end{equation*}
of Shimura varieties defined over $E$ indexed by neat open compact subgroups $\mdc K\le \mbf G(\Ade_{F, f})$ (as defined in \cite[\S 0.6]{Pin90}), with complex uniformization
\begin{equation*}
\bSh_{\mdc K}(\Res_{F/\bb Q}\mbf G,\mbf X)\otimes_{E,\tau_0}\bb C\cong\mbf G(F)\bsh\paren{\mbf X\times\mbf G(\Ade_{F,f})/\mdc K}.
\end{equation*}
Its complex dimension is $\dim_{\bb C}(\mbf X)$. When $E=F$, we tacitly replace this system by its base change to $F_1$ and retain the same notation.

Similarly, there is a projective system
\begin{equation*}
\{\bSh_{\mdc K^\sharp}(\mbf G^\sharp, \mbf X^\sharp)\}_{\mdc K^\sharp}
\end{equation*}
of Shimura varieties over the reflex field $E^\sharp$ indexed by neat open compact subgroups $\mdc K^\sharp\le \mbf G^\sharp(\Ade_f)$.\footnote{Note that the reflex field $E^\sharp$ may be bigger than $F_1$.} Each member has complex dimension $\dim_{\bb C}(\mbf X)$. If $\mdc K$ is the image of $\mdc K^\sharp$ in $\mbf G(\Ade_{F,f})$, then there is a morphism
\begin{equation}\label{finieihheinfiejsims}
\bSh_{\mdc K^\sharp}(\mbf G^\sharp,\mbf X^\sharp)\longrightarrow\bSh_{\mdc K}(\Res_{F/\bb Q}\mbf G,\mbf X)_{E^\sharp}\end{equation}
functorial in $\mdc K^\sharp$.

Finally, we check that $(\mbf G^\sharp, \mbf X^\sharp)$ is a Shimura datum of Hodge type and thus $(\Res_{F/\bb Q}\mbf G, \mbf X)$ is a Shimura datum of Abelian type. In Case U, this follows from \cite[\S 3.2]{RSZ20}, and $(\mbf G^\sharp, \mbf X^\sharp)$ is of PEL type. In Case O, let $c_\daleth$ be the nontrivial element of $\Gal(F(\daleth)/F)$. Let $\mbf H=\Cl(\mbf V)$, viewed as an $F$-representation of $\GSpin(\mbf V)$ via left multiplication, and set
\[
\mbf H_\daleth=\mbf H\otimes_F F(\daleth).
\]
We define an $F$-linear anti-involution $\dagger$ on $\mbf H_\daleth$ by
\[
(x\otimes a)^\dagger=x^*\otimes a^{c_\daleth},
\]
where $*$ is the main anti-involution on $\mbf H$; see~\cite[\S 1.1]{Mad16}. The use of the conjugation $c_\daleth$ is essential in order to obtain the norm character in the similitude factor below.

Let $\beta_\daleth\in \mbf H_\daleth^\times$ satisfy
\[
\beta_\daleth^\dagger=-\beta_\daleth.
\]
We define an $F$-valued pairing on $\mbf H_\daleth$, viewed as an $F$-vector space, by
\begin{equation*}
\psi_{\beta_\daleth,\daleth}(x,y)
=
\tr_{F(\daleth)/F}
\paren{
\trrd_{\mbf H_\daleth/F(\daleth)}
\paren{x\beta_\daleth y^\dagger}
},
\qquad x,y\in \mbf H_\daleth.
\end{equation*}
This pairing is nondegenerate and alternating. Indeed, using
\[
\trrd_{\mbf H_\daleth/F(\daleth)}(z^\dagger)
=
\trrd_{\mbf H_\daleth/F(\daleth)}(z)^{c_\daleth},
\]
the invariance of $\tr_{F(\daleth)/F}$ under $c_\daleth$, and the identity
$\beta_\daleth^\dagger=-\beta_\daleth$, we have
\begin{align*}
\psi_{\beta_\daleth,\daleth}(y,x)
&=
\tr_{F(\daleth)/F}
\paren{
\trrd_{\mbf H_\daleth/F(\daleth)}
\paren{y\beta_\daleth x^\dagger}
} \\
&=
\tr_{F(\daleth)/F}
\paren{
\trrd_{\mbf H_\daleth/F(\daleth)}
\paren{(y\beta_\daleth x^\dagger)^\dagger}
} \\
&=
\tr_{F(\daleth)/F}
\paren{
\trrd_{\mbf H_\daleth/F(\daleth)}
\paren{x\beta_\daleth^\dagger y^\dagger}
} \\
&=
-\psi_{\beta_\daleth,\daleth}(x,y).
\end{align*}

For $(g, t)\in \GSpin(\mbf V)\times F(\daleth)^\times$, let $(g,t)$ act on $\mbf H_\daleth$ by
\[
x\longmapsto gxt.
\]
Then
\begin{align*}
\psi_{\beta_\daleth,\daleth}(gxt,gyt)
&=
\tr_{F(\daleth)/F}
\paren{
\trrd_{\mbf H_\daleth/F(\daleth)}
\paren{gxt\beta_\daleth(gyt)^\dagger}
} \\
&=
\tr_{F(\daleth)/F}
\paren{
\trrd_{\mbf H_\daleth/F(\daleth)}
\paren{gx\,t\beta_\daleth t^{c_\daleth}y^\dagger g^*}
} \\
&=
\nu(g)\Nm_{F(\daleth)/F}(t)\,
\tr_{F(\daleth)/F}
\paren{
\trrd_{\mbf H_\daleth/F(\daleth)}
\paren{x\beta_\daleth y^\dagger}
} \\
&=
\nu(g)\Nm_{F(\daleth)/F}(t)\,
\psi_{\beta_\daleth,\daleth}(x,y),
\end{align*}
where we used $g^*g=\nu(g)$ and the cyclicity of the reduced trace. The anti-diagonally embedded $\Res_{F/\bb Q}\GL_1$ acts trivially on $\mbf H_\daleth$, and the above action therefore descends to an embedding
\begin{align*}
a:\mbf G_\daleth=\paren{(\Res_{F/\bb Q}\GSpin(\mbf V))\times\Res_{F(\daleth)/\bb Q}\GL_1}/\Res_{F/\bb Q}\GL_1&\inj\Res_{F/\bb Q}\GSp_F\paren{\mbf H_\daleth,\psi_{\beta_\daleth,\daleth}}, \\
(g,t)&\longmapsto (x\mapsto gxt).
\end{align*}
Moreover, if $\nu_{\GSp}$ denotes the similitude character on
$\GSp_F\paren{\mbf H_\daleth,\psi_{\beta_\daleth,\daleth}}$, then
\begin{equation*}
\Res_{F/\bb Q}(\nu_{\GSp})\circ a=\nu^\sharp.
\end{equation*}

Now define the $\bb Q$-valued symplectic form
\begin{equation*}
\psi_{\beta_\daleth,\daleth}^{\bb Q}=\tr_{F/\bb Q}\circ\psi_{\beta_\daleth,\daleth}.
\end{equation*}
By the definition of $\mbf G^\sharp$, the restriction of $\nu^\sharp$ to $\mbf G^\sharp$ takes values in the diagonal subtorus
\begin{equation*}
\GL_1\subset \Res_{F/\bb Q}\GL_1.
\end{equation*}
Consequently, after viewing $\mbf H_\daleth$ as a $\bb Q$-vector space, the restriction of $a$ gives an embedding of reductive groups over $\bb Q$
\begin{equation*}
\rho: \mbf G^\sharp\inj\GSp\paren{\mbf H_\daleth,\psi_{\beta_\daleth,\daleth}^{\bb Q}},
\end{equation*}
where the similitude factor is the above diagonal $\bb Q$-valued character.

It remains to choose $\beta_\daleth$ so that the above symplectic
representation is a morphism of Shimura data. Equivalently, for every real
embedding $\tau:F\hookrightarrow\bb R$, if
\[
J_\tau=\rho(h_0^\sharp(\ii))
\]
is the induced complex structure on
$\mbf H_\daleth\otimes_{F,\tau}\bb R$, then the symmetric form
\[
(x,y)\longmapsto
\psi_{\beta_\daleth,\daleth,\tau}(x,J_\tau y)
\]
must be definite.

Such a choice of $\beta_\daleth$ exists by the standard Kuga--Satake
linear algebra. At the distinguished real place $\tau_0$, where
$\mbf V_{\tau_0}$ has signature $(\dim\mbf V-2,2)$, this is precisely the
usual choice of a skew Clifford element giving a polarization of the
Clifford representation; see~\cite[3.5]{Mad16}. At the remaining real
places $\tau\ne\tau_0$, the space $\mbf V_\tau$ is positive definite and
the complex structure comes only from the $F(\daleth)^\times$-factor, so a
suitable real multiple of $1\otimes\daleth$ gives the required definite
form. These local positivity conditions define nonempty open subsets of
\[
\{b\in \mbf H_\daleth\otimes_{F,\tau}\bb R: b^\dagger=-b\}.
\]
Together with the open condition of invertibility, weak approximation gives
a global element
\[
\beta_\daleth\in\mbf H_\daleth^\times,\qquad
\beta_\daleth^\dagger=-\beta_\daleth,
\]
satisfying the required positivity at all Archimedean places.

For such a choice of $\beta_\daleth$, the above embedding sends $h_0^\sharp$ to the Siegel double space attached to the symplectic space
\begin{equation*}
\paren{\mbf H_\daleth,\psi_{\beta_\daleth,\daleth}^{\bb Q}}.
\end{equation*}
Hence it induces an embedding of Shimura data
\begin{equation*}
(\mbf G^\sharp,\mbf X^\sharp)\inj\paren{\GSp\paren{\mbf H_\daleth,\psi_{\beta_\daleth,\daleth}^{\bb Q}},\mcl X},
\end{equation*}
where $\mcl X$ is the union of Siegel upper half-spaces attached to
$\paren{\mbf H_\daleth,\psi_{\beta_\daleth,\daleth}^{\bb Q}}$.
Thus $(\mbf G^\sharp,\mbf X^\sharp)$ is of Hodge type. Since
$\mbf G^\sharp$ is a central extension of $\Res_{F/\bb Q}\mbf G$, the Shimura datum
$(\Res_{F/\bb Q}\mbf G,\mbf X)$ is of Abelian type.

\subsection{Langlands--Kottwitz method}\label{Lanlgan-Kotiemethiems}

In this subsection, we apply the Langlands--Kottwitz method to relate the action of Frobenius elements at primes of good reduction to the Hecke action on the compactly supported cohomology of orthogonal or unitary Shimura varieties. We adopt the notation from \S\ref{seubsienfiehtoenfeis}, and assume:

\begin{note}\enskip
\begin{itemize}
\item
$F$ has a finite place $\mfk q$ whose residue characteristic is odd and which is inert in $F_1$ in Case U.
\item
$\bigstar$ is a control tuple (see \textup{Definition~\ref{deieiutesuroeneis}}) such that  $\mfk q\in \Pla^\St$ and $\Pla$ is of the form $\Pla=\Pla_F(S_{\bx{bad}})$, where $S_{\bx{bad}}$ is a finite set of rational primes containing $2$ and all rational primes ramified in $F$.
\item
$(\mbf G, \bm\varrho,\bm z)$ is a $\bigstar$-good pure inner form of $\mbf G^*$.
\item
$\pi$ is a $\bigstar$-good automorphic representation of $\mbf G$.
\item
$\mdc K\le \mbf G(\Ade_{F, f})$ is a $\bigstar$-split compact open subgroup with $\pi^{\mdc K}\ne0$.
\item
$\ell$ is a rational prime together with an isomorphism $\iota_\ell: \bb C\xr\sim\ovl{\bb Q_\ell}$.
\item
$(\mfk X, \chi)$ is a central character datum for $\mbf G^*$ (see \textup{Definition~\ref{isnsieieifmeips}}) with
\begin{equation*}
\mfk X=(\bm\varrho^{-1}(\mdc K)\cap Z_{\mbf G^*}(\Ade_{F, f}))\times Z_{\mbf G^*}(F\otimes\bb R),
\end{equation*}
and $\chi$ is the inverse of the central character of $\xi$ (extended from $Z_{\mbf G^*}(F\otimes\bb R)$ to $\mfk X$ trivially on $\bm\varrho^{-1}(\mdc K)\cap Z_{\mbf G^*}(\Ade_{F, f})$).
\end{itemize}
\end{note}

\begin{defi}\label{automroehisliinifheis}
Suppose $(\bb G, \bb X)$ is any Shimura datum with reflex field $E\subset \bb C$ with associated projective system of Shimura varieties $\{\bSh_{\mdc K}(\bb G, \bb X)\}_{\mdc K}$ defined over $E$, indexed by the set of neat compact open subgroups $\mdc K\le\bb G(\Ade_f)$. Let $Z_{\bx a}$ be the maximal anisotropic $\bb Q$-subtorus of $Z(\bb G)$, and let $Z_{\bx{ac}}$ be the smallest $\bb Q$-subgroup of $Z_{\bx a}$ whose base change to $\bb R$ contains the maximal $\bb R$-split subtorus of $Z_{\bx a}$; see~\cite[Definition 1.5.4]{KSZ21}. For each irreducible algebraic representation $\xi$ of $\bb G_{\bb C}$ that is trivial on $Z_{\bx{ac}}$, there exists a compatible system of lisse $\ovl{\bb Q_\ell}$-local systems $\mrs L_{\iota_\ell\xi}$ on this projective system of Shimura varieties associated with $\xi$; see~\cite[1.5.8]{KSZ21}. For each $i\in\bb N$, we define
\begin{equation*}
\cetH^i(\bSh, \mrs L_{\iota_\ell\xi})\defining \ilim_{\mdc K\to \uno}\cetH^i\paren{\bSh_{\mdc K}(\bb G, \bb X)_{\ovl E}, \mrs L_{\iota_\ell\xi}}.
\end{equation*}
This is a $\bb G(\Ade_f)\times\Gal_E$-module with admissible $\bb G(\Ade_f)$-action and continuous $\Gal_E$-action.

For each admissible representation $\Pi$ of $\bb G(\Ade)$, we define
\begin{equation}\label{imfihifmos}
\cetH^i(\bSh, \mrs L_{\iota_\ell\xi})[\Pi^\infty]\defining \Hom_{\bb G(\Ade_f)}\paren{\iota_\ell\Pi^\infty, \cetH^i(\bSh, \mrs L_{\iota_\ell\xi})}.
\end{equation}
This is a finite dimensional representation of $\Gal_E$, unramified at all but finitely many places. We set  \begin{equation*}
\cetH^i(\bSh,\mrs L_{\iota_\ell\xi})^\sems[\Pi^\infty]\defining\paren{\cetH^i(\bSh,\mrs L_{\iota_\ell\xi})[\Pi^\infty]}^\sems
\end{equation*}
for its semisimplification.

Similarly, if $\bb G'$ is any reductive group over $\bb Q$ such that  $\bb G'(\bb R)$ is compact, then for each irreducible algebraic representation $(\xi, V_\xi)$ of $\bb G'_{\bb C}$, we define the injective system of \tbf{algebraic automorphic forms} valued in $\iota_\ell V_\xi$ as
\begin{equation*}
\Brace{\mcl A\paren{\bb G'(\bb Q)\bsh \bb G'(\Ade_f)/\mdc K, \mrs L_{\iota_\ell\xi}}}_{\mdc K}
\end{equation*}
indexed by compact open subgroups $\mdc K\le \bb G'(\Ade_f)$, where $\mcl A\paren{\bb G'(\bb Q)\bsh \bb G'(\Ade_f)/\mdc K, \mrs L_{\iota_\ell\xi}}$ consists of maps $\phi: \bb G'(\Ade_f)\to \iota_\ell V_\xi$ such that  $\phi(gk)=\phi(g)$ and $\phi(\gamma g)=\gamma.\phi(g)$ for any $g\in \bb G'(\Ade_f), \gamma\in \bb G'(\bb Q)$ and $k\in \mdc K$. For each compact open subgroup $\mdc K\le \bb G'(\Ade_f^p)$, we write
\begin{equation*}
\mcl A\paren{\bb G'(\bb Q)\bsh \bb G'(\Ade_f)/\mdc K^p, \mrs L_{\iota_\ell\xi}}\defining \ilim_{\mdc K_p}\mcl A\paren{\bb G'(\bb Q)\bsh \bb G'(\Ade_f)/\mdc K_p\mdc K^p, \mrs L_{\iota_\ell\xi}}
\end{equation*}
where $\mdc K_p$ runs through compact open subgroups of $\bb G'(\bb Q_p)$.
\end{defi}

The following theorem describes the Galois action on the cohomology of orthogonal or unitary Shimura varieties, which is the main result of \cite{KSZ21} (cf.~\cite[Theorem 7.3]{K-S23}).

\begin{thm}\label{stabilzienKlanglakOtiwinform}
Suppose $\mdc K_p\le\mbf G(F\otimes\bb Q_p)$ is hyperspecial for some $p\ge 3$, and $\iota_p: \bb C\xr\sim\ovl{\bb Q_p}$ is an isomorphism such that  $\iota_p\circ\tau_0: F_1\to \ovl{\bb Q_p}$ induces a finite place $\mfk p\in\Pla_{F_1}(\{p\})$. Define a test function $f^\infty=f^{\infty, p}f_p\in \mcl H\paren{\Res_{F/\bb Q}\mbf G(\Ade_f), \mdc K}$ with $f_p=\uno_{\mdc K_p}$. Then there exists $j_0\in\bb Z_+$ such that
\begin{equation*}
\sum_{i=0}^{2\dim_{\bb C}(\mbf X)}(-1)^i\iota_\ell^{-1}\Tr\paren{\iota_\ell f^\infty\sigma_{\mfk p}^j|\cetH^i(\bSh, \mrs L_{\iota_\ell\xi})}=\sum_{\mfk e\in \mcl E_\ellip(\mbf G)}\iota(\mfk e)\ST_{\ellip, \chi}^{\mbf G^{\mfk e}}(h^{\mbf G^{\mfk e}}_{\xi, j})
\end{equation*}
for all positive integers $j\ge j_0$. Here $\ST^{\mbf G^{\mfk e}}_{\ellip, \chi}$ is the elliptic stable distribution associated with $\mfk e$ (see \textup{Definition~\ref{staieliaisieuels}}), $\iota(\mfk e)\in\bb Q$ is the global coefficient introduced by Kottwitz and Shelstad (cf.~\cite[Equation~(3.2.4)]{Art13}), and $h^{\mbf G^{\mfk e}}_{\xi, j}=h^{\mbf G^{\mfk e}, p\infty}h_{p, j}^{\mbf G^{\mfk e}}h_{\infty, \xi}^{\mbf G^{\mfk e}}\in \mcl H(\mbf G^{\mfk e}(\Ade_F), \chi^{-1})$ are defined in \textup{\cite{Kot90, KSZ21}}. In particular, $\iota(\mfk e)=1$ if $\mbf G^{\mfk e}=\mbf G^*$, and 
\begin{itemize}
\item
$h^{\mbf G^*, p\infty}$ is an endoscopic transfer of $f^{p\infty}$, (Note that such a transfer exists in the fixed-central character setting, by first lifting $f$ along the averaging map $\mcl H(\mbf G(\Ade_F))\to \mcl H(\mbf G(\Ade_F), \chi^{-1})$, and then taking the transfer to $\mcl H(\mbf G^*(\Ade_F))$, and finally taking the image along the averaging map $\mcl H(\mbf G^*(\Ade_F))\to \mcl H(\mbf G^*(\Ade_F), \chi^{-1})$),
\item
$h_{p, j}^{\mbf G^*}$ is the base change transfer of
\begin{equation*}
\phi_j\defining\uno_{\mcl G_p\paren{\bb Z_{\norml{\mfk p}^j}}\mu(p^{-1})\mcl G_p\paren{\bb Z_{\norml{\mfk p}^j}}}\in\mcl H\paren{\mbf G(F\otimes\bb Z_{\norml{\mfk p}^j}), \mcl G_p\paren{\bb Z_{\norml{\mfk p}^j}}}
\end{equation*}
(where we write $\mcl G_p$ for $\prod_{v\in\Pla_F(\{p\})}\Res_{\mcl O_{F_v}/\bb Z_p}\mcl G_v$ and write $\bb Z_{\norml{\mfk p}^j}$ for the integer ring of the unramified extension $\bb Q_{\norml{\mfk p}^j}$ of $\bb Q_p$ of degree $j\cdot\log_p(\norml{\mfk p})$) down to $\mcl H\paren{\mbf G(F\otimes\bb Q_p), \mcl G_p(\bb Z_p)}$,
\item
$h_{\infty,\xi}^{\mbf G^*}\defining\#\Pi_\xi(\mbf G^*(F\otimes\bb R))^{-1}\sum_{\tau_\infty\in \Pi_\xi(\mbf G^*(F\otimes\bb R))}f_{\tau_\infty}$, that is, the average of the pseudo-coefficients for the discrete series $L$-packet of $\mbf G^*(F\otimes\bb R)$ associated with $\xi$.
\end{itemize}
\end{thm}

\begin{defi}\label{ieindiiivinereheieis}
Let $[\pi]$ be the set of isomorphism classes of $\bigstar$-good automorphic representations $\dot\pi$ of $\mbf G(\Ade_F)$ such that $\tilde{\dot\pi}_v\cong\tilde\pi_v$ for each $v\in\fPla_F$. Consider two automorphic representations $\dot\pi_1, \dot\pi_2\in [\pi]$ equivalent and write $\dot\pi_1\sim\dot\pi_2$ if $\dot\pi_1^\infty\cong \dot\pi_2^\infty$. In particular, $[\pi]/\sim$ is a singleton in Case U and Case O1.

Choose one representative from each equivalence class in $[\pi]/\sim$. We define the virtual Galois representation
\begin{equation*}
\rho_\bSh^\pi\defining (-1)^{\dim_{\bb C}(\mbf X)}\sum_{\dot\pi\in [\pi]/\sim}\sum_{i=0}^{2\dim_{\bb C}(\mbf X)}(-1)^i\cetH^i(\bSh, \mrs L_{\iota_\ell\xi})^\sems[\dot\pi^\infty]\in\bx K_0\paren{\ovl{\bb Q_\ell}[\Gal_{F_1}]}, 
\end{equation*}
where $\bx K_0\paren{\ovl{\bb Q_\ell}[\Gal_{F_1}]}$ is the Grothendieck group of finite dimensional continuous representations of $\Gal_{F_1}$ with $\ovl{\bb Q_\ell}$-coefficients unramified at all but finitely many places.
\end{defi}

We will also need the following cohomology spaces to deal with non-compact Shimura varieties: Let
\begin{equation*}
\bx H^i_{(2)}(\bSh, \mrs L_\xi)\defining \ilim_{\mdc K\to 1}\bx H^i_{(2)}(\bSh_{\mdc K}(\mbf G, \mbf X), \mrs L_\xi)
\end{equation*}
be the $L^2$-cohomology of $\bSh(\mbf G, \mbf X)\times_{F_1, \tau_0}\bb C$ as defined in \cite[\S 6]{Fal83}, and let 
\begin{equation*}
\bx{IH}^*(\bSh, \mrs L_{\iota_\ell\xi})\defining \ilim_{\mdc K\to 1}\bx{IH}^*(\bSh_{\mdc K}(\mbf G, \mbf X), \mrs L_{\iota_\ell\xi})
\end{equation*}
be the $\ell$-adic intersection cohomology of $\bSh(\mbf G, \mbf X)$. These two cohomologies are equipped with admissible $\mbf G(\Ade_{F, f})$-actions defined by Hecke correspondences. There are natural $\mbf G(\Ade_{F, f})$-equivariant maps
\begin{equation}\label{maosinfiehfeis}
\cetH^i(\bSh, \mrs L_{\iota_\ell\xi})\to \iota_\ell\bx H^i_{(2)}(\bSh, \mrs L_\xi)\to \etH^i(\bSh, \mrs L_{\iota_\ell\xi}),
\end{equation}
\begin{equation*}
\cetH^i(\bSh, \mrs L_{\iota_\ell\xi})\to \bx{IH}^i(\bSh, \mrs L_{\iota_\ell\xi}),
\end{equation*}
and it follows from Zucker's conjecture \cite{Loo88, L-R91, S-S90} that there is a $\mbf G(\Ade_{F, f})$-equivariant commutative diagram
\begin{equation}\label{liniihsniefiens}
\begin{tikzcd}
\cetH^i(\bSh, \mrs L_{\iota_\ell\xi})\ar[r]\ar[rd] &\bx{IH}^i(\bSh, \mrs L_{\iota_\ell\xi})\ar[d, "\cong"]\\
& \iota_\ell\bx H^i_{(2)}(\bSh, \mrs L_\xi)
\end{tikzcd}.
\end{equation}

\begin{lm}\label{temrpeodounirnifehis}
The maps in~\textup{\eqref{maosinfiehfeis}} induce isomorphisms
\begin{equation*}
\cetH^i(\bSh, \mrs L_{\iota_\ell\xi})[\pi^\infty]\cong\iota_\ell\bx H_{(2)}^i(\bSh, \mrs L_\xi)[\pi^\infty]\cong\etH^i(\bSh, \mrs L_{\iota_\ell\xi})[\pi^\infty].
\end{equation*}
Moreover, $\dim\cetH^i(\bSh, \mrs L_{\iota_\ell\xi})^\sems[\pi^\infty]=\dim\cetH^i(\bSh, \mrs L_{\iota_\ell\xi})[\pi^\infty]$.
\end{lm}
\begin{proof}
The isomorphisms follow from Franke's spectral sequence \cite[Theorem 19]{Fra98} and the last assertion follows from Borel--Casselman's decomposition of $\bx H_{(2)}^i(\bSh, \mrs L_\xi)$ as direct sums of certain multiplicities of $\pi_f$ for each $\pi\in L^2_\disc(\mbf G(\Ade_F))^\sm$ (thus $\bx H_{(2)}^i(\bSh, \mrs L_\xi)$ is semisimple as a $\mbf G(\Ade_{F, f})$-module); see~\cite{B-C83}. The proof is the same as that of \cite[Lemma 8.1(1)]{K-S23}, thus omitted here.
\end{proof}

\begin{thm}\label{eieninifhenfiews}
For all but finitely many finite places $\mfk p$ of $F_1$ not lying over $\ell$ or places in $\Pla$, and for all sufficiently large positive integers $j$ (depending on $\mfk p$),
\begin{equation*}
\tr\paren{\sigma_{\mfk p}^j|\rho_\bSh^\pi}=m(\pi)\cdot\norml{\mfk p}^{\frac{\dim_{\bb C}(\mbf X)}{2}\cdot j}\cdot\tr\paren{\iota_\ell\tilde\phi_{\pi_{\mfk p_\flat}}^\GL(\sigma_{\mfk p}^j)},
\end{equation*}
where $\tilde\phi_{\pi_{\mfk p_\flat}}$ is the classical $L$-parameter of $\pi_{\mfk p_\flat}$. Moreover, the only nonzero term in the definition of $\rho_\bSh^\pi$ (see \textup{Definition~\ref{ieindiiivinereheieis}}) appears in the middle degree $\dim_{\bb C}(\mbf X)$. In particular, $\rho_\bSh^\pi$ is a genuine representation of $\Gal_{F_1}$.
\end{thm}
\begin{proof}
We imitate the arguments of \cite{Kot92} and \cite[Proposition 8.2]{K-S23}. Consider the test function $f=f_\infty\otimes f_{\Pla^\St}\otimes f^{\Pla^\St\cup\infPla_F}$ on $\mbf G(\Ade_F)$ such that 
\begin{itemize}
\item
$f_\infty=\paren{\#\Pi_\xi(\mbf G(F\otimes\bb R))}^{-1}\sum_{\tau_\infty\in \Pi_\xi(\mbf G(F\otimes\bb R))}f_{\tau_\infty}$, i.e., the average of the pseudo-coefficients for the discrete series $L$-packet of $\mbf G(F\otimes\bb R)$ associated with $\xi$.
\item
$f^{\Pla^\St\cup\infPla_F}\in\tilde{\mcl H}\paren{\mbf G(\Ade_{F,f}^{\Pla^\St})}$ so that, on the finite set of cuspidal automorphic representations $\dot\pi$ satisfying
\begin{equation*}
(\dot\pi_f)^{\mdc K}\ne0\qquad\text{and}\qquad\tr(f_\infty|\dot\pi_\infty)\ne0,
\end{equation*}
it acts by the scalar $1$ on those $\dot\pi$ satisfying
\begin{equation*}
\dot\pi^\Pla\text{ has the same }\tilde{\bb T}^\Pla\text{-character as }\pi^\Pla
\end{equation*}
and
\begin{equation*}
\tilde{\dot\pi}_v\cong\tilde\pi_v\qquad(v\in\Pla\setm\Pla^\St),
\end{equation*}
and by the scalar $0$ on the remaining representations in this finite set. Let $\Pla_{\bx{aux}}\subset\fPla_F\setm\Pla^\St$ be the finite set of places at which the local component of $f^{\Pla^\St\cup\infPla_F}$ is not the unit of the corresponding spherical Hecke algebra.
\item
$f_{\Pla^\St}=\otimes_{v\in\Pla^\St}f_{\Lef, v}^{\mbf G}$, where each $f_{\Lef, v}^{\mbf G}$ is a Lefschetz function (see \textup{Definition~\ref{noninIAhienihifneins}}).
\end{itemize}
Choose a finite place $\mfk p$ of $F_1$ whose underlying place $\mfk p_\flat$ of $F$ does not belong to $\Pla\cup\Pla^{\bx{aux}}$ and whose residue characteristic $p$ is different from $\ell$. Choose an isomorphism $\iota_p: \bb C\xr\sim \ovl{\bb Q_p}$ such that $\iota_p\circ\tau_0: F_1\to \ovl{\bb Q_p}$ induces the place $\mfk p$. Then the stabilized Langlands--Kottwitz formula (Theorem~\ref{stabilzienKlanglakOtiwinform}) simplifies to
\begin{equation*}
\sum_{i=0}^{2\dim_{\bb C}(\mbf X)}(-1)^i\iota_\ell^{-1}\Tr\paren{\iota_\ell f^\infty\sigma_{\mfk p}^j|\cetH^i(\bSh, \mrs L_{\iota_\ell\xi})}=\ST_{\ellip, \chi}^{\mbf G^*}(h^{\mbf G^*}_{\xi, j}),
\end{equation*}
because the stable orbital integrals of $h_v^{\mbf  G^{\mfk e}}$ vanish for $\mbf G^{\mfk e}\ne\mbf G^*$ as they equal $\kappa$-orbital integrals of $f_{\Lef, v}^{\mbf G}$ with $\kappa\ne\uno$ up to a nonzero constant, and $f_{\Lef, v}^{\mbf G}$ is stabilizing (see \textup{Definition~\ref{noninIAhienihifneins}}); see~ \cite[Theorem 4.3.4]{Lab99}. Note that the left-hand side equals
\begin{equation*}
(-1)^{\dim_{\bb C}(\mbf X)+\sum_{v\in\Pla^\St}q(\mbf G_v)}\tr\paren{\sigma_{\mfk p}^j|\rho_\bSh^\pi}
\end{equation*}
by definition of $f^\infty=f_{\Pla^\St}\otimes f^{\Pla^\St\cup\infPla_F}$, where $q(\mbf G_v)$ is the $F_v$-rank of $\mbf G_v$.

Then it follows from the simple stable trace formula (Theorem~\ref{simpltraceofmrualoens}) and Definition~\ref{Lelfinedihfinies} that
\begin{equation}\label{osinehfiniiws}
\ST_{\ellip, \chi}^{\mbf G^*}(h^{\mbf G^*}_{\xi, j})=\bx T_{\cusp, \chi}^{\mbf G}(f^{p\infty}f_{p, j}'f_\infty)=\sum_{\dot\pi\in\cuspRep_\chi(\mbf G)}m(\dot\pi)\tr(f^{p\infty}|\dot\pi^{p\infty})\tr(f_{p, j}'|\dot\pi_p)\tr(f_\infty|\dot\pi_\infty)
\end{equation}
where $f_{p, j}'=h_{p, j}^{\mbf G^*}$ via the identification $\varrho_p: \mbf G^*(F\otimes\bb Q_p)\xr\sim \mbf G(F\otimes\bb Q_p)$. The term on the right-hand side vanishes unless $\dot\pi_p$ is unramified and both traces $\tr(f^{p\infty}|\dot\pi^{p\infty})$ and $\tr(f_\infty|\dot\pi_\infty)$ are nonzero. If a summand indexed by $\dot\pi$ contributes, then the choice of
$f^{\Pla^\St\cup\infPla_F}$ implies that $\dot\pi^\Pla$ has the same $\tilde{\bb T}^\Pla$-character as $\pi^\Pla$, and that
\begin{equation*}
\tilde{\dot\pi}_v\cong\tilde\pi_v
\qquad
(v\in\Pla\setm\Pla^\St).
\end{equation*}
By strong multiplicity one for the strong functorial transfers and Corollary~\ref{sotnienfiehineiws}, $\dot\pi$ and $\pi$ have the same local classical $L$-parameter at every finite place. At $v\in\Pla^\St$ the packet is a singleton, and at $v\notin\Pla$ both representations are spherical. Hence
$\dot\pi\in[\pi]$. In particular, $\dot\pi_v\cong \pi_v$ for $v\in \Pla_F(\{p\})$ as $\pi_v$ is unramified. By Definition~\ref{Lelfinedihfinies}, the right-hand side of Equation~\eqref{osinehfiniiws} equals
\begin{align*}
&\frac{1}{\#\Pi_\xi(\mbf G(F\otimes\bb R))}\sum_{\dot\pi\in[\pi]}m(\dot\pi)(-1)^{\sum_{v\in\Pla^\St}q(\mbf G_v)}\bx{ep}(\dot\pi_\infty\otimes\xi)\tr\bigg(h^{\mbf G^*}_{p, j}|\prod_{v\in\Pla_F(\{p\})}\pi_v\bigg)\\
&=(-1)^{\dim_{\bb C}(\mbf X)+\sum_{v\in\Pla^\St}q(\mbf G_v)}m(\pi)\tr\bigg(h^{\mbf G^*}_{p, j}|\prod_{v\in\Pla_F(\{p\})}\pi_v\bigg).
\end{align*}
Here we use that $m(\dot\pi)=m(\pi)$ for each $\dot\pi\in [\pi]$ by Arthur's multiplicity formula. Thus
\begin{equation*}
\tr\paren{\sigma_{\mfk p}^j|\rho_\bSh^\pi}=m(\pi)\tr\bigg(h^{\mbf G^*}_{p, j}|\prod_{v\in\Pla_F(\{p\})}\pi_v\bigg).
\end{equation*}

To analyze the right-hand side, consider the conjugacy class
\begin{equation*}
\{\mu\}: \GL_{1, \bb C}\to \paren{\Res_{F/\bb Q}\mbf G}_{\bb C}
\end{equation*}
of Hodge cocharacters defined in \S\ref{Shimreiejvanieniocinis}. The irreducible highest-weight representation
\begin{equation*}
\paren{\Res_{F/\bb Q}\mbf G}^\wedge\cong\prod_{\tau\in\infPla_F}\hat{\mbf G_\tau}
\end{equation*}
associated with $\{\mu\}$ is the standard representation on the $\tau_0$ component and trivial on other components. Furthermore, the Satake parameter of $\pi_v$ belongs to the $\sigma_v$-coset of $\paren{(\Res_{F/\bb Q}\mbf G)\otimes\bb Q_p}^\wedge$, identified with $\paren{\Res_{F/\bb Q}\mbf G}^\wedge$ via $\hat{\iota_p}$, in the $L$-group. Thus the $\tau_0$-component of the Satake parameter of the representation $\prod_{v\in\Pla_F(\{p\})}\pi_v$ of $\paren{\Res_{F/\bb Q}\mbf G}\otimes\bb Q_p$ is identified with the Satake parameter of $\pi_{\mfk p_\flat}$, since the composite
\begin{equation*}
F\xr{\tau_0}\bb C\xr{\iota_p}\ovl{\bb Q_p}
\end{equation*}
induces the place $\mfk p_\flat$. Then it follows from \cite[(2.2.1)]{Kot84} that
\begin{equation*}
\tr\bigg(h^{\mbf G^*}_{p, j}|\prod_{v\in\Pla_F(\{p\})}\pi_v\bigg)=\norml{\mfk p}^{\frac{\dim_{\bb C}(\mbf X)}{2}j}\tr\paren{\iota_\ell\tilde\phi^\GL_{\pi_{\mfk p_\flat}}(\sigma_{\mfk p}^j)},
\end{equation*}
and the first assertion follows.

For every $i$, define the genuine semisimple Galois representation
\begin{equation*}
V_i\defining\bplus_{\dot\pi\in[\pi]/\sim}\paren{\bx{IH}^i\paren{\bSh,\mrs L_{\iota_\ell\xi}}[\dot\pi^\infty]}^{\sems}.
\end{equation*}
By Lemma~\ref{temrpeodounirnifehis} and~\eqref{liniihsniefiens}, we obtain
\begin{equation*}
[\rho_\bSh^\pi]=\sum_{i=0}^{2\dim_{\bb C}(\mbf X)}(-1)^{i-\dim_{\bb C}(\mbf X)}[V_i].
\end{equation*}

The coefficient local system $\mrs L_{\iota_\ell\xi}$ is pure of weight $0$. Hence $V_i$ is pure of weight $i$, by Pink's purity result~\cite[Proposition 5.6.2]{Pin92} and the purity of intersection cohomology. (Note that the weight cocharacter of the Shimura datum which appears in \cite[\S 5.4]{Pin92} is trivial because $Z(\mbf G)$ is anisotropic).

Fix a place $\mfk p$ for which the first assertion of the theorem holds. Since $\dot\pi_{\mfk p_\flat}$ is tempered for every $\dot\pi\in[\pi]$, the first part of the theorem and Theorem~\ref{coeleninifheihsn} imply that, for
all sufficiently large $j$,
\begin{equation*}
\sum_{i=0}^{2\dim_{\bb C}(\mbf X)}(-1)^{i-\dim_{\bb C}(\mbf X)}\iota_{\ell}^{-1}\tr\paren{\sigma_{\mfk p}^j|V_i}
\end{equation*}
is a power sum all of whose bases have complex absolute value $\norml{\mfk p}^{\frac{j\dim_{\bb C}(\mbf X)}{2}}$.

Since $V_i$ is pure of weight $i$ and all constituents of a fixed $V_i$ occur with
the same sign, no cancellation is possible for $i\ne \dim_{\bb C}(\mbf X)$. The linear independence of the sequences $j\mapsto\alpha^j$ implies that $V_i=0$ for $i\ne \dim_{\bb C}(\mbf X)$. Thus only the middle degree contributes, and
\begin{equation*}
\rho_\bSh^\pi\cong V_{\dim_{\bb C}(\mbf X)}
\end{equation*}
is a genuine representation of $\Gal_{F_1}$.
\end{proof}

\begin{cor}\label{iodnifneidkhnsiws}
The Galois representation $\rho_\bSh^\pi$ has contribution only coming from the middle degree $\dim_{\bb C}(\mbf X)$, and for each $w\in \fPla_{F_1}$ with underlying place $v\in\fPla_F$ it has a subquotient whose semisimplification is isomorphic to
\begin{equation*}
\paren{\rho_{\pi, \ell}|_{W_{F_{1,w}}}}^\sems\otimes\iota_\ell\largel{-}_{F_{1,w}}^{\frac{N(\mbf G)-1-\dim_{\bb C}(\mbf X)}{2}}\cong\iota_\ell\paren{\tilde\phi_{\pi_v}^\GL\otimes\largel{-}_{F_{1,w}}^{-\frac{\dim_{\bb C}(\mbf X)}{2}}}
\end{equation*}
as a $W_{F_{1,w}}$-module, where $\rho_{\pi, \ell}$ is defined in \textup{Theorem~\ref{Galosiniautoenrieeis}}.
\end{cor}
\begin{proof}
This follows from Theorem~\ref{eieninifhenfiews}, together with the local-global compatibility and irreducibility of $\rho_{\pi, \ell}$ (Theorem~\ref{Galosiniautoenrieeis}), the Brauer--Nesbitt theorem, and the Chebotarev density theorem.
\end{proof}

To end this section, we recall a hypothesis describing the Galois representation appearing in the generic part of the middle degree $\ell$-adic cohomology of orthogonal and unitary Shimura varieties. It will later be used in the proof of the strong Kottwitz conjecture.

\begin{hyp}\label{hsisieteijeiureis}
Suppose $\xi$ is the trivial representation of $(\Res_{F/\bb Q}\mbf G)\otimes_{\bb Q}\bb C$, except in Case O2, where its highest weight is
\begin{equation*}
(1, \ldots, 1)
\end{equation*}
at every infinite place. Suppose $\pi$ is a cuspidal automorphic representation of $\mbf G(\Ade_F)$ that is cohomological for $\xi$, with generic elliptic global parameter
\begin{equation*}
\bm\psi=\Pi_1+\cdots+\Pi_r,
\end{equation*}
where each $\Pi_j$ is a conjugate self-dual cuspidal automorphic representation of $\GL_{n_j}(\Ade_{F_1})$ for each $1\le j\le r$, and $n_1+\cdots+n_r=N(\mbf G)$. We set
\begin{equation*}
\mfk I\defining
\begin{cases}
\{N(\mbf G), N(\mbf G)-2, \ldots, 2, -2, \ldots, -N(\mbf G)\} &\text{in case O2,}\\
\{N(\mbf G)-1, N(\mbf G)-3, \ldots, 3-N(\mbf G), 1-N(\mbf G)\} &\otherwise.
\end{cases}
\end{equation*}
There is a unique partition $\mfk I=\mfk I_1\coprod \cdots \coprod \mfk I_r$ such that, for every $j$,
\begin{equation*}
\phi_{\Pi_{j,\tau_0}}|_{\bb C^\times}\cong\bplus_{i\in\mfk I_j}\arg^i.
\end{equation*}

The Galois representation $\rho_\bSh^\pi$ has contribution only coming from the middle degree $\dim_{\bb C}(\mbf X)$, by the same argument as in the last part of the proof of Theorem~\ref{eieninifhenfiews}. If $\rho^\pi_\bSh\ne 0$, then $\pi_{\tau_0}$ is the discrete series representation of $\mbf G_{\tau_0}(\bb R)$ that corresponds to a character of $\mfk S_{\tilde{\bm\psi}_{\tau_0}}$. Let $\chi_\pi$ be the pullback of this character along the localization map
\begin{equation*}
\mfk S_{\bm\psi}\to \mfk S_{\tilde{\bm\psi}_{\tau_0}}.
\end{equation*}
Identify the global component group $\mfk S_{\bm\psi}$ with the kernel of the homomorphism
\begin{equation*}
(\bb Z/2\bb Z)^{\mfk I}\to \bb Z/2\bb Z\colon \sum_{i\in \mfk I}a_ie_i\mapsto \sum_{i\in\mfk I}a_in_i.
\end{equation*}
Let $1\le j(\pi)\le r$ be an integer such that $\chi_\pi$ equals the projection of $(\bb Z/2\bb Z)^{\mfk I}$ to the $j(\pi)$-th coordinate. In Case O2, two distinct coordinate characters may have the same restriction to $\mfk S_{\bm\psi}$, so $\mfk J(\pi)$ may have two elements. We let $\mfk J(\pi)$ be the subset of $[r]_+$ consisting of all such choices of $j(\pi)$, then there is an isomorphism of Galois representations
\begin{equation*}
\rho^\pi_\bSh\cong\bigg(\bplus_{j\in \mfk J(\pi)}\rho_{\Pi_j, \ell}\bigg)^{\oplus m(\pi)}: \Gal_{F_1}\to \GL_{m(\pi)\sum_{j\in\mfk J(\pi)}n_j}(\ovl{\bb Q_\ell}).
\end{equation*}
Here $\rho_{\Pi_j, \ell}$ is the Galois representation attached to $\Pi_j$, for example as in~\cite[Theorem~2.1.1]{LGGT}\footnote{In that reference, $\pi$ is assumed to be a regular $C$-algebraic cuspidal essentially self-dual representation of $\GL_n$, but the Galois representation is associated with $\pi\otimes|\det|^{(1-n)/2}$, which is regular $L$-algebraic.}.
\end{hyp}
\begin{rem}
In Case O, if $F=\bb Q$, then Hypothesis~\ref{hsisieteijeiureis} is a corollary of \cite[Corollary~9.8.10]{Zhu24a}. In general, Hypothesis~\ref{hsisieteijeiureis} will be established in a sequel to \cite{KSZ21}, assuming the full endoscopic classification for the corresponding unitary and orthogonal groups, cf.~\cite[Hypothesis~6.6, Remark~6.7]{L-L21}. Note that the generic part of the endoscopic classification for such groups is established by Chen-Zou \cite[Theorem~2.6]{C-Z24}.
\end{rem}

\section{Local and global Shimura varieties}\label{local-LGoabsiSHimief}

In this section, we relate local shtuka spaces with minuscule $\mu$ (or local Shimura varieties) to global Shimura varieties, in order to prove a key result on the cohomology of local shtuka spaces, namely Corollary~\ref{ckeyeeriefiensiw}, using global methods.

Let $p$ be a rational prime, let $\iota_p:\bb C\xr\sim\ovl{\bb Q_p}$ be a fixed isomorphism, and let
$K$ be a finite extension of $\bb Q_p$. Let $\ell\ne p$ be a rational prime and fix an isomorphism $\iota_\ell:\bb C\xr\sim\ovl{\bb Q_\ell}$. The isomorphism $\iota_\ell$ fixes a square root of $p$ in $\ovl{\bb Q_\ell}$. The latter is an $\ell$-adic unit and therefore has a reduction in $\ovl{\bb F_\ell}^\times$.

\subsection{Basic uniformization of Shimura varieties}\label{baisniseufniemeos}

In this subsection we briefly review the basic uniformization of the generic fiber of a Shimura variety following \cite[\S 3.1]{Han20} and \cite[\S 4]{Ham22}. We adopt the notation on local shtuka spaces from~\S\ref{secitonautnciniauotenries}.

Let $\bb G$ be a connected reductive group over $\bb Q$ and $(\bb G, \bb X)$ a Shimura datum of Abelian type with associated conjugacy class of Hodge cocharacters $\{\mu\}:\GL_{1, \bb C}\to \bb G_{\bb C}$. Suppose $\msf G=\bb G\otimes\bb Q_p$ is unramified, and we fix a Borel pair $(\msf B, \msf T)$ for $\msf G$, then we get from $\{\mu\}$ and $\iota_p$ a dominant cocharacter $\mu$ for $\msf G_{\ovl{\bb Q_p}}$, with reflex field $E_\mu/\bb Q_p$.

For each neat compact open subgroup $\mdc K=\mdc K_p\mdc K^p\le \bb G(\Ade_f)$, we have the adic space $\mcl S_{\mdc K_p\mdc K^p}(\bb G, \bb X)$ over $\Spa(E_\mu)$ associated with the Shimura variety $\bSh_{\mdc K_p\mdc K^p}(\bb G, \bb X)$, and we define
\begin{equation*}
\mcl S_{\mdc K^p}(\bb G, \bb X)\defining \plim_{\mdc K_p\to\uno}\mcl S_{\mdc K_p\mdc K^p}(\bb G, \bb X),
\end{equation*}
which is representable by a perfectoid space because $(\bb G, \bb X)$ is of Abelian type, see~\cite{Sch15, She17}. By the result of \cite{Han20}, there exists a canonical $\msf G(\bb Q_p)$-equivariant Hodge--Tate period map defined in \cite{Sch15, C-S17}
\begin{equation*}
\pi_{\bx{HT}}: \mcl S_{\mdc K^p}(\bb G, \bb X)\to\Gr_{\msf G, \mu},
\end{equation*}
where $\Gr_{\msf G, \mu}$ is the Schubert cell of the $\mbf B_\dR^+$-affine Grassmannian of \cite{S-W20} indexed by $\mu$, defined over $\Spa(E_\mu)$.

We write $\mu^\bullet=-w_0(\mu)\in X_\bullet(\msf G)$, where $w_0$ is the longest Weyl group element. Let $\msf b\in B(\msf G, \uno, \mu^\bullet)_\bas$ be the unique basic element and $\uno\in B(\msf G)_\bas$ be the trivial element, then we obtain the open basic Newton stratum $\Gr^{\msf b}_{\msf G, \mu}\subset \Gr_{\msf G, \mu}$ as defined in \cite[\S 3.5]{C-S17}, and let $\mcl S_{\mdc K^p}(\bb G, \bb X)^{\msf b}\subset \mcl S_{\mdc K^p}(\bb G, \bb X)$ be the preimage of $\Gr^{\msf b}_{\msf G, \mu}$ under $\pi_{\bx{HT}}$, called the basic Newton stratum.

Suppose $p>2$ and $\msf G$ is unramified, then Shen's result~\cite[Theorem 1.2]{She20} together with \cite[Proposition~3.1]{Han20} state that $(\bb G, \bb X)$ has a basic uniformization at $p$ in the following sense:

\begin{thm}\label{sisnihfienisw}
There exists a $\bb Q$-inner form $\bb G'$ of $\bb G$ such that 
\begin{itemize}
\item
$\bb G'\otimes\Ade_f^p\cong \bb G\otimes\Ade_f^p$ as algebraic groups over $\Ade_f^p$,
\item
$\bb G'\otimes\bb Q_p\cong \msf G_{\msf b}$,
\item
$\bb G'(\bb R)$ is compact modulo center,
\end{itemize}
together with a $\bb G(\Ade_f)$-equivariant isomorphism of diamonds over $\breve E\defining \breve{\bb Q_p}E_\mu$:
\begin{equation*}
\plim_{\mdc K^p\to\uno}\mcl S_{\mdc K^p}(\bb G, \bb X)^{\msf b}\cong\paren{\udl{\bb G'(\bb Q)\bsh\bb G'(\Ade_f)}\times_{\Spd(\breve E)}\Sht(\msf G,\msf b, \uno, \mu^\bullet)}/\udl{\msf G_{\msf b}(\bb Q_p)},
\end{equation*}
where $\msf G_{\msf b}(\bb Q_p)$ acts diagonally and $\bb G(\Ade_f)\cong \bb G'(\Ade_f^p)\times \msf G(\bb Q_p)$ acts on the right-hand side via the natural action of $\bb G'(\Ade_f^p)$ on $\bb G'(\bb Q)\bsh \bb G'(\Ade_f)$ and the action of $\msf G(\bb Q_p)$ on $\Sht(\msf G,\msf b,\uno, \mu^\bullet)$. Moreover, under the identification $\Gr^{\msf b}_{\msf G, \mu}\cong \Sht(\msf G, \msf b, \uno, \mu^\bullet)/\udl{\msf G_{\msf b}(\bb Q_p)}$, the Hodge--Tate period map:
\begin{equation*}
\pi_{\bx{HT}}: \plim_{\mdc K^p\to\uno}\mcl S_{\mdc K^p}(\bb G, \bb X)^{\msf b}\to \Gr^{\msf b}_{\msf G, \mu}
\end{equation*}
identifies with the natural projection
\begin{equation*}
\paren{\udl{\bb G'(\bb Q)\bsh\bb G'(\Ade_f)}\times_{\Spd(\breve E)}\Sht(\msf G,\msf b, \uno, \mu^\bullet)}/\udl{\msf G_{\msf b}(\bb Q_p)}\to \Sht(\msf G, \msf b, \uno, \mu^\bullet)/\udl{\msf G_{\msf b}(\bb Q_p)}. 
\end{equation*}
\end{thm} 

\vm

This basic uniformization at $p$ will allow us to deduce an isomorphism
\begin{align*}
\bx R\Gamma_c(\msf G,\msf b, \uno, \mu^\bullet)\otimes\iota_\ell\largel{-}_{E_\mu}^{\frac{-\dim_{\bb C}(\bb X)}{2}}[\dim_{\bb C}(\bb X)]&\Ltimes_{\msf G_{\msf b}(\bb Q_p)}\mcl A\paren{\bb G'(\bb Q)\bsh \bb G'(\Ade_f)/\mdc K^p, \mrs L_{\iota_\ell\xi}}\\
&\cong \bx R\Gamma_c\paren{\mcl S_{\mdc K^p}(\bb G, \bb X)^{\msf b}, \mcl L_{\iota_\ell\xi}}
\end{align*}
of $\msf G(\bb Q_p)\times W_{E_\mu}$-modules, for each algebraic representation $\xi$ of $\bb G_{\bb C}$ that is trivial on $Z_{\bx{ac}}$ (see \textup{Definition~\ref{automroehisliinifheis}}), where $\mcl L_{\iota_\ell\xi}$ is the rigid analytification of the lisse $\ovl{\bb Q_\ell}$-sheaf $\mrs L_{\iota_\ell\xi}$ associated with $\xi$ (see \textup{Definition~\ref{automroehisliinifheis}}). When composed with the morphism
\begin{equation*}
\bx R\Gamma_c\paren{\mcl S_{\mdc K^p}(\bb G, \bb X)^{\msf b}, \mcl L_{\iota_\ell\xi}}\to \bx R\Gamma_c\paren{\mcl S_{\mdc K^p}(\bb G, \bb X), \mcl L_{\iota_\ell\xi}}
\end{equation*}
coming from excision with $\bb Z_\ell$-coefficients with respect to the open basic stratum $\mcl S(\bb G, \bb X)^{\msf b}_{\mdc K^p}\subset \mcl S(\bb G, \bb X)_{\mdc K^p}$, this isomorphism gives us a uniformization map between cohomologies. We recall the following result from \cite[Proposition~4.1]{Ham22}.

\begin{prop}\label{basieuniformisihka}
Suppose $p>2$, then there exists a $\msf G(\bb Q_p)\times W_{E_\mu}$-equivariant map
\begin{align*}
\Theta: \bx R\Gamma_c(\msf G,\msf b, \uno, \mu^\bullet)\otimes\iota_\ell\largel{-}_{E_\mu}^{\frac{-\dim_{\bb C}(\bb X)}{2}}[\dim_{\bb C}(\bb X)]&\Ltimes_{\msf G_{\msf b}(\bb Q_p)}\mcl A\paren{\bb G'(\bb Q)\bsh \bb G'(\Ade_f)/\mdc K^p, \mrs L_{\iota_\ell\xi}}\\
&\to\bx R\Gamma_c\paren{\mcl S(\bb G, \bb X)_{\mdc K^p}, \mcl L_{\iota_\ell\xi}}
\end{align*}
functorial with respect to  $\mdc K^p$.
\end{prop}

Next, we apply `Boyer's trick,' an analogue of a result of \cite{Boy99}, relating the supercuspidal part of the cohomology of the local Shimura variety to the cohomology of the global Shimura variety. We recall the definition of being fully Hodge--Newton decomposable in the sense of \cite[Definition 4.28]{R-V14} and \cite[Definition 3.1]{GHN19}:

\begin{defi}
Let $\msf G$ be a quasisplit reductive group over $\bb Q_p$ with a Borel pair $(\msf B, \msf T)$, and let $\mu$ be a dominant cocharacter for $\msf G_{\ovl{\bb Q_p}}$. Then $(\msf G, \mu)$ is called \tbf{fully Hodge--Newton decomposable} if for every non-basic $\vp_K$-conjugacy class $\msf b\in B(\msf G, \uno, \mu)$ (see~\eqref{seinieunifnws}), there exists a proper standard Levi subgroup $\msf M$ of $\msf G$, a dominant cocharacter $\mu_{\msf M}$ of $\msf M_{\ovl{\bb Q_p}}$ and an element $\msf b_{\msf M}\in B(\msf M, \uno, \mu_{\msf M})$ such that  $\mu_{\msf M}$ is conjugate to $\mu$ under $\msf G(\ovl{\bb Q_p})$-action, and $\msf b_{\msf M}$ is mapped to $\msf b$ under the natural map $B(\msf M)\to B(\msf G)$.\footnote{For a fixed $\msf b\in B(\msf G,\uno, \mu)$, the condition that $b$ is Hodge--Newton decomposable with respect to some proper Levi in the sense of \cite[Definition~3.1]{GHN19} is exactly the condition that the local Shimura datum $(b,\mu)$ is Hodge--Newton reducible in the sense of \cite[Definition~4.28]{R-V14}.}
\end{defi}
\begin{exm}\label{ienieniunis}
Suppose 
\begin{itemize}
\item
$\msf G=\Res_{K/\bb Q_p}G^*$ where $K/\bb Q_p$ is unramified and $G^*$ is a quasisplit reductive group over $K$ defined in \S\ref{theogirneidnis}, moreover we assume that $G_\ad^*$ is geometrically simple,
\item
$\mu$ is the dominant cocharacter of
\begin{equation*}
\paren{\Res_{K/\bb Q_p}G^*}_{\ovl{\bb Q_p}}\cong\prod_{v\in\Hom(K,\ovl{\bb Q_p})}G^*_{\ovl K}
\end{equation*}
that equals $\mu_1$ defined in~\eqref{idnifdujiherheins} on one factor and trivial on the other factors,
\end{itemize}
then $(\Res_{K/\bb Q_p}G^*, \mu)$ is fully Hodge--Newton decomposable: By \cite[Theorem 3.3, Proposition 3.4]{GHN19}, it suffices to show that $\{\mu\}$ is minute for $\Res_{K/\bb Q_p}G^*_\ad$ as defined in \cite[Definition 3.2]{GHN19}. Then it suffices to show that $\{\mu_1\}$ is minute for $G^*$ by \cite[\S 3.4]{GHN19}, and this follows from the classification in \cite[Theorem 3.5]{GHN19}; cf. \cite[\S 3.7]{GHN19}.
\end{exm}

Finally, we recall the following result of Hamann~\cite[Proposition 4.4]{Ham22}.

\begin{prop}\label{baiseniunifnoeineiss}
If $p>2$, $\msf G$ is unramified, and $(\msf G,\mu^\bullet)$ is fully Hodge--Newton decomposable, then the uniformization map $\Theta$ in \textup{Proposition~\ref{basieuniformisihka}} induces an isomorphism of $W_{E_\mu}$-modules 
\begin{align*}
\Theta_{\bx{sc}}: \bx R\Gamma_c(\msf G,\msf b, \uno, \mu^\bullet)_{\bx{sc}}\otimes\iota_\ell\largel{-}_{E_\mu}^{\frac{-\dim_{\bb C}(\bb X)}{2}}[\dim_{\bb C}(\bb X)]&\Ltimes_{\msf G_{\msf b}(\bb Q_p)}\mcl A\paren{\bb G'(\bb Q)\bsh \bb G'(\Ade_f)/\mdc K^p, \mrs L_{\iota_\ell\xi}}\\
&\xr\sim\bx R\Gamma_c\paren{\mcl S_{\mdc K^p}(\bb G, \bb X), \mcl L_{\iota_\ell\xi}}_{\bx{sc}}
\end{align*}
on the summands where $\msf G(\bb Q_p)$ acts via a supercuspidal representation.
\end{prop}
\begin{proof}
The proofs of \cite[Lemma 4.3, Proposition 4.4]{Ham22} go through verbatim. The key point is that the non-basic Newton strata of the flag varieties $\Gr_{\msf G, \mu}^{\msf b}$ are all parabolically induced from Newton strata on flag varieties associated with properly contained Levi subgroups of $\msf G$, thus do not contribute to the supercuspidal part of cohomology.
\end{proof}

\subsection{Globalization}\label{glaosbsialisniejfid}

In this subsection, assume that $K/\bb Q_p$ is unramified. Let $G$ be one of the quasisplit reductive groups over $K$ defined in \S\ref{theogirneidnis}, and we adopt the notation there. By a simple application of Krasner's lemma and the weak approximation theorem (see, for example, \cite[Lemma 6.2.1]{Art13}, \cite[Theorem F.1]{C-Z24}), we may choose a CM or totally real number field $F_1\subset \bb C$ with maximal totally real subfield $F\subset F_1$, together with distinct rational primes $p, q, r$, such that
\begin{itemize}
\item
$p, q$ are inert in $F_1$,
\item
\begin{equation*}
(F_1\otimes\bb Q_p)/(F\otimes\bb Q_p)\cong K_1/K
\end{equation*}
as extensions, 
\item
$r\ge N(G)+2$, and
\item
$r$ is inert in $F$, and moreover $F_1\otimes\bb Q_r$ is a ramified quadratic extension of $F\otimes\bb Q_r$ in case U.
\end{itemize}
Write $\mfk p_1=(p), \mfk q_1=(q), \mfk r_1=(r)\in \fPla_{F_1}$ with underlying places $\mfk p, \mfk q, \mfk r\in\fPla_F$, respectively, and let $\tau_0$ denote the fixed embedding $F\inj\bb C$. We fix a vector space $\mbf V$ over $F_1$ equipped with a nondegenerate Hermitian $\cc$-sesquilinear form on $\mbf V$ such that
\begin{itemize}
\item
$(\mbf V, \mbf G=\bx U(\mbf V)^\circ)$ is standard indefinite (see Definition~\ref{staninihidnfies}),
\item
The Hasse--Witt invariant of $\mbf V$ is trivial outside $\{\mfk q\}\cup\infPla_F$, and
\item
$\mbf V\otimes_FF_{\mfk p}$ is isomorphic to $V$ as a $\cc$-Hermitian space over $F_{\mfk p}$.
\end{itemize}
Adopt the notation from \S\ref{seubsienfiehtoenfeis} for $\mbf V$. In particular, we have reductive groups $\mbf G, \mbf G^*, \mbf G^\sharp$ over $F$ with $\mbf G_\ad$ geometrically simple; if $F_1\ne F$, we have fixed a totally imaginary element $\daleth\in F_1^\times$, so each embedding $\tau: F\to \bb C$ extends to an embedding $\tau: F_1\to \bb C$ sending $\daleth$ to $\bb R_+\ii$. Then $G\cong \mbf G_{\mfk p}$.

Let $(\Res_{F/\bb Q}\mbf G, \mbf X)$ be the Shimura datum defined in \S\ref{Shimreiejvanieniocinis} with a conjugacy class of Hodge cocharacters $\{\mu\}$. Via $\iota_p$, we may regard $\{\mu\}$ as a conjugacy class of cocharacters
\begin{equation*}
\GL_{1, \ovl{\bb Q_p}}\to\paren{\Res_{F/\bb Q}\mbf G}_{\ovl{\bb Q_p}}=\prod_{v\in\Hom(F, \ovl{\bb Q_p})}\mbf G\otimes_{F,v}\ovl{\bb Q_p}.
\end{equation*}
Then it contains a dominant cocharacter $\mu$ that equals $\mu_1$ defined in~\eqref{idnifdujiherheins} on one factor and trivial on the other factors. In particular, $\paren{\Res_{K/\bb Q_p}G, \mu}$ is fully Hodge--Newton decomposable by Example~\ref{ienieniunis}.

Let $(\mbf V', \mbf G'=\bx U(\mbf V')^\circ)$ be a standard definite pure inner twist of $\mbf G^*$ (see \textup{Definition~\ref{staninihidnfies}}) such that the Hasse--Witt invariant of $\mbf V'$ is nontrivial at $\mfk p$ and trivial outside $\{\mfk p, \mfk q\}\cup\infPla_F$ (such a pure inner twist exists uniquely), and define $J\defining\mbf G'_{\mfk p}$, which is isomorphic to $G_{b_1}$, where $b_1$ is the unique nontrivial basic element in $B(G)_\bas$. We also regard $b_1$ as the unique nontrivial basic element in $B(\Res_{K/\bb Q_p}G)_\bas$ via the isomorphism $B(\Res_{K/\bb Q_p}G)_\bas\cong B(G)_\bas$.

We have the following globalization result.
\begin{prop}\label{gooalbsisneihfienis}
Suppose $\rho\in\Pi_{\bx{sc}}(J)$ is supercuspidal, then there exists:
\begin{itemize}
\item
a control tuple $\bigstar$ for $\mbf G^*$ (see \textup{Definition~\ref{deieiutesuroeneis}}) such that $\mfk p\in \Pla^{\bx{sc}}, \Pla^\circ=\{\mfk p, \mfk q\}$ and $\Pla^\St=\{\mfk q\}$.
\item
a $\bigstar$-split compact open subgroup $\mdc K^{\mfk p}\le\mbf G'(\Ade_{F, f}^{\mfk p})$ (see \textup{Definition~\ref{Soainsilsienicmosn}}); and
\item
a $\bigstar$-good automorphic representation $\Pi'=\otimes'_v\Pi'_v$ of $\mbf G'(\Ade_F)$ (see \textup{Definition~\ref{olakisniuateineeoifneis}}) such that  $\Pi'_{\mfk p}\cong \rho$ and $(\Pi_f^{\prime\mfk p})^{\mdc K^{\mfk p}}\ne 0$.
\end{itemize}
Such a tuple $(\bigstar, \mdc K^{\mfk p}, \Pi')$ is called a good globalization of $\rho$.
\end{prop}
\begin{proof}
Choose a finite place $\mfk r$ of $F$ not in $\{\mfk p, \mfk q\}$, and let $\Pla^{\bx{sc}}\defining\{\mfk p, \mfk r\}$. It follows from \cite[Theorems~C.3 and~D.2]{C-Z24} that there exists a supercuspidal representation $\pi_{\mfk r}$ of $\mbf G'(F_{\mfk r})$ with a simple supercuspidal $L$-parameter. Since $\mbf G'(F\otimes\bb R)$ admits discrete series, the standard Plancherel density theorem, for example \cite[Theorem~1.1.(\rmnum1)]{Shi12}, applies. In particular, if we choose an algebraic irreducible representation $\xi$ of $\paren{\Res_{F/\bb Q}\mbf G'}\otimes\bb C$, then it follows from \cite[Theorem~1.1.(\rmnum1)]{Shi12} that there exists a $\xi$-cohomological cuspidal automorphic representations $\Pi'$ of $\mbf G'(\Ade_F)$ such that $\Pi'_{\mfk q}$ is Steinberg, $\Pi'_{\mfk p}\cong\rho$, and $\Pi'_{\mfk r}\cong \pi_{\mfk r}$.

By choosing a sufficiently large finite subset of places $\Pla$ of $F$ containing $\{\mfk p, \mfk q, \mfk r\}$ and $\infPla_F$, we can make sure $\Pi'_v$ is unramified for all $v\in \fPla_F\setm\Pla$. We can also choose a $\bigstar$-split compact open subgroup $\mdc K^{\mfk p}\le \mbf G'(\Ade_{F, f}^{\mfk p})$ such that $(\Pi_f^{\prime\mfk p})^{\mdc K^{\mfk p}}\ne 0$.
\end{proof}

Suppose $\rho\in\Pi_{\bx{sc}}(J)$ is supercuspidal with supercuspidal classical $L$-parameter $\tilde\phi$, and let $(\bigstar, \mdc K^{\mfk p}, \Pi')$ be a good globalization of $\rho$. Let 
\begin{equation*}
\phi^\Pla_{\Pi'}: \tilde{\bb T}^\Pla\to \bb C,
\end{equation*}
be the Hecke character associated with $\Pi'$, and set $\mfk m\defining\iota_\ell\ker(\phi^\Pla_{\Pi'})$, which is a maximal ideal of $\iota_\ell\tilde{\bb T}^\Pla$. Consider the uniformization map
\begin{align*}
\Theta: \bx R\Gamma_c(G, b_1, \uno, \mu_1)\otimes\iota_\ell\largel{-}_{E_\mu}^{\frac{-\dim_{\bb C}(\mbf X)}{2}}[\dim_{\bb C}(\mbf X)]&\Ltimes_{J(K)}\mcl A(\mbf G'(F)\bsh \mbf G'(\Ade_{F, f})/\mdc K^{\mfk p}; \mrs L_{\iota_\ell\xi})\\
&\to \bx R\Gamma_c(\mcl S_{\mdc K^{\mfk p}}(\Res_{F/\bb Q}\mbf G, \mbf X), \mcl L_{\iota_\ell\xi})
\end{align*}
which is $G(K)\times W_{K_1}$-equivariant and functorial with respect to  $\mdc K^{\mfk p}$ by Proposition~\ref{basieuniformisihka}. Here we use that $\Sht(\Res_{K/\bb Q_p}G, b_1, \uno, \mu^\bullet)$ is naturally isomorphic to $\Sht(G, b_1, \uno, \mu_1)$; see~\S\ref{secitonautnciniauotenries}. We localize both sides at $\mfk m$ and restrict to the parts on both sides where $G(K)$ acts via supercuspidal representations to get an isomorphism
\begin{align*}
\Theta_{\mfk m, \bx{sc}}: \bx R\Gamma_c(G, b_1, \uno, \mu_1)_{\bx{sc}}\otimes\iota_\ell\largel{-}_{E_\mu}^{\frac{-\dim_{\bb C}(\mbf X)}{2}}[\dim_{\bb C}(\mbf X)]&\Ltimes_{J(K)}\mcl A(\mbf G'(F)\bsh \mbf G'(\Ade_{F, f})/\mdc K^{\mfk p}, \mrs L_{\iota_\ell\xi})_{\mfk m}\\
&\xr\sim \bx R\Gamma_c(\mcl S_{\mdc K^{\mfk p}}(\Res_{F/\bb Q}\mbf G, \mbf X), \mcl L_{\iota_\ell\xi})_{\mfk m}
\end{align*}
by the strong multiplicity one result Corollary~\ref{sotnienfiehineiws} and the basic uniformization result Proposition~\ref{baiseniunifnoeineiss} and Example~\ref{ienieniunis}. Note that, by Corollary~\ref{sotnienfiehineiws} again, the left-hand side decomposes as a direct sum
\begin{equation}\label{iidifhiefimwow}
\bplus_{\dot\Pi'}\paren{\bx R\Gamma_c(G, b_1, \uno, \mu_1)_{\bx{sc}}\otimes\iota_\ell\largel{-}_{E_\mu}^{\frac{-\dim_{\bb C}(\mbf X)}{2}}[\dim_{\bb C}(\mbf X)]\Ltimes_{J(K)}\iota_\ell(\dot\Pi'_f)^{\mdc K^{\mfk p}}}^{\oplus m(\dot\Pi')},
\end{equation}
where $\dot\Pi'$ runs through $\bigstar$-good automorphic representations of $\mbf G'(\Ade_F)$ such that $(\dot\Pi')^\Pla$ has associated Hecke character $\phi_{\Pi'}^\Pla$ and $\dot\Pi'_{\mfk p}$ has classical $L$-parameter $\tilde\phi$.

\begin{cor}\label{coroellineihhihrnais}
Both sides of $\Theta_{\mfk m,\bx{sc}}$ are concentrated in cohomological degree $\dim_{\bb C}(\mbf X)$. Thus $\Theta_{\mfk m,\bx{sc}}$ identifies their only nonzero cohomology groups as $G(K)\times W_{K_1}$-modules. Moreover, this $W_{K_1}$-module has a subquotient whose semisimplification is isomorphic to
\begin{equation*}
\iota_\ell\paren{\tilde\phi^\GL\otimes\largel{-}_{K_1}^{-\frac{\dim_{\bb C}(\mbf X)}{2}}}.
\end{equation*}
\end{cor}
\begin{proof}

By the direct sum decomposition~\eqref{iidifhiefimwow}, it suffices to prove the concentration statement for each summand
\begin{equation*}
\bx R\Gamma_c(G,b_1,\uno,\mu_1)_{\bx{sc}}\otimes\iota_\ell\largel{-}_{E_\mu}^{-\frac{\dim_{\bb C}(\mbf X)}{2}}[\dim_{\bb C}(\mbf X)]\Ltimes_{J(K)}\iota_\ell(\dot\Pi'_f)^{\mdc K^{\mfk p}},
\end{equation*}
where $\dot\Pi'$ runs through $\bigstar$-good automorphic representations of $\mbf G'(\Ade_F)$ such that $(\dot\Pi')^\Pla$ has associated Hecke character $\phi_{\Pi'}^\Pla$ and $\dot\Pi'_{\mfk p}$ has classical $L$-parameter $\tilde\phi$. Fix such a $\dot\Pi'$, and let $\dot\Pi$ be a $\bigstar$-good transfer of $\dot\Pi'$ to $\mbf G$; see Theorem~\ref{strongleienineirnes}. Let
\begin{equation*}
\rho_{\Pi',\ell}:\Gal_{F_1}\to\GL_{N(\mbf G)}(\ovl{\bb Q_\ell})
\end{equation*}
be the Galois representation associated with $\Pi'$ by Theorem~\ref{Galosiniautoenrieeis}. Since $(\dot\Pi')^\Pla$ has associated Hecke character $\phi_{\Pi'}^\Pla$, the same theorem and the strong multiplicity one result imply that $\rho_{\Pi',\ell}$ is also the Galois representation associated with $\dot\Pi'$ and to its transfer $\dot\Pi$. Under the uniformization isomorphism $\Theta_{\mfk m,\bx{sc}}$, the above summand corresponds to the $\dot\Pi^\Pla$-isotypic part of the target. By Corollary~\ref{iodnifneidkhnsiws}, this isotypic part occurs only in degree $\dim_{\bb C}(\mbf X)$. Hence every summand in~\eqref{iidifhiefimwow} is concentrated in degree $\dim_{\bb C}(\mbf X)$, and therefore both sides of $\Theta_{\mfk m,\bx{sc}}$ are concentrated in degree $\dim_{\bb C}(\mbf X)$.

It remains to prove the asserted existence of the $W_{K_1}$-subquotient. For this, take the particular summand corresponding to the original globalization $\dot\Pi'=\Pi'$, which occurs in~\eqref{iidifhiefimwow}. Let $\Pi$ be a $\bigstar$-good transfer of $\Pi'$ to $\mbf G$. Applying Corollary~\ref{iodnifneidkhnsiws} to the $\Pi^\Pla$-isotypic part of the target gives a $W_{K_1}$-subquotient whose semisimplification is isomorphic to
\begin{equation*}
\iota_\ell\paren{\tilde\phi^\GL\otimes\largel{-}_{K_1}^{-\frac{\dim_{\bb C}(\mbf X)}{2}}}.
\end{equation*}
Transporting this subquotient through the isomorphism $\Theta_{\mfk m,\bx{sc}}$ proves the claim.
\end{proof}

In particular, by the direct sum decomposition~\eqref{iidifhiefimwow}, we get the following key corollary: 

\begin{cor}\label{ckeyeeriefiensiw}
If $\tilde\phi\in\tilde\Phi(J)$ is a supercuspidal $L$-parameter with associated packet $\tilde\Pi_{\tilde\phi}(J)$, then the direct summand of
\begin{equation*}
\bplus_{\tilde\rho'\in\tilde\Pi_{\tilde\phi}(J)}\bx R\Gamma_c(G, b_1, \uno, \mu_1)[\iota_\ell\tilde\rho']
\end{equation*}
where $G(K)$ acts by supercuspidal representations, denoted by
\begin{equation*}
\bplus_{\tilde\rho'\in\tilde\Pi_{\tilde\phi}(J)}\bx R\Gamma_c(G, b_1, \uno, \mu_1)[\iota_\ell\tilde\rho']_{\bx{sc}},
\end{equation*}
is concentrated in the middle degree 0, and it has a subquotient whose semisimplification is isomorphic to $\iota_\ell\tilde\phi^\GL$ as a $W_{K_1}$-module.
\end{cor}

\section{Proof of the compatibility property}\label{secitonpaoroifnienocosmw}

In this section we prove Theorem~\ref{compaiteiehdnifds}. Here $p$ is a rational prime, $K/\bb Q_p$ is an unramified finite extension, $\ell$ is a rational prime distinct from $p$ with a fixed isomorphism $\iota_\ell: \bb C\xr\sim \ovl{\bb Q_\ell}$, and $(G=G^*_{b_0}, \varrho_{b_0}, z_{b_0})$ is an extended pure inner form of $G^*$ of Case O or Case U as in \S\ref{ndinlocalLanlgnnis}, such that $G$ splits over an unramified finite extension of $K$. Let $\pi\in \Pi(G)$ be an irreducible smooth representation with classical $L$-parameter $\tilde\phi\in \tilde\Phi(G^*)$. We will show $\tilde\phi^\sems=\tilde\phi_\pi^\FS|_{W_{K_1}}$ using induction on $n(G)$.

In Case O2 or U, when $n(G)=1$, $G$ is a torus. Thus the assertion is known by compatibility of Fargues--Scholze LLC with local class field theory, see~Theorem~\ref{compaitbsilFaiirfies}.

In Case U, when $n(G)=2$, choose a quaternion algebra $D$ over $K$ containing $K_1$ such that the Hermitian space $V$ is identified with $D$, viewed as a two-dimensional right $K_1$-vector space with the Hermitian form induced by the reduced norm. $G$ is contained in the unitary similitude group
\begin{equation*}
G^\sharp\defining \GU(V)\cong (\GL_1(D)\times\Res_{K_1/K}\GL_1)/\GL_1,
\end{equation*}
where $\GL_1$ acts anti-diagonally. By \cite[Proposition 2.2]{Tad92}, for any $\pi\in\Pi(G)$, there exists $\pi^\sharp\in\Pi(G^\sharp)$ such that $\pi$ is a subrepresentation of $\pi^\sharp|_{G(K)}$. So the assertion follows from compatibility for $G^\sharp$ (see~\cite[Lemma~4.13(3)]{H-L24}) and compatibility of Fargues--Scholze LLC with central extensions (see~Theorem~\ref{compaitbsilFaiirfies}).

In Case O1, when $n(G)=1$, $G$ is of the form $\PGL_1(D)$ for some quaternion algebra $D$ over $K$, and the LLC for $G$ defined in Theorem~\ref{coeleninifheihsn} equals the LLC for $G$ via the LLC for $\GL_1(D)$ constructed in \cite{DKV84, Rog83} and the projection $\GL_1(D)\to \PGL_1(D)$; see~\cite[pp.~385--386]{A-G17}. Thus the main theorem follows from compatibility for inner forms of general linear groups \cite[Theorem 6.6.1]{HKW22} and compatibility of Fargues--Scholze LLC with central extensions; see~Theorem~\ref{compaitbsilFaiirfies}.

In Case O2, when $n(G)=2$, $G$ is isomorphic to
\begin{equation*}
\paren{\Res_{K'/K}\SL_1(D)_{K'}}/\mu_2,
\end{equation*}
where $K'$ is either $K\times K$ or the unique unramified quadratic extension of $K$, and $D$ is a quaternion algebra over $K$; cf.~\cite[\S 0]{K-R99}. In fact we can prove the main theorem whenever $K'$ is an \eTale extension of $K$ of rank at most two. The $L$-parameter is constructed for $G$ as follows (not only up to outer automorphism): For any $\pi\in\Pi(G)$, by \cite[Proposition 2.2]{Tad92} there exists $\pi^\sharp\in\Pi(G^\sharp)$, where
\begin{equation*}
G^\sharp=\paren{\Res_{K'/K}\GL_1(D)_{K'}}/\GL_1
\end{equation*}
is a group containing $G$, such that  $\pi$ is a subrepresentation of the restriction of $\pi^\sharp$ to $G(K)$. Then the $L$-parameter $\phi_\pi$ is given by $\phi_{\pi^\sharp}$ composed with the natural map $\LL G^\sharp\to \LL G$; see~\cite[pp.~385--386]{A-G17}. Thus, as before, the main theorem follows from compatibility for inner forms of general linear groups \cite[Theorem 6.6.1]{HKW22} and compatibility of Fargues--Scholze LLC with central extensions, Theorem~\ref{compaitbsilFaiirfies}.

Then, in all remaining cases $G_\ad$ is geometrically simple. We suppose throughout this section that the assertion is known for $G(n_0)$ for each $n_0<n(G)$.

First, if $\pi$ is non-supercuspidal, then the assertion is true:

\begin{prop}\label{supisisbiariainifidns}
If $\pi\in\Pi(G)$ is a subquotient of a parabolic induction, then $\tilde\phi_\pi^\sems=\tilde\phi_\pi^\FS|_{W_{K_1}}$.
\end{prop}
\begin{proof}
Suppose $\pi$ is a subquotient of $\bx I_P^G(\sigma)$ where $P\le G$ is a proper parabolic subgroup with Levi subgroup $M$, and $\sigma\in\Pi(M)$. By compatibility of Fargues--Scholze LLC with parabolic induction (see Theorem~\ref{compaitbsilFaiirfies}) and compatibility of classical LLC with parabolic induction (see Proposition~\ref{compaitebirelclaisirelsi}), the assertion for $\pi$ follows from the assertion for $\sigma$. The Levi subgroup $M$ of $G$ is of the form
\begin{equation*}
G(n_0)\times\Res_{K_1/K}(H)
\end{equation*}
for some integer $n_0<n(G)$, and $H$ is a product of general linear groups. So the assertion follows from the induction hypothesis and compatibility of Fargues--Scholze LLC with products (see Theorem~\ref{compaitbsilFaiirfies}) and compatibility of Fargues--Scholze LLC with classical LLC for general linear groups (see Theorem~\ref{compaitbsilFaiirfies}).
\end{proof}

Assume now that $\pi$ is supercuspidal, so $\tilde\phi$ is a discrete $L$-parameter by Theorem~\ref{coeleninifheihsn}. There are two cases that can occur for the $L$-packet $\tilde\Pi_{\tilde\phi}(G^*)$:

\begin{enumerate}
\item
Case~(1): $\tilde\Pi_{\tilde\phi}(G^*)$ consists entirely of supercuspidal representations,
\item
Case~(2): $\tilde\Pi_{\tilde\phi}(G^*)$ contains a non-supercuspidal representation.
\end{enumerate}

In the following subsections, we prove the main theorem in each case.
\subsection{The first case}

In Case~(1), $\tilde\Pi_{\tilde\phi}(G^*)$ consists entirely of supercuspidal representations, so it follows from Corollary~\ref{superisnidnLpaifniesues} that $\tilde\phi$ is a supercuspidal $L$-parameter.

We first prove the compatibility for representations of $G_{b_1}$, where $b_1$ is the unique nontrivial basic element in $B(G)_\bas$.

\begin{prop}\label{proefnehifeniss}
If $\tilde\phi\in\tilde\Phi_{\bx{sc}}(G^*)$ is a supercuspidal $L$-parameter, then $\tilde\phi_{\tilde\rho}^\FS|_{W_{K_1}}=\tilde\phi$ for any $\tilde\rho\in\tilde\Pi_{\tilde\phi}(G_{b_1})$.
\end{prop}
\begin{proof}
First assume that $\kappa_{b_0}(-1)=1$, so $G\cong G^*$. Write
\begin{equation*}
\tilde\phi^\GL=\phi_1\oplus\cdots\oplus\phi_r,
\end{equation*}
where the $\phi_i$ are irreducible representations of $W_{K_1}$. By Corollary~\ref{ckeyeeriefiensiw},~Equation~\eqref{IIEHienitehiremids} and the compatibility of Fargues--Scholze LLC with contragredients, Theorem~\ref{compaitbsilFaiirfies}, we see that $\phi_i$ appears in
\begin{equation*}
\bplus_{\tilde\rho\in\tilde\Pi_{\tilde\phi}(G_{b_1})}\Mant_{G, b_1, \mu_1}(\iota_\ell\tilde\rho)
\end{equation*}
as a representation of $W_{K_1}$, for each $i$. It follows that $\phi_i$ is an irreducible subquotient of $(\tilde\phi_{\tilde\rho}^\FS|_{W_{K_1}})^\GL$ for each $i$ and each $\tilde\rho\in\tilde\Pi_{\tilde\phi}(G_{b_1})$, by Corollary~\ref{coroellaoslisenifes} and \cite[Theorem 1.3]{Kos21}.\footnote{Note that the notation $\mcl M_{(G, b, \mu), K}$ in \cite{Kos21} is just our $\Sht_{\mdc K}(G, b, b_0, \mu^\bullet)$, so there exists no dual appearing; cf. \cite[Remark 3.8]{Ham22}.} Since $\tilde\phi^\GL$ and $(\tilde\phi_{\tilde\rho}^\FS|_{W_{K_1}})^\GL$ are both semisimple with dimension $N(G)$, we deduce that they are equal. Thus it follows from \cite[Theorem 8.1.(\rmnum{2})]{GGP12} that $\tilde\phi=\tilde\phi_{\tilde\rho}^\FS|_{W_{K_1}}$.

Now assume that $\kappa_{b_0}(-1)=-1$, so $G_{b_1}\cong G^*$. It follows from the weak Kottwitz conjecture Corollary~\ref{Haninidkeiehiens} that some $\tilde\pi\in \tilde\Pi_{\tilde\phi}(G^*)$ appears in $\Mant_{G, b_1, \mu_1}(\iota_\ell\tilde\rho)$. We then have $\tilde\phi_{\tilde\pi}^\FS=\tilde\phi_{\tilde\rho}^\FS$ by Proposition~\ref{ienfimeoifiems}. Thus
\begin{equation*}
\tilde\phi_{\tilde\rho}^\FS|_{W_{K_1}}=\tilde\phi_{\tilde\pi}^\FS|_{W_{K_1}}=\tilde\phi,
\end{equation*}
where the second equality follows from the first case, upon replacing $(G=G^*_{b_0}, \varrho_{b_0}, z_{b_0})$ with $(G_{b_1}, \varrho_{b_0+b_1}, z_{b_0+b_1})$.
\end{proof}

Finally, to deduce the compatibility for $\tilde\pi\in\tilde\Pi_{\tilde\phi}(G)$, we apply this result with $(G=G^*_{b_0}, \varrho_{b_0}, z_{b_0})$ replaced by $(G_{b_1}, \varrho_{b_0+b_1}, z_{b_0+b_1})$, because then the unique nontrivial basic element in $B(G_{b_1})_\bas$ induces an inner form of $G_{b_1}$ that is isomorphic to $G$.

\subsection{The second case}
In the second case where $\tilde\Pi_{\tilde\phi}(G^*)$ contains a non-supercuspidal representation, 
we write
\begin{equation*}
\tilde\phi^\GL=\phi_1\oplus\cdots\oplus\phi_k\oplus\phi_{k+1}\oplus\cdots\oplus\phi_r,
\end{equation*}
where the $\phi_i$ are irreducible representations of $W_{K_1}\times\SL_2(\bb C)$ of dimension $d_i$ for each $i$, and $d_i$ is odd if and only if $i\le k$.

We now prove the main theorem in this case.
\begin{prop}
If $\tilde\Pi_{\tilde\phi}(G^*)$ contains a non-supercuspidal representation $\tilde\rho^*_{\bx{nsc}}$, then $\tilde\phi_{\tilde\pi}^\FS|_{W_{K_1}}=\tilde\phi^\sems$ for every $\tilde\pi\in\tilde\Pi_{\tilde\phi}(G)$.
\end{prop}
\begin{proof}
Write $\tilde\rho_{\bx{nsc}}^*=\tilde\pi_{[I]}$ for some $I\in \mrs P([r]_+)/\sim_k$ with $\#[I]\equiv 0\modu2$. We prove by induction on $m\defining \#J$ that
\begin{equation*}
\tilde\phi_{\tilde\pi}^\FS|_{W_{K_1}}=\tilde\phi^\sems
\end{equation*}
for every
\begin{equation*}
\tilde\pi=\tilde\pi_{[I\oplus J]}\in \tilde\Pi_{\tilde\phi}(G)\cup\tilde\Pi_{\tilde\phi}(G_{b_1}).
\end{equation*}

If $m=0$, then $\tilde\pi=\tilde\rho_{\bx{nsc}}^*$, and the assertion follows from Proposition~\ref{supisisbiariainifidns}. Assume $m>0$ and that the result is known for smaller values of $m$. Let $\tilde\pi=\tilde\pi_{[I\oplus J]}$ with $\#J=m$. If $\tilde\pi$ is not supercuspidal, the assertion again follows from Proposition~\ref{supisisbiariainifidns}.

Suppose $\tilde\pi$ is supercuspidal, choose $J'\subset J$ with $\#J'=m-1$, and set $\tilde\pi'=\tilde\pi_{[I\oplus J']}$. Then $\tilde\pi'\in\tilde\Pi_{\tilde\phi}(G^*_{b'})$, where $b'\in B(G^*)_\bas$ is the unique basic element with $\kappa_{G^*}(b')(-1)=(-1)^{m-1}$, and the induction hypothesis gives
\begin{equation*}
\tilde\phi_{\tilde\pi'}^\FS|_{W_{K_1}}=\tilde\phi^\sems.
\end{equation*}

Corollary~\ref{Haninidkeiehiens} shows that $\tilde\pi$ occurs in
\begin{equation*}
\Mant_{G^*_{b'}, b_1, \{\mu_1\}}(\tilde\pi_{[I\oplus J']}),
\end{equation*}
where $b_1\in B(G^*_{b'})_\bas$ is the unique nontrivial basic element. Indeed, the error term has no supercuspidal constituents. Proposition~\ref{ienfimeoifiems} therefore gives
\begin{equation*}
\tilde\phi_{\tilde\pi}^\FS|_{W_{K_1}}=\tilde\phi_{\tilde\pi_{[I\oplus J']}}^\FS|_{W_{K_1}}=\tilde\phi^\sems.
\end{equation*}
This completes the induction.
\end{proof}

\section{Applications}
\subsection{Unambiguous local Langlands correspondence for even orthogonal groups}\label{unambinloclianisske}
Combining the Fargues--Scholze correspondence with the classical correspondence, we obtain an unambiguous local Langlands correspondence for even special orthogonal groups: its parameters are defined up to conjugation by $\SO_{2n(G)}(\bb C)$, rather than only by $\bx O_{2n(G)}(\bb C)$. We retain the notation of \S\ref{ndinlocalLanlgnnis}. In particular, $p$ is a rational prime, $K/\bb Q_p$ is unramified, and $(G=G^*_{b_0},\varrho_{b_0},z_{b_0})$ is a pure inner twist of $G^*$ of type O2.

\begin{thm}\label{LLCienifheis}
In Case O2, let $(G=G^*_{b_0}, \varrho_{b_0}, z_{b_0})$ be a pure inner twist of $G^*$ such that $\ord_K(\disc(G))\equiv 0\modu2$. Then there exists a map
\begin{equation*}
\rec_G^\natural:\Pi(G)\longrightarrow\Phi(G)
\end{equation*}
fitting into the following commutative diagram:
\begin{equation*}
\begin{tikzcd}[sep=large]
\Pi(G)\ar[r, "\rec_G^\natural"]\ar[d] & \Phi(G)\ar[d]\\
\tilde\Pi(G)\ar[r, "\rec_G"] &\tilde\Phi(G).
\end{tikzcd}
\end{equation*}
For any $\phi\in\Phi(G)$, we write $\Pi_\phi(G)\defining(\rec_G^\natural)^{-1}(\phi)$, called the $L$-packet for $\phi$. This correspondence satisfies the following properties:
\begin{enumerate}
\item
If $\phi\in\Phi(G)$ is not relevant for $G$ in the sense of \textup{\cite[Definition 0.4.14]{KMSW}}, then $\Pi_\phi(G)=\vn$.
\item
For each $\phi\in\Phi(G)$ and $\pi\in\Pi_\phi(G)$, $\pi$ is tempered if and only if $\phi$ is tempered, and $\pi$ is a discrete series representation if and only if $\phi$ is discrete. 
\item
$\rec_G^\natural$ only depends on $G$ but not on $\varrho_{b_0}$ and $z_{b_0}$. For the fixed Whittaker datum $\mfk w$ of $G^*$, there exists a canonical bijection
\begin{equation*}
\iota_{\mfk w, b_0}: \Pi_\phi(G)\xr\sim\Irr(\mfk S_\phi; \kappa_{b_0})
\end{equation*}
for each $\phi\in\Phi(G)$, where $\Irr(\mfk S_\phi; \kappa_{b_0})$ is the set of characters $\eta$ of $\mfk S_\phi$ such that  $\eta(z_\phi)=\kappa_{b_0}(-1)$. We write $\pi=\pi_{\mfk w, b_0}(\phi, \eta)$ if $\pi\in\Pi_\phi(G)$ corresponds to $\eta\in\Irr(\mfk S_\phi; \kappa_{b_0})$ under $\iota_{\mfk w, b_0}$.
\item(Compatibility with Langlands quotient)
Let $P\le G$ be a parabolic subgroup of $G$ with a Levi factor 
\begin{equation*}
M\cong \Res_{K_1/K}\GL_{d_1}\times\cdots\times\Res_{K_1/K}\GL_{d_r}\times G(n_0),
\end{equation*}
such that  $M=\varrho_{b_0}(M^*)$ for a standard Levi subgroup $M^*$ of $G^*$. Suppose $\pi\in\Pi(G)$ is the unique irreducible quotient of
\begin{equation*}
\bx I_P^G\paren{\paren{\tau_1\otimes\nu^{s_1}}\boxtimes\cdots\boxtimes\paren{\tau_r\otimes\nu^{s_r}}\boxtimes\pi_0},
\end{equation*}
where
\begin{equation*}
d_1+\cdots+d_r+n_0=n(G), \qquad s_1\ge\cdots\ge s_r>0,
\end{equation*}
$\pi_0\in \Pi_\temp(G(n_0))$ has parameter $\phi_0\defining \rec_{G(n_0)}^\natural(\pi_0)$, and $\tau_i\in\Pi_{2, \temp}(\GL_{d_i})$ has parameter $\phi_i$. Then $\rec_G^\natural(\pi)$ equals the image under $\LL M\to\LL G$ of
\begin{equation*}
\paren{\phi_1\largel{-}_{K_1}^{s_1}}\times\cdots\times\paren{\phi_r\largel{-}_{K_1}^{s_r}}\times\phi_0\in\Phi(M).
\end{equation*}
Moreover, under the natural identification $\mfk S_{\phi_0}\cong \mfk S_\phi$, one has
\begin{equation*}
\iota_{\mfk w, b_0}(\pi)=\iota_{\mfk w_0, b_0}(\pi_0),
\end{equation*}
where $\mfk w_0$ is the induced Whittaker datum on $M^*$.
\item(Compatibility with standard $\gamma$-factors)
Suppose $\pi\in\Pi(G)$ with $\phi\defining\rec_G^\natural(\pi)$, then for any character $\chi$ of $K^\times$,
\begin{equation*}
\gamma(\pi, \chi, \uppsi_K; s)=\gamma(\phi^\GL\otimes\chi, \uppsi_K; s),
\end{equation*}
where the left-hand side is the standard $\gamma$-factor defined by Lapid--Rallis using the doubling zeta integral \textup{\cite{L-R05}} with the normalization modified in \textup{\cite{G-I14}}, and the right-hand side is the $\gamma$-factor defined in \textup{\cite{Tat79}}.
\item(Compatibility with Plancherel measures)
Suppose $\pi\in\Pi(G)$ with $\phi\defining\rec_G^\natural(\pi)$, then for any $d\in \bb Z_+$ and any $\tau\in\Pi(\GL_{d, K_1})$ with $L$-parameter $\phi_\tau$,
\begin{align*}
\mu_{\uppsi_K}(\tau\otimes\nu^s\boxtimes\pi)=&\gamma(\phi_\tau\otimes(\phi^\GL)^\vee, \uppsi_{K_1}; s)\cdot\gamma(\phi_\tau^\vee\otimes\phi^\GL, \uppsi_K^{-1}, -s)\\
&\times \gamma(\wedge^2(\phi_\tau), \uppsi_K; 2s)\cdot \gamma(\wedge^2(\phi_\tau^\vee), \uppsi_K^{-1}; -2s),
\end{align*}
where the left-hand side is the Plancherel measure defined in \textup{\cite[\S 12]{G-I14}}; cf.~\textup{\cite[\S A.7]{G-I16}}.
\item(Local intertwining relations)
Suppose $P\le G$ is a maximal parabolic subgroup with a Levi factor
\begin{equation*}
M\cong \GL_d\times G(n-d),
\end{equation*}
$\tau\in\Pi_{2,\temp}(\GL_{d, K_1})$, $\pi_0\in\Pi_\temp(G(n-d))$, $\phi_0=\rec_{G(n-d)}^\natural(\pi_0)$, and $\pi\in\Pi_\temp(G)$ is a subrepresentation of $\bx I_P^G(\tau\boxtimes\pi_0)$. Assume that $M=\varrho_{b_0}(M^*)$ where $M^*$ is a standard Levi subgroup of $G^*$, and $\mfk w_0$ is the induced Whittaker datum on $M^*$, then $\phi\defining\rec_G^\natural(\pi)$ equals the image of 
\begin{equation*}
\phi_\tau\times\phi_0\in\Phi(M)
\end{equation*}
composed with the canonical embedding $\LL M\to \LL G$. Furthermore, if $\phi_\tau$ is self-dual of sign 1 and the normalized intertwining operator
\begin{equation*}
\bx R_{\mfk w}(w, \tau\boxtimes\pi_0)\in\End_{G(K)}\paren{\bx I_P^G(\tau\boxtimes\pi_0)}
\end{equation*}
defined in \textup{\cite[\S7.1]{C-Z21}} acts on $\pi$ by an element $\eps\in\{\pm1\}$, where $w$ is the unique nontrivial element in the relative Weyl group for $M$, then 
\begin{equation*}
\eta\defining \iota_{\mfk w, b_0}(\pi)\in\Irr(\mfk S_\phi)
\end{equation*}
restricts to $\iota_{\mfk w_0, b_0}(\pi_0)\in\Irr(\mfk S_{\phi_0})$ under the natural embedding $\mfk S_{\phi_0}\inj \mfk S_\phi$, and satisfies $\eta(e_\tau)=\eps$, where $e_\tau$ is the element of $\mfk S_\phi$ corresponding to $\phi_\tau$.
\item
(Compatibility with Fargues--Scholze LLC)
Suppose $\ell$ is a rational prime distinct from $p$ with a fixed isomorphism $\iota_\ell: \bb C\xr\sim\ovl{\bb Q_\ell}$. For any $\pi\in\Pi(G)$ with $L$-parameter $\phi=\rec_G^\natural(\pi)$, the semisimplification satisfies $\phi^\sems=\iota_\ell^{-1}\phi_{\iota_\ell\pi}^\FS\in\Phi^\sems(G)$.
\end{enumerate}
\end{thm}
\begin{proof}
We first explain why the Fargues--Scholze parameter singles out one of the two possible lifts of the classical parameter.

Let $\tilde\phi\in\tilde\Phi_2(G^*)$. By \cite[Theorem 8.1(\rmnum{2})]{GGP12}, the fiber of
\[
\Phi_2(G^*)\longrightarrow \tilde\Phi_2(G^*)
\]
over $\tilde\phi$ has cardinality two precisely when every irreducible summand
\[
\rho\boxtimes\sp_a
\]
of the standard representation $\tilde\phi^\GL$ has even dimension. We claim that, in this case, the same property remains true after semisimplification. Indeed, the semisimplification of such a summand is
\begin{equation*}
(\rho\boxtimes\sp_a)^\sems
=
\paren{\rho\otimes\largel{-}_{K_1}^{\frac{a-1}{2}}}
\oplus
\paren{\rho\otimes\largel{-}_{K_1}^{\frac{a-3}{2}}}
\oplus\cdots\oplus
\paren{\rho\otimes\largel{-}_{K_1}^{\frac{1-a}{2}}}.
\end{equation*}
Thus every irreducible summand of $(\rho\boxtimes\sp_a)^\sems$ has dimension $\dim(\rho)$. Suppose, for contradiction, that $\dim(\rho)$ is odd. Since
\begin{equation*}
\dim(\rho\boxtimes\sp_a)=a\dim(\rho)
\end{equation*}
is even by assumption, $a$ must be even. On the other hand, for an odd-dimensional conjugate self-dual representation $\rho$, its sign is $b(\rho)=1$. Hence
\begin{equation*}
b(\rho\boxtimes\sp_a)=b(\rho)(-1)^{a-1}=-1.
\end{equation*}
This contradicts the fact that $\rho\boxtimes\sp_a$ is a summand of the standard parameter for $G^*$, whose relevant summands have sign $b(G)$. Therefore $\dim(\rho)$ is even. Consequently every irreducible summand of $\tilde\phi^{\GL,\sems}$ has even dimension.

Applying again \cite[Theorem 8.1(\rmnum{2})]{GGP12}, we see that the fiber of
\begin{equation*}
\Phi^\sems(G^*)\longrightarrow \tilde\Phi^\sems(G^*)
\end{equation*}
over $\tilde\phi^\sems$ has the same cardinality as the fiber of
\begin{equation*}
\Phi_2(G^*)\longrightarrow \tilde\Phi_2(G^*)
\end{equation*}
over $\tilde\phi$. In particular, the semisimplification map induces a bijection
\[
\left\{
\phi\in\Phi_2(G^*)\mid \phi\mapsto\tilde\phi
\right\}
\xrightarrow{\;\sim\;}
\left\{
\phi'\in\Phi^\sems(G^*)\mid \phi'\mapsto\tilde\phi^\sems
\right\}.
\]
This is the key point: the two possible lifts of $\tilde\phi$ remain distinct after semisimplification, so a semisimple Fargues--Scholze parameter can distinguish them.

Now let $\pi\in\Pi_{2,\temp}(G)$, and let
\[
\tilde\phi=\rec_G(\pi)\in\tilde\Phi_2(G^*)
\]
be its classical parameter. By Theorem~\ref{compaiteiehdnifds}, the Fargues--Scholze parameter of $\pi$ maps to the same ambiguous semisimple parameter:
\[
\iota_\ell^{-1}\phi_{\iota_\ell\pi}^\FS
\longmapsto
\tilde\phi^\sems
\qquad\text{in }\tilde\Phi^\sems(G^*).
\]
Hence $\iota_\ell^{-1}\phi_{\iota_\ell\pi}^\FS$ belongs to the fiber over $\tilde\phi^\sems$. By the bijectivity just proved, there exists a unique element
\[
\phi\in\Phi_2(G^*)
\]
lying over $\tilde\phi$ such that
\[
\phi^\sems=\iota_\ell^{-1}\phi_{\iota_\ell\pi}^\FS .
\]
We define
\[
\rec_G^\natural(\pi)\defining \phi.
\]
This definition removes the outer-automorphism ambiguity in the classical correspondence: the classical parameter $\tilde\phi$ determines only an $\bx O_{2n}(\bb C)$-conjugacy class, while the Fargues--Scholze parameter is a canonical semisimple $\SO_{2n}(\bb C)$-conjugacy class. Since the two possible $\SO_{2n}(\bb C)$-lifts have distinct semisimplifications, the Fargues--Scholze parameter selects exactly one of them.

This defines $\rec_G^\natural$ on discrete tempered representations. We extend it to all tempered representations by induction using the local intertwining relation in \textup{(7)}, exactly as in the proof of Proposition~\ref{compaitebirelclaisirelsi}. More precisely, if a tempered representation occurs as a constituent of a normalized parabolic induction from a proper Levi subgroup, then the parameter on the Levi is already defined by the induction hypothesis, and \textup{(7)} determines both the image parameter in $\Phi(G)$ and the internal parametrization of the packet. The preceding discrete tempered construction gives the initial case.

Finally, for a general irreducible smooth representation $\pi\in\Pi(G)$, we define $\rec_G^\natural(\pi)$ by the Langlands quotient property \textup{(4)}. Namely, write $\pi$ as the unique irreducible quotient of the standard module
\[
\bx I_P^G\paren{
\paren{\tau_1\otimes\nu^{s_1}}\boxtimes\cdots\boxtimes
\paren{\tau_r\otimes\nu^{s_r}}\boxtimes\pi_0
},
\]
with $s_1\ge\cdots\ge s_r>0$, $\tau_i$ essentially square-integrable representations of the relevant general linear groups, and $\pi_0$ tempered. We then set $\rec_G^\natural(\pi)$ to be the image of
\[
\paren{\phi_1\largel{-}_{K_1}^{s_1}}\times\cdots\times
\paren{\phi_r\largel{-}_{K_1}^{s_r}}\times\rec_{G(n_0)}^\natural(\pi_0)
\]
under the canonical embedding $\LL M\to\LL G$.

The commutative diagram follows from the construction. Properties \textup{(1)}--\textup{(3)} are the corresponding properties of the classical correspondence, together with the definition of $\rec_G^\natural$ as the unique Fargues--Scholze-compatible lift. Properties \textup{(4)} and \textup{(7)} hold by construction. Properties \textup{(5)} and \textup{(6)} follow from the same properties for the classical correspondence in Theorem~\ref{coeleninifheihsn}, since the standard representation $\phi^\GL$ is unchanged by passing from the ambiguous parameter $\tilde\phi$ to its lift $\phi$. Finally, property \textup{(8)} is exactly the defining compatibility condition for discrete tempered representations, and for tempered and general representations it follows from the compatibility of the Fargues--Scholze correspondence with parabolic induction and Langlands quotients, Theorem~\ref{compaitbsilFaiirfies}.
\end{proof}

Furthermore, the weak version of the Kottwitz conjecture Theorem~\ref{coalidhbinifw222} can be strengthened as follows:

\begin{thm}\label{Kottiwnsihsinfinsi}
Suppose $(G=G^*_{b_0}, \varrho_{b_0}, z_{b_0})$ is a pure inner twist of $G^*$ associated with $b_0\in B(G^*)_\bas$, $\mu$ is a dominant cocharacter of $G^*_{\ovl K}$ and $b\in B(G^*, b_0, \mu)_\bas$ is the unique basic element. If $\phi\in\Phi_2(G^*)$ is a discrete $L$-parameter and $\rho\in\Pi_\phi(G_b)$, then \begin{equation*}
\Mant_{G, b, \mu}(\iota_\ell\rho)=\sum_{\pi\in\Pi_\phi(G)}\dim\Hom_{\mfk S_\phi}(\delta[\pi,\rho], \mcl T_\mu)[\iota_\ell\pi]+\Err
\end{equation*}
in $\bx K_0(G, \ovl{\bb Q_\ell})$, where $\Err\in\bx K_0(G, \ovl{\bb Q_\ell})$ is a virtual representation whose character vanishes on $G(K)_{\sreg, \ellip}$.

Moreover, if $\phi$ is supercuspidal, then $\Err=0$.
\end{thm}
\begin{proof}
This follows from Theorem~\ref{coalidhbinifw222} and Proposition~\ref{ienfimeoifiems}, by extracting the terms whose Fargues--Scholze parameters equal $\phi^\sems$, noting that $\Mant_{G, b, \{\mu\}}$ commutes with $\varsigma$.
\end{proof}

We conjecture that the correspondence defined in Theorem~\ref{LLCienifheis} satisfies the ``unambiguous'' endoscopic character identities of \cite{Kal16a}. This should reflect the expected compatibility of the Fargues--Scholze correspondence with endoscopic transfer. More precisely, we expect that the endoscopic character identities follow from an ambiguous version of the conjectural averaging formula of Shin stated in \cite[Conjecture~C.2]{Ham24}, where ambiguity means we conflate representations conjugated by some outer automorphism. For example, in the trivial endoscopy case, using the Kottwitz conjecture Theorem~\ref{Kottiwnsihsinfinsi}, we prove the following endoscopic character identity between $G$ and $G^*$.

\begin{thm}
Suppose $\phi\in\Phi_2(G^*)$, and $g\in G(K)_{\sreg, \ellip}, h\in G^*(K)_{\sreg,\ellip}$ are stably conjugate, then
\begin{equation*}
e(G)\sum_{\rho\in\Pi_\phi(G)}\Theta_\rho(g)=
\bx S\Theta_\phi(h)=\sum_{\pi\in\Pi_\phi(G^*)}\Theta_\pi(h),
\end{equation*}
\end{thm}
\begin{proof}
This is established by reversing the argument in the proof of Theorem~\ref{coalidhbinifw222}. Given $g, h$, we may choose a dominant cocharacter $\mu$ of $G_{\ovl K}^*$ such that  $(g, h, \mu)\in \bx{Rel}_{b_0}$. Here we identify $G_{b_0}$ with $G^*$. We apply Theorem~\ref{Kottiwnsihsinfinsi}, and it follows from Proposition~\ref{ienicehifnHneisn} and the proof of Proposition~\ref{hteoeomiiiehinidiss} that for any $\pi\in\Pi_\phi(G^*)$,
\begin{equation}\label{eimeiihnifnims}
e(G)\sum_{(h', \lbd)}\Theta_\pi(h')\dim\mcl T_\mu[\lbd]=\sum_{\rho\in\Pi_\phi(G)}\dim\Hom_{\mfk S_\phi}(\delta[\rho, \pi],\mcl T_\mu)\Theta_\rho(g).
\end{equation}
For each $\rho\in\Pi_\phi(G)$, when $\pi$ runs through all elements of $\Pi_\phi(G^*)$, $\delta[\rho, \pi]$ runs through every element of $\Irr(\mfk S_\phi; \kappa_{b_0})$ exactly once. Note that for the fixed $g$, a character $\lbd\in X_\bullet(Z_G(g))=X^\bullet(\hat T)$ can be extended uniquely to a triple $(g, h', \lbd)\in \bx{Rel}_{b_0}$ if and only if the restriction of $\lbd$ to $Z(\hat G)^{\Gal_K}$ equals $\kappa_{b_0}$; see~\cite[p.~17]{HKW22}. So we sum Equation~\eqref{eimeiihnifnims} over $\pi\in\Pi_\phi(G^*)$ to get

\begin{align}\label{isnsihehifneis}
&e(G)\sum_{(h', \lbd)}\dim\mcl T_\mu[\lbd]\bx S\Theta_\phi(h')\\
&=\dim\Hom_{Z(\hat G)^{\Gal_K}}(\kappa_{b_0},\mcl T_\mu)\sum_{\rho\in\Pi_\phi(G)}\Theta_\rho(g)\\
&=\bigg(\sum_{(h', \lbd)}\dim\mcl T_\mu[\lbd]\bigg)\cdot\sum_{\rho\in\Pi_\phi(G)}\Theta_\rho(g).
\end{align}

If there exists a pair $(h', \lbd)$ such that  $\mcl T_\mu[\lbd]\ne 0$, $(g, h', \lbd)\in \bx{Rel}_{b_0}$ and $\lbd\ne \mu$, then we may replace $\mu$ by the dominant representative of $\lbd$ to get a new equation. By the highest weight theory of representations, after finitely many steps we may replace the original $\mu$ by another $\mu'$ such that the only term in the summation of~\eqref{isnsihehifneis} is indexed by $(h, \mu')$. Thus we get:
\begin{equation*}
e(G)\sum_{\rho\in\Pi_\phi(G)}\Theta_\rho(g)=\bx S\Theta_\phi(h).
\end{equation*}
\end{proof}

The arguments above show a general strategy to eliminate ambiguity in local Langlands correspondence caused by outer automorphisms: If we can construct a coarse local Langlands correspondence for all extended pure inner twists of a quasisplit reductive group $\msf G^*$ over $K$ up to action of a finite group $\mfk A$ acting by outer automorphisms, character twists or taking contragredients, and verify endoscopic character identities in the sense of \cite{Kal16a} but up to action by $\mfk A$, then we may use it to deduce a weak version of the Kottwitz conjecture up to action by $\mfk A$. If we can also show the local Langlands correspondence constructed is compatible with Fargues--Scholze parameters up to action by $\mfk A$ in the sense of Theorem~\ref{compaiteiehdnifds}, and the semisimplification map $\Phi(G)\to \Phi^\sems(G)$ is injective on each orbit of the $\mfk A$-action, then we may use the action of $\mfk A$ on the local shtuka space and \cite[Theorem 1.3]{Kos21} to extract a local Langlands correspondence not up to action by $\mfk A$. For example, we expect the strategy to hold for constructing local Langlands correspondence for pure inner forms of the even rank unitary similitude group $\GU(2n)$ with respect to  unramified quadratic extensions or general even special orthogonal groups $\bx{GSO}(2n)$, following work of Xu \cite{Xu16}.

\subsection{Naturality of Fargues--Scholze LLC}\label{naturalllFairngiesHOS}

In this subsection, we will prove the following ``naturality'' property of Fargues--Scholze LLC for $G$, which is precisely \cite[Assumption~6.5]{Ham24}, and a weaker result for a central extension of $\Res_{K/\bb Q_p}G$, which will be used to prove a vanishing result for relevant Shimura varieties in \S\ref{avnainsishenifeihhsinsniw}. We first recall some notation from \cite[\S4.2.1]{H-L24}: For a quasisplit reductive group $\msf G$ over a non-Archimedean local field $K$ of characteristic zero with a Borel pair $(\msf B, \msf T)$, if $\msf b\in B(\msf G)_{\bx{un}}$, there exists a standard Levi subgroup $\msf M_{\msf b}$ of $\msf G$ contained in a standard parabolic subgroup $\msf P_{\msf b}$, such that $\msf G_{\msf b}\cong\msf M_{\msf b}$ under the inner twisting by $\msf b$, and $\msf B\cap\msf M_{\msf b}\le\msf M_{\msf b}$ transfers to a Borel subgroup $\msf B_{\msf b}\le\msf G_{\msf b}$. Let $W_{\msf b}\defining W_{\msf G}/W_{\msf M_{\msf b}}$, and we identify any element of $W_{\msf b}$ with a representative in $W_{\msf G}$ of minimal length. For a character $\chi$ of $\msf T(K)$ and $\msf w\in W_{\msf b}$, we set
\begin{equation}\label{msinifdenideis}
\rho^\chi_{\msf b,\msf w}\defining \bx I_{\msf B_{\msf b}}^{\msf G_{\msf b}}(\chi^{\msf w})\otimes \delta_{\msf P_{\msf b}}^{-1/2}.
\end{equation}

We now verify a property of the Fargues--Scholze LLC, which is stated as an assumption in~\cite[Assumption~6.5]{Ham24}.

\begin{prop}\label{renfienihfneis}\enskip
\begin{enumerate}
\item
For each $b\in B(G^*)$, the classical LLC $\rec_{G^*_b}: \Pi(G^*_b)\to \Phi(G^*_b)$ is compatible with the Fargues--Scholze LLC, i.e., $\iota_\ell^{-1}\phi_{\iota_\ell\pi}^\FS=\phi_\pi^\sems$ for each $\pi\in\Pi(G^*_b)$.
\item
For $b\in B(G^*)$ and $\rho\in\Pi(G^*_b)$, let $\phi\in\Phi(G^*)$ be the composition of $\phi_\rho\in\Phi(G^*_b)$ with the twisted embedding
\begin{equation*}
\LL G^*_b\xr\sim\LL M_b\longrightarrow \LL G^*,
\end{equation*}
as defined in \cite[\S~\Rmnum{9}.7.1]{F-S24}, then $\phi$ factors through the canonical embedding $\LL T^*\to \LL G^*$ only if $b\in B(G^*)_{\bx{un}}$.
\item
In the situation of (2), if $b$ is unramified and $\phi$ factors through $\phi_{T^*}\in\Phi(T^*)$, then $\rho$ is isomorphic to an irreducible constituent of $\rho_{b, w}^\chi$, where $w\in W_b$ and $\chi$ is the character of $T^*$ attached to $\phi_{T^*}$ via local Langlands correspondence for tori.
\end{enumerate}
\end{prop}
\begin{proof}
For (1): each $G_b$ is an inner form of a Levi subgroup $M_b$ of a parabolic subgroup of $G^*$, thus is of the form $\Res_{K_1/K}H\times G'$, where $G'$ is a special orthogonal or unitary group that splits over an unramified quadratic extension of $K$, and $H$ is a product of general linear groups. So the assertion follows from \cite[Theorem 6.6.1]{HKW22}, Theorem~\ref{compaiteiehdnifds}, and Theorem~\ref{LLCienifheis}.

For (2): suppose $b\notin B(G^*)_{\bx{un}}$. By \cite[Lemma~2.12]{Ham24}, $G_b$ is not quasisplit. Consequently, $T^*$ is not relevant for $G_b$ in the sense of \cite[Definition 0.4.14]{KMSW}, and hence $\Pi_{\phi_\rho}(G_b)$ is empty. Indeed, by Jacquet--Langlands correspondence \cite{DKV84}, this follows from Theorem~\ref{LLCienifheis} together with the corresponding results for inner forms of general linear groups. This yields a contradiction.

For (3): first, we note $\phi$ is semisimple, since $\LL T^*$ consists of semisimple elements. Since $G_b\cong M_b$, the preimage of $\phi$ under the natural embedding
\begin{equation*}
\Phi(M^*_b)\longrightarrow\Phi(G^*)
\end{equation*}
is parametrized by a set of minimal length representatives of $W_b$. The assertion follows from the compatibility of the classical LLC with parabolic induction (Theorem~\ref{LLCienifheis}). The factor $\delta_{P_b^*}^{-1/2}$ is included precisely to cancel the modulus-character twist appearing in the twisted $L$-embedding $\LL G^*_b \cong\LL M_b \to\LL G^*$ defined in \cite[\S\Rmnum{9}.7.1]{F-S24}.
\end{proof}

We also need to prove a weaker result for a central extension of $\Res_{K/\bb Q_p}G^*$. More generally, we impose the following global setup for future use: 

\begin{setup}\label{setuepnifneism}\enskip
\begin{itemize}
\item
Let $F$ be a totally real number field and $F_1$ be either $F$ or a CM field containing $F$, and let $\cc\in \Gal(F_1/F)$ be the element with fixed field $F$,
\item
Let $p$ be a rational prime that is unramified in $F$, with a fixed isomorphism $\iota_p: \bb C\xr\sim \ovl{\bb Q_p}$, and we write $K=F\otimes\bb Q_p$, which is a finite product of unramified finite extensions of $\bb Q_p$,
\item
Let $\mbf G$ be a standard indefinite special orthogonal or unitary group over $F$ defined by a $\cc$-Hermitian space $\mbf V$ as in Definition~\ref{staninihidnfies}, such that  $G^*\defining\mbf G\otimes_FK$ is quasisplit and splits over an unramified finite extension of $\bb Q_p$.
\item
for each quadratic imaginary element $\daleth\in \bb R_+\ii$ (in particular $\daleth^2\in \bb Q_-$), we have defined in \S\ref{seubsienfiehtoenfeis} a central extension
\begin{equation*}
1\to\mbf Z^{\bb Q}\to \mbf G^\sharp\to \Res_{F/\bb Q}\mbf G\to 1
\end{equation*}
where
\begin{equation*}
\mbf Z^{\bb Q}=
\begin{cases}
\{z\in \Res_{F(\daleth)/\bb Q}\GL_1: \Nm_{F(\daleth)/F}(z)\in\bb Q^\times\}& \text{in Case O}\\
\{z\in \Res_{F_1/\bb Q}\GL_1: \Nm_{F_1/F}(z)\in\bb Q^\times\}& \text{in Case U}.
\end{cases}
\end{equation*}
Moreover, this central extension splits in Case U. In Case O, we assume that $\bb Q(\daleth)/\bb Q$ is split at $p$, so
\begin{equation*}
\mbf Z^{\bb Q}\otimes\bb Q_p\cong \GL_1\times \Res_{K/\bb Q_p}\GL_1, \quad \mbf G^\sharp\otimes\bb Q_p\cong\GL_1\times \Res_{K/\bb Q_p}\GSpin(V^*).
\end{equation*}
Set
\begin{equation}\label{isnidfheinsm}
G^\sharp\defining \mbf G^\sharp\otimes\bb Q_p,
\end{equation}
which is a central extension of $\Res_{K/\bb Q_p}G^*$. We fix a Borel pair $(B^\sharp, T^\sharp)$ for $G^\sharp$ with image $(B, T)$ in $\Res_{K/\bb Q_p}G^*$.
\end{itemize}
\end{setup}

We now prove a version of Proposition~\ref{renfienihfneis} for those $L$-parameters of $G^\sharp$ coming from $L$-parameters of $\Res_{K/\bb Q_p}G^*$:

\begin{thm}\label{ienigehifniehsins}
Let $\phi\in\Phi^\sems(\Res_{K/\bb Q_p}G^*)$ and let $\phi^\sharp$ be the image of $\phi$ under the natural $L$-homomorphism $\LL(\Res_{K/\bb Q_p}G^*)\to \LL G^\sharp$.
\begin{enumerate}
\item
Suppose $b^\sharp\in B(G^\sharp)$ and $\rho^\sharp\in\Pi(G_{b^\sharp}^\sharp)$. Assume that the image of $\phi_{\iota_\ell\rho^\sharp}^\FS$ under the twisted embedding
\begin{equation*}
\LL G^\sharp_{b^\sharp}\cong \LL M_{b^\sharp}\to \LL G^\sharp,
\end{equation*}
as defined in \textup{\cite[\S \Rmnum{9}.7.1]{F-S24}}, equals $\phi^\sharp$. Then $\rho^\sharp$ factors through a representation $\rho\in\Pi((\Res_{K/\bb Q_p}G^*)_b)$, where $b$ is the image of $b^\sharp$ under the map $B(G^\sharp)\to B(\Res_{K/\bb Q_p}G^*)$, and the classical LLC $\phi_\rho\in\Phi((\Res_{K/\bb Q_p}G^*)_b)$ is defined, and the image of $\phi_\rho^\sems$ under the natural $L$-homomorphism $\LL(\Res_{K/\bb Q_p}G^*)_b\to \LL G^\sharp_{b^\sharp}$ equals $\iota_\ell^{-1}\phi_{\iota_\ell\rho^\sharp}^\FS$.
\item
Situation as in (1), if $\phi^\sharp$ factors through the canonical embedding $\LL T^\sharp\to \LL G^\sharp$, then $b^\sharp\in B(G^\sharp)_{\bx{un}}$.
\item
In the situation of (1), if $b^\sharp$ is unramified and $\phi^\sharp$ factors through $\phi_{T^\sharp}\in\Phi^\sems(T^\sharp)$, then $\rho^\sharp$ is isomorphic to an irreducible constituent of $\rho_{b^\sharp, w^\sharp}^{\chi^\sharp}$, where $w^\sharp\in W_{b^\sharp}$ and $\chi^\sharp$ is the character of $T^\sharp$ attached to $\phi_{T^\sharp}$ via local Langlands correspondence for tori.
\end{enumerate}
\end{thm}
\begin{proof}
Note that $G^\sharp_{b^\sharp}$ is a product of Weil restrictions of general linear groups, unitary similitude groups or general spinor groups, so for any $b^\sharp\in B(G^\sharp)$ and $\rho^\sharp\in\Pi(G_{b^\sharp}^\sharp)$ as in (1), $\rho^\sharp$ is trivial on the kernel $Z$ of the map $G^\sharp_{b^\sharp}(\bb Q_p)\to (\Res_{K/\bb Q_p}G^*)_b(\bb Q_p)$, by compatibility of Fargues--Scholze LLC with central characters, Theorem~\ref{compaitbsilFaiirfies}. Thus $\rho^\sharp$ factors through a representation of $(\Res_{K/\bb Q_p}G^*)_b(\bb Q_p)$ because either the central extension
\begin{equation*}
1\to Z\to G^\sharp\to \Res_{K/\bb Q_p}G^*\to 1
\end{equation*}
is split or the kernel $Z$ is an induced torus, i.e., a product of tori of the form $\Res_{L_i/\bb Q_p}\GL_1$ for finite extensions $L_i/\bb Q_p$. Now (1) follows from compatibility of Fargues--Scholze LLC with central extensions,4 Theorem~\ref{compaitbsilFaiirfies}, and (2)--(3) follow from Proposition~\ref{renfienihfneis} and the following Lemma~\ref{kemeienfiemsow} showing that $\Phi^\sems(\Res_{K/\bb Q_p}G^*)\to \Phi^\sems(G^\sharp)$ is injective.
\end{proof}
\begin{lm}\label{kemeienfiemsow}
Suppose $K$ is a non-Archimedean local field of characteristic zero and $1\to\msf Z\to \msf G'\to\msf G\to 1$ is a central extension of reductive groups over $K$ such that  $\msf Z$ is a torus and $\msf G'(K)\to \msf G(K)$ is surjective, then the natural homomorphism $\LL\msf G\to \LL\msf G'$ induces an injection $\Phi^\sems(\msf G)\inj \Phi^\sems(\msf G')$.
\end{lm}
\begin{proof}
It follows from \cite[(1.9.1)]{Kot85} that there is a long exact sequence of pointed sets
\begin{equation*}
1\to \msf Z(K)\to \msf G'(K)\to \msf G(K)\to B(\msf Z)\to B(\msf G')\to B(\msf G).
\end{equation*}
So the hypothesis implies that $B(\msf Z)\to B(\msf G')$ is injective. By the Kottwitz isomorphism for tori, this implies that
\begin{equation*}
Z(\hat{\msf G'})^{\Gal_K}\to \hat{\msf Z}^{\Gal_K}
\end{equation*}
is surjective, and hence so is
\begin{equation*}
\hat{\msf G'}^{\Gal_K}\to \hat{\msf Z}^{\Gal_K}
\end{equation*}
Now apply the non-abelian cohomology exact sequence attached to $1\to \hat{\msf G}\to \hat{\msf G'}\to \hat{\msf Z}\to 1$. The preceding surjectivity implies that the induced map on continuous cocycles \begin{equation*}
\bx H^1(W_K,\hat{\msf G})\to\bx H^1(W_K,\hat{\msf G'})
\end{equation*}
has trivial kernel. Thus $\Phi^\sems(\msf G)\to\Phi^\sems(\msf G')$ is injective.
\end{proof}

\subsection{Strong Kottwitz conjecture}\label{Kotinitbeinss}

We now revisit the Kottwitz conjecture discussed in \S\ref{secitonautnciniauotenries} and \S\ref{Kottwisniconeidinfs}. We combine the compatibility result and the Act-functors defined in \S\ref{snihsnfifies} to describe the complex of $G(K)\times W_{E_\mu}$-modules $\bx R\Gamma_c(G, b, \mu)[\rho]$ without passing to the Grothendieck group and without the condition that $\mu$ is minuscule. We adopt the notation from~\S\ref{local-abglancorrepsodeifn}, but without quotienting by outer automorphisms in Case O2, because we now have the unambiguous LLC Theorem~\ref{LLCienifheis}. In particular, $K$ is a non-Archimedean local field with residue characteristic $p$, and $\ell$ is a rational prime different from $p$ with a fixed isomorphism $\iota_\ell: \bb C\xr\sim \ovl{\bb Q_\ell}$.

We first recall the following general result of Hansen~\cite[Theorem 1.1]{Han20}.

\begin{thm}\label{isnishenfeis}
Suppose $\msf G$ is a quasisplit reductive group over $\bb Q_p$ with a Borel pair $(\msf B, \msf T)$, $\msf b_0\in B(\msf G)_\bas$ is a basic element, $\mu$ is a minuscule dominant cocharacter of $\msf G_{\ovl{\bb Q_p}}$, and $\msf b\in B(\msf G, \msf b_0, \mu^\bullet)$ (see~\textup{\eqref{seinieunifnws}}), so $(\msf G_{\msf b_0}, \msf b, \mu^\bullet)$ is a local Shimura datum in the sense of \textup{\cite[Definition 5.1]{R-V14}}, and $\rho\in\Pi(\msf G_{\msf b}, \ovl{\bb Q_\ell})$. Suppose the following conditions hold:
\begin{enumerate}
\item
$\Sht(\msf G, \msf b, \msf b_0, \mu)$ appears in the basic uniformization at $p$ of a global Shimura variety in the sense of \textup{Theorem~\ref{sisnihfienisw}}.
\item
The Fargues--Scholze $L$-parameter $\phi_\rho^\FS$ is supercuspidal.
\end{enumerate}
Then the complex $\bx R\Gamma_c(\msf G, \msf b, \msf b_0, \mu^\bullet)[\rho]$ is concentrated in middle degree, which is 0 under our normalization.
\end{thm}

Note that $\Sht(G^*, b_1, \uno, \mu_1)$ appears in the basic uniformization of a global Shimura variety of Abelian type defined in \S\ref{Shimreiejvanieniocinis}, where $b_1\in B(G^*)_\bas$ is the unique nontrivial basic element and $\mu_1$ is defined in~\eqref{idnifdujiherheins}.

For the remainder of this subsection, fix a supercuspidal $L$-parameter $\phi\in\Phi_{\bx{sc}}(G^*)$ such that
\begin{equation*}
\tilde\phi^\GL=\phi_1\oplus\cdots\oplus\phi_k\oplus\phi_{k+1}\oplus\cdots\oplus\phi_r,
\end{equation*}
where the $\phi_i$ are distinct irreducible representations of $W_{K_1}$ of dimension $d_i$, and $d_i$ is odd if and only if $i\le k$. We use the combinatorial notation on $L$-parameters introduced in \S\ref{combianosiLoifn}. By the weak version of the Kottwitz conjecture together with the preceding theorem, the following holds.

\begin{cor}\label{corliemisheifneis}
For each subset $I\subset [r]_+$ with $\#I\equiv 1\modu2$, there is an isomorphism
\begin{equation*}
\bx R\Gamma_c(G^*, b_1, \uno, \mu_1)[\iota_\ell\pi_{[I]}]\cong \bplus_{i=1}^rd_i\iota_\ell\pi_{[I\oplus\{i\}]}
\end{equation*}
of complexes of representations of $G^*(K)$.
\end{cor}
\begin{proof}
It follows from Theorem~\ref{compaiteiehdnifds} that $\iota_\ell^{-1}\phi_{\iota_\ell\pi_{[I]}}^\FS=\phi$ is supercuspidal, so Theorem~\ref{Kottiwnsihsinfinsi} and Proposition~\ref{speeorjoiehgiejmsiws} imply that
\begin{equation*}
\Bkt{\bx R\Gamma_c(G^*, b_1, \uno, \mu_1)[\iota_\ell\pi_{[I]}]}=\sum_{i=1}^rd_i[\iota_\ell\pi_{[I\oplus\{i\}]}]\in\bx K_0(G^*, \ovl{\bb Q_\ell}).
\end{equation*}
Moreover, Theorem~\ref{isnishenfeis} implies that
$\bx R\Gamma_c(G^*, b_1, \uno, \mu_1)[\iota_\ell\pi_{[I]}]$ is concentrated in degree $0$. Hence it is a finite-length smooth representation of $G^*(K)$ that admits a finite filtration whose graded pieces consist of $d_i$ copies of $\iota_\ell\pi_{[I\oplus\{i\}]}$ for each $i\in [r]_+$. By Corollary~\ref{superisnidnLpaifniesues}, $\pi_{[I\oplus\{i\}]}\in\Pi(G^*)$ is supercuspidal for each $i\in [r]_+$. Since supercuspidal representations are injective and projective in the category of smooth representations with a fixed central character, this filtration splits, and the assertion follows.
\end{proof}

The Act-functors defined in \S\ref{snihsnfifies} relate to the cohomology of local shtuka spaces via the following result of Fargues--Scholze~\cite[\S \Rmnum{10}.2]{F-S24} and Hamann~\cite[Corollary~3.11]{Ham22}. Recall that we write $\phi^\natural: W_K\to \LL G$ for the homomorphism corresponding to the supercuspidal $L$-parameter $\phi\in\Phi^{\bx{sc}}(G)$; see~\S\ref{IFNieniehifeniws}.

\begin{thm}\label{nsishinies}
Suppose that $\mu$ is a dominant cocharacter of $G^*_{\ovl K}$ and $b\in B(G^*, b_0, \mu)_\bas$. The highest weight tilting module $\mcl T_\mu$ of $\hat G$ naturally extends to a representation $\mcl T_\mu$ of $\hat G\rtimes W_{E_\mu}$ as defined in \cite[Lemma 2.1.2]{Kot84}. Moreover, for each $\rho\in\Pi_\phi(G^*_b)$ there is an isomorphism
\begin{equation*}
\bx R\Gamma_c(G^*, b, b_0, \mu)[\iota_\ell\rho]\cong\bplus_{\eta\in\Irr(\mfk S_\phi)}\bx{Act}_\eta(\iota_\ell\rho)\boxtimes\iota_\ell\Hom_{\mfk S_\phi}\paren{\eta, \mcl T_\mu\circ\paren{\phi^\natural|_{W_{E_\mu}}}}
\end{equation*}
of $G^*_{b_0}(K)\times W_{E_\mu}$-modules.
\end{thm}

Combining this theorem with the monoidal property of Act-functors, we deduce the analogous result for non-minuscule $\mu$, extending \cite[Theorem 8.2]{Ham22} and \cite[Theorem 4.6, Theorem 4.22]{MHN24} to special orthogonal groups and unitary groups.

\begin{thm}\label{enifeinifehniiwms}
For each $I\subset [r]_+$, there exists a bijection of multisets
\begin{equation}\label{biejifeiheisnrrr}
\{\bx{Act}_{[\{x\}]}(\pi_{[I]})\}_{x\in [r]_+}\cong \{\pi_{[I\oplus\{x\}]}\}_{x\in [r]_+}.
\end{equation}
We choose a permutation $\sigma_\vn$ of $[r]_+$ such that 
\begin{equation*}
\bx{Act}_{[\{x\}]}(\pi_{[\vn]})\cong \pi_{[\{\sigma_\vn(x)\}]}
\end{equation*}
for each $x\in [r]_+$. For each $b_0\in B(G^*)_\bas$ and dominant cocharacter $\mu$ of $G^*_{\ovl K}$ with reflex field $E_\mu$, if $b\in B(G^*, b_0, \mu)_\bas$, we write
\begin{equation*}
\mcl T_\mu\circ\paren{\phi^\natural|_{W_{E_\mu}}}=\bplus_{j=1}^{m_\mu}\eta_{[I_j^\mu]}\boxtimes\sigma_j^\mu
\end{equation*}
as a sum of irreducible representations of $\mfk S_\phi\times W_{E_\mu}$, where $I_j^\mu\in\mrs P([r]_+)/\sim_k$ for $1\le j\le m_\mu$. Then for any subset $I\subset [r]_+$ with $\#I\equiv\frac{\kappa_b(-1)-1}{2}\modu2$, there is an isomorphism of $G^*_{b_0}(K)\times W_{E_\mu}$-modules
\begin{equation}\label{eneifiemsis}
\iota_\ell^{-1}\bx R\Gamma_c(G^*, b, b_0, \mu)[\iota_\ell\pi_{[I]}]\cong\bplus_{j=1}^{m_\mu}\pi_{[I\oplus\sigma_\vn(I_j^\mu)]}\boxtimes \sigma_j^\mu.
\end{equation}
\end{thm}
\begin{proof}
If we apply Theorem~\ref{nsishinies} to $b=b_1, b_0=\uno$ and $\mu=\mu_1$, we get an isomorphism
\begin{equation*}
\bx R\Gamma_c(G^*, b_1, \uno, \mu_1)[\iota_\ell\pi_{[I]}]\cong\bplus_{j=1}^r\iota_\ell\bx{Act}_{[\{j\}]}(\pi_{[I]})\boxtimes\phi_j
\end{equation*}
of representations of $G^*(K)\times W_{K_1}$ for each $I\subset[r]_+$ with $\#I\equiv 1\modu2$. So it follows from Corollary~\ref{corliemisheifneis} and Schur's lemma (noting that each $\bx{Act}_{[j]}(\pi_{[I]})$ is irreducible) that there exists a bijection of multisets
\begin{equation}\label{biejifeiheisn}
\{\bx{Act}_{[\{x\}]}(\pi_{[I]})\}_{x\in [r]_+}\cong \{\pi_{[I\oplus\{x\}]}\}_{x\in [r]_+}
\end{equation}
(The only subtlety comes from the case when $k=2$, but it is easy to verify that this bijection holds in this case). 

Next we consider $[I']\in\mrs P([r]_+)/\sim_k$ with $\#[I']\equiv 0\modu2$. For each $j\in [r]_+$, it follows from the bijection~\eqref{biejifeiheisn} for $[I]=[I']\oplus[\{j\}]$ that there exists $x_j\in [r]_+$ such that  $\bx{Act}_{[\{x_j\}]}(\pi_{[I']\oplus[\{j\}]})\cong \pi_{[I']}$. Thus it follows from the monoidal property of $\bx{Act}$ recalled in~\S\ref{snihsnfifies} that $\bx{Act}_{[\{x_j\}]}(\pi_{[I']})=\pi_{[I']\oplus[\{j\}]}$. We then get a bijection of multisets
\begin{equation}
\{\bx{Act}_{[\{x_j\}]}(\pi_{[I']})\}_{j\in [r]_+}\cong\{\pi_{[I']\oplus[\{j\}]}\}_{j\in [r]_+}.
\end{equation}
Note that if $x_j=x_{j'}$, then $[I'\oplus\{j\}]=[I'\oplus\{j'\}]$. It follows from numerical counting that the left-hand side must equal the multiset $\{\bx{Act}_{[\{j\}]}(\pi_{[I']})\}_{j\in [r]_+}$, and there is a bijection of multisets
\begin{equation}\label{bienieifehnimws}
\{\bx{Act}_{[\{j\}]}(\pi_{[I']})\}_{j\in [r]_+}\cong \{\pi_{[I']\oplus[\{j\}]}\}_{j\in [r]_+}.
\end{equation}
Now the first assertion follows from~\eqref{biejifeiheisn} or~\eqref{bienieifehnimws} depending on the cardinality of $I$.

For the second assertion, by Theorem~\ref{nsishinies}, it suffices to prove that
\begin{equation}\label{Actsigmaperm}
\bx{Act}_{[J]}(\pi_{[I]})=\pi_{[I\oplus\sigma_\vn(J)]}
\end{equation}
for all $I, J\subset [r]_+$. By the monoidal property of $\bx{Act}$-functors recalled in~\S\ref{snihsnfifies}, it suffices to prove
\begin{equation}\label{sniffheinsm}
\bx{Act}_{[I]}(\pi_{[\vn]})=\pi_{[\sigma_\vn(I)]}
\end{equation}
for all $I\subset [r]_+$. We prove~\eqref{sniffheinsm} by induction on $\#I$. The cases $\#I=0$ and $\#I=1$ follow respectively from the identity property of $\bx{Act}_{[\vn]}$ and from the definition of $\sigma_\vn$. Assume $\#I\ge2$. For each $i\in I$, the induction hypothesis and monoidal property give
\begin{equation}\label{sinsiueihfheinis-new}
\bx{Act}_{[I]}(\pi_{[\vn]})\cong\bx{Act}_{[\{i\}]}\bigl(\pi_{[\sigma_\vn(I\setm\{i\})]}\bigr).
\end{equation}
By the first assertion already proved, there is a bijection of multisets
\begin{equation*}
\{\bx{Act}_{[\{j\}]}(\pi_{[\sigma_\vn(I\setm\{i\})]})\}_{j\in [r]_+}\cong\{\pi_{[\sigma_\vn(I\setm\{i\})]\oplus[\{\sigma_\vn(j)\}]}\}_{j\in [r]_+},
\end{equation*}
For $j\in I\setm\{i\}$, the induction hypothesis gives
\begin{equation*}
\bx{Act}_{[\{j\}]}(\pi_{[\sigma_\vn(I\setm\{i\})]})\cong\pi_{[\sigma_\vn(I\setm\{i,j\})]}=\pi_{[\sigma_\vn(I\setm\{i\})\oplus\{\sigma_\vn(j)\}]}.
\end{equation*}
Removing these common terms from the two multisets, we obtain
\begin{equation}\label{iaineiefhsnimsmiw-new}
\{\bx{Act}_{[\{j\}]}(\pi_{[\sigma_\vn(I\setm\{i\})]})\}_{j\in [r]_+\setm(I\setm\{i\})}\cong\{\pi_{[\sigma_\vn(I\setm\{i\})\oplus\{\sigma_\vn(j)\}]}\}_{j\in [r]_+\setm(I\setm\{i\})}.
\end{equation}
For each $i\in I$, the object $\bx{Act}_{[I]}(\pi_{[\vn]})$ occurs in the left-hand multiset of \eqref{iaineiefhsnimsmiw-new}, by~\eqref{sinsiueihfheinis-new}. The intersection, over all $i\in I$, of the supports of the corresponding right-hand multisets consists only of $\pi_{[\sigma_\vn(I)]}$. Hence
\begin{equation*}
\bx{Act}_{[I]}(\pi_{[\vn]})\cong \pi_{[\sigma_\vn(I)]},
\end{equation*}
which proves~\eqref{sniffheinsm}, and therefore~\eqref{Actsigmaperm}.
\end{proof}

In particular, understanding the part of the cohomology of local shtuka spaces with supercuspidal $L$-parameters is reduced to understanding the decomposition of $\mcl T_\mu\circ\paren{\phi^\natural|_{W_{E_\mu}}}$ as $\mfk S_\phi\times W_{E_\mu}$-modules, up to permutation of the $L$-packet $\Pi_\phi(G)$.

Theorem~\ref{enifeinifehniiwms} implies that the strong Kottwitz conjecture~\textup{\cite[Conjecture~1.0.1]{HKW22}} is equivalent to the triviality of the permutation $\sigma_\vn$. For example, this is known for $G=\bx U(V)$ or $\GU(V)$ where $V$ is an odd-dimensional Hermitian space with respect to the unramified quadratic extension $\bb Q_{p^2}/\bb Q_p$ by \cite{M-N23}.

\begin{thm}
If \textup{Hypothesis~\ref{hsisieteijeiureis}} holds, then $\sigma_\vn$ is trivial, and the strong Kottwitz conjecture holds.
\end{thm}
\begin{proof}
Fix $I\subset [r]_+$ with $\#I\equiv1\modu2$. By Theorem~\ref{enifeinifehniiwms}, there is an isomorphism of $G^*(K)\times W_{K_1}$-modules
\begin{equation}\label{lslsiitiyyeineifsss}
\iota_\ell^{-1}\bx R\Gamma_c(G^*, b_1, \uno, \mu_1)[\iota_\ell\pi_{[I]}]\cong\bplus_{j=1}^{r}\pi_{[I\oplus\{\sigma_\vn(j)\}]}\boxtimes\phi_j.
\end{equation}

Following \S\ref{glaosbsialisniejfid}, we globalize $K_1/K$ to $F_1/F$ and $G^*/K$ to $\mbf G/F$, and retain the notation from there. In particular, let $(\mbf V,\mbf G=\bx U(\mbf V)^\circ)$ be the standard indefinite pure inner form of $\mbf G^*$ with $\mbf G_{\mfk p}\cong G^*$, and let $(\mbf V',\mbf G'=\bx U(\mbf V')^\circ)$ be the standard definite pure inner form of $\mbf G^*$ with $\mbf G'_{\mfk p}\cong G^*_{b_1}$.

Set $\rho\defining \pi_{[I]}\in \Pi(G^*_{b_1})$. We use the following globalization of $\rho$ different from Proposition~\ref{gooalbsisneihfienis}.

\begin{lm}\label{lslsiiehteibencies}
There exists
\begin{itemize}
\item
a control tuple $\bigstar$ for $\mbf G^*$ (see \textup{Definition~\ref{deieiutesuroeneis}}) such that $\Pla^{\bx{sc}}=\Pla^\circ=\{\mfk p\}$ and $\Pla^\St=\vn$, such that the highest weight of $\xi$ is trivial except in Case O2, where it is
\begin{equation*}
(1,\ldots,1)
\end{equation*}
at every infinite place of $F$;\footnote{The representation $\xi$ cannot be chosen to be trivial: otherwise, the formal parameter $\bm\psi$ of $\Pi'$ would fail to be regular algebraic at the infinite places. For such parameters, local-global compatibility for the associated Galois representations is not yet known at ramified places.}
\item
a $\bigstar$-split compact open subgroup $\mdc K^{\mfk p}\le\mbf G'(\Ade_{F, f}^{\mfk p})$ (see \textup{Definition~\ref{Soainsilsienicmosn}});
\item
an automorphic representation $\Pi'=\otimes'_v\Pi'_v$ of $\mbf G'(\Ade_F)$ such that
\begin{itemize}
\item
$\Pi'_{\mfk p}\cong \rho$ and $(\Pi_f^{\prime\mfk p})^{\mdc K^{\mfk p}}\ne 0$;
\item
$\Pi'_v$ is unramified for all $v\in \fPla_F\setm\Pla$;
\item
$\Pi'_\infty$ is cohomological for $\xi$, i.e.,
\begin{equation*}
\bx H^i(\Lie(\mbf G(F\otimes\bb R)), \mdc K_\infty, \Pi'_\infty\otimes_{\bb C}\xi)\ne 0
\end{equation*}
for some $i\in\bb N$;
\item
The natural localization map
\begin{equation*}
\mfk S_{\bm\psi}\to \mfk S_\phi
\end{equation*}
associated with the formal parameter $\bm\psi$ of $\Pi'$ is an isomorphism.
\end{itemize}
\end{itemize}
\end{lm}
\begin{proof}
Recall that
\begin{equation*}
\tilde\phi^\GL=\phi_1\oplus\cdots\oplus\phi_k\oplus\phi_{k+1}\oplus\cdots\oplus\phi_r.
\end{equation*}
We define global classical groups $\mbf G_i$ over $F$ for $1\le i\le r$ as follows. In case U, let $\mbf G_i$ denote the quasisplit global unitary group over $F$ of geometric rank $d_i$ that splits over $F_1$. In case O1, $d_i$ is even for every $1\le i\le r$, and we let $\mbf G_i\defining \SO_{d_i+1}$. In case O2, we choose elements $\mfk D_i\in F^\times$ such that
\begin{itemize}
\item
$(-1)^{d_i}\mfk D_i$ is totally positive for every $1\le i\le r$.
\item
The extension $F_{\mfk p}(\sqrt{\mfk D_i})/F_{\mfk p}$ corresponds to $\det\circ\phi_i$ under the local Langlands correspondence for every $1\le i\le r$, and
\item
$\disc(\mbf G)\cdot\prod_{1\le i\le r}\mfk D_i$ belongs to $(F^\times)^2$.
\end{itemize}
Then we set $\mbf G_i=\Sp_{d_i-1}$ for $1\le i\le k$, and set $\mbf G_i=\SO_{d_i}^{\mfk D_i}$ for $k+1\le i\le r$.

Fix, for each $1\le i\le r$, an irreducible representation $\pi_i$ of $\mbf G_i(F_{\mfk p})$ with $L$-parameter
\begin{equation*}
\begin{cases}
\phi_i &\text{in case U}\\
\phi_i\otimes(\det\circ\phi_i) &\text{in case O}.
\end{cases}
\end{equation*}
Since $\mbf G_i(F\otimes\bb R)$ admits discrete series by Harish-Chandra's criterion and our choice of $\mfk D_i$, the standard Plancherel density theorem, for example \cite[Theorem~1.1.(\rmnum1)]{Shi12}, applies. In particular, if we choose an algebraic irreducible representation $\xi_i$ of $\paren{\Res_{F/\bb Q}\mbf G_i}\otimes\bb C$, then it follows from \cite[Theorem~1.1.(\rmnum1)]{Shi12} that there exists a $\xi_i$-cohomological cuspidal automorphic representation $\Pi_i$ of $\mbf G_i(\Ade_F)$ such that $\Pi_{i, \mfk p}\cong \pi_i$.

For each $1\le i\le r$, Arthur's multiplicity formula, Theorem~\ref{endoslcinidhnfineism}, implies that there exists an elliptic global $A$-parameter $\bm\psi_i$ for $\mbf G_i$ such that $\Pi_i$ belongs to the corresponding packet, i.e. $\Pi_i\in \mcl A_{2, \bm\psi_i}(\mbf G_i)$. Since the localization $\tilde{\bm\psi}_{i,\mfk p}$ is isomorphic to $\phi_i\otimes(\det\circ\phi_i)$, $\bm\psi_i$ must be in fact an irreducible self-dual cuspidal automorphic representation of $\GL_{N(\mbf G_i)}(\Ade_F)$. We then define an elliptic global $A$-parameter
\begin{equation*}
\bm\psi=
\begin{cases}
\bm\psi_1+\cdots+\bm\psi_r &\text{in case U}\\
\paren{\bm\psi_1\otimes\chi_{\mfk D_1}}+\cdots+\paren{\bm\psi_r\otimes\chi_{\mfk D_r}} &\text{in case O}
\end{cases}
\end{equation*}
for $\mbf G$.

By choosing algebraic irreducible representations $\xi_i$ of $\paren{\Res_{F/\bb Q}\mbf G_i}\otimes\bb C$ suitably, we may ensure that the restriction of the Archimedean $L$-parameter attached to the localization of the isobaric sum
\begin{equation*}
\bm\psi_1\boxplus\cdots\boxplus\bm\psi_r
\end{equation*}
at every infinite place to $W_{\bb C}\cong \bb C^\times$ is isomorphic to
\begin{equation*}
\bplus_{i\in \mfk I}\arg^i
\end{equation*}
for
\begin{equation*}
\mfk I\defining
\begin{cases}
\{N(\mbf G), N(\mbf G)-2, \ldots, 2, -2, \ldots, -N(\mbf G)\} &\text{in case O2,}\\
\{N(\mbf G)-1, N(\mbf G)-3, \ldots, 3-N(\mbf G), 1-N(\mbf G)\} &\otherwise.
\end{cases}
\end{equation*}

Finally, it follows from Arthur's multiplicity formula for $\mbf G$ that there exists a $\xi$-cohomological automorphic representation $\Pi'=\otimes'_v\Pi_v'$ of $\mbf G'(\Ade_F)$ with formal parameter $\bm\psi$, satisfying $\Pi'_{\mfk p}\cong \rho$. By choosing a sufficiently large finite subset of places $\Pla$ of $F$ containing $\{\mfk p\}$ and $\infPla_F$, we may assume that $\Pi'_v$ is unramified for all $v\in \fPla_F\setm\Pla$. We can also choose a $\bigstar$-split compact open subgroup $\mdc K^{\mfk p}\le \mbf G'(\Ade_{F, f}^{\mfk p})$ such that $(\Pi_f^{\prime\mfk p})^{\mdc K^{\mfk p}}\ne 0$.
\end{proof}

Choose $\bigstar, \mdc K^{\mfk p}$, and $\Pi'$ as in Lemma~\ref{lslsiiehteibencies}. Let
\begin{equation*}
\phi^\Pla_{\Pi'}: \tilde{\bb T}^\Pla\to \bb C
\end{equation*}
be the associated Hecke character of $\Pi'$, and define $\mfk m\defining\iota_\ell\ker(\phi^\Pla_{\Pi'})$, which is a maximal ideal of $\iota_\ell\tilde{\bb T}^\Pla$. Following the argument of \S\ref{glaosbsialisniejfid}, we obtain an isomorphism
\begin{align*}
\Theta_{\mfk m, \bx{sc}}: \bx R\Gamma_c(G^*, b_1, \uno, \mu_1)_{\bx{sc}}\otimes\iota_\ell\largel{-}_{K_1}^{\frac{-\dim_{\bb C}(\mbf X)}{2}}[\dim_{\bb C}(\mbf X)]&\Ltimes_{J(K)}\mcl A(\mbf G'(F)\bsh \mbf G'(\Ade_{F, f})/\mdc K^{\mfk p}, \mrs L_{\iota_\ell\xi})_{\mfk m}\\
&\xr\sim \bx R\Gamma_c(\mcl S_{\mdc K^{\mfk p}}(\Res_{F/\bb Q}\mbf G, \mbf X), \mcl L_{\iota_\ell\xi})_{\mfk m}.
\end{align*}
By further applying a Hecke projector, we may restrict to the terms corresponding to those automorphic representations $\dot\Pi'$ of $\mbf G'(\Ade_F)$ cohomological for $\xi$ such that $(\dot\Pi'_f)^\Pla$ has associated Hecke character $\phi_{\Pi'}^\Pla$. But each such $\dot\Pi'$ must satisfy $\tilde{\dot\Pi}'_{\mfk p}\cong \tilde\Pi'_{\mfk p}$, by our choice of $\Pi'$ and Arthur's multiplicity formula.\footnote{Here we use that $\mbf G'(F\otimes\bb R)$ is compact, so each discrete series packet is a singleton.} By an analogue of Theorem~\ref{strongleienineirnes}, for each $1\le j\le r$, there exists an automorphic representation $\Pi_j$ of $\mbf G(\Ade_F)$ satisfying that
\begin{itemize}
\item
$\tilde\Pi_{j, v}\cong \tilde\Pi'_v$ via the isomorphism $\mbf G(F_v)\cong \mbf G'(F_v)$ induced by the inner twists for every $v\in \fPla_F\setm\Pla$, and
\item
$\tilde\Pi_{j, \mfk p}\cong \tilde\pi_{[I\oplus\{\sigma_\vn(j)\}]}$.
\end{itemize}
Combining with Equation~\eqref{lslsiitiyyeineifsss}, we see that $\phi_j$ is a submodule of the $W_{K_1}$-module $\rho^{\Pi_j}_\bSh$ for each $1\le j\le r$. However, it follows from Arthur's multiplicity formula that $\sigma_\vn(j)$ is contained in $\mfk J(\Pi_j)$, and it follows from \textup{Hypothesis~\ref{hsisieteijeiureis}} that $j$ is contained in $\mfk J(\Pi_j)$. Thus $j$ and $\sigma_\vn(j)$ are both contained in $\mfk J(\Pi_j)$ for each $1\le j\le r$. By the same hypothesis, either $\sigma_\vn(j)=j$ or $\mfk J(\Pi_j)=\{j, \sigma_\vn(j)\}$. In either case we have
\begin{equation*}
\{j\}\oplus\{\sigma_\vn(j)\}\sim_k\vn.
\end{equation*}
As a result, $\sigma_\vn$ acts trivially on the quotient labelling, and we may replace it by the trivial permutation in Theorem~\ref{enifeinifehniiwms}.
\end{proof}

Theorem~\ref{enifeinifehniiwms} can be formulated more naturally in terms of eigensheaves. We recall from \cite[Conjecture~4.4]{Far16} that a \tbf{Hecke eigensheaf} for $\phi$ is an object $\mcl G_\phi\in \bx D_{\bx{lis}}(\Bun_{G^*}, \ovl{\bb Q_\ell})$ such that, for any finite index set $I$ and $(V, r_V)\in \bx{Rep}_{\ovl{\bb Q_\ell}}(\LL G^{*I})$, there exists an isomorphism
\begin{equation}\label{einiefmies}
\eta_{V, I}:\bx T_V(\mcl G_\phi)\xr\sim \mcl G_\phi\boxtimes r_V\circ\phi\in\bx D_{\bx{lis}}(\Bun_{G^*}, \ovl{\bb Q_\ell})^{BW_K^I}
\end{equation}
that is natural in $I$ and $V$ and compatible with compositions of Hecke operators.

Following \cite[\S 4.1.3]{MHN24}, we construct eigensheaves attached to the supercuspidal $L$-parameters $\phi$: 

\begin{thm}\label{sinieikIInifhfiens}
Set
\begin{equation*}
\mcl G_\phi\defining\bplus_{\eta\in\Irr(\mfk S_\phi)}\bx{Act}_\eta(\pi_{[\vn]})\in \bx D_{\bx{lis}}(\Bun_{G^*}, \ovl{\bb Q_\ell}).
\end{equation*}
Let $\mfk S_\phi$ act on the summand $\bx{Act}_\eta(\pi_{[\vn]})$ through the character $\eta$.
\begin{itemize}
\item
$\mcl G_\phi$ is supported on $B(G^*)_\bas\subset\largel{\Bun_{G^*}}$, i.e., on the semi-stable locus of $\Bun_{G^*}$.
\item
For each $b\in B(G^*)_\bas$, under the natural identification $\bx D_{\bx{lis}}(\Bun_{G^*}^b, \ovl{\bb Q_\ell})$ with $\bx D_{\bx{lis}}(G^*_b, \ovl{\bb Q_\ell})$, there exists an isomorphism
\begin{equation*}
i_b^*\mcl G_\phi\cong \bplus_{\substack{[I]\in\mrs P([r]_+)/\sim_k,\\ \#I\equiv \frac{\kappa_b(-1)-1}{2}\modu2}}\eta_{[I]}\boxtimes \pi_{[\sigma_\vn(I)]}
\end{equation*}
of representations of $\mfk S_\phi\times G_b(K)$.
\item
$\mcl G_\phi$ is a Hecke eigensheaf for $\phi$, i.e.,~\eqref{einiefmies} holds.
\end{itemize}
\end{thm}
\begin{proof}
These assertions follow from the bijection~\eqref{biejifeiheisnrrr} and the symmetric monoidal property of the Act-functors recalled in~\S\ref{snihsnfifies}, in the same way as in the proof of~\cite[Proposition~4.18, Theorem~4.19]{MHN24}.
\end{proof}

\section{A vanishing result for torsion cohomology of Shimura varieties}\label{avnainsishenifeihhsinsniw}

We use the compatibility result to prove a vanishing result for the generic part of the cohomology of orthogonal or unitary Shimura varieties with torsion coefficients.

We first recall the general torsion vanishing conjecture of \cite{Han23, Ham24}. Let $(\bb G, \bb X)$ be a Shimura datum with reflex field $E\subset \bb C$ (which is a number field), and let $p$ be a rational prime coprime to $2\cdot \#\pi_1([\bb G, \bb G])$, with a fixed isomorphism $\iota_p: \bb C\xr\sim \ovl{\bb Q_p}$. The isomorphism $\iota_p$ induces a place $\mfk p$ of $E$ over $p$, and we write $\bb C_p$ for the completion of the algebraic closure of $E_{\mfk p}\subset \ovl{\bb Q_p}$. We write $\msf G\defining \bb G\otimes\bb Q_p$. Let $\ell$ be a rational prime that is coprime to $p\sdot \#\pi_0(Z(\msf G))$, with a fixed isomorphism $\iota_\ell: \bb C\to \ovl{\bb Q_\ell}$, which fixes a square root $\sqrt p\in\ovl{\bb Z_\ell}$ thus also $\sqrt p\in \ovl{\bb F_\ell}$. Let $\Lbd\in\{\ovl{\bb Q_\ell}, \ovl{\bb F_\ell}\}$. Whenever we consider $\ovl{\bb F_\ell}$-coefficients, we assume that $\pi_0(Z(\msf G))$ is invertible in $\Lbd$ to avoid complications in this $\ell$-modular setting.

For neat compact open subgroup $\mdc K\le \bb G(\Ade_f)$, let $\mcl S_{\mdc K}(\bb G, \bb X)$ be the adic space over $\Spa(E_{\mfk p})$ associated with the Shimura variety $\bSh_{\mdc K}(\bb G, \bb X)$. If $\mdc K^p\le\bb G(\Ade_f^p)$ is a neat compact open subgroup, we define
\begin{equation*}
\mcl S_{\mdc K^p}(\bb G, \bb X)\defining \plim_{\mdc K_p}\mcl S_{\mdc K_p\mdc K^p}(\bb G, \bb X),
\end{equation*}
where $\mdc K_p$ runs through all open compact subgroups of $\bb G(\bb Q_p)$.

 The $\msf G(\bb Q_p)\times W_{E_{\mfk p}}$-representation on 
\begin{equation*}
\bx R\Gamma_c(\mcl S(\bb G, \bb X)_{\mdc K^p, \bb C_p}, \Lbd)
\end{equation*}
decomposes as
\begin{equation*}
\bx R\Gamma_c(\mcl S_{\mdc K^p}(\bb G, \bb X)_{\bb C_p}, \Lbd)=\bplus_\phi\bx R\Gamma_c(\mcl S_{\mdc K^p}(\bb G, \bb X)_{\bb C_p}, \Lbd)_\phi
\end{equation*}
according to Fargues--Scholze parameters of irreducible subquotients, where $\phi$ runs through semisimple $L$-parameters $\phi\in\Phi^\sems(\msf G; \Lbd)$; see~\cite[Corollary~4.4]{H-L24}.

We now recall the concept of (weakly) Langlands--Shahidi type $L$-parameters, as defined in~\cite[Definition~6.2]{H-L24}.

\begin{defi}\label{dienfiehIGNieneifmeis}
Let $\msf G$ be a quasisplit reductive group over a non-Archimedean local field $K$ of characteristic zero with a Borel pair $(\msf B, \msf T)$. Let $\phi_{\msf T}\in\Phi^\sems(\msf T, \Lbd)$ be a semisimple $L$-parameter, and let
\begin{equation*}
\phi\in\Phi^\sems(\msf G,\Lbd)
\end{equation*}
be its image under the natural embedding $\LL\msf T\to \LL\msf G$. We write $\phi_{\msf T}^\vee$ for the Chevalley dual of $\phi_{\msf T}$.
\begin{itemize}
\item
we say $\phi$ is \tbf{generic} (or of Langlands--Shahidi type) if the following Galois cohomology complexes vanish
\begin{equation*}
\bx R\Gamma(W_K, \LL\mcl T_\mu\circ\phi_{\msf T}), \quad \bx R\Gamma(W_K, \LL\mcl T_\mu\circ\phi_{\msf T}^\vee)
\end{equation*}
are both trivial for each positive coroot $\mu\in \Phi^\vee(\msf G, \msf T)^+\subset X_\bullet(\msf T)$.\footnote{Recall that $\LL T_\mu\in \bx{Rep}_\Lbd(\LL\msf T)$ is the extended highest weight tilting module attached to $\mu$ as defined in \eqref{ieneineifeilsws}.}
\item
we say $\phi$ is of \tbf{weakly Langlands--Shahidi type} if the Galois cohomology groups
\begin{equation*}
\bx H^2(W_K, \LL\mcl T_\mu\circ\phi_{\msf T}), \quad \bx H^2(W_K, \LL\mcl T_\mu\circ\phi_{\msf T}^\vee)
\end{equation*}
are both trivial for each positive coroot $\mu\in \Phi^\vee(\msf G, \msf T)^+\subset X_\bullet(\msf T)$.
\end{itemize}
By \cite[Remark~6.3]{H-L24}, these conditions depend only on the image $\phi$ of $\phi_{\msf T}$ under the natural embedding $\LL\msf T\to \LL\msf G$; we call $\phi$ a semisimple toral $L$-parameter for $\msf G$. For a semisimple toral $L$-parameter $\phi$, we also write $\phi^\vee$ for the image of $\phi_{\msf T}^\vee$ under the natural embedding $\LL\msf T\to \LL\msf G$.

Moreover, it follows from \cite[Lemma~4.24]{H-L24} that, for a finite splitting field extension $K'/K$ for $\msf G$, $\phi$ is generic (resp. of weakly Langlands--Shahidi type) if and only if $\phi|_{W_{K'}}$ is. By local Tate duality, genericity for $\phi$ is equivalent to $\alpha\circ \phi|_{W_{K'}}$ not equaling $\uno$ or $\largel{-}_{K'}^{\pm1}$ for each coroot $\alpha$ of $\msf G$.
\end{defi}

Going back to the global situation, we recall the following conjecture by Hamann and Lee~\cite[Conjecture~6.6]{H-L24} on the vanishing of cohomology of Shimura varieties with torsion coefficients:

\begin{conj}\label{condinfihenis}
Let $\phi\in\Phi^\sems(\msf G, \ovl{\bb F_\ell})$ be a semisimple toral $L$-parameter of weakly Langlands--Shahidi type, then the complex $\bx R\Gamma_c(\mcl S(\bb G, \bb X)_{\mdc K^p, \bb C_p}, \ovl{\bb F_\ell})_\phi$ (resp. $\bx R\Gamma(\mcl S(\bb G, \bb X)_{\mdc K^p, \bb C_p}, \ovl{\bb F_\ell})_\phi$) is concentrated in degrees $0\le i\le\dim_{\bb C}(\bb X)$ (resp. $\dim_{\bb C}(\bb X)\le i\le 2\dim_{\bb C}(\bb X)$).
\end{conj}
\begin{rem}
Suppose $F^+\ne \bb Q$ is a totally real field and $\bb G=\Res_{F/\bb Q}\bx U(n, n)$ is the Weil restriction of a quasisplit unitary group of even rank, and we assume that $\bb G$ splits at $p$, i.e.
\begin{equation*}
\bb G\otimes\bb Q_p\cong \prod_{i=1}^{[F^+: \bb Q]}\GL_{2n, \bb Q_p}.
\end{equation*}
Then the conjecture is true for any semisimple toral $L$-parameter $\phi\in \Phi^\sems(\bb G\otimes\bb Q_p, \ovl{\bb F_\ell})$ of weakly Langlands--Shahidi type by \cite[Theorem 1.1]{C-S24}. Note that if $\phi$ is an unramified character
\begin{equation*}
\phi=\diag(\chi_1, \ldots, \chi_{2n})\in \Phi(\GL_{2n, \bb Q_p}, \ovl{\bb F_\ell}),
\end{equation*}
then $\phi$ is of weakly Langlands--Shahidi type if and only if $\chi_i\ne \chi_j\otimes\largel{-}_{\bb Q_p}$ for any distinct $i, j\in [2n]_+$.
\end{rem}

\subsection{Generic semisimple \texorpdfstring{$L$}{L}-parameters}\label{generisoemisiejfmes}

In this subsection, we study generic semisimple toral $L$-parameters. We import Setup~\ref{setuepnifneism}. In particular, $F$ is a totally real number field unramified at a prime $p$, $\mbf G$ is a special orthogonal or unitary group over $F$ with $G^*=\mbf G\otimes_F(F\otimes\bb Q_p)$, and $G^\sharp$ is a central extension of $\mbf G$. We consider the Hodge cocharacter $\mu^\sharp$ of $\mbf G^\sharp$ corresponding to the Deligne homomorphism $h_0^\sharp$ (see~\eqref{delfieheihsojiems}). When viewed as a cocharacter of 
\begin{equation}\label{sisnihfiemss}
G^\sharp_{\ovl{\bb Q_p}}\cong
\begin{cases}
\GL_{1, \ovl{\bb Q_p}}\times \paren{\Res_{K/\bb Q_p}\GSpin(V^*)}_{\ovl{\bb Q_p}}&=\GL_{1, \ovl{\bb Q_p}}\times \prod\limits_{v\in \Hom(K, \ovl{\bb Q_p})}\GSpin(V^*)\otimes_{K, v}\ovl{\bb Q_p}\\
&\text{in Case O}\\
\GL_{1, \ovl{\bb Q_p}}\times \paren{\Res_{K/\bb Q_p}\GU(V^*)}_{\ovl{\bb Q_p}}&=\GL_{1, \ovl{\bb Q_p}}\times \prod\limits_{v\in \Hom(K, \ovl{\bb Q_p})}\GU(V^*)\otimes_{K, v}\ovl{\bb Q_p}\\
 &\text{in Case U}
\end{cases}
\end{equation}
via the isomorphism $\iota_p$, it is the inverse of the identity map on the $\GL_1$-factor and nontrivial on exactly one other factor, where it is a lift $\mu^\sharp_1$ of the dominant inverse of the cocharacter $\mu_1$ of $G^*_{\ovl K}=G(V^*)^\circ_{\ovl K}$ defined in \eqref{idnifdujiherheins}.

We will prove a special case of Conjecture~\ref{condinfihenis} for generic toral parameters $\phi$ with an additional condition called regularity; cf.~\cite[Definition~4.15]{H-L24}.

\begin{defi}\label{normialiweihnifenis}
Let $\msf G$ be a quasisplit reductive group over a $p$-adic field $K$ with Borel pair $(\msf B,\msf T)$. Let $\phi\in\Phi^\sems(\msf G,\ovl{\bb F_\ell})$ be a generic toral semisimple $L$-parameter, and let $\chi$ be the character of $\msf T(K)$ attached to $\phi_{\msf T}$ by the local Langlands correspondence for tori. We say that $\phi$ is \tbf{regular} if
\begin{equation*}
\chi\ncong \chi^w
\end{equation*}
for every nontrivial element $w\in W_{\msf G}$. 
\end{defi}

We have the following descent property for regularity conditions.

\begin{lm}\label{ianfihefneis}
Let $\msf G$ be a quasisplit reductive group over a $p$-adic field $K$ with Borel pair $(\msf B,\msf T)$, and let $K'/K$ be a finite extension. If $\phi\in \Phi^\sems(\msf G, \ovl{\bb F_\ell})$ is a semisimple toral $L$-parameter, then $\phi$ is regular, provided that $\phi|_{W_{K'}}$ has the same property.
\end{lm}
\begin{proof}
Let $\chi$ be the character of $\msf T(K)$ corresponding to $\phi_{\msf T}$. By local class field theory for tori, the character corresponding to $\phi_{\msf T}|_{W_{K'}}$ is
\begin{equation*}
\chi\circ \Nm_{K'/K}:\msf T(K')\to \ovl{\bb F_\ell}^\times,
\end{equation*}
where $\Nm_{K'/K}:\msf T(K')\to\msf T(K)$ denotes the norm map. This norm map is equivariant for the natural Weyl group action. Hence, if $\chi=\chi^w$ for some nontrivial $w\in W_{\msf G}$, then
\begin{equation*}
\chi\circ\Nm_{K'/K}=(\chi\circ\Nm_{K'/K})^w,
\end{equation*}
so the restriction $\phi|_{W_{K'}}$ is not regular. This proves the lemma.
\end{proof}

For later use, we prove the following result regarding regularity of $L$-parameters and central extensions.

\begin{lm}\label{iansihsienifhnss}
Suppose
\begin{equation*}
1\to Z\to \msf G^\sharp\to \msf G\to 1
\end{equation*}
is a central extension of quasisplit reductive groups over a non-Archimedean local field $K$ of characteristic zero, and assume that either this extension splits or $Z$ is an induced torus. Let $(\msf B^\sharp, \msf T^\sharp)$ be a Borel pair of $\msf G^\sharp$ with image $(\msf B, \msf T)$ in $\msf G$. If $\phi\in\Phi^\sems(\msf G, \ovl{\bb F_\ell})$ is a semisimple toral $L$-parameter which may be regarded as a semisimple toral $L$-parameter $\phi^\sharp\in\Phi^\sems(\msf G^\sharp, \ovl{\bb F_\ell})$ via the canonical embedding
\begin{equation*}
\LL\msf G(\ovl{\bb F_\ell})\to \LL\msf G^\sharp(\ovl{\bb F_\ell}),
\end{equation*}
then $\phi^\sharp$ is regular if and only if $\phi$ is.
\end{lm}
\begin{proof}
The map $\msf T^\sharp(K)\to\msf T(K)$ is surjective: this is clear if the extension splits, and follows from Shapiro's lemma when $Z$ is induced. Moreover, $W_{\msf G^\sharp}\cong W_{\msf G}$. The assertion follows immediately from the definition of regularity.
\end{proof}

We next discuss when genericity implies regularity. In Cases U and O1, genericity is sufficient for regularity. In Case O2 this is no longer true in general, because the $D_{n(G^*)}$-coroots do not include the characters $2\ve_i$. We therefore impose the following additional genericity condition in Case O2.

\begin{defi}\label{tttiitienfieifs}
Assume we are in Case O2. Let $\msf H$ be $\Res_{K/\bb Q_p}G^*$ or $G^\sharp$. We say that a semisimple toral $L$-parameter
\[
\phi_{\msf H}\in \Phi^\sems(\msf H,\Lbd)
\]
is \tbf{$\hat\Std$-generic} if, after restricting to $W_K$ and composing with the standard representation of any factor of the dual groups corresponding to some $v\in\Hom(K, \ovl{\bb Q_p})$ (see~\eqref{sisnihfiemss}), the resulting semisimple toral parameter for $\GL_{2n(G^*)}$ is generic in the sense of Definition~\ref{dienfiehIGNieneifmeis}.

Equivalently, after base change to a finite splitting field $K'/K$ and writing the noncentral toral part on a simple factor as
\begin{equation*}
\phi_T=(\chi_1,\ldots,\chi_{n(G^*)}),
\end{equation*}
the characters occurring in $\hat\Std\circ\phi_T$ are
\begin{equation*}
\eta_1,\ldots,\eta_{2n(G^*)}=\chi_1,\ldots,\chi_{n(G^*)}, \chi_{n(G^*)}^{-1},\ldots,\chi_1^{-1},
\end{equation*}
and for every pair of distinct entries $\eta_a,\eta_b$ with $a\ne b$, one has
\begin{equation*}
\eta_a\eta_b^{-1}\ncong \uno,
\qquad
\eta_a\eta_b^{-1}\ncong \largel{-}_{K'}^{\pm1}.
\end{equation*}
This condition is stable under Chevalley duality.
\end{defi}

The following result generalizes \cite[Lemma~4.25]{H-L24} to special orthogonal groups and unitary groups.

\begin{lm}\label{geoeirnieinnifeiss}
If $\phi\in \Phi^\sems(\Res_{K/\bb Q_p}G^*, \ovl{\bb F_\ell})$ is a generic semisimple toral $L$-parameter, and assume that $\phi$ is $\hat{\Std}$-generic in case O2 (see \textup{Definition~\ref{tttiitienfieifs}}), then $\phi$ is regular. Furthermore, if we regard $\phi$ as a semisimple toral $L$-parameter $\phi^\sharp\in\Phi^\sems(G^\sharp, \ovl{\bb F_\ell})$ via the natural embedding
\begin{equation*}
\LL(\Res_{K/\bb Q_p}G^*)(\ovl{\bb F_\ell})\to \LL G^\sharp(\ovl{\bb F_\ell}),
\end{equation*}
then it is also regular as a parameter for $G^\sharp(\ovl{\bb F_\ell})$.
\end{lm}
\begin{proof}
By Lemma~\ref{iansihsienifhnss}, the second assertion follows from the first. For the first assertion, by Lemma~\ref{ianfihefneis} and \cite[Lemma~4.24]{H-L24}, it suffices to treat the split groups $\SO_{d(G^*), K}$ (in Case O) and the general linear group $\GL_{d(G^*), K_1}$ (in Case U). Under the isomorphism 
\begin{equation*}
\GL_1^{n(G^*)}\cong T^*: (t_1, \ldots, t_{n(G^*)})\mapsto
\begin{cases}
\diag(t_1, \ldots, t_{n(G^*)}) &\text{in Case U},\\
\diag(t_1, \ldots, t_{n(G^*)}, 1, t_{n(G^*)}^{-1}, \ldots, t_1^{-1}) &\text{in Case O1},
\end{cases}
\end{equation*}
$W_{G^*}$ is $\Sym_{n(G^*)}$ in case U, and is the semi-direct product of the group $\Sym_{n(G^*)}$ acting by permutation on the group $\{\pm1\}^{n(G^*)}$ in Case O1., and is the semi-direct product of the group $\Sym_{n(G^*)}$ acting by permutation on the kernel of the determinant map $\det: \{\pm1\}^{n(G^*)}\to\{\pm1\}: (\eps_1, \ldots, \eps_{n(G^*)})\mapsto \prod_i\eps_i$ in Case O2.

For each $i\in[n(G^*)]_+$, denote by $\ve_i$ the cocharacter $\GL_1\to T^*: t\mapsto (1, \ldots, t, 1, \ldots)$ where the $t$ is at the $i$-th coordinate; and set $\ve_{-i}\defining -\ve_i$.

We write $\phi=\chi_1\boxtimes\cdots\boxtimes \chi_{n(G^*)}$, and write $\chi_{-i}\defining \chi_i^{-1}$ for $i\in [n(G^*)]$. Suppose, for contradiction, that
\begin{equation}\label{eniefheinss}
\chi=\chi^w
\end{equation}
for some nontrivial element $w\in W_{G^*}$. Suppose $\ve_j$ is not fixed by $W_{G^*}$, and we write $\ve_j^w=\ve_k$, then we evaluate Equation~\eqref{eniefheinss} at $\ve_j$ to get
\begin{equation*}
\chi_j=\chi_k.
\end{equation*}
In Case O1 or Case U, this contradicts genericity because $\ve_j-\ve_k$ is a coroot of $\SO_{d(G^*)}$ or $\GL_{d(G^*)}$. In Case O2, this contradicts genericity because the composition of $\ve_j-\ve_k$ with the standard representation of $\SO_{2n(G^*)}$ is a coroot of $\GL_{2n(G^*)}$.
\end{proof}

In the setting of Definition~\ref{normialiweihnifenis}, assume that $\phi$ is generic and regular, then by \cite[Theorem~9.10]{Ham24} there exists an object
\begin{equation*}
\bx{nEis}(\mcl S_{\phi_{\msf T}})\in\bx D_{\bx{lis}}(\Bun_{\msf G}, \ovl{\bb F_\ell}),
\end{equation*}
which is a perverse filtered Hecke eigensheaf on $\Bun_{\msf G}$ with eigenvalue $\phi$ in the sense of \cite[Theorem~6.1]{Ham24}. In particular, if $\mu$ is a dominant cocharacter of $\msf G_{\ovl K}$ with extended highest weight tilting module $\LL\mcl T_\mu\in \bx{Rep}_{\ovl{\bb F_\ell}}(\LL\msf G)$ as defined in \eqref{ieneineifeilsws}, and $\bx T_\mu$ is the Hecke operator attached to $\LL\mcl T_\mu$ as defined in \S\ref{icoroenfiehsnw}, then $\bx T_\mu(\bx{nEis}(\mcl S_{\phi_{\msf T}}))$ admits a $W_K$-equivariant filtration indexed by $\Gal_K$-orbits $\Gal_K.\nu_i$ in $X_\bullet(\msf T)$. The graded piece indexed by $\Gal_K.\nu_i$ is
\begin{equation*}
\bx{nEis}(\mcl S_{\phi_{\msf T}})\otimes(\LL\mcl T_{\nu_i}\circ\phi_{\msf T})\otimes\LL\mcl T_\mu[\Gal_K.\nu_i].
\end{equation*}
Moreover, when this filtration splits, there is an isomorphism
\begin{equation*}
\bx T_\mu(\bx{nEis}(\mcl S_{\phi_{\msf T}}))\cong \mcl S_{\phi_{\msf T}}\otimes\LL\mcl T_\mu\circ\phi
\end{equation*}
of sheaves in $\bx D_{\bx{lis}}(\Bun_{\msf G}, \ovl{\bb F_\ell})^{\bx BW_K}$. We say that $\phi$ is \tbf{$\mu$-regular} if this filtration splits. By \cite[Theorem~9.10]{Ham24}, this $\mu$-regularity condition is implied by the following strongly $\mu$-regularity condition; cf.~\cite[Definition~4.16]{H-L24}.

\begin{defi}\label{sotneiegneimsiujsnw}
For a quasisplit reductive group $\msf G$ over a non-Archimedean local field $K$ of characteristic zero with a Borel pair $(\msf B, \msf T)$, and for a dominant cocharacter $\mu$ of $\msf G_{\ovl K}$, a toral semisimple $L$-parameter $\phi\in\Phi^\sems(\msf G, \ovl{\bb F_\ell})$ is called \tbf{strongly $\mu$-regular} if
\begin{equation*}
\bx R\Gamma(W_K, \LL\mcl T_{\nu-\nu'}\circ \phi_{\msf T})
\end{equation*}
is trivial for any weights $\nu,\nu'$ of the extended highest weight tilting module $\LL\mcl T_\mu$ attached to $\mu$ that lie in distinct $\Gal_K$-orbits.
\end{defi}

For later use, we prove the following result regarding $\mu$-regularity of $L$-parameters and central extensions:

\begin{lm}\label{neineiihenfies}
Assume $\ell\ne 2$, and that $(\ell, n(G^*)!)=1$ if we are in Case O1. Let
\begin{equation*}
\phi\in\Phi^\sems(G^*, \ovl{\bb F_\ell})
\end{equation*}
be a generic semisimple toral $L$-parameter. In Case O2, assume in addition that $\phi$ is $\hat\Std$-generic in the sense of \textup{Definition~\ref{tttiitienfieifs}}. Then $\phi$ is $\mu$-regular for every dominant cocharacter $\mu\in X_\bullet(G^*).$

Similarly, let
\begin{equation*}
\phi^\sharp\in \Phi^\sems(G^\sharp, \ovl{\bb F_\ell})
\end{equation*}
be a generic semisimple toral $L$-parameter. In Case O2, assume in addition that $\phi^\sharp$ is $\hat\Std$-generic. Then $\phi^\sharp$ is $\mu^\sharp$-regular for every dominant cocharacter $\mu^\sharp\in X_\bullet(G^\sharp).$ In particular, $\mu^\sharp$ can be chosen to be not fixed by any nontrivial element of $W_{G^\sharp}$.
\end{lm}
\begin{proof}
By base change \cite[Lemma~4.24]{H-L24} and the isomorphism~\eqref{sisnihfiemss}, it suffices to work after the relevant group is split. The assertion for general linear groups, i.e. Case U, is the argument of \cite[Lemma~4.25]{H-L24}. We treat the orthogonal cases.

Let $\ve_1,\ldots,\ve_{n(G^*)}$ be the standard coordinates of a split maximal torus. The weights of the standard representation $\hat\Std$ of the dual group are
\begin{equation*}
\pm\ve_1,\ldots,\pm\ve_{n(G^*)}.
\end{equation*}

The standard representation extends to the standard representation of the corresponding dual similitude group, namely $\GSp_{2n(G^*)}$ in Case O1 and $\bx{GSO}_{2n(G^*)}$ in Case O2. We denote its highest weight by $\omega_1^\sharp$ on the similitude side.

In Case O1, the differences of two distinct weights of $\hat\Std$ are of the form
\begin{equation*}
\pm\ve_i\pm\ve_j\quad (i\ne j),\qquad\pm2\ve_i,
\end{equation*}
and these are coroots of $G^*$. Hence genericity implies strong $\omega_1$-regularity for $\phi$. The same argument applies to $\phi^\sharp$ for $\omega_1^\sharp$, since the central similitude character contributes the same scalar to every weight and therefore cancels in every weight difference.

In Case O2, the differences
\begin{equation*}
\pm\ve_i\pm\ve_j\quad (i\ne j)
\end{equation*}
are coroots of $G^*$ and are controlled by genericity. The remaining differences are
\begin{equation*}
\pm2\ve_i,
\end{equation*}
which are not coroots of the $D_{n(G^*)}$-root system. These are exactly the additional ratios controlled by the $\hat\Std$-genericity assumption. Thus $\phi$ is strongly $\omega_1$-regular, and $\phi^\sharp$ is strongly $\omega_1^\sharp$-regular. Therefore they are $\omega_1$-regular and $\omega_1^\sharp$-regular by \cite[Theorem~9.10]{Ham24}.

It remains to pass from the standard representation to arbitrary dominant cocharacters. In Case O1, the highest weight tilting module with highest weight $\omega_i$ is realized as a direct summand of the appropriate exterior power $\wedge^i(\hat\Std)$ for $1\le i\le n(G^*)$; the same statement holds for the corresponding highest weights $\omega_i^\sharp$ of the dual similitude group, up to a central character. This uses the assumption $(\ell,n!)=1$; cf.~\cite[pp.~286--287]{Jan03} and \cite[\S 9.1, Appendix~B.2]{Ham24}.

In Case O2, with $\ovl{\bb Q_\ell}$-coefficients one has
\begin{equation*}
\wedge^i(\hat\Std)\cong \mcl T_{\omega_i}\qquad 1\le i\le n(G^*)-2,
\end{equation*}
\begin{equation*}
\wedge^{n(G^*)-1}(\hat\Std)\cong\mcl T_{\omega_{n(G^*)-1}+\omega_{n(G^*)}},
\end{equation*}
and
\begin{equation*}
\wedge^{n(G^*)}(\hat\Std)\cong\mcl T_{2\omega_{n(G^*)-1}}\oplus\mcl T_{2\omega_{n(G^*)}}.
\end{equation*}
These decompositions extend to the corresponding highest weights of the dual $\bx{GSO}_{2n(G^*)}$, again up to central characters. Since $\ell\ne 2$, the same direct-summand statements hold with $\ovl{\bb F_\ell}$-coefficients; cf.~\cite[Theorem~5.5.13]{G-W09}, \cite[pp.~286--287]{Jan03}, and \cite[\S 9.1, Appendix~B.2]{Ham24}.

The noncentral parts of the dominant cocharacter monoids in the orthogonal cases are generated by the highest weights just listed. These highest weight tilting modules occur as direct summands of tensor products of $\hat\Std$. Hence repeated application of \cite[Proposition~9.12]{Ham24} shows that $\phi$ is $\mu$-regular for every dominant $\mu$ of $G^*$, and that $\phi^\sharp$ is $\mu^\sharp$-regular for every dominant $\mu^\sharp$ of $G^\sharp$. Central characters do not affect the condition, since they add a common character to all weights and hence disappear from weight differences.

Finally, we may choose $\tilde\mu^\sharp$ so that under the isomorphism~\eqref{sisnihfiemss}
\begin{equation*}
G^\sharp_{\ovl{\bb Q_p}}\cong
\begin{cases}
\GL_{1, \ovl{\bb Q_p}}\times \prod_{v\in \Hom(K, \ovl{\bb Q_p})}\GSpin(V^*)\otimes_{K, v}\ovl{\bb Q_p} &\text{in Case O},\\
\GL_{1, \ovl{\bb Q_p}}\times \prod_{v\in \Hom(K, \ovl{\bb Q_p})}\GU(V^*)\otimes_{K, v}\ovl{\bb Q_p} &\text{in Case U}
\end{cases}
\end{equation*}
it is of the form $(0,\tilde\mu^{\sharp\prime},\ldots,\tilde\mu^{\sharp\prime})$, where $\tilde\mu^{\sharp\prime}$ is not fixed by any nontrivial Weyl group element. Then $\tilde\mu^\sharp$ is not fixed by any nontrivial element of $W_{G^\sharp}$.
\end{proof}

\subsection{Perverse \texorpdfstring{$t$}{t}-exactness and vanishing results}\label{peomeofeniemfiesiws}

In this subsection, we prove a perverse $t$-exactness result for Hecke operators, and deduce a vanishing result for cohomology of Shimura varieties with torsion coefficients. We adopt the notation related to $\Bun_{\msf G}$ from~\S\ref{icoroenfiehsnw}.

For any reductive group $\msf G$ over a non-Archimedean local field $K$ of characteristic zero and any open substack $U\subset \Bun_{\msf G}$, there exists a perverse $t$-structure on $\bx D_{\bx{lis}}(\Bun_{\msf G}, \ovl{\bb F_\ell})$ defined as follows: For each $\msf b\in B(\msf G)$, we define $d_{\msf b}\defining \bra{2\rho_{\msf G}, \nu_{\msf b}}$, where $\nu_{\msf b}$ is the slope homomorphism of $\msf b$, cf.~\cite[Definition~4.14]{H-L24}. Then an object $A$ is contained in ${}^p\bx D^{\le 0}(U, \ovl{\bb F_\ell})$ if $i_{\msf b}^*A\in\bx D^{\le d_{\msf b}}(\msf G_{\msf b}, \Lbd)$, and $A$ is contained in ${}^p\bx D^{\ge 0}(U, \ovl{\bb F_\ell})$ if $i_{\msf b}^!A\in\bx D^{\ge d_{\msf b}}(\msf G_{\msf b}, \Lbd)$. Here we recall that $i_{\msf b}$ is the inclusion $\Bun^{\msf b}_{\msf G}\subset \Bun_{\msf G}$.

Let
\begin{equation*}
\bx D^{\bx{ULA}}(\Bun_{\msf G}, \ovl{\bb F_\ell})\subset \bx D_{\bx{lis}}(\Bun_{\msf G}, \ovl{\bb F_\ell})
\end{equation*}
be the full subcategory of universally locally acyclic (ULA) objects; see~~\cite[Definition~\Rmnum4.2.22]{F-S24}. By~\cite[Theorem~\Rmnum5.7.1]{F-S24}, $\bx D^{\bx{ULA}}(\Bun_{\msf G}, \ovl{\bb F_\ell})$ consists of objects $A$ such that  $i_{\msf b}^*A\in \bx D^\adm(\msf G_{\msf b}, \ovl{\bb F_\ell})$ for each $\msf b\in B(\msf G)$. 

We import Setup~\ref{setuepnifneism}. In particular, $F$ is a totally real number field unramified at a prime $p$, $\mbf G$ is a special orthogonal or unitary group over $F$ with $G^*=\mbf G\otimes_F(F\otimes\bb Q_p)$, and $\msf G^\sharp$ is a central extension of $\Res_{F/\bb Q}\mbf G$ with $G^\sharp=\mbf G^\sharp\otimes\bb Q_p$. We then have the following local result on the perverse $t$-exactness of Hecke operators, which generalizes \cite[Corollary~4.27]{H-L24} to the present setting: 

\begin{thm}\label{teniherieniehrnis}
Suppose
\begin{equation*}
\phi\in \Phi^\sems(\Res_{K/\bb Q_p}G^*, \ovl{\bb F_\ell})
\end{equation*}
is generic in the sense of \textup{Definition~\ref{dienfiehIGNieneifmeis}}. In Case O2, assume further that $\phi$ is $\hat{\Std}$-generic in the sense of \textup{Definition~\ref{tttiitienfieifs}}. We regard $\phi$ as a semisimple toral $L$-parameter $\phi^\sharp\in\Phi^\sems(G^\sharp, \ovl{\bb F_\ell})$ via the natural embedding
\begin{equation*}
\LL(\Res_{K/\bb Q_p}G^*)(\ovl{\bb F_\ell})\to \LL G^\sharp(\ovl{\bb F_\ell}).
\end{equation*}
Assume further that $\ell\ne 2$, and that $(\ell, n(G^*)!)=1$ in Case O1. Then, for any dominant cocharacter $\mu^\sharp$ of $G^\sharp_{\ovl{\bb Q_p}}$, the Hecke operator $\bx T_{\mu^\sharp}$ attached to the extended highest weight tilting module $\LL \mcl T_{\mu^\sharp}$, as defined in \eqref{Heivneinekdnfienis}, preserves ULA objects. Moreover, the induced functor
\begin{equation*}
i_\uno^*\bx T_{\mu^\sharp}: \bx D^{\bx{ULA}}(\Bun_{G^\sharp}, \ovl{\bb F_\ell})_{\phi^\sharp}\to \bx D^\adm(G^\sharp, \ovl{\bb F_\ell})_{\phi^\sharp},
\end{equation*}
where $\uno\in B(G^\sharp)$ is the trivial element, is exact with respect to  the perverse $t$-structure on the source and the natural t-structure on the target.
\end{thm}
\begin{proof}
By \cite[Theorem~4.23]{H-L24} and \cite[Theorem~1.13]{Ham24}, it suffices to check that all of the following claims hold: 
\begin{itemize}
\item
\cite[Assumption~6.5]{Ham24} holds for $\phi^\sharp$.
\item
$\phi^\sharp$ is regular.
\item
$\phi^\sharp$ is $\mu^\sharp$-regular, and there exists a cocharacter $\tilde\mu^\sharp$ of $G^\sharp_{\ovl{\bb Q_p}}$ that is not fixed by any nontrivial element $w\in W_{G^\sharp}$, such that  $\phi^\sharp$ is $\tilde\mu^\sharp$-regular.
\item
$\rho^{\chi^\sharp}_{b^\sharp, w^\sharp}$, as defined in \eqref{msinifdenideis}, is irreducible for any $b^\sharp\in B(G^\sharp)_{\bx{un}}$ and $w^\sharp\in W_{b^\sharp}$. Here the character $\chi^\sharp$ is attached to $\phi^\sharp_{\msf T^\sharp}$ via local Langlands correspondence for tori.
\end{itemize}
The first claim follows from Theorem~\ref{ienigehifniehsins}. The regularity of $\phi^\sharp$ follows from Lemma~\ref{iansihsienifhnss} and Lemma~\ref{geoeirnieinnifeiss}. The third claim follows from Lemma~\ref{neineiihenfies}.

It remains to prove the last claim. The group $G_{b^\sharp}^\sharp$ is isomorphic to a Levi factor of a parabolic subgroup of $G^\sharp$, which is of the form
\begin{equation*}
\GL_{1, \bb Q_p}\times\Res_{K/\bb Q_p}G'\times \Res_{K_1/\bb Q_p}H,
\end{equation*}
where $H$ is a product of general linear groups and $G'$ is a general spinor or general unitary group over $K$ that splits over an unramified quadratic extension. The preceding three claims hold with $G^\sharp$ replaced by $G^\sharp_{b^\sharp}$, and the desired irreducibility follows from \cite[Lemma~4.21]{H-L24} and \cite[Proposition~A.2]{Ham24}.
\end{proof}

Next, we recall a perversity result, which will be a crucial ingredient in the proof of torsion vanishing result later. We impose the following global assumptions for future use: 

\begin{setup}\label{ienfienfiesmwss}\enskip
\begin{itemize}
\item
$\bSh(\bb G, \bb X)$ is proper, and there exists a Shimura datum of Hodge type $(\bb G^\sharp, \bb X^\sharp)$ with a map of Shimura data $(\bb G^\sharp, \bb X^\sharp)\to (\bb G, \bb X)$ such that  $\bb G^\sharp_\ad\to \bb G_\ad$ is an isomorphism and $\#\pi_0(Z(\bb G^\sharp))$ is coprime to $\ell$. Let $E^\sharp\subset \bb C$ be the common reflex field.
\item
$\bb G$ and $\bb G^\sharp$ are unramified at $p$. Set
\begin{equation*}
\msf G^\sharp\defining \bb G^\sharp\otimes\bb Q_p, \qquad \msf G\defining \bb G\otimes\bb Q_p,
\end{equation*}
and fix a Borel pair $(\msf B^\sharp, \msf T^\sharp)$ of $\msf G^\sharp$ with image $(\msf B, \msf T)$ in $\msf G$.
\item
The central extension $\msf G^\sharp\to \msf G$ extends to a map of reductive integral models $\mcl G^\sharp\to \mcl G$ over $\bb Z_p$, and we define
\begin{equation*}
\mdc K_p^\sharp\defining \mcl G^\sharp(\bb Z_p), \qquad \mdc K_p\defining \mcl G(\bb Z_p).
\end{equation*}
\item
$\mdc K^{\sharp p}\le \bb G^\sharp(\Ade_f^p)$ is a neat compact open subgroup with image $\mdc K^p$ in $\bb G(\Ade_f^p)$. Set
\begin{equation*}
\mdc K^\sharp\defining \mdc K_p^\sharp\mdc K^{\sharp p}\le \bb G^\sharp(\Ade_f), \qquad \mdc K\defining \mdc K_p\mdc K^p\le \bb G(\Ade_f).
\end{equation*}
\item
We write
\begin{equation*}
\mcl H_{\mdc K_p^\sharp}\defining \ovl{\bb F_\ell}[\mdc K_p^\sharp\bsh \msf G^\sharp(\bb Q_p)/\mdc K_p^\sharp], \qquad \mcl H_{\mdc K_p}\defining \ovl{\bb F_\ell}[\mdc K_p\bsh \msf G(\bb Q_p)/\mdc K_p]
\end{equation*}
for the corresponding Hecke algebras with $\ovl{\bb F_\ell}$-coefficients.
\item
Let $\mfk m\subset \mcl H_{\mdc K_p}$ be a maximal ideal with inverse image $\mfk m^\sharp\subset  \mcl H_{\mdc K_p^\sharp}$ and corresponding semisimple toral $L$-parameters
\begin{equation*}
\phi_{\mfk m}\in \Phi^\sems(\msf G, \ovl{\bb F_\ell}), \qquad \phi_{\mfk m^\sharp}\in \Phi^\sems(\msf G^\sharp, \ovl{\bb F_\ell}),
\end{equation*}
respectively.
\end{itemize}
\end{setup}

We then have a finite Galois covering of Shimura varieties
\begin{equation*}
\bSh_{\mdc K^\sharp}(\bb G^\sharp, \bb X^\sharp)\to \bSh_{\mdc K}(\bb G, \bb X).
\end{equation*}
over $E^\sharp$. Note that $\iota_p: \bb C\to \ovl{\bb Q_p}$ induces an embedding $E^\sharp\to \ovl{\bb Q_p}$.

We now use the Igusa stacks for Hodge type Shimura varieties defined in \cite{DHKZ}. If we write
\begin{equation*}
\bSh_{\mdc K^{\sharp p}}(\bb G^\sharp, \bb X^\sharp)\defining \plim_{\tilde{\mdc K}_p}\bSh_{\tilde{\mdc K}_p\mdc K^{\sharp p}}(\bb G^\sharp, \bb X^\sharp),
\end{equation*}
where $\tilde{\mdc K}_p$ runs through all compact open subgroups of $\msf G^\sharp(\bb Q_p)$, then we have the Hodge--Tate period map
\begin{equation*}
\pi_{\bx{HT}}: \bSh_{\mdc K^{\sharp p}}(\bb G^\sharp, \bb X^\sharp)^{\bx{an}}\to \Gr_{\msf G^\sharp, \mu^\sharp}.
\end{equation*}
We then have the following Igusa stack $\Igu_{\mdc K^{\sharp p}}(\bb G^\sharp, \bb X^\sharp)$; see~\cite[Theorem~\Rmnum{1}]{DHKZ}.

\begin{thm}\label{heinifehifenis}
There is an Artin $v$-stack $\Igu_{\mdc K^{\sharp p}}(\bb G^\sharp, \bb X^\sharp)$ on $\Perfd_{\ovl\kappa}$ sitting in a Cartesian diagram 
\begin{equation*}
\begin{tikzcd}[sep=large]
\bSh_{\mdc K^{\sharp p}}(\bb G^\sharp, \bb X^\sharp)\ar[r, "\pi_{\bx{HT}}"]\ar[d, "\pr_{\Igu}"] &\Gr_{\msf G^\sharp, \mu^\sharp}\ar[d, "\bx{BL}"]\\
\Igu_{\mdc K^{\sharp p}}(\bb G^\sharp,\bb X^\sharp)\ar[r, "\pi^\Igu_{\bx{HT}}"] &\Bun_{\msf G^\sharp, \mu^{\sharp\bullet}},
\end{tikzcd}
\end{equation*}
where $\bx{BL}$ is the Beauville--Laszlo map from \textup{\cite[Proposition \Rmnum{3}.3.1]{F-S24}}. Moreover, $\Igu_{\mdc K^{\sharp p}}(\bb G^\sharp,\bb X^\sharp)$ is $\ell$-cohomologically smooth of dimension 0, and its dualizing sheaf is isomorphic to $\ovl{\bb F_\ell}[0]$.
\end{thm}

Define
\begin{equation*}
\mcl F\defining \bx R\pi^\Igu_{\bx{HT}, *}(\ovl{\bb F_\ell})\in \bx D_{\bx{lis}}(\Bun_{\msf G^\sharp, \mu^{\sharp\bullet}}, \ovl{\bb F_\ell}),
\end{equation*}
which is universally locally acyclic by \cite[Corollary~8.5.4]{DHKZ}. Moreover, since $\bSh(\bb G^\sharp, \bb X^\sharp)$ is proper, we recall the following perversity result from~\cite[Theorem~8.6.3]{DHKZ}.

\begin{thm}\label{vieiniefienmsis}
$\mcl F$ is perverse, i.e., 
\begin{equation*}
\mcl F\in {}^p\bx D^{\ge 0}(\Bun_{\msf G^\sharp, \mu^{\sharp\bullet}}, \ovl{\bb F_\ell})\cap {}^p\bx D^{\le 0}(\Bun_{\msf G^\sharp, \mu^{\sharp\bullet}}, \ovl{\bb F_\ell}).
\end{equation*}
\end{thm}

The object $\mcl F$ is related to cohomology of Shimura varieties as follows. By Theorem~\ref{heinifehifenis}, \cite[Theorem~8.4.10]{DHKZ} yields the following relation between the cohomology of the relevant Shimura varieties and the value of the corresponding Hecke operator on $\mcl F$.

\begin{thm}\label{ienifeihenifeismws}
There is an isomorphism
\begin{equation*}
\bx R\Gamma(\bSh_{\mdc K^{\sharp p}}(\bb G^\sharp, \bb X^\sharp)_{\bb C_p}, \ovl{\bb F_\ell})\cong i_\uno^*\bx T_{\mu^{\sharp\bullet}}(\mcl F[-\dim_{\bb C}(\bb X)])\paren{-\frac{\dim_{\bb C}(\bb X)}{2}}
\end{equation*}
in $\bx D(\msf G^\sharp, \ovl{\bb F_\ell})$.
\end{thm}

We now state our first main theorem on the torsion vanishing of the cohomology of orthogonal or unitary Shimura varieties away from central dimension. We impose Setup~\ref{setuepnifneism} with $F\ne \bb Q$, and extend the map $G^\sharp\to\Res_{K/\bb Q_p}G$ to a map of reductive integral models $\mcl G^\sharp\to \mcl G$ over $\bb Z_p$ by fixing a $\bb Z_p$-lattice of $\Res_{K/\bb Q_p}(\mbf V\otimes_FK)$, and define $\mdc K_p^\sharp\defining \mcl G^\sharp(\bb Z_p), \mdc K_p\defining \mcl G(\bb Z_p)$. Let $\mdc K^{p, \sharp}\le \mbf G^\sharp(\Ade_f^p)$ be a compact open subgroup with image $\mdc K^p\le \paren{\Res_{F/\bb Q}\mbf G}(\Ade_f^p)$, such that
\begin{equation*}
\mdc K\defining \mdc K_p\mdc K^p\le\Res_{F/\bb Q}\mbf G(\Ade_f)\quad\text{and}\quad\mdc K^\sharp\defining \mdc K_p^\sharp\mdc K^{\sharp p}\le \mbf G^\sharp(\Ade_f)
\end{equation*}
are both neat. Let $\ell$ be a rational prime coprime to $p$, and we assume moreover $(\ell, n(G)!)=1$ in Case O1. Let
\begin{equation*}
\mcl H_{\mdc K_p^\sharp}\defining \ovl{\bb F_\ell}[\mdc K_p^\sharp\bsh G^\sharp(\bb Q_p)/\mdc K_p^\sharp], \qquad \mcl H_{\mdc K_p}\defining \ovl{\bb F_\ell}[\mdc K_p\bsh G^*(K)/\mdc K_p]
\end{equation*}
be Hecke algebras with $\ovl{\bb F_\ell}$-coefficients. Let $\mfk m\subset \mcl H_{\mdc K_p}$ be a maximal ideal with inverse image $\mfk m^\sharp\subset  \mcl H_{\mdc K_p^\sharp}$ and corresponding semisimple toral $L$-parameters
\begin{equation*}
\phi_{\mfk m}\in \Phi^\sems(\Res_{K/\bb Q_p}G, \ovl{\bb F_\ell}), \quad \phi_{\mfk m^\sharp}\in \Phi^\sems(G^\sharp, \ovl{\bb F_\ell}),
\end{equation*}
respectively. Moreover, we assume that $p$ is coprime to $2\dim(\mbf V)$ in Case U, so we work in the setting of \textup{Setup~\ref{ienfienfiesmwss}}.

\begin{thm}\label{ninIEnifeifneims}
Suppose $F\ne\bb Q$, and suppose
\begin{equation*}
\phi_{\mfk m}\in\Phi^\sems(\Res_{K/\bb Q_p}G, \ovl{\bb F_\ell})
\end{equation*}
is generic. Assume further that $\phi_{\mfk m}$ is $\hat{\Std}$-regular in the sense of~\textup{Definition~\ref{tttiitienfieifs}} in case O2.
\begin{enumerate}
\item
$\etH^i(\bSh_{\mdc K^\sharp}(\mbf G^\sharp, \mbf X^\sharp)_{\ovl{E^\sharp}}, \ovl{\bb F_\ell})_{\mfk m^\sharp}$ vanishes unless $i=\dim_{\bb C}(\mbf X)$.
\item
$\etH^i(\bSh_{\mdc K}(\Res_{F/\bb Q}\mbf G, \mbf X)_{\ovl E}, \ovl{\bb F_\ell})_{\mfk m}$ vanishes unless $i=\dim_{\bb C}(\mbf X)$.
\end{enumerate}
\end{thm}
\begin{proof}
The second assertion follows from the first one: $\bSh_{\mdc K^\sharp}(\mbf G^\sharp, \mbf X^\sharp)_{\ovl{E^\sharp}}$ is a finite Galois covering of an open and closed subset $M$ of $\bSh_{\mdc K}(\Res_{F/\bb Q}\mbf G, \mbf X)_{\ovl{E^\sharp}}$ with Galois group denoted by $\mfk T$. The cohomology of $M$ is equipped with Hecke action by $\mcl H_{\mdc K_p}$. By the Hochschild--Serre spectral sequence,\begin{equation*}
\bx R\Gamma\paren{M, \ovl{\bb F_\ell}}_{\mfk m}=\bx R\Gamma\paren{\mfk T, \bx R\Gamma(\bSh_{\mdc K^\sharp}(\mbf G^\sharp, \mbf X^\sharp)_{\ovl{E^\sharp}}, \ovl{\bb F_\ell})_{\mfk m^\sharp}}_{\mfk m},
\end{equation*}
which is concentrated in degree at least $\dim_{\bb C}(\mbf X)$ by (1). Now $\bx R\Gamma\paren{\bSh_{\mdc K}(\Res_{F/\bb Q}\mbf G, \mbf X)_{\ovl E}, \ovl{\bb F_\ell}}_{\mfk m}$ is a finite direct sum of copies of the complex $\bx R\Gamma\paren{M, \ovl{\bb F_\ell}}_{\mfk m}$, which is also concentrated in degree at least $\dim_{\bb C}(\mbf X)$.

But by the \Poincare duality and \cite[Corollary~A.7]{H-L24}, it follows that
\begin{equation*}
\bx R\Gamma(\bSh_{\mdc K}(\mbf G, \mbf X)_{\ovl E}, \ovl{\bb F_\ell})_{\mfk m}\cong \bx R\Gamma(\bSh_{\mdc K}(\mbf G, \mbf X)_{\ovl E}, \ovl{\bb F_\ell})_{\mfk m^\vee}(\dim_{\bb C}(\mbf X))[2\dim_{\bb C}(\mbf X)]
\end{equation*}
is concentrated in degree at most $\dim_{\bb C}(\mbf X)$, where $\mfk m^\vee$ is the maximal ideal corresponding to $\phi_{\mfk m}^\vee$. 

For the first assertion, note that $\bSh_{\mdc K^\sharp}(\mbf G^\sharp, \mbf X^\sharp)$ is proper because $[\mbf G^\sharp, \mbf G^\sharp]$ is anisotropic, so the assertion follows from Theorems~\ref{vieiniefienmsis},~\ref{ienifeihenifeismws},~\ref{teniherieniehrnis} and the \Poincare duality, as $\phi_{\mfk m}$ and $\phi_{\mfk m}^\vee$ are both generic.
\end{proof}
For Shimura varieties of Abelian type localized at a split place, the same argument gives a more general statement. In fact, we do not need the full strength of the compatibility of Fargues--Scholze LLC with ``classical local Langlands correspondence'' in the sense of \cite[Assumption~6.5]{Ham24},  but only one property of the Fargues--Scholze LLC: 

\begin{hyp}\label{axioaisijenifes}
Suppose $\msf G/K$ is a quasisplit reductive group with a Borel pair $(\msf B, \msf T)$ and $\phi\in\Phi^\sems(\msf G, \ovl{\bb Q_\ell})$ is a semisimple generic toral $L$-parameter. Then for any $\msf b\in B(\msf G)$ and any $\rho\in\Pi(\msf G_{\msf b}, \ovl{\bb Q_\ell})$, if the composition of $\phi_\rho^\FS: W_K\to\LL\msf G_{\msf b}(\ovl{\bb Q_\ell})$ with the twisted embedding $\LL\msf G_{\msf b}(\ovl{\bb Q_\ell})\to\LL\msf G(\ovl{\bb Q_\ell})$ (as defined in \cite[\S \Rmnum{9}.7.1]{F-S24}) equals $\phi$, then $\msf b$ is unramified. 
\end{hyp}
\begin{rem}
Note that, by \cite[Lemma~3.14]{Ham24}, this hypothesis is a consequence of \cite[Assumption~6.5]{Ham24}: The hypothesis that $\phi_{\iota_\ell^{-1}\rho}^\sems=\iota_\ell^{-1}\phi_\rho^\FS$ factors through a generic parameter $\phi_{\msf T}$ of $\msf T$ implies that $\phi_{\iota_\ell^{-1}\rho}=\phi_{\iota_\ell^{-1}\rho}^\sems$ by \cite[Lemma~3.14]{Ham24}, which implies $\msf T$ is relevant for $\msf G_{\msf b}$ by Theorem~\ref{coeleninifheihsn}. So $\msf G_{\msf b}$ is quasisplit, or equivalently, $\msf b$ is unramified.
\end{rem}

Observe that this hypothesis is stable under duality: if it holds for $\phi$, then it holds for $\phi^\vee$. The following invariance property may be of independent interest.

\begin{prop}\label{ivnairenpeireimries}
Suppose $\msf G'\to \msf G$ is a map of quasisplit reductive groups over a non-Archimedean local field $K$ of characteristic zero that induces an isomorphism on adjoint groups. Let $(\msf B, \msf T)$ and $(\msf B', \msf T')$ be compatible Borel pairs of $\msf G$ and $\msf G'$, respectively. If $\phi\in\Phi^\sems(\msf G, \ovl{\bb Q_\ell})$ is a semisimple generic toral $L$-parameter, and we write $\phi'\in \Phi^\sems(\msf G', \ovl{\bb Q_\ell})$ for the image of $\phi$ under the natural map $\Phi^\sems(\msf G, \ovl{\bb Q_\ell})\to \Phi^\sems(\msf G', \ovl{\bb Q_\ell})$. Then \textup{Hypothesis~\ref{axioaisijenifes}} for $\phi$ implies that it holds for $\phi'$.

In particular, if $\msf G'\to \msf G$ is an injection, then \textup{Hypothesis~\ref{axioaisijenifes}} for $\msf G$ implies that it holds for $\msf G'$.
\end{prop}
\begin{proof}
Suppose $\msf b'\in B(\msf G')$ maps to $\msf b\in B(\msf G)$, then there is a map $\msf G_{\msf b'}\to \msf G_{\msf b}$ that induces an isomorphism on adjoint groups. Suppose $\rho'\in \Pi(\msf G'_{\msf b'}, \ovl{\bb Q_\ell})$ such that  the composition of $\phi_{\rho'}^\FS: W_K\to\LL\msf G'_{\msf b'}(\ovl{\bb Q_\ell})$ with the twisted embedding $\LL\msf G'_{\msf b'}(\ovl{\bb Q_\ell})\to\LL\msf G'(\ovl{\bb Q_\ell})$ equals $\phi'\in\Phi^\sems(\msf G', \ovl{\bb Q_\ell})$. Then it follows from the compatibility of Fargues--Scholze parameters with central characters Theorem~\ref{compaitbsilFaiirfies} that $\rho'$ factors through $\Im(\msf G'_{\msf b'}(K)\to \msf G_{\msf b}(K))$. Then it follows from \cite[Lemma 2.3]{G-K82} that there exists an irreducible smooth representation $\rho\in\Pi(\msf G_{\msf b}, \ovl{\bb Q_\ell})$ such that  $\rho'$ is a subquotient of $\rho|_{\msf G'_{\msf b'}(K)}$. It follows from compatibility of Fargues--Scholze correspondence with central extensions, Theorem~\ref{compaitbsilFaiirfies} that $\phi_{\rho'}^\FS$ is the image of $\phi_\rho^\FS$ under the natural map $\Phi^\sems(\msf G_{\msf b}, \ovl{\bb Q_\ell})\to \Phi^\sems(\msf G'_{\msf b'}, \ovl{\bb Q_\ell})$. So the composition of $\phi_\rho^\FS: W_K\to\LL\msf G_{\msf b}(\ovl{\bb Q_\ell})$ with the twisted embedding $\LL\msf G_{\msf b}(\ovl{\bb Q_\ell})\to\LL\msf G(\ovl{\bb Q_\ell})$ is a parameter $\tilde\phi$ whose image under the natural map $\Phi^\sems(\msf G, \ovl{\bb Q_\ell})\to \Phi^\sems(\msf G', \ovl{\bb Q_\ell})$ equals $\phi'$. So it follows from the compatibility of Fargues--Scholze correspondence with character twists and \cite[Appendix A]{Xu17} that we may twist $\rho$ by a character of $\coker(\msf G'_{\msf b'}(K)\to \msf G_{\msf b}(K))$ to make sure that $\tilde\phi=\phi$. Now the axiom for $\phi$ implies that $\msf b$ is unramified. In particular, $\msf G_{\msf b}$ is quasisplit. Suppose $\msf B_{\msf b}\le \msf G_{\msf b}$ is a Borel subgroup, then $\msf B_{\msf b'}=\msf B_{\msf b}\cap \msf G'_{\msf b'}$ is a Borel subgroup of $\msf G'_{\msf b'}$. Thus $\msf b'$ is unramified.

For the last assertion, it suffices to note that if $\phi\in\Phi^\sems(\msf G, \ovl{\bb Q_\ell})$ has toral generic image $\phi'\in\Phi^\sems(\msf G', \ovl{\bb Q_\ell})$, then $\phi$ is itself toral and generic.
\end{proof}

We now prove the second main theorem on vanishing result for torsion cohomology of Shimura varieties of Abelian type under Hypothesis~\ref{axioaisijenifes}. For a semisimple toral parameter $\ovl\phi\in\Phi^\sems(\msf G,\ovl{\bb F_\ell})$, we say that
\textup{Hypothesis~\ref{axioaisijenifes}} holds for $\ovl\phi$ if it holds for every semisimple generic toral lift $\phi\in\Phi^\sems(\msf G,\ovl{\bb Q_\ell})$ whose reduction modulo $\ell$ is $\ovl\phi$.

\begin{thm}\label{ieifmeimfeos}
We work in the setting of \textup{Setup~\ref{ienfienfiesmwss}}, and assume further that the set of unramified $\mu_\ad^\bullet$-acceptable elements $B(\msf G_\ad, \mu_\ad^\bullet)_{\bx{un}}$ is a singleton. Suppose $\phi_{\mfk m}$ is generic and \textup{Hypothesis~\ref{axioaisijenifes}} holds for $\phi_{\mfk m}$ in the modulo $\ell$ sense.
\begin{enumerate}
\item
$\etH^i(\bSh_{\mdc K^\sharp}(\bb G^\sharp, \bb X^\sharp)_{\ovl{E^\sharp}}, \ovl{\bb F_\ell})_{\mfk m^\sharp}$ vanishes unless $i=\dim_{\bb C}(\bb X)$.
\item
$\etH^i(\bSh_{\mdc K}(\bb G, \bb X)_{\ovl{E^\sharp}}, \ovl{\bb F_\ell})_{\mfk m}$ vanishes unless $i=\dim_{\bb C}(\bb X)$.
\end{enumerate}
\end{thm}
\begin{rem}
By \cite[Corollary 4.2.4]{X-Z17}, the condition that $B(\msf G_\ad, \mu_\ad^\bullet)_{\bx{un}}$ is a singleton (which is the $\mu_\ad$-ordinary element) is guaranteed when $\msf G_\ad$ is a product of unramified Weil restrictions of split simple groups $\prod_{i=1}^k\Res_{L_i/\bb Q_p}\msf H_i$, and the conjugacy class of Hodge cocharacters $\{\mu\}$ associated with $\bb X^\sharp$ induces a dominant cocharacter $\mu_\ad=(\mu_1, \ldots, \mu_k)$ of $\msf G_{\ovl{\bb Q_p}}$ via $\iota_p$, such that  each $\mu_i$ is trivial on all except possibly one simple factor of
\begin{equation*}
(\Res_{L_i/\bb Q_p}\msf H_i)_{\ovl{\bb Q_p}}\cong\prod_{v:L_i\hookrightarrow\ovl{\bb Q_p}}\msf H_i\otimes_{L_i,v}\ovl{\bb Q_p}.
\end{equation*}
\end{rem}
\begin{rem}
The first assertion is established in \cite[Theorem~10.1.6]{DHKZ} under the assumption that the Fargues--Scholze LLC for $\msf G^\sharp$ is ``natural'' in the sense of~\cite[Assumption~6.5]{Ham24}. However, this naturality is established in limited cases. For example, ``classical local Langlands correspondence'' for pure inner forms of $\GSpin_N$ when $N\ge 8$ has not been constructed, except for those irreducible representations with central character being a square of another character; see~\cite{G-T19}. Nonetheless, we can still establish the theorem under this weaker axiom by modifying the argument in \cite{DHKZ}.
\end{rem}
\begin{proof}
First, the second assertion follows from the first. Write $\mcl H'=\Im(\mcl H_{\mdc K_p^\sharp}\to \mcl H_{\mdc K_p})$, and let $\mfk m'$ be the inverse image of $\mfk m$ in $\mcl H'$. By functoriality of Shimura varieties, $\bSh_{\mdc K^\sharp}(\bb G^\sharp, \bb X^\sharp)$ is a finite Galois covering of an open and closed subset $M$ of $\bSh_{\mdc K}(\bb G, \bb X)_{\ovl{E^\sharp}}$, with Galois group denoted by $\mfk T$. The cohomology of $M$ is equipped with an action of $\mcl H'$. By the Hochschild--Serre spectral sequence,
\begin{equation*}
\bx R\Gamma\paren{M, \ovl{\bb F_\ell}}_{\mfk m'}=\bx R\Gamma\paren{\mfk T, \bx R\Gamma(\bSh_{\mdc K^\sharp}(\bb G^\sharp, \bb X^\sharp)_{\ovl{E^\sharp}}, \ovl{\bb F_\ell})_{\mfk m^\sharp}}_{\mfk m'}.
\end{equation*}
By (1), this complex is concentrated in degrees $\ge \dim_{\bb C}(\bb X)$. Hence
\begin{equation*}
\bx R\Hom_{\mcl H'_{\mfk m'}}\paren{(\mcl H_{\mdc K_p})_{\mfk m}, \bx R\Gamma\paren{M, \ovl{\bb F_\ell}}_{\mfk m'}}
\end{equation*}
is concentrated in degrees $\ge \dim_{\bb C}(\bb X)$, and so is
\begin{equation*}
\bx R\Gamma\paren{\bSh_{\mdc K}(\bb G, \bb X)_{\ovl{E^\sharp}}, \ovl{\bb F_\ell}}_{\mfk m}.
\end{equation*}

Applying the same argument to the dual maximal ideal $\mfk m^\vee$ corresponding to $\phi_{\mfk m}^\vee$, and using \Poincare duality together with \cite[Corollary~A.7]{H-L24}, we obtain
\begin{equation*}
\bx R\Gamma(\bSh_{\mdc K}(\bb G, \bb X)_{\ovl{E^\sharp}}, \ovl{\bb F_\ell})_{\mfk m}\cong \bx R\Gamma(\bSh_{\mdc K}(\bb G, \bb X)_{\ovl{E^\sharp}}, \ovl{\bb F_\ell})_{\mfk m^\vee}(\dim_{\bb C}(\bb X))[2\dim_{\bb C}(\bb X)].
\end{equation*}
Thus the $\mfk m$-localized cohomology is also concentrated in degrees $\le \dim_{\bb C}(\bb X)$, proving (2).

For the first assertion, note that $\phi_{\mfk m^\sharp}$ and $\phi_{\mfk m^\sharp}^\vee$ are both generic. By Theorem~\ref{vieiniefienmsis} and~\ref{ienifeihenifeismws}, together with \Poincare duality, it suffices to show the $t$-exactness of
\begin{equation*}
i_\uno^*\bx T_{\mu^{\sharp\bullet}}: \bx D^{\bx{ULA}}(\Bun_{\msf G^\sharp}, \ovl{\bb F_\ell})_{\ovl\phi^\sharp}\to \bx D^\adm(\msf G^\sharp, \ovl{\bb F_\ell})_{\ovl\phi^\sharp}
\end{equation*}
(as defined in Theorem~\ref{teniherieniehrnis}) for any semisimple generic toral $L$-parameter $\ovl\phi\in\Phi^\sems(\msf G, \ovl{\bb F_\ell})$ satisfying \textup{Hypothesis~\ref{axioaisijenifes}} in the modulo $\ell$ sense, where $\ovl\phi^\sharp$ denotes its image under
\begin{equation*}
\LL\msf G(\ovl{\bb F_\ell})\to \LL\msf G^\sharp(\ovl{\bb F_\ell}).
\end{equation*}
By \cite[Corollary 4.2.4]{X-Z17}, the natural image of $B(\msf G^\sharp,\mu^{\sharp\bullet})_{\bx{un}}$ in $B(\msf G_\ad,\mu_\ad^\bullet)_{\bx{un}}$ is a singleton by assumption. Thus, by \cite[Proposition~10.2.4]{DHKZ}, it remains to show that $\bx D_{\bx{lis}}(\Bun_{\msf G^\sharp},\ovl{\bb F_\ell})_{\ovl\phi^\sharp}$ is supported on the unramified strata.

Let
\begin{equation*}
\ovl\rho^\sharp\in\bx D(\Bun_{\msf G^\sharp}^{\msf b^\sharp},\ovl{\bb F_\ell})_{\ovl\phi^\sharp}\cong\bx D(\msf G^\sharp_{\msf b^\sharp},\ovl{\bb F_\ell})
\end{equation*}
be an irreducible admissible representation of $\msf G^\sharp_{\msf b^\sharp}(\bb Q_p)$. By \cite[Lemma~4.3(1)]{H-L24} and \cite[\S \Rmnum{9}.7.1]{F-S24}, the Fargues--Scholze parameter of $\ovl\rho^\sharp$, composed with the twisted embedding
\begin{equation*}
\LL\msf G^\sharp_{\msf b^\sharp}(\ovl{\bb F_\ell})\to\LL\msf G^\sharp(\ovl{\bb F_\ell}),
\end{equation*}
is $\ovl\phi^\sharp$.

By \cite[Lemma 6.8]{Dat05}, lift $\ovl\rho^\sharp$ to an irreducible admissible $\ovl{\bb Q_\ell}$-representation $\rho^\sharp$ admitting a $\msf G^\sharp_{\msf b^\sharp}(\bb Q_p)$-stable $\ovl{\bb Z_\ell}$-lattice such that $\ovl\rho^\sharp$ occurs as a subquotient of its reduction modulo $\ell$. By compatibility of Fargues--Scholze parameters with reduction modulo $\ell$, Theorem~\ref{compaitbsilFaiirfies}, the parameter $\phi_{\rho^\sharp}^\FS$ has reduction equal to $\phi_{\ovl\rho^\sharp}^\FS$. Hence its composition with the twisted embedding
\begin{equation*}
\LL\msf G^\sharp_{\msf b^\sharp}(\ovl{\bb Q_\ell})\to\LL\msf G^\sharp(\ovl{\bb Q_\ell})
\end{equation*}
is a semisimple generic toral lift of $\ovl\phi^\sharp$. By the modulo $\ell$ form of \textup{Hypothesis~\ref{axioaisijenifes}} and Proposition~\ref{ivnairenpeireimries}, the element $\msf b^\sharp$ is unramified. This proves the required support statement.
\end{proof}

\appendix

\section{(Twisted) Endoscopy theory}\label{oniinssmows}

In this appendix we review some standard definitions related to the trace formulas used in the main text and fix notation. \S\ref{neodosiiitromeifsl} recalls (twisted) endoscopic triples for reductive groups over both local and global fields. \S\ref{transofenrienfieslsss} reviews local transfers and pseudo-coefficients for square-integrable irreducible admissible representations. \S\ref{simpletrianidnifnienis} records Lefschetz functions and also a simple stable trace formula that will be used in the Langlands--Kottwitz method \S\ref{Lanlgan-Kotiemethiems}.

\subsection{Endoscopic triples}\label{neodosiiitromeifsl}

We recall here some general definitions of extended endoscopic triples from \cite[\S 1.2]{L-S87} and \cite[\S 2.1]{K-S99}. Note that the more general notion of endoscopic data is not needed in the cases we consider in this paper.

Let $F$ be a local or global field of characteristic zero, and consider a pair $(\msf G^*, \theta^*)$ where 
\begin{enumerate}
\item
$\msf G^*$ is a quasisplit reductive group over $F$ with a fixed $\Gal_F$-invariant pinning on $\hat{\msf G^*}$, 
\item
$\theta^*$ is a pinned automorphism of $\msf G^*$.
\end{enumerate}
Then $\theta^*$ induces an automorphism $\hat\theta$ of $\hat{\msf G^*}$ preserving the fixed $\Gal_F$-invariant pinning of $\hat{\msf G}$ \cite[\S 1.2]{K-S99}. Set $\LL\theta=\hat\theta\times\id_{W_F}$, which is an automorphism of $\LL\msf G$. Then an \tbf{extended endoscopic triple} for $(\msf G^*, \theta^*)$ is a triple $\mfk e=(\msf G^{\mfk e},\msf s^{\mfk e}, \LL\xi^{\mfk e})$, where $\msf G^{\mfk e}$ is a quasisplit reductive group over $F$, $\msf s^{\mfk e}\in \hat{\msf G^*}$, and $\LL\xi^{\mfk e}:\LL \msf G^{\mfk e}\to \LL\msf G^*$ is an $L$-homomorphism such that 
\begin{enumerate}
\item
$\Ad(\msf s^{\mfk e})\circ\hat\theta$ preserves a pair of a Borel subgroup and a maximal torus in $\hat{\msf G}$, and $\Ad(\msf s^{\mfk e})\circ\hat\theta\circ\LL\xi^{\mfk e}=\LL\xi^{\mfk e}$,
\item
$\LL\xi^{\mfk e}(\hat{\msf G^{\mfk e}})$ is the connected component of the subgroup of $\Ad(\msf s^{\mfk e})\circ\hat\theta$-fixed elements in $\hat{\msf G^*}$.
\end{enumerate}
$\mfk e$ is called \tbf{elliptic} if $\LL\xi^{\mfk e}\big(Z(\hat{\msf G^{\mfk e}})^{\Gal_F}\big)^\circ\subset Z(\hat{G^*})$.

Following \cite{KMSW}, we define an isomorphism between two extended endoscopic triples $\mfk e, \mfk e'$ to be an element $\msf g\in \hat{\msf G^*}$ such that
\begin{equation*}
\msf g\LL\xi^{\mfk e}(\LL{\msf G}^{\mfk e})\msf g^{-1}=\LL\xi^{\mfk e'}(\LL{\msf G}^{\mfk e'}), \quad g\msf s^{\mfk e}\hat\theta(g)^{-1}=\msf s^{\mfk e'}\modu{Z(\hat{\msf G})}.
\end{equation*}
Denote by $\mcl E(\msf G^*\rtimes\theta^*)$ the set of isomorphism classes of extended endoscopic triples for $(\msf G^*, \theta^*)$. When $\theta^*=\id_{\msf G^*}$, we also write $\mcl E(\msf G^*)$ for $\mcl E(\msf G^*\rtimes\theta^*)$. Also we write $\mcl E_\ellip(\msf G^*\rtimes\theta^*)$ or $\mcl E_\ellip(\msf G^*)$ for the subset of elliptic extended endoscopic triples.

Suppose $\mfk e\in \mcl E(\msf G^*\rtimes\theta^*)$. For each $g^{\mfk e}\in\hat{\msf G}^{\mfk e}$, the $\LL\xi^{\mfk e}(g^{\mfk e})\in\LL G$ induces an automorphism of $\mfk e$. Define the outer automorphism group of $\mfk e$ by
\begin{equation}\label{equirenrinfids}
\OAut(\mfk e)\defining \Aut(\mfk e)/\LL\xi^{\mfk e}(\hat{\msf G^{\mfk e}}).
\end{equation}

\subsection{Transfer of orbital integrals}\label{transofenrienfieslsss}

Here we recall some notions on the theory of transfer, following Arthur~\cite[\S 2.1]{Art13} and Mok~\cite[\S 3.1]{Mok15}.

Let $K$ be a local field of characteristic zero, $(\msf G^*, \theta^*)$ be a pair as in Appendix~\S\ref{neodosiiitromeifsl}, and $(\msf G, \varrho, z)$ be a pure inner twist of $\msf G^*$. We get an automorphism $\theta\defining \varrho\circ\theta^*\circ\varrho^{-1}$, which we assume to be a rational automorphism of $\msf G$. Moreover, we fix a Whittaker datum $\mfk w$ for $\msf G^*$. We write $\msf G\times\theta$ for the twisted group (or bitorsor) over $\msf G$. If $\theta$ is the identity, then of course $\msf G\times\theta=\msf G$ is just the trivial bitorsor. Given $\delta\in \msf G\times\theta$, we write $Z_{\msf G}(\delta)$ for the centralizer of $\delta$ in $\msf G$. We write $(\msf G\times\theta)(K)_\sreg\subset(\msf G\times\theta)(K)$ for the open subset of strongly regular semisimple elements, meaning those regular semisimple elements whose centralizer is connected, i.e., a maximal torus. We fix a Haar measure on $\msf G(K)$. For $\delta\in (\msf G\times\theta)(K)_\sreg$, the Weyl discriminant of $\delta$ is defined as 
\begin{equation*}
D^{\msf G}(\delta)\defining \det(1-\Ad(\delta)|\mfk g/\mfk g_\delta)\in K^\times,
\end{equation*}
where $\mfk g$ and $\mfk g_\delta$ are the Lie algebras of $\msf G$ and $Z_{\msf G}(\delta)$, respectively. We fix a Haar measure on the torus $Z_{\msf G}(\delta)$, which induces a quotient measure on $Z_{\msf G}(\delta)(K)\bsh \msf G(K)$.

If $K$ is \na, we let $\mcl H(\msf G\times\theta)$ be the space of smooth compactly supported functions on $(\msf G\times\theta)(K)$ with complex coefficients. If $K=\bb R$, we fix a maximal compact subgroup $\mdc K$ of $\msf G(\bb R)$ and let $\mcl H(\msf G\times\theta)$ be the space of bi-$\mdc K$-finite smooth compactly supported functions on $(\msf G\times\theta)(\bb R)$ with complex coefficients.

For $f\in \mcl H(\msf G\times\theta)$ and $\delta\in (\msf G\times\theta)(K)_\sreg$, the normalized orbital integral of $f$ along the conjugacy class of $\delta$ is defined as
\begin{equation*}
\Orb_\delta(f)\defining\largel{D^{\msf G}(\delta)}^{\frac{1}{2}}\int_{Z_{\msf G}(\delta)(K)\bsh \msf G(K)}f(x^{-1}\delta x)\bx dx,
\end{equation*}
Moreover, when $\msf G$ is quasisplit and $\theta^*=\id$, the normalized stable orbital integral of $f$ along $\delta$ is defined as
\begin{equation*}
\SOrb_\delta(f)\defining\sum_{\delta'}\Orb_{\delta'}(f),
\end{equation*}
where $\delta'$ runs over a set of representatives for the $\msf G(K)$-conjugacy classes of elements that are $\msf G(\ovl K)$-conjugate to $\delta$.

Assume that either $(\msf G, \varrho, z)=(\msf G^*, \id, \uno)$ or $\theta=\id$. For $\mfk e\in \mcl E(\msf G^*\rtimes\theta^*)$, a \tbf{transfer factor}
\begin{equation*}
\Delta[\mfk w, \mfk e, z]:\msf G^{\mfk e}(K)_\sreg\times(\msf G\times\theta)(K)_\sreg\to \bb C
\end{equation*}
is defined in \cite[\S 1.1.2]{KMSW}, such that $\Delta[\mfk w, \mfk e, z]$ is a function on stable conjugacy classes of $\msf G^{\mfk e}(K)_\sreg$ and $\msf G(K)$-conjugacy classes of $(\msf G\times\theta)(K)_\sreg$. With the transfer factor in hand, we now recall the notion of matching test functions from \cite[\S 5.5]{K-S99}.

\begin{defi}\label{mathcineinitenisehnis}
Two functions $f^{\msf G^{\mfk e}}\in\mcl H(\msf G^{\mfk e})$ and $f\in \mcl H(\msf G\times\theta)$ are called ($\Delta[\mfk w, \mfk e, z]$-) \tbf{matching test functions} if, for every $\gamma\in\msf G^{\mfk e}(K)_\sreg$,
\begin{equation*}
\SOrb_\gamma(f^{\msf G^{\mfk e}})=\sum_{\delta\in(\msf G\times\theta)(K)_\sreg/\msf G(K)\dash\Conj}\Delta[\mfk w, \mfk e, z](\gamma, \delta)\Orb_\delta(f).
\end{equation*}
For brevity, say that $f^{\msf G^{\mfk e}}$ is a transfer of $f$ to $\msf G^{\mfk e}$.
\end{defi}
\begin{rem}
Since the orbital integrals $\Orb_\delta(f)$ depend on the choices of measures on $\msf G(K)$ and $Z_{\msf G}(\delta)(K)$, the concept of matching functions also depends on the choice of Haar measures on $\msf G(K)$ and $\msf G^{\mfk e}(K)$, and all tori in $\msf G$ and $\msf G^{\mfk e}$. There is a way to synchronize the various tori; cf.~\cite[Remark 5.1.2]{A-K24}.
\end{rem}

We now state a theorem asserting existence of transfer of orbital integrals. When $K=\bb R$, it is a fundamental result of Shelstad \cite{She82, She08}. When $K$ is \na, it is a culmination of the work of many people, including Langlands and Shelstad \cite{L-S87, L-S90}, Waldspurger \cite{Wal97, Wal06}, and Ng\^o \cite{Ngo10}.

\begin{thm}\label{sisienfnieies}
Let $f\in \mcl H(\msf G\times\theta)$ and $\mfk e\in \mcl E_\ellip(\msf G)$. Then there exists a transfer $f^{\msf G^{\mfk e}}$ of $f$ to $\msf G^{\mfk e}$.

Moreover, suppose $\mdc K\subset \msf G(K)$ is a $\theta$-stable hyperspecial maximal compact open subgroup. Then there exists a hyperspecial maximal compact open subgroup $\mdc K^{\mfk e}\subset \msf G^{\mfk e}(K)$ such that the characteristic function $\uno_{\mdc K^{\mfk e}}$ is a transfer of $\uno_{\mdc K\times\theta}$ to $\msf G^{\mfk e}$, provided the Haar measure is chosen such that
\begin{equation*}
\Vol(\mdc K)=\Vol(\mdc K^{\mfk e})=1.
\end{equation*}
\end{thm}

\subsection{Cuspidal functions}

In this subsection, we recall the definition of cuspidal and stabilizing functions, following Labesse~\cite[Definition 3.8.1, 3.8.2]{Lab99}. Recall that an element $\gamma\in\msf G(K)$ is called an elliptic element if the maximal split subtorus of the center of $Z_{\msf G}(\gamma)$ is equal to the maximal split subtorus of $Z(\msf G)$.

\begin{defi}\label{stabbizleirnign-fucniotns}
Suppose $K$ is non-Archimedean of characteristic zero. A function $f\in\mcl H(\msf G)$ is called
\begin{itemize}
\item
\tbf{cuspidal} if the orbital integral of $f$ vanishes on all regular semisimple non-elliptic elements.
\item
\tbf{strongly cuspidal} if the orbital integral of $f$ vanishes outside regular semisimple elliptic elements.
\item
\tbf{stabilizing} if it is cuspidal and the $\kappa$-orbital integral of $f$ (as defined in \cite[p.~68]{Lab99}) vanishes on all semisimple elements $\gamma$ and all nontrivial $\kappa$.
\end{itemize}
\end{defi}

We recall the notion of pseudo-coefficients of \cite{Kaz86, Clo86}:

\begin{defi}\label{pseeidifniehdins}
Fix a Haar measure on $\msf G(K)$. For a square-integrable irreducible admissible representation $\pi$ of $\msf G(K)$, a function $f_\pi\in \mcl H(\msf G)$ is called a \tbf{pseudo-coefficient} for $\pi$ if
\begin{equation*}
\tr(f_\pi|\pi')=\delta_{\pi, \pi'}
\end{equation*}
for each tempered representation $\pi'$ of $\msf G(K)$.
\end{defi}

\begin{prop}[\cite{Kaz86, Clo86}]\label{jislisienicienfienieihdinf}
Suppose $K$ is non-Archimedean of characteristic zero, and fix a Haar measure on $\msf G(K)$. For any square-integrable irreducible admissible representation $\pi$ of $\msf G(K)$, there exists a pseudo-coefficient $f_\pi$ for $\pi$. Moreover, for any such $f_\pi$:
\begin{enumerate}
\item
$\tr(f_\pi|\pi')=0$ for any finite-length admissible smooth representation $\pi'$ of $\msf G(K)$ that is parabolically induced from a proper parabolic subgroup of $\msf G$.
\item
For every regular elliptic element $\gamma\in\msf G(K)$,
\begin{equation*}
\int_{Z_{\msf G}(\gamma)(K)\bsh\msf G(K)}f_\pi(g^{-1}\gamma g)\bx dg=\Theta_\pi(\gamma),
\end{equation*}
where $\Theta_\pi$ denotes the Harish-Chandra character of $\pi$.
\item
$f_\pi$ is cuspidal.
\item
If $\pi$ is supercuspidal, then $\tr(f_\pi|\pi')=\delta_{\pi,\pi'}$ for every irreducible admissible representation $\pi'$ of $\msf G(K)$.
\end{enumerate}
\end{prop}
\begin{proof}
The existence of $f_\pi$ is established by Kazhdan~\cite[Theorem K]{Kaz86} and Clozel~\cite[Proposition 1]{Clo86}. The trace values $\Theta_{\pi'}(f_\pi)$ for admissible representations $\pi'$ of $\msf G(K)$ and the orbital integrals
\begin{equation*}
\int_{Z_{\msf G}(\gamma)(K)\bsh\msf G(K)}f_\pi(g^{-1}\gamma g)\bx dg
\end{equation*}
for any regular elliptic elements $\gamma\in\msf G(K)$ are independent of the choice of $f_\pi$ by \cite[Theorem~0]{Kaz86}. Assertions (i) and (iii) are built into Kazhdan's construction \cite[Theorem~K, Theorem~A.(b)]{Kaz86}, and (ii) follows from \cite[Theorem~K]{Kaz86}.

For (iv), if $\pi'$ is supercuspidal, the assertion follows from the defining property of a pseudo-coefficient, since supercuspidal representations are tempered. If $\pi'$ is not supercuspidal, then its image in the Grothendieck group of finite-length admissible representations is a linear combination of representations induced from proper parabolic subgroups of $\msf G$ \cite[Proposition~2]{Clo86}; hence (i) gives $\tr(f_\pi|\pi')=0$.
\end{proof}

\subsection{Simple stable trace formulas}\label{simpletrianidnifnienis}

Here we recall some results on simple stable trace formulas from \cite{K-S23, Ham22}. We work in the following setting.

\begin{note}\enskip
\begin{itemize}
\item
$F$ is a totally real number field.
\item
$\bb G^*$ is a quasisplit reductive group over $F$ that is simple over $\ovl F$; assume $\bb G(F\otimes\bb R)$ admits discrete series.
\item
$(\bb G, \varrho, z)$ is a pure inner twist of $\bb G^*$.
\item
Let $Z$ denote the center of $\bb G^*$, and let $A_Z$ denote the maximal split torus of $\Res_{F/\bb Q}Z$; write $A_{Z, \infty}\defining A_Z(\bb R)^\circ$.
\item
Let $\mdc K_\infty=\prod_{v\in\infPla_F}\mdc K_v\le \bb G(F\otimes\bb R)$ be the product of a maximal compact subgroup with $Z_{\bb G}(F\otimes\bb R)$.
\item
For each finite place $v$ of $F$, set $q(\bb G_v)$ to be the $F_v$-rank of $\bb G_{v, \ad}$.
\item
For each infinite place $v$ of $F$, set $q(\bb G_v)$ to be the real dimension of the locally symmetric space $\bb G(F_v)/\mdc K_v$.
\item
Set
\begin{equation*}
\bb G(\Ade_F)^1\defining \bigcap_{\chi\in X^*(\bb G)}\ker(\norml{-}\circ\chi: \bb G(\Ade_F)\to \bb R_+).
\end{equation*}
In particular, $\bb G(\Ade_F)=\bb G(\Ade_F)^1\times A_{Z, \infty}$.
\end{itemize}
\end{note}

\begin{defi}\label{isnsieieifmeips}
A \tbf{central character datum} for $\bb G$ is a pair $(\mfk X, \chi)$ where $\mfk X$ is a closed subgroup of $Z(\Ade_F)$ containing $A_{Z, \infty}$ such that  $Z(F)\mfk X$ is closed in $Z(\Ade_F)$, and $\chi: (\mfk X\cap Z(F))\bsh \mfk X\to \bb C^\times$ is a continuous character. In particular, $Z(F)\mfk X$ is cocompact in $Z(\Ade_F)$ because $Z(F)\bsh Z(\Ade_F)/A_{Z, \infty}$ is compact.

For our purposes, it suffices to consider the cases when $\mfk X=\prod_{v\in\Pla_F}\mfk X_v$ where $\mfk X_\tau=Z(F_\tau)$ for each $\tau\in\infPla_F$.

Note that the center of $\bb G$ is isomorphic to $Z$ via $\varrho$, so any central character datum for $\bb G^*$ may be regarded as a central character datum for $\bb G$.
\end{defi}

\begin{rem}
We suppress the choice of Haar measures for various groups below as they are standard.
\end{rem}

\begin{defi}
Given a central character datum $(\mfk X, \chi)$ for $\bb G^*$ of the form $\mfk X=\prod_{v\in\Pla_F}\mfk X_v$, for each $v\in \Pla_F$, let $\mcl H(\bb G(F_v), \chi_v^{-1})$ be the space of smooth functions on $\bb G(F_v)$ that are compactly supported modulo center and transform under $\mfk X_v$ by the character $\chi_v^{-1}$ and moreover $\mdc K_v$-finite when $v\in\infPla_F$.

Given a semisimple element $\gamma_v\in\bb G(F_v)$ with $I_{\gamma_v}=Z_{\bb G_{F_v}}(\gamma_v)^\circ$, we define the orbital integral on $\mcl H(\bb G(F_v), \chi_v^{-1})$ to be 
\begin{equation*}
\Orb_{\gamma_v}(f_v)\defining \int_{I_{\gamma_v}(F_v)\bsh \bb G(F_v)}f_v(x_v^{-1}\gamma_vx_v)\bx dx_v,
\end{equation*}
where $I_{\gamma_v}\bsh \bb G(F_v)$ is given the Euler--\Poincare measure as defined in \cite[\S 1]{Kot88}. If $\pi_v$ is an admissible representation of $\bb G(F_v)$ with central character $\chi_v$ on $\mfk X_v$, we define the trace character on $\mcl H(\bb G(F_v), \chi_v^{-1})$ to be
\begin{equation*}
\Theta_{\pi_v}(f_v)\defining\tr\bigg(\int_{\bb G(F_v)/\mfk X_v}f_v(g)\pi_v(g)\bx dg\bigg).
\end{equation*}

We also define the adelic Hecke algebra $\mcl H(\bb G(\Ade_F), \chi^{-1})$ as well as adelic orbital integrals and adelic trace characters by taking restricted tensor product over the local cases considered above.
\end{defi}

\begin{defi}\label{transofmieinfieidjef}
Given a central character datum $(\mfk X, \chi)$ for $\bb G^*$, we write 
\begin{itemize}
\item
$\Gamma_{\ellip, \mfk X}(\bb G)$ for the set of $\mfk X$-orbits of elliptic conjugacy classes in $\bb G(F)$,
\item
$\Sigma_{\ellip,\mfk X}(\bb G)$ for the set of $\mfk X$-orbits of elliptic stable conjugacy classes in $\bb G(F)$,
\item
$L^2_{\disc, \chi}(\bb G(F)\bsh \bb G(\Ade_F))$ for the space of measurable functions on $\bb G(F)\bsh \bb G(\Ade_F)$ transforming under $\mfk X$ by $\chi$ and square-integrable on $\bb G(F)\bsh\bb G(\Ade_F)^1/\paren{\mfk X\cap \bb G(\Ade_F)^1}$,
\item
$L^2_{\cusp, \chi}(\bb G(F)\bsh \bb G(\Ade_F))$ for the space of cuspidal measurable functions on $\bb G(F)\bsh \bb G(\Ade_F)$ transforming under $\mfk X$ by $\chi$ and square-integrable on $\bb G(F)\bsh\bb G(\Ade_F)^1/\paren{\mfk X\cap \bb G(\Ade_F)^1}$,
\item
$\cuspRep_\chi(\bb G)$ the set of isomorphism classes of cuspidal automorphic representations of $\bb G(\Ade_F)$ whose central characters under $\mfk X$ are $\chi$.
\end{itemize}
We define the following invariant distributions on $\mcl H(\bb G(\Ade_F), \chi^{-1})$:
\begin{equation*}
\bx T_{\ellip, \chi}^{\bb G}(f)\defining\sum_{\gamma\in \Gamma_{\ellip, \mfk X}(\bb G)}\frac{1}{\#\pi_0(Z_{\bb G}(\gamma))}\Vol\paren{I_\gamma(F)\bsh I_\gamma(\Ade_F)/\mfk X}\Orb_\gamma(f),
\end{equation*}
\begin{equation*}
\bx T_{\disc, \chi}^{\bb G}(f)\defining\tr\paren{f|L^2_{\disc, \chi}(\bb G(F)\bsh\bb G(\Ade_F))},
\end{equation*}
\begin{equation*}
\bx T_{\cusp, \chi}^{\bb G}(f)\defining\tr\paren{f|L^2_{\cusp, \chi}(\bb G(F)\bsh\bb G(\Ade_F))},
\end{equation*}
\end{defi}

Next, we recall the definitions of unramified twists of Steinberg representations and Lefschetz functions.

\begin{defi}\label{unraimfiehidtiensitineis}
The Steinberg representation $\St_{\bb G_v}$ is the discrete series representation defined in \cite[10.4.6]{B-W00}. An unramified twist of $\St_{\bb G_v}$ is the twist of the Steinberg representation by an unramified character of $\bb G(F_v)$, where a character of $\bb G(F_v)$ is unramified if it is trivial on all compact subgroups of $\bb G(F_v)$; see~\cite[p.~17]{Cas95}.
\end{defi}

By \cite[Theorem~2 and Theorem~2']{Kot88} and \cite[Proposition~A.1, A.4 and Lemma~A.7]{K-S23}, we introduce the following non-Archimedean Lefschetz function:

\begin{defi}\label{noninIAhienihifneins}
Let $v$ be a finite place of $F$. If $Z_{F_v}$ is anisotropic, there exists a \tbf{Lefschetz function} $f_{\Lef, v}^{\bb G}\in \mcl H(\bb G(F_v))$ such that 
\begin{itemize}
\item
If $\bb G_\ad$ is simple, then for each irreducible admissible representation $\pi$ of $\bb G(F_v)$,
\begin{equation*}
\tr\paren{f_{\Lef, v}^{\bb G}|\pi}=
\begin{cases}1 &\pi=\uno,\\
(-1)^{q(\bb G_v)} &\pi=\St_{\bb G_v},\\
0 &\text{otherwise}.
\end{cases}
\end{equation*}
\item
If $\gamma_v\in \bb G(F_v)$ is semisimple with $I_{\gamma_v}=Z_{\bb G_v}(\gamma_v)^\circ$, then the orbital integral
\begin{equation*}
\Orb_{\gamma_v}(f_{\Lef, v}^{\bb G})=\int_{I_{\gamma_v}(F_v)\bsh\bb G(F_v)}f_{\Lef, v}^{\bb G}(g_v^{-1}\gamma_v g_v)\bx dg_v
\end{equation*}
vanishes unless $Z(I_{\gamma_v}(F_v))$ is compact, in which case $\Orb_{\gamma_v}(f_{\Lef, v}^{\bb G})=1$. Here $I_{\gamma_v}(F_v)\bsh \bb G(F_v)$ is endowed with the Euler--\Poincare measure.
\end{itemize}

In general, set $A_v$ to be the maximal split torus in the center of $\bb G_v$, and set $\bb G_v'\defining\bb G_v/A_v$. Let
\begin{equation*}
\nu:\bb G(F_v)\to X^\bullet(A_v)\otimes\bb R
\end{equation*}
denote the valuation map as in \cite[\S 3.9]{Lab99}, with kernel $\bb G(F_v)^1$. We define
\begin{equation*}
f_{\Lef, v}^{\bb G_v}\defining\uno_{\bb G(F_v)^1}\cdot f_{\Lef, v}^{\bb G_v'}.
\end{equation*}
Then:
\begin{itemize}
\item
$f_{\Lef, v}^{\bb G}$ is strongly cuspidal and stabilizing (see \textup{Definition~\ref{stabbizleirnign-fucniotns}});
\item
if $\tr(f_{\Lef, v}^{\bb G}|\pi_v)\ne 0$ for some irreducible unitary representation $\pi_v$ of $\bb G(F_v)$, then $\pi_v$ is an unramified character twist of either the trivial representation or the Steinberg representation;
\item
$(-1)^{q(\bb G_v)}f_{\Lef, v}^{\bb G}$ and $(-1)^{q(\bb G_v^*)}f_{\Lef, v}^{\bb G^*}$ are associated.\end{itemize}
\end{defi}

Assume that $\bb G(F\otimes\bb R)$ admits discrete series. We introduce the following Archimedean Lefschetz function.

\begin{defi}\label{Lelfinedihfinies}
Suppose $\tau\in \infPla_F$ and $\xi_\tau$ is an irreducible algebraic representation of $\bb G_\tau$ with regular highest weight. Denote by $\chi_{\xi_\tau}: Z(F_\tau)\to \bb C^\times$ the inverse of the central character of $\xi_\tau$. Then there exists a \tbf{Lefschetz function}
\begin{equation*}
f_{\xi_\tau}^{\bb G}\in \mcl H(\bb G(F_\tau), \chi_{\xi_\tau}^{-1})
\end{equation*}
associated with $\xi_\tau$ such that 
\begin{equation*}
\tr(f_{\xi_\tau}|\pi_\tau)=\bx{ep}_{\mdc K_\tau}(\pi_\tau\otimes\xi_\tau)\defining \sum_{i\in\bb N}(-1)^i\dim\bx H^i(\Lie(\bb G_\tau(\bb R)), \mdc K_\tau; \pi_\tau\otimes\xi_\tau)
\end{equation*}
for each irreducible admissible representation $\pi_\tau$ of $\bb G_\tau(\bb R)$ whose central character equals the inverse of the central character of $\xi_\tau$.

For any irreducible admissible $(\mfk g_\tau, \mdc K_\tau)$-module $\pi_\tau$ such that $\tr\big(f_{\xi_\tau}^{\bb G}|\pi_\tau\big)\ne 0$, the representation $\pi_\tau$ is a discrete series representation cohomological for $\xi_\tau$; equivalently, $\pi_\tau$ has the same central character and infinitesimal character as $\xi^\vee$. Moreover,
\begin{equation*}
\tr\big(f_{\xi_\tau}^{\bb G}|\pi_\tau\big)=(-1)^{q(\bb G_\tau)}.
\end{equation*}
This follows from the Vogan--Zuckerman classification of unitary cohomological representations; cf.~\cite[Lemma~2.7]{Shi12}.
\end{defi}


\begin{defi}\label{staieliaisieuels}
Given a central character datum $(\mfk X, \chi)$ for $\bb G^*$, we define the stably invariant distributions on $\mcl H(\bb G^*(\Ade_F), \chi^{-1})$ by
\begin{equation*}
\ST_{\ellip, \chi}^{\bb G^*}(f^*)\defining \tau(\bb G^*)\sum_{\gamma\in \Sigma_{\ellip, \mfk X}(\bb G^*)}\frac{1}{\#\paren{\pi_0(Z_{\bb G^*}(\gamma))}^{\Gal_F}}\SOrb_{\gamma,\chi}^{\bb G^*}(f^*),
\end{equation*}
where $\tau(\bb G^*)$ is the Tamagawa number of $\bb G^*$ and
\begin{equation*}
\SOrb_{\gamma, \chi}^{\bb G^*}(f^*)=\sum_{\gamma'}\Orb_{\gamma', \chi}^{\bb G^*}(f^*)
\end{equation*}
is the stable orbital integral at $\gamma$, with $\gamma'$ running through a set of representatives for the $F$-conjugacy classes inside the stable conjugacy class of $\gamma$.
\end{defi}

\begin{thm}\label{simpltraceofmrualoens}
Fix a central character datum $(Z(F\otimes\bb R), \chi)$, where $\chi$ is the inverse of the central character of an irreducible representation
\begin{equation*}
\xi=\boxtimes_{\tau\in\infPla_F}\xi_\tau
\end{equation*}
of $\bb G_{\bb C}$ with regular highest weight, restricted to $Z(F\otimes\bb R)$. Suppose $f\in \mcl H(\bb G(\Ade_F), \chi^{-1})$ satisfies
\begin{itemize}
\item
$f_\infty=\bigotimes_{\tau\in\infPla_F}f_{\xi_\tau}^{\bb G}$ is an Archimedean Lefschetz function attached to $\xi$;
\item
there exists a finite place $v$ of $F$ such that $f_v=f_{\Lef, v}^{\bb G}$ is a non-Archimedean Lefschetz function.
\end{itemize}
Then
\begin{equation*}
\ST_{\ellip, \chi}^{\bb G^*}(f^*)=\bx T_{\ellip, \chi}^{\bb G}(f)=\bx T_{\disc, \chi}^{\bb G}(f)=\bx T_{\cusp, \chi}^{\bb G}(f),
\end{equation*}
where $f^*\in \mcl H(\bb G^*(\Ade_F), \chi^{-1})$ is a transfer of $f$ to $\bb G^*$.
\end{thm}
\begin{proof}
This follows from \cite[Lemma 6.1, 6.2]{K-S23}.
\end{proof}

\subsection*{Acknowledgements}
I would like to thank Linus Hamann for his interest in this project and helpful discussions, and for teaching me the theory of local shtuka spaces. I am grateful to my Ph.D. advisor Wei Zhang for carefully reading an early draft and for many valuable suggestions, and also to David Hansen, Si Ying Lee, and Xu Shen for comments and discussions. Special thanks go to Rui Chen and Jialiang Zou for discussions concerning the local Langlands correspondence for even orthogonal groups, and to Shengkai Mao for pointing out the results on the affineness of partial minimal compactifications of Igusa varieties in his thesis. I also benefited from discussions with Weixiao Lu, Yu Luo, and Sian Nie. Finally, I thank the anonymous referee for a careful reading of the manuscript and for many helpful comments and suggestions.

\bibliographystyle{alpha}
\bibliography{bibliography}

\vspace{2em}
\noindent\textsc{Hao Peng}\\
Department of Mathematics, Massachusetts Institute of Technology, Cambridge, MA 02139, USA\\
\textit{Email}: \texttt{hao\_peng@mit.edu}

\end{document}